\documentclass[a4paper]{article}
\title{Symplectic wheelgebras and noncommutative geometry}
\author{David Fernández and Estanislao Herscovich}
\date{}

\usepackage{amsmath,amsthm,amsfonts,amssymb,stmaryrd}
\usepackage{enumitem}
\usepackage{latexsym}
\usepackage{titlesec}
\usepackage{mathrsfs}
\usepackage[bbgreekl]{mathbbol}
\usepackage{a4,xcolor,palatino,fancyhdr}
\usepackage[top=1in,bottom=1in,left=1.5in,right=1.5in]{geometry}
\colorlet{mycyan}{cyan!40!gray}
\colorlet{myblue}{blue!40!gray}
\colorlet{myred}{red!40!gray}
\colorlet{mygreen}{green!20!gray}
\definecolor{ultramarine}{RGB}{0,32,96}
\definecolor{light-gray}{gray}{0.8}
\colorlet{myultramarine}{ultramarine!20!gray}
\usepackage{graphicx}%poner [dvips] para gráficos
\usepackage{float}  %Este packete permite controlar la flotacion de figuras al poner [H] o [h] como opcion de \begin{figure}
\usepackage[small, bf, margin=90pt, tableposition=bottom]{caption}
\usepackage[utf8]{inputenc}
\usepackage[T1]{fontenc}
\usepackage{marginnote}
\usepackage[colorlinks=true,citecolor=mygreen, linkcolor=mygreen, urlcolor=myultramarine,breaklinks, naturalnames]{hyperref}
\usepackage[backrefs]{amsrefs}
\usepackage{longtable}

\usepackage{tikz}
\usepackage{tikz}
\usetikzlibrary{
	arrows,
	arrows.meta,
	cd,
	decorations, 
	decorations.markings,
	external,
	fit, 
	intersections,
	matrix,
	decorations.pathreplacing, 
	calc, 
	positioning,
	shapes.geometric,
	trees
}

\input{xy}
\xyoption{all}
\xyoption{poly}
\usepackage[all]{xy}
\usepackage{mathtools}
\mathtoolsset{showonlyrefs}

\numberwithin{equation}{section}
\numberwithin{table}{section}

\newtheorem{theorem}{Theorem}[section]
\newtheorem{proposition}[theorem]{Proposition}
\newtheorem{definition}[theorem]{Definition}
\newtheorem{lemma}[theorem]{Lemma}
\newtheorem{lemma-definition}[theorem]{Lemma/Definition}
\newtheorem{corollary}[theorem]{Corollary}
\newtheorem{remark}[theorem]{Remark}
\newtheorem{example}[theorem]{Example}
\newtheorem{fact}[theorem]{Fact}

\newtheorem{notation}[theorem]{Notation}

\DeclareFontFamily{U}{BOONDOX-calo}{\skewchar\font=45 }
\DeclareFontShape{U}{BOONDOX-calo}{m}{n}{
  <-> s*[1.05] BOONDOX-r-calo}{}
\DeclareFontShape{U}{BOONDOX-calo}{b}{n}{
  <-> s*[1.05] BOONDOX-b-calo}{}
\DeclareMathAlphabet{\mathcalboondox}{U}{BOONDOX-calo}{m}{n}
\SetMathAlphabet{\mathcalboondox}{bold}{U}{BOONDOX-calo}{b}{n}
\DeclareMathAlphabet{\mathbcalboondox}{U}{BOONDOX-calo}{b}{n}

\DeclareSymbolFontAlphabet{\mathbb}{AMSb}
\DeclareSymbolFontAlphabet{\mathbbl}{bbold}

\newcommand{\lr}[1]{% \llrrparen{..}
  \{\mkern-6mu\{#1\}\mkern-6mu\}}

\numberwithin{equation}{section}

\usepackage{longtable}
\usepackage{booktabs}
\usepackage{array}

\def\place{{-}}
\def\id{{\operatorname{id}}}
\def\derdifw{{\mathcalboondox{d}}}
\def\derdifwbul{{\mathcalboondox{D}}}
\def\iotaw{{\mathcalboondox{i}}}
\def\Liew{{\mathcalboondox{L}}}

\newcommand\ZZ{{\mathbb{Z}}}
\newcommand\NN{{\mathbb{N}}}
\newcommand\g{{\mathfrak{g}}}  %general Lie algebra
\newcommand{\kk}{{\Bbbk}}      %general field
\newcommand{\SG}{{\mathbb{S}}} %Symmetric group
\newcommand\EE{{\mathbb{E}}}   %endomorphism S-bimodule
\newcommand\TT{{\mathbb{T}}}   %endomorphism S-bimodule

\def\sump{{\mathbbl{s}}}
\def\block{{\mathbbl{b}}}
\def\linc{{\mathbbl{e}}}
\def\rinc{{\mathbbl{f}}}
\def\double{{\mathbbl{d}}}
\def\wad{{\mathcalboondox{a\!d}}}
\newcommand\omicron{\mathcalboondox{o}}
\newcommand\incl{\mathcalboondox{inc}}

\newcommand\cano{\mathcalboondox{can}}   %Canonical morphism from Der to Omega1 for wheelbimodules

\DeclareRobustCommand\ebseries{\fontseries{eb}\selectfont}
\DeclareTextFontCommand{\texteb}{\ebseries}

\def\Ker{{\operatorname{Ker}}}
\def\Img{{\operatorname{Img}}}
\def\Hom{{\operatorname{Hom}}}
\def\End{{\operatorname{End}}}  %endomophism of vector spaces
\def\Sym{{\operatorname{Sym}}}
\def\IHom{{\mathcal{H}om}}      %internal hom
\def\Diff{{\operatorname{Diff}}}          %diff op
\def\WDiff{{\mathcalboondox{Diff}}}      %wheel diff op
\def\IDer{{\mathcalboondox{Der}}}      %wheel derivations
\newcommand{\DDer}{\operatorname{\mathbb{D}er}}
\newcommand{\nc}{\operatorname{nc}}
\def\diff{{\Omega}^1_{\nc}}   % noncommutative 1-forms 
\def\diffn{{\Omega}^n_{\nc}}  % noncommutative n-forms
\def\diffb{{\Omega}^{\bullet}_{\nc}}  % noncommutative \bullet-forms
\newcommand{\Der}{\operatorname{Der}}
\newcommand{\du}{\operatorname{d}\!}%{d} % The differential
\newcommand{\DR}{\operatorname{DR}}

\newcommand{\out}{\operatorname{out}}
\newcommand{\Jac}{\operatorname{Jac}}

\def\Res{{\operatorname{Res}}}
\def\Ind{{\operatorname{Ind}}}
\def\inc{{\operatorname{inc}}}
\def\AD{{\operatorname{AD}}}
\def\Lie{{\operatorname{Lie}}}
\def\I{{\mathbb{I}}}
\def\IA{{\mathcalboondox{I\!a}}} 
\def\Fo{{\mathbbl{Fo}}}
\def\gFo{{\mathbbl{gFo}}}
\def\Bi{{\operatorname{Bi}}}
\def\Bil{{\operatorname{Bil}}}
\def\sh{{\operatorname{sh}}}

\def\Fock{{\mathcalboondox{F}}}
\def\WW{{\mathcalboondox{W}}}
\def\alg{{\operatorname{a}}}
\def\bmod{{\operatorname{bm}}}
\def\Spec{{\operatorname{Spec}}}
\def\ext{{\operatorname{ext}}}

\def\ord{{\mathtt{O}}} %Order or partition
\def\Part{{\mathtt{P}}} %Partitions 
\def\Set{{\mathtt{Set}}}
\def\len{{\mathtt{l}}}

\def\El{{\mathcalboondox{El}}}
\def\ZSE{{\operatorname{Sqz}}}
\def\Gen{{\mathbb{G}}} %generators as alg diagonal S-bimod

\def\Wi{{\mathcalboondox{i}}}

\def\Ec{{\mathcalboondox{E}}}

\def\Mod{{\operatorname{Mod}}}
\def\Vect{{\operatorname{Vect}}}
\def\SMod{{\mathbbl{Mod}_{\SG}}}
\def\DMod{{\mathbbl{DMod}_{\SG^\env}}}
\def\WhMod{{\mathcalboondox{M\!o\!d}}}
\def\GWhMod{{\mathcalboondox{G\!M\!o\!d}}}
\def\WMod{{\mathcalboondox{W}}}
\def\WAlg{{\mathcalboondox{Al\!g}}}
\def\Adm{{\mathcalboondox{Al\!g}}}
\def\CAdm{{\mathcalboondox{CAl\!g}}}
\def\AMod{{\mathcalboondox{M\!o\!d}}}
\def\pWMod{{\mathcalboondox{p\!W}}}
\def\GWMod{{\mathcalboondox{g\!W}}}
\def\Alg{{\operatorname{Alg}}}
\def\uAlg{{\operatorname{uAlg}}}
\def\CAlg{{\operatorname{CAlg}}}
\def\uCAlg{{\operatorname{uCAlg}}}
\def\aAlg{{\operatorname{aAlg}}}
\def\Ring{{\operatorname{Ring}}}
\def\aRing{{\operatorname{aRing}}}

\def\aRng{{\operatorname{aRng}}}
\def\catA{{\mathfrak{A}}}

\def\ab{{\operatorname{ab}}}
\def\cyc{{\operatorname{cyc}}}
\def\cycm{{\natural}}  %M_{\cycm} = M/[A,M]
\def\env{{\operatorname{e}}}
\def\op{{\operatorname{op}}}
\def\wh{{\operatorname{wh}}}
\def\gwh{{\mathcalboondox{g\!W}}}
\def\s{{\operatorname{s}}}
\def\gr{{\operatorname{gr}}}
\def\mod{{\operatorname{mod}}}

\newcommand\C{{\mathcal{C}}}
\newcommand\D{{\mathcal{D}}}
\def\CAlg{{\operatorname{CAlg}}}
\def\Alg{{\operatorname{Alg}}}
\def\Rep{{\operatorname{Rep}}}

\def\wA{{\mathscr{A}}}
\def\wC{{\mathscr{C}}}

\def\wM{{\mathcalboondox{M}}}
\def\wN{{\mathcalboondox{N}}}

\def\wcatA{{\mathcalboondox{G}}}

\newcommand{\sqcapi}{\mathchoice
  {\kern0.1em{\sqcap}\kern-0.4em\raisebox{-0ex}{\tiny i}\kern0.45em}% \displaystyle
  {\kern0.1em{\sqcap}\kern-0.4em\raisebox{-0ex}{\tiny i}\kern0.45em}% \textstyle
  {\kern0.1em{\sqcap}\kern-0.35em\raisebox{-0ex}{\tiny i}\kern0.25em}% \scriptstyle
  {\kern0.1em{\sqcap}\kern-0.35em\raisebox{-0ex}{\tiny i}\kern0.25em}% \scriptscriptstyle
  }

\newcommand{\sqcapj}{\mathchoice
  {\kern0.1em{\sqcap}\kern-0.4em\raisebox{-0ex}{\tiny j}\kern0.45em}% \displaystyle
  {\kern0.1em{\sqcap}\kern-0.4em\raisebox{-0ex}{\tiny j}\kern0.45em}% \textstyle
  {\kern0.1em{\sqcap}\kern-0.35em\raisebox{-0ex}{\tiny j}\kern0.25em}% \scriptstyle
  {\kern0.1em{\sqcap}\kern-0.35em\raisebox{-0ex}{\tiny j}\kern0.25em}% \scriptscriptstyle
  }

\newcommand{\sqcapv}{\mathchoice
  {\kern0.1em{\sqcap}\kern-0.54em\raisebox{-0ex}{\tiny $\bar{v}$}\kern0.3em}% \displaystyle
  {\kern0.1em{\sqcap}\kern-0.54em\raisebox{-0ex}{\tiny $\bar{v}$}\kern0.3em}% \textstyle
  {\kern0.1em{\sqcap}\kern-0.48em\raisebox{-0ex}{\tiny $\bar{v}$}\kern0.1em}% \scriptstyle
  {\kern0.1em{\sqcap}\kern-0.48em\raisebox{-0ex}{\tiny $\bar{v}$}\kern0.1em}% \scriptscriptstyle
  }

\newcommand{\sqcapu}{\mathchoice
  {\kern0.1em{\sqcap}\kern-0.53em\raisebox{-0ex}{\tiny $\bar{u}$}\kern0.5em}% \displaystyle
  {\kern0.1em{\sqcap}\kern-0.53em\raisebox{-0ex}{\tiny $\bar{u}$}\kern0.5em}% \textstyle
  {\kern0.1em{\sqcap}\kern-0.48em\raisebox{-0ex}{\tiny $\bar{u}$}\kern0.1em}% \scriptstyle
  {\kern0.1em{\sqcap}\kern-0.48em\raisebox{-0ex}{\tiny $\bar{u}$}\kern0.1em}% \scriptscriptstyle
  }

\newcommand{\sqcapone}{\mathchoice
  {\kern0em{\sqcap}\kern-0.5em\raisebox{-0ex}{\tiny $1$}\kern0.2em}% \displaystyle
  {\kern0em{\sqcap}\kern-0.5em\raisebox{-0ex}{\tiny $1$}\kern0.2em}% \textstyle
  {\kern0em{\sqcap}\kern-0.45em\raisebox{-0ex}{\tiny $1$}\kern0.1em}% \scriptstyle
  {\kern0em{\sqcap}\kern-0.45em\raisebox{-0ex}{\tiny $1$}\kern0.1em}% \scriptscriptstyle
  }

\newcommand{\sqcapze}{\mathchoice
  {\kern0em{\sqcap}\kern-0.5em\raisebox{-0ex}{\tiny $0$}\kern0.2em}% \displaystyle
  {\kern0em{\sqcap}\kern-0.5em\raisebox{-0ex}{\tiny $0$}\kern0.2em}% \textstyle
  {\kern0.1em{\sqcap}\kern-0.45em\raisebox{-0ex}{\tiny $0$}\kern0.1em}% \scriptstyle
  {\kern0.1em{\sqcap}\kern-0.45em\raisebox{-0ex}{\tiny $0$}\kern0.1em}% \scriptscriptstyle
  }

\titlespacing*{\section}
{0pt}{1.25ex plus 1ex minus .2ex}{1.25ex plus .2ex}

\newcommand{\ard}{\mathchoice
  {\kern0em\raisebox{-0.1ex}{\rotatebox[origin=c]{-90}{$\curvearrowright$}}\kern0.1em}% \displaystyle
  {\kern0em\raisebox{-0.1ex}{\rotatebox[origin=c]{-90}{$\curvearrowright$}}\kern0.1em}% \textstyle
  {\kern0em\raisebox{-0ex}{\rotatebox[origin=c]{-90}{$\curvearrowright$}}\kern0.2em}% \scriptstyle
  {\kern0em\raisebox{-0ex}{\rotatebox[origin=c]{-90}{$\curvearrowright$}}\kern0.2em}% \scriptscriptstyle
  }

\newcommand{\aru}{\mathchoice
  {\kern0em\raisebox{0.1ex}{\rotatebox[origin=c]{90}{$\curvearrowright$}}\kern0.0em}% \displaystyle
  {\kern0em\raisebox{0.1ex}{\rotatebox[origin=c]{90}{$\curvearrowright$}}\kern0.0em}% \textstyle
  {\kern0em\raisebox{-0ex}{\rotatebox[origin=c]{90}{$\curvearrowright$}}\kern0.2em}% \scriptstyle
  {\kern0em\raisebox{-0ex}{\rotatebox[origin=c]{90}{$\curvearrowright$}}\kern0.2em}% \scriptscriptstyle
  }

\usepackage{soul}

\def\whsqrt{\sqrt}

\usepackage{tikz}
\usepackage{tikz}
\usetikzlibrary{
	arrows,
	arrows.meta,
	cd,
	decorations, 
	decorations.markings,
	external,
	fit, 
	intersections,
	matrix,
	decorations.pathreplacing, 
	calc, 
	positioning,
	shapes.geometric,
	trees
}

\input{xy}
\xyoption{all}
\xyoption{poly}
\usepackage[all]{xy}
\begin{document}

\maketitle
\hrulefill
%%%%%%%%%%%%%%%%%%%%%%%%%%%%%%%%%%%%%%%%%%%%%%%%%%%%%%%%%%%%%%%%%%%%%%%%%%%%%%%%%%%%%%%%%%%%%%%%%%%%%%%%%%%%%%%%%%%%%%%%%%%%%%%%%%%%%%%%%%%%%%%%
\begin{abstract} 
In this article, we explore the following statement made by V. Ginzburg and T. Schedler in \cite{MR2734329}: ``\emph{an adequate framework for doing noncommutative differential geometry is provided by the notion of wheelspace. Wheelspaces form a symmetric monoidal category}''.
However, the category of wheelspaces turns out not to be monoidal. 
To address this, we introduce generalized wheelspaces, which do form a symmetric monoidal category and provide solid ground for the theory of wheelgebras.
To support their first claim, Ginzburg and Schedler defined Poisson (Fock) wheelgebras in connection with Van den Bergh's double Poisson algebras via the Fock functor. 
We provide strong evidence to their claim by introducing symplectic wheelgebras and prove that the Fock functor
sends smooth bisymplectic algebras, as defined by W. Crawley-Boevey, V. Ginzburg and P. Etingof, into our symplectic wheelgebras.
In the process, we develop a Cartan calculus adapted to this wheeled context.
Moreover, we present a wheeled version of the significant Van den Bergh functor, which facilitates a formalization of the Kontsevich--Rosenberg principle, bridging the noncommutative and commutative frameworks.
After establishing that the classical Van den Bergh functor factors through our wheeled version, we show that symplectic Fock wheelgebras naturally induce symplectic algebras on representation schemes.
\end{abstract}



\hrulefill

\tableofcontents

%%%%%%%%%%%%%%%%%%%%%%%%%%%%%%%%%%%%%%%%%%%%%%%%%%%%%%%%%%%%%%%%%%%%%%
\section{Introduction}

\noindent\textbf{A fundamental dichotomy.}
In many areas of mathematics, one often encounters a fundamental dichotomy: either working with simple definitions in a complicated setting or employing more intricate definitions in a simpler context. 
A classic illustration of this can be seen in the construction of the symplectic form on the moduli space of flat connections over a principal bundle on a compact Riemann surface. 
M. Atiyah and R. Bott achieved this structure by using familiar moment maps within the delicate framework of Banach manifolds. 
In contrast, A. Alekseev, A. Malkin, and E. Meinrenken approached the same problem through quasi-Hamiltonian manifolds and Lie group-valued moment maps, but in the more conventional setting of finite-dimensional spaces.

Another compelling example of this dichotomy arises in the fact that the tensor algebra $T_\kk V$ of a vector space $V$ over a field $\kk$ of characteristic zero can be viewed in two different ways: either as an associative algebra object in the symmetric monoidal category of vector spaces, or as a commutative algebra object in the symmetric monoidal category of $\SG$-modules---graded vector spaces endowed with actions of the symmetric group on their homogeneous components. 
Building on this, in \cite{MR2568415}, Thm. 2.8, T. Schedler extended this idea to provide a novel perspective on noncommutative Poisson geometry. 
Specifically, he showed that for a prime noncommutative associative algebra $A$, twisted Poisson structures on $T_\kk A$ that satisfy a Leibniz-like identity are equivalent to double Poisson structures on $A$. 
M. Van den Bergh had earlier defined in \cite{MR2425689} a double Poisson bracket on an associative algebra $A$ as a bilinear map $\lr{-,-}\colon A\otimes A\to A\otimes A$ that satisfies highly non-trivial adaptations of the usual axioms of symmetry, the Leibniz and the Jacobi identity. 
Twisted Poisson structures, by constrast, can be simply understood as Poisson algebra objects in the category of $\SG$-modules (see \S\ref{sec:Lie and Poisson algebras} and \S\ref{sec:algebraic-structures-diagonal-bimodules}).\\

\noindent\textbf{Wheelspaces.}
However, twisted commutative algebras are insufficient for the needs of noncommutative geometry, as tensor algebras over associative algebras (rather than over fields) play a key role (see, for example, \cite{MR2294224}, \S5, or \cite{MR2425689}, \S3).
As explained in \cite{MR2734329}, within the context of quiver path algebras, tensoring over an associative algebra can be interpreted as a gluing operation that connects heads and tails of various paths.
If we join the head and tail of the same path, we obtain a ``wheel''.
To introduce a suitable notion of differential operators in noncommutative geometry, V. Ginzburg and Schedler \cite{MR2734329} abstracted this operation, developing the  category of wheelspaces. 

In simple terms, a wheelspace is a diagonal $\SG$-bimodule---a sequence $\big(S(n)\big)_{n\in\mathbb{N}_0}$, where each $S(n)$ is a $\kk\SG_n$-bimodule for all $n$---further endowed with a trace/contraction operation ${}^n_jt_i: S(n)\to S(n-1)$ that satisfies suitable compatibility properties with the symmetric action (note that in \S\ref{subsection:wh-basic-2}, we provide an equivalent but simpler definition of wheelspaces).
For example, given a finite dimensional vector space $V$, a prototypical wheelspace in this article is $\big(\Hom(V^{\otimes n},V^{\otimes n})\big)_{n\in\mathbb{N}_0}$, endowed with the action that permutes tensor factors (see Example \ref{example:ten-1}), and contractions ${}^n_it_j\colon \Hom(V^{\otimes n},V^{\otimes n}) \to \Hom(V^{\otimes (n-1)},V^{\otimes (n-1)})$, given by natural evaluation (see Example \ref{example:end-3}).
Therefore, while diagonal $\SG$-bimodules can be visualized as boxes with an equal number of inputs and outputs (see the left picture of Figure \ref{fig:wheels}), with the symmetric action permuting them, a wheelspace introduces an additional layer of complexity through contractions: in the right picture of Figure \ref{fig:wheels}, for instance, we depict the contraction ${}^{\phantom{x-}n}_{n-1}t_2$ that joins the second input to the $(n-1)$-th output, forming a wheel.

%%%%%%%%%%%%%%%
\begin{figure}[H]
    \centering
\begin{tikzpicture}
\def\d{7};

    % Draw the left rectangle
    \draw[thick] (0,0) rectangle (4,2) node[pos=.5] {};
    
        % Draw the inputs above
    \foreach \i in {1,2,4,5} {
     \draw[thick, ->](4.5 - \i * 0.8,3)  -- (4.5 - \i * 0.8,2);}
     \foreach \i in {1,2} {
     \node at ( -0.3+\i * 0.8,3.2) {$\i$};
      %  \node at (-1, 4.8 - \i * 0.8) {In \i};
    }
    \node at ( -0.3+3 * 0.8,2.5) {$\hdots$};
    \node at ( -0.3+4 * 0.8,3.2) {$n-1$};
    \node at ( -0.3+5 * 0.8,3.2) {$n$};
    
        % Draw the outputs below
    \foreach \i in {1,2,4,5} {
       \draw[thick, ->] (4.5 - \i * 0.8, 0) -- ( 4.5 - \i * 0.8, -1);}
    \foreach \i in {1,2} {   
       \node at ( -0.3+\i * 0.8,-1.3) {$\i$};
    }
    \node at ( -0.3+3 * 0.8,-0.5) {$\hdots$};    
    \node at ( -0.3+4 * 0.8,-1.3) {$n-1$};
    \node at ( -0.3+5 * 0.8,-1.3) {$n$};

    % Draw the right rectangle
    \draw[thick] (0+\d,0) rectangle (4+\d,2) node[pos=.5] {};  

    % Draw the inputs above
    \foreach \i in {1,2,5} {
     \draw[thick, ->](\d + 4.5 - \i * 0.8,3)  -- (\d + 4.5 - \i * 0.8,2);}
    \draw[thick, -](\d + 4.5 - 4 * 0.8,3)  -- (\d + 4.5 - 4 * 0.8,2);
     \node at (\d -0.3+ 0.8,3.2) {$1$};
     \node (A) at (\d -0.5+ 2 * 0.8,3.2) {$2$};     
      %  \node at (-1, 4.8 - \i * 0.8) {In \i};
    \node at (\d -0.3+3 * 0.8,2.5) {$\hdots$};
    \node at (\d -0.3+4 * 0.8,3.2) {$n-1$};
    \node at (\d -0.3+5 * 0.8,3.2) {$n$};
    
        % Draw the outputs below
    \foreach \i in {1,4,5} {
       \draw[thick, ->] (\d + 4.5 - \i * 0.8, 0) -- ( \d + 4.5 - \i * 0.8, -1);}
    \draw[thick, -](\d + 4.5 - 2 * 0.8,0)  -- (\d + 4.5 - 2 * 0.8,-1);   
    \foreach \i in {1,2} {   
       \node at (\d -0.3+\i * 0.8,-1.3) {$\i$};
    }
    \node at (\d -0.3+3 * 0.8,-0.5) {$\hdots$};    
    \node (B) at (\d -0.8+4 * 0.8,-1.3) {$n-1$};
    \node at (\d -0.3+5 * 0.8,-1.3) {$n$};

    \draw[thick,-%, color=myultramarine
    ] (\d + 2.9, -1) 
    to[out=-90,in=-90, distance=1.8cm] (\d + 6, -1);
    \draw[thick,-%, color=myultramarine
    ] (\d + 1.3, 3) 
    to[out=90,in=90, distance=3cm] (\d + 6, 3);
    \draw[thick,-%, color=myultramarine
    ] (\d + 6, -1) 
    to (\d + 6, 3);
%    \draw[thick,-] (\d + 2.9, -1) arc[start angle=-110, end angle=120, radius=2.5] -- (\d+1.3, 3);
   \end{tikzpicture}
     \captionsetup{width=.8\linewidth}
    \caption{Left: Pictorial representation of a diagonal $\SG$-bimodule. Right: Pictorial representation of a wheelspace with trace/contraction operation ${}^{\phantom{x-}n}_{n-1}t_2$.}
    \label{fig:wheels}
\end{figure}
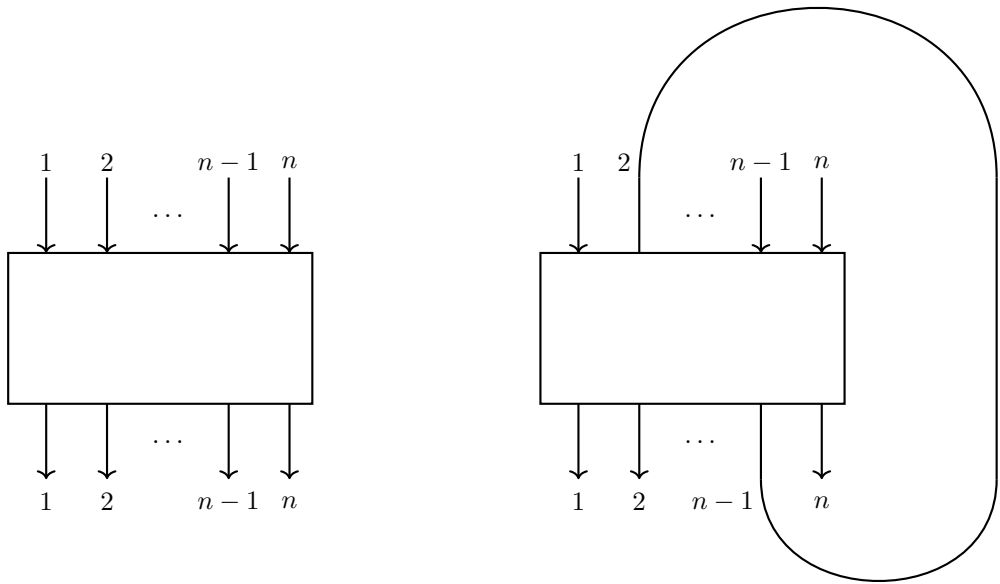       

In \cite{MR2734329}, Rk. 3.1.9, Ginzburg and Schedler noted that commutative wheelgebras are, in fact, special cases of (nonunital) wheeled PROPs---PROPs in which every object has a dual. 
Wheeled PROPs were first introduced by M. Markl, S. Merkulov and S. Shadrin in \cite{MR2483835} to address a key limitation of ordinary operads and PROPs, which lack of place to encode traces and the master equation in physics that involves a divergence operator.
In recent years, wheeled PROPs have experienced a resurgence, appearing in a wide range of areas: from welded tangles in low-dimensional topology and the Kashiwara--Vergne problem \cites{MR4577332, MR4265709}, 
to invariant theory \cite{MR4568126}, deformation quantization \cites{MR2641194, MR4436207}, amplitudes in quantum field theory \cite{MR4507248}, and even in the context of the stochastic heat equation \cite{MR4322389}.

The goal of this article is to delve into the following statement made in the introduction of \cite{MR2734329}: ``\emph{an adequate framework for doing noncommutative differential geometry is provided by the notion of wheelspace. Wheelspaces form a symmetric monoidal category}''.\\

\noindent\textbf{(Generalized) wheelspaces and their monoidal structure.} 
As noted in \cite{arXiv:2211.11370}, \S B.1 (see also \cite{MR4568126}, \S2), the vertical composition of wheeled PROPs can be visualized by stacking the outputs of one diagram onto the inputs of another, while the tensor product (horizontal composition) of the PROP is represented by placing the diagrams side by side. 
Since commutative wheelgebras turn out to be special cases of wheeled PROPs, one might expect the category of wheelspaces $ \WMod$ to be monoidal, as claimed in \cite{MR2734329}, pp. 690--692 (see especially Def. 3.1.6), with a straightforwardly defined tensor product. 
However, in Proposition \ref{proposition:contradiction} we prove that the category of wheelspaces $\WMod$ is \underline{not} monoidal.
This is due to the fact that, for any two wheelspaces $S$ and $S'$, the natural bifunctor $\Bil_{\gwh}(S,S';-)|_{\WMod}$ of bilinear maps defined in \eqref{index:functor-restriction-gen-wheelsp-wheelsp} is not representable.

To overcome this critical issue, we were compelled to broaden our framework and introduce the concept of \emph{generalized} wheelspaces. 
Given $\bar{n} = (n_{1},\dots,n_{\ell})$ in $\Part(n)$, the set of all partitions of $n\in\mathbb{N}$, a generalized wheelspace is defined as a diagonal $\SG$-bimodule $S=\big( S(n) \big)_{n\in\mathbb{N}_0}$, together with a distinguished decomposition of $\SG_n$-bimodules $S(n) = \bigoplus_{\bar{n} \in \Part(n)} \Ind_{\Bbbk\sump_{\bar{n}}^{\env}}\big(S(n)_{\bar{n}}\big)$, and a family of contractions ${}^{\bar{n}}_{j}t_{i} $ 
(see Definition \ref{definition:gwh}). 
After defining their morphisms, we consider the category $\GWMod$ of generalized wheelspaces, which we equip with a tensor product $\otimes_{\gwh}$ and the same unit as the category of diagonal $\SG$-bimodules, $\DMod$.
This turns $\GWMod$ into a symmetric monoidal category, where the forgetful functor $\gFo : \GWMod \to \DMod$ is braided strong monoidal. 
This leads to Proposition \ref{proposition:gfo}, which forms the first main result of this article.\\

\noindent\textbf{Poisson wheelgebras and noncommutative (Poisson) geometry.} 
Strictly speaking, the adequate framework for noncommutative differential geometry is better captured by Fock algebras rather than wheelspaces. 
Introduced by Ginzburg and Schedler, the Fock algebra $\Fock(B)$ for any associative algebra $B$ is a commutative wheelgebra, analogous to the Fock space construction. 
In the context of wheeled PROPs, having $m$ inputs and $n$ outputs corresponds to a $B^{\otimes m}$-$B^{\otimes n}$-bimodule structure. 
As written in \cite{MR2734329}, p. 693, $\Fock(B)$ is the commutative wheelgebra freely generated by $B$ in degree 1, with the condition that the contraction operation $B\otimes B\to B$ coincides with the multiplication on $B$. However, the authors did not characterize this construction via a universal property. 

Let $ \CAdm $ denote the category of commutative wheelgebras.
In Theorem \ref{theorem:fock-adjoint}, we prove that the Fock construction, defined by the functor $ \Fock : \Alg \to \CAdm$, which sends a nonunitary associative algebra $B$ to the Fock wheelgebra $\Fock(B)$, is left adjoint to a natural functor from the category of commutative wheelgebras to the category of algebras. 
To define this functor, we unveil a key technical property that all wheelgebras satisfy, which we call \emph{admissibility} (see Definition \ref{definition:admissible-wh}).
As explained in Remark \ref{rem:symmetric-Fock-alg}, this result also implies that if $A$ is a tensor algebra, $\Fock(A)$ can be expressed as an appropriate symmetric algebra.
Furthermore, in Lemma \ref{lemma:fock-adjoint}, we extend this perspective based on universal properties from associative algebras to bimodules via the functor $ \WW\colon {}_B \Mod_B\to {}_{\Fock(B)}\AMod$, which maps $B$-bimodules to $\Fock(B)$-wheelmodules by setting $\WW(M)=\Fock(T_BM)_1$; see \eqref{index:functor-WW}. 
Methodologically, throughout this article we adopt a categorical approach, leveraging seven key functors to establish a range of significant results. In addition to the aforementioned bifunctor $\Bil_{\gwh}(S,S';-)|_{\WMod}$ and the functors $\Fock$ and $\WW$, we highlight the critical role of the following functors: the wheeled $k$-shifted functor $  \sh_{k}^{\wh}$ \eqref{eq:shifted-w-mod}, $\alg$ \eqref{eq:wheel-alg}, $\bmod$ \eqref{eq:wheel-mod},
and $\IDer(\wC, -)$ (\S\ref{sec:Cartan calculus for commutative wheelgebras}).

In \cite{MR2734329}, Ginzburg and Schedler assert that \emph{``an adequate framework for doing noncommutative differential geometry is provided by the notion of wheelspace''}, a statement largely based on their Thm. 3.6.7. Specifically, when $A$ is a smooth algebra and $\DDer A:=\Der(A, A\otimes A)$ denotes the $A$-bimodule of double derivations, the associated graded algebra of wheeled differential operators on $\Fock(A)$, $\operatorname{gr}\WDiff\big(\Fock(A)\big)$, is isomorphic to $\Fock\big(T_A(\DDer A)\big)$, with  the induced Poisson wheelgebra structure closely related to Van den Bergh's double Schouten--Nijenhuis bracket. 
Notably, in Rk. 3.5.19, they observe that \emph{``it is not true that a wheeled Poisson bracket on $\Fock(A)$ is equivalent to a double bracket on $A$. Wheeled Poisson brackets are a generalization of double Poisson brackets''}.
In fact, every double Poisson bracket on an algebra $A$ induces a Poisson wheelgebra structure on $\Fock(A)$.
Unfortunately, they did not provide a proof of this crucial result. 
We address this gap in Proposition \ref{prop:wheel-Poisson-double} by levering the Fock functor $\Fock$ and Lemma \ref{fact:wh-poisson-algebra-ex}:
it suffices to show that the bracket induced by the forgetful functor $\Fo : \WMod \to \DMod$, which is a Poisson bracket in the symmetric monoidal category of diagonal $\SG$-bimodules, becomes a morphism of \emph{generalized} wheelspaces.\\

\noindent\textbf{Symplectic wheelgebras and noncommutative symplectic geometry.} 
Given a smooth associative algebra $B$, the tensor algebra $T^*B:=T_B(\DDer B)$ can be interpreted as the noncommutative cotangent bundle.
One of the main results in \cite{MR2294224} establishes that if $B$ is a smooth associative algebra, then the algebra $T^*B$ possesses a canonical bisymplectic form: a noncommutative 2-form $\omega$ for which the reduced contraction operator $\iota(\omega)\colon \DDer(T^*B)\to \Omega^1_{\nc}(T^*B)$ becomes an isomorphism of $T^*B$-bimodules (see Definition \ref{def:bisymplectic-algebras}).

As a consequence, while Ginzburg and Schedler have developed Poisson geometry within the wheeled framework, it is natural to explore the development of wheeled symplectic geometry.
One of the main goals of this article is to introduce symplectic wheelgebras and establish clear grounds to study their features. 
We expect that this new area may provide a symplectic upgrade to the aforementioned isomorphism between $\operatorname{gr}\WDiff\big(\Fock(A)\big)$ and $\Fock\big(T^*A\big)$, potentially yielding to interesting applications in deformation quantization.
This article represents an initial step toward that goal.

To formulate an appropriate definition of a symplectic structure in any new setting, we must address the following key tasks: define the analogue of the space of K\"ahler differentials forms, specify the dg algebra of differential forms along with a closedness condition, and establish an (iso)morphism between the space of derivations and that of K\"ahler differential forms. 
In the context of the wheeled setting, and to introduce the main new object of this article---symplectic wheelgebras (Definition \ref{def:symplectic-wheelgebra-proto})---we detail in \S\ref{sec:Cartan calculus for commutative wheelgebras} in which sense a commutative wheelgebra $\wC$ admits a wheelmodule of K\"ahler wheeldifferential forms, it admits a dg wheelgebra of wheeldifferential forms, and $\wC$ has wheelcontractions. The central element is the introduction of the functor $ \IDer(\wC, -) :{}_{\wC}\AMod \rightarrow {}_{\wC}\AMod$, which is representable: the representing object is the $\wC$-wheelmodule $\Omega^{1}_{\wh}\wC$ (see \eqref{eq:wheel-der-rep}) that comes equipped with the universal derivation $\derdifw_{\wC}$.

At the level of generality of commutative wheelgebras that admit K\"ahler wheeldifferential forms and have wheelcontractions, we introduce Lie derivatives via Cartan's magical formula. This allows us to prove a remarkable Cartan calculus in Proposition \ref{lemma:cartan-calculus}, where the occurring brackets are (graded) commutators.
It is worthwhile to compare this result with Thm. 4.4 in \cite{MR4640979}, which presents a Cartan calculus in the context of double derivations and Van den Bergh's double Schouten--Nijenhuis brackets---and where some rather unexpected terms appear. This provides yet another instance of the dichotomy highlighted in the first paragraph of this introduction.

We can specialize the new concept of a symplectic wheelgebra structure by considering the case where the commutative wheelgebra $\wC$ is $\Fock(B)$, the Fock algebra of an associative algebra $B$.
Building on the relationship between double Poisson algebras and Poisson wheelgebras, it is natural to ask whether there exists a link between 
the bisymplectic algebras of \cite{MR2294224} and
symplectic (Fock) wheelgebras as introduced in Definition \ref{def:symplectic-wheelgebra-proto}.
The second main result of this article, presented in Theorem \ref{theorem:bisymplectic-gives-wheel-symp}, confirms this connection---if $B$ is a \emph{smooth} bisymplectic algebra, then $\Fock(B)$ is a symplectic wheelgebra.
This result provides strong evidence to the claim that an adequate framework for doing noncommutative differential geometry is
provided by the notion of Fock algebras.
Moreover, a significant consequence of this theorem is that the crucial $\Fock(T^*B)$ becomes a symplectic wheelgebra.

To prove Theorem \ref{theorem:bisymplectic-gives-wheel-symp}, we begin by showing that $\Fock(B)$ admits a wheelmodule of K\"ahler wheeldifferential forms (Proposition \ref{prop:Fock-has-wheeldiff-forms}), a dg wheelgebra of K\"ahler wheeldifferential forms (Lemma \ref{lemma:dgwheel-fock}), and it has wheelcontractions (Lemma \ref{lemma:cont-wheel-fock}). Consequently, Fock algebras are equipped with a Cartan calculus.
Furthermore, since the above-mentioned functor $\WW\colon {}_A \Mod_B\to {}_{\Fock(B)}\AMod$ gives rise to a connection between $B$-bimodules and $\Fock(B)$-wheelmodules, the core idea of the proof of Theorem \ref{theorem:bisymplectic-gives-wheel-symp} is to apply this functor $\WW$ to the definition of bisymplectic algebras. Hence, we obtain the $\Fock(B)$-wheelmodules $\WW\!\DDer B$ and $\WW \diff B$, which are isomorphic to $\Omega_{\wh}^{1}\Fock(B)$ and $ \IDer(\Fock(B))$, respectively (see Propositions \ref{prop:Fock-has-wheeldiff-forms} and 
\ref{lemma:wwderdoble-wder-iso}).
Finally, to realize the isomorphism $\cano_{\varpi}$ between $\Omega_{\wh}^{1}\Fock(B)$ and $ \IDer(\Fock(B))$ in \eqref{eq:wheelcont}, we require a noncommutative version of the big bracket introduced by B. Kostant and S. Sternberg in the context of the classical BRST theory, later employed by P. Lecomte and C. Roger in the theory of Lie bialgebras. If $\mathbb{E}=\DDer B\oplus \Omega^1_{\nc}B$, the double big bracket $\lr{-,-}^{\imath}$ is a double Poisson bracket of degree -2 on $T_B\mathbb{E}$, which essentially pairs a double derivation with a noncommutative 1-form through contraction (see \eqref{eq:double-big-bracket}). By Proposition \ref{prop:wheel-Poisson-double}, which connects double Poisson algebras and Poisson wheelgebras, the double big bracket$\lr{-,-}^{\imath}$ induces the \emph{wheeled} big bracket $\{-,-\}^{\imath}_{n,m}\colon\Fock(T_B\mathbb{E})(n)\otimes \Fock(T_B\mathbb{E})(m)\rightarrow \Fock(T_B\mathbb{E})(n+m)$, which realizes the isomorphism $\iota(\omega)$ under the functor $\WW$, where $\omega$ stands for the given bisymplectic form on $B$.  
We expect that both the newly double big brackets and wheeled big brackets will be of independent interest and find further applications.\\

\noindent\textbf{The Kontsevich--Rosenberg principle within the wheeled framework.} 
If wheelspaces (or more specifically, Fock algebras) provide the adequate framework for doing noncommutative differential geometry, it is natural to explore how this wheeled theory connects to usual commutative geometry. 
In the context of noncommutative geometry based on an associative algebra $A$, the guiding principle is the Kontsevich--Rosenberg principle \cite{MR1731635}: any noncommutative algebro-geometric structure on $A$ should induce the corresponding usual commutative algebro-geometric structure on the representation affine scheme $\Rep(A,V)$ of $A$ for any finite dimensional vector space $V$. 
Building on the works of G. Bergman \cite{MR357503} and P. Cohn \cite{MR555546}, a convenient approach is to introduce the representable functor $ \Rep_V A\colon \CAlg\rightarrow \Set $, whose representing object is the commutative algebra $A_V$. Therefore, we have $\Rep(A,V)=\Spec(A_V)$.
Notably, the algebra $A_V$ can be explicitly described in terms of generators and relations, as it is the abelianization of the associative algebra $\sqrt[V]{A}$.
In the enlightening articles \cites{Khachatryan, arXiv:1010.4901, MR3084440,MR3204869}, their authors leveraged of a functorial approach based on the introduction of the two functors $\sqrt[V]{-}\colon \Alg \rightarrow \Alg$ and $(-)_V\colon \Alg \rightarrow \CAlg$; see  \eqref{eq:root-VdB-def-BKR}.

Building on \cites{Khachatryan, arXiv:1010.4901, MR3084440,MR3204869}, in \S\ref{sec:wheeled-KR}, we develop a functorial approach within the wheeled framework by introducing the functor $\whsqrt[V,\wh]{\place} : \Adm \rightarrow \Alg$, which is the wheeled counterpart of the functor $\sqrt[V]{-}$.
In fact,  if $\wA$ is a wheelgebra, we conveniently express $\whsqrt[V,\wh]{\wA}$ as the quotient of the tensor algebra $T=T\big(\oplus_{n\in\mathbb{N}_0}\wA(n)\otimes\End(V^{\otimes n})\big)$ by an ideal generated by the relations \ref{item:wa1}--\ref{item:wa4}.
As a result, we can define the functor $ (\place)_{V,\wh} : \Adm \rightarrow \CAlg$ as the composition of the functor $\sqrt[V,\wh]{\place}$ with the abelianization functor. Both $ (\place)_{V,\wh}$ and $\sqrt[V,\wh]{\place}$ satisfy appropriate adjunctions, as shown in Corollary \ref{corollary:functor-alg-whalg} (to be compared with Proposition \ref{prop:BKR-representability-Rep}).

Nevertheless, the systematic realization of the Kontsevich--Rosenberg is achieved through the so-called Van den Bergh functor $(-)_V\colon {}_{A}\Mod_{A} \rightarrow\Mod_{A_V} $, which transforms $A$-bimodules into $A_V$-modules. Prior to its introduction in \cite{MR2397630}, various instances of the Kontsevich--Rosenberg principle appeared somewhat \emph{ad hoc}. However, this additive functor provides a unified approach; for example, $(\diff A)_V=\Omega^1_{A_V}$, and $(\DDer A)_V\cong \Der(A_V)$ when $A$ is a smooth associative algebra. Interestingly, the authors of \cite{MR3084440} described the Van den Bergh functor as the composition of the functor $\sqrt[V]{\place}\colon {}_A\Mod_{A}\rightarrow {}_{\sqrt[V]{A}}\Mod_{\sqrt[V]{A}}$ in \eqref{eq:VdB-root-functor-def.a} and the abelianization functor \eqref{eq:abel-mod}. 
In this article, we extend this approach to the wheeled setting by introducing the `wheeled Van den Bergh functor' $(\place)_{V,\wh}$ in \eqref{eq:root-wh-mod-comm} . This functor is defined as the composition of the functors $\whsqrt[V,\wh]{\place} : {}_{\wA}\AMod_{\wA} \rightarrow {}_{\whsqrt[V,\wh]{\wA}}\Mod_{\whsqrt[V,\wh]{\wA}}$ (see \eqref{eq:root-wh-mod-adj}) and the abelianization functor. 

A key theme of this article is that the wheeled framework is shaped by the functors $\Fock$ in \eqref{eq:functor-fock} and $\WW$ in \eqref{index:functor-WW}. Naturally, this raises the question of how these functors interact with the four introduced in the preceding two paragraphs. This interaction is established in Corollaries \ref{corollary:functor-alg-whalg-2} and \ref{corollary:functor-alg-whalg-mod-2}. Given an associative algebra $A$ and an $A$-bimodule $N$, we prove the following isomorphisms: $\whsqrt[V,\wh]{\Fock(A)} \cong \whsqrt[V]{A}$, $\Fock(A)_{V,\wh} \cong A_{V}$, $\whsqrt[V,\wh]{\WW(N)} \cong \whsqrt[V]{N}$ and $\WW(N)_{V,\wh} \cong N_{V}$---therefore, $(\WW\!\DDer A)_{V,\wh}\cong (\DDer A)_V\cong \Der(A_V)$ and $(\WW \diff A)_{V,\wh}\cong (\diff A)_V\cong \Omega^1_{A_V}$.
As a result, these corollaries show that the functors $(\place)_V\colon  \Alg \rightarrow \Alg$ and the usual Van den Bergh functor $(-)_V\colon {}_{
B}\Mod_{B} \rightarrow\Mod_{B_V} $ factor through their wheeled counterparts.

Building on Corollaries \ref{corollary:functor-alg-whalg-2} and \ref{corollary:functor-alg-whalg-mod-2}, a natural question arises: how does the wheeled Van den Bergh functor $(\place)_{V,\wh}$ interact with symplectic Fock wheelgebras? One would expect that $(\place)_{V,\wh}$ maps these to symplectic algebras on representation schemes, suggesting a compelling  `wheeled Kontsevich--Rosenberg principle'.
The final main result of this article is Theorem \ref{theorem:w-symp-K-R}, where we prove that if $B$ is a smooth associative algebra and $\Fock(B)$ carries a wheelsymplectic form, then $\Fock(B)_{V,\wh}$ possesses a symplectic form.
Note that this result is more general than the combination of Theorem \ref{theorem:bisymplectic-gives-wheel-symp}, Corollary \ref{corollary:functor-alg-whalg-mod-2} and the Kontsevich--Rosenberg principle for smooth bisymplectic algebras (\cite{MR2294224}), as it does not require that the wheelsymplectic form to be induced from a bisymplectic form. \\

\noindent\textbf{Differential operators on associative algebras.}
As written at the beginning of the introduction, Ginzburg and Schedler's primary motivation in \cite{MR2734329} for working within the wheeled framework was to develop a new notion of differential operators for associative algebras, as the classical Grothendieck definition does not always provide a satisfactory notion (see their Rks. 2.1.3 and 2.1.9). 
To address this issue, they introduce a definition of differential operator on a commutative algebra object in an arbitrary symmetric monoidal category, which they then apply to the category of wheelspaces. 
However, one of the central themes of our article is the necessity of generalized wheelspaces because the category of wheelspaces itself is not monoidal. 
Moreover, in Thm. 6.1.8, (iii), of \cite{MR2734329}, the authors claim that there is an isomorphism of algebras
\[    
 \operatorname{RA}_{V}\Big(\WDiff(\Fock(A))\Big) \overset{\cong}{\longrightarrow} \Diff(A_{V}),    
 \] 
where $\WDiff(\Fock(A))$ denotes the wheelgebra of differential operators of the Fock wheelgebra associated to a smooth associative algebra $A$. 
However, the noncommutative algebra $\operatorname{RA}_{V}(\wA)$ associated to a wheelgebra $\wA$, which we denote as $\sqrt[V,\wh]{\wA}$,
is not almost commutative in general. This contrasts with $\Diff(A_{V})$, the algebra of differential operators of $A_{V}$. 
For example, when $\wA = \Fock(\Bbbk[x])$, Corollary \ref{corollary:functor-alg-whalg-2} (\textit{cf.} \cite{MR2734329}, Thm. 6.1.8, (i)) gives $\operatorname{RA}_{V}(\wA) \cong \sqrt[V]{\Bbbk[x]} \cong \Bbbk \langle x_{i,j} : 1 \leq i,j \leq \dim V\rangle$, which is not almost commutative if $\dim(V) > 1$, since it has exponential growth in that case. 
To overcome these two important issues, we propose in Example \ref{example:wheeldiff} a new definition of differential operators at the level of a commutative wheelgebra $\wC$ with values in a $\wC$-wheelmodule  $\wN$, which we denote as $\WDiff(\wC,\wN)$.
It uses the map $\wad_{c}(f) \in \IHom_{\WMod}(\wM,\wM')(n+p)$, as defined in \eqref{eq:wad}, so it is a wheelspace and even a wheelgebra.
We defer the study of the main properties of $\WDiff(\wC,\wN)$ and its relation to Grothendieck's differential operators via the wheeled Kontsevich--Rosenberg principle to a future article.\\

\noindent\textbf{Layout of the article.}
Section \ref{section:preliminaries} introduces the foundational notions of algebraic structures in symmetric monoidal categories that we will use, along the basic constructions for symmetric groups. 
Section \ref{subsection:s-mod} reviews the basic theory of $\SG$-modules and diagonal $\SG$-bimodules, which serve as the ground for addressing (generalized) wheelspaces.
These sections are mainly intended to establish the notation and terminology used throughout the article, as they do not contain any new results.

Section \ref{section:wheelspaces} presents the theory of generalized wheelspaces, which form a symmetric monoidal category that includes the wheelspaces considered in \cite{MR2734329} (see Proposition \ref{proposition:gfo}). 
As previously noted, contrary to the claims in \cite{MR2734329}, there is no monoidal structure on the category of wheelspaces representing the bifunctor Ginzburg and Schedler are interested in (see Proposition \ref{proposition:contradiction}). This prompted the introduction of generalized wheelspaces. 
Notably, in Subsection \ref{subsec:wheeldif-forms}, we introduce several algebraic constructions within the category of wheelspaces that have geometric flavor and were not considered in \cite{MR2734329}, such as the existence of a Cartan calculus (see Proposition \ref{lemma:cartan-calculus}) and symplectic wheelgebras (see Definition \ref{def:symplectic-wheelgebra-proto}). 

In Section \ref{section:main-ex-Fock-alg}, we provide a very precise exposition of the Fock wheelgebra associated with a nonunitary algebra, which was originally introduced in \cite{MR2734329}. 
Our main result is Theorem \ref{theorem:fock-adjoint}, which gives a universal characterization of the Fock functor. From this, we derive several key properties of Fock wheelgebras (see Corollaries \ref{corollary:fock-adjoint} and \ref{corollary:fock-adjoint-2}) and establish the universal property of the Fock functor for bimodules (see Lemma \ref{lemma:fock-adjoint}). 
Interestingly, in the process of identifying the universal characterization of the Fock functor, we uncover a natural property verified by all wheelgebras, which we call admissibility (see Lemmas \ref{lemma:wheel-adm} and \ref{lemma:wheelbimod-adm})---this property had not been noticed in \cite{MR2734329}. 

In Section \ref{section:fock-geometry}, we show that various noncommutative notions for algebras are transformed into their commutative counterparts on Fock wheelgebras under the Fock functor.
More precisely, we begin by outlining the basics on noncommutative geometry in Subsection \ref{subsec:bisymplectic-algebras}, and then in Subsection \ref{sec:double-Poisson-and-wheelgebras}, we provide a complete proof that the Fock wheelgebra of a double Poisson algebra is naturally equipped with a Poisson structure (see Proposition \ref{prop:wheel-Poisson-double}); this result was stated without proof in \cite{MR2734329}, Def. 3.5.17. 
In the subsequent two subsections, we prove that Fock wheelgebras possess a Cartan calculus (see Lemma \ref{lemma:cont-wheel-fock}) and that the Fock wheelgebra of a smooth bisymplectic algebra has a symplectic wheelgebra structure (see Theorem \ref{theorem:bisymplectic-gives-wheel-symp}). 

In the final section, after reviewing the basic properties of the Van den Bergh functor for algebras and bimodules, along the Kontsevich--Rosenberg principle for bisymplectic algebras in Subsection \ref{sec:KR-after-BKR}, we proceed to introduce the Van den Bergh functor on wheelgebras and wheelbimodules in  Subsection \ref{sec:wheeled-KR}, and we show that the usual Van den Bergh functors for algebras and bimodules factor through their wheeled counterparts. 
The section culminates with Theorem \ref{theorem:w-symp-K-R}, which establishes that our definition of a symplectic wheelgebra satisfies the Kontsevich--Rosenberg principle for any Fock wheelgebra associated with a smooth associative  algebra.

The appendix provides a glossary of notation, which we believe can be useful for readers given the cornucopia and complexity of the objects considered throughout the article. \\

\noindent\textbf{Acknowledgments.}
We would like to thank Travis Schedler for useful email correspondence. We are also grateful to the authors of \cite{MR4507248} for discussing several results of their work with us, and to Leonid Ryvkin, whose question during a seminar given by the second author in Lyon prompted this project. The first author also thanks Luis Álvarez-Cónsul for bringing \cite{MR2734329} to his attention many years ago, and for his continued interest in this project.

%%%%%%%%%%%%%%%%%%%%%%%%%%%%%%%%%%%%%%%%%%%%%%%%%%%%%%%%%%%%%%%%%%%%%%
\section{Preliminaries}
\label{section:preliminaries}
\allowdisplaybreaks

%%%%%%%%%%%%%%%%%%%%%%%%%%%%%%%%%%%%%%%%%%%%%%%%%%%%%%%%%%%%%%%%%%%%%%
\subsection{Notation and conventions}
\label{subsection:not-conv} 

We denote by $\NN_{0} \label{eq:NN_0}$ (resp., $\NN \label{eq:NN}$) the set of nonnegative (resp., positive) integers $\{ 0, 1, 2, \dots \}$ (resp., $\{ 1, 2, \dots \}$), $\ZZ$ the set of all integers $\{ 0, \pm 1, \pm 2, \dots  \}$ and given $a, b \in \ZZ$ we denote by $\llbracket a , b \rrbracket \label{eq:interval-integers}$ the set $\{ n \in \ZZ : a \leq n \leq b \}$. 
If $X$ is a set we denote by $\label{eq:set-tuples} X^{(\NN)} = \sqcup_{\ell \in \NN} X^{\ell}$ the set of all finite tuples of elements of $X$. 

Given $\ell \in \NN$ and $\label{eq:tuple}\bar{n} = (n_{1},\dots,n_{\ell}) \in \NN_{0}^{\ell}$, set $ \label{eq:length-tuple} |\bar{n}| = \sum_{i=1}^{\ell} n_{i}$. 
Moreover, given $\ell \in \NN$ and $p \in \llbracket 1 , \ell \rrbracket$, we will denote by $\label{eq:vector-Kronecker} \bar{e}^{p,\ell} \in \NN_{0}^{\ell}$ the unique vector satisfying that $|\bar{e}^{p,\ell}| = 1$ and such that all of its components vanish except for the $p$-th component. 
We will simply denote $\bar{e}^{p,\ell}$ by $\bar{e}^{p}$ if $\ell$ is clear from the context. 

Given $n \in \NN_{0}$, recall that an \textbf{\textcolor{myblue}{ordered partition}} of $n$, sometimes called a \textbf{\textcolor{myblue}{composition}}, is a tuple $\bar{n} = (n_{1}, \dots, n_{\ell}) \in \NN_{0}^{\ell}$ with $\ell \in \NN$ such that $n_{1} + \dots + n_{\ell} = n$. 
Since all the partitions we will consider are ordered, we will refer to them simply as partitions. 
We call $\ell$ the \textbf{\textcolor{myblue}{length of the partition}} $\bar{n}$, and also denote it by $\len(\bar{n})\label{eq:length-partition}$. 
We will denote by $\Part(n) \label{eq:set-partitions}$ (resp., $\Part_{\ell}(n) \label{eq:set-partitions-given-length}$) the set of all partitions of $n \in \NN$ (resp., of length $\ell$). 
Moreover, given $n \in \NN$, it will be useful to consider the map 
\begin{equation}
\label{eq:map-partition-O}
\ord : \Part(n) \longrightarrow \NN^{\llbracket 1 , n \rrbracket}
\end{equation}
that sends a partition $\bar{n} = (n_{1}, \dots, n_{\ell})$ of $n \in \NN$ to the map $\ord_{\bar{n}} : \llbracket 1 , n \rrbracket \rightarrow \NN$ sending an element of $\llbracket 1 + \sum_{k=1}^{p-1} n_{k}, \sum_{k=1}^{p} n_{k} \rrbracket$ for $p \in \llbracket 1 , \ell \rrbracket$ to $p$. 
More concretely, if we write $\llbracket 1,n\rrbracket$ as the disjoint union of `packages' 
\[     \llbracket 1, n_1\rrbracket\sqcup \llbracket n_1+1, n_1+n_2\rrbracket \sqcup \cdots \sqcup \bigg\llbracket 1 + \sum_{k=1}^{p-1} n_{k}, \sum_{k=1}^{p} n_{k} \bigg\rrbracket \sqcup \cdots \sqcup \bigg\llbracket 1 + \sum_{k=1}^{\ell-1} n_{k}, n \bigg\rrbracket     \]
then, if $s\in \llbracket 1 + \sum_{k=1}^{p-1} n_{k}, \sum_{k=1}^{p} n_{k} \rrbracket$ for $p\in\llbracket 1,\ell\rrbracket$, we have that  $\ord_{\bar{n}}(s)= p$. 

Given $n \in \NN_{0}$ and $\bar{n} \in \Part_{\ell}(n)$ for $\ell \in \NN$, 
an \textbf{\textcolor{myblue}{agglutination}} of $\bar{n}$ 
is a partition $\bar{m} \in \Part_{k}(n)$ such that there exists 
a partition $\bar{\ell} = (\ell_{1}, \dots, \ell_{k}) \in \NN^{m}$ of $\ell$ satisfying that the $j$-th entry of $\bar{m}$ is given by
\[     \sum_{i=1+L_{j-1}}^{L_{j}} n_{i},     \]
for $j \in \llbracket 1 , k \rrbracket$, where $L_{j} = \sum_{h=1}^{j}\ell_{h}$. 
We further say that a mapping $\varphi : \Part(n) \rightarrow \Part(n)$ is an \textbf{\textcolor{myblue}{agglutination map}} if $\varphi(\bar{n})$ is an agglutination of $\bar{n}$, for all $\bar{n} \in \Part(n)$. 

Finally, given $n, n' \in \NN_{0}$, define the map 
\begin{equation}
\label{eq:cup-par}
  \sqcup : \Part(n) \times \Part(n') \longrightarrow \Part(n+n')
\end{equation} 
sending $(\bar{n},\bar{n}') \in \Part_{\ell}(n) \times \Part_{\ell'}(n')$ to the partition $\bar{n}'' = \bar{n} \sqcup \bar{n}' \in \Part_{\ell+\ell'}(n+n')$ given by 
$(n_{1}, \dots, n_{\ell}, n'_{1}, \dots, n'_{\ell'})$, where $\bar{n} = (n_{1}, \dots, n_{\ell})$ and $\bar{n}' = (n'_{1}, \dots, n'_{\ell'})$. 
More generally, given a set $\mathtt{A}\label{index:arbitrary-set}$, $\bar{\alpha} = (\alpha_{1},\dots,\alpha_{n}) \in \mathtt{A}^{n}$ and $\bar{\beta} = (\beta_{1},\dots,\beta_{m}) \in \mathtt{A}^{m}$, 
we will also write $\bar{\alpha} \sqcup \bar{\beta} = (\alpha_{1}, \dots, \alpha_{n}, \beta_{1}, \dots, \beta_{m}) \in \mathtt{A}^{n+m}$. 

Recall that, given $n \in \NN_{0}$, $\SG_{n}\label{symmetric group}$ denotes the group of automorphisms of the set $\llbracket 1, n \rrbracket$ with the product given by composition, which is written in the usual way from right to left, so $\SG_{n}$ has exactly $n!$ elements. 
In particular, $\SG_{0}$ is the group formed by the empty function $\emptyset \rightarrow \emptyset$, 
and $\SG_{1}$ is the group formed by the identity map of $\{ 1 \}$. 

%%%%%%%%%%%%%%%%%%%%%%%%%%%%%%%%%%%%%%%%%%%%%%%%%%%%%%%%%%%%%%%%%%%%%%
\subsection{Categorical definitions}
\label{subsection:categorical}

Throughout the article, we fix a field $\Bbbk \label{eq:field-char-0}$ of characteristic zero. 
For the basic definitions on category theory we refer the reader to \cite{MR0202787}, and also to \cite{MR1712872} for the basics on symmetric monoidal categories. 
Throughout this subsection, we fix a locally small complete and cocomplete $\Bbbk$-linear category $\C\label{eq:generic-category}$. 
Assume further that we have a monoidal structure $(\C, \otimes_{\C}, \mathbf{I}_{\C}) \label{eq:generic-monoidal-str}$ on $\C$ with a symmetric braiding $\tau^{\C} \label{eq:generic-braiding}$ such that $\otimes_{\C}$ commutes with arbitrary colimits in $\C$ on each side. 

%%%%%%%
\begin{example}
\label{example:vect}
The basic example we can consider is the category $\C = {}_{\Bbbk}\Vect$ of vector spaces and linear morphisms endowed with the usual tensor product $\otimes_{\Bbbk}$, that we will simply write as $\otimes$, and with the braiding 
given by the \textbf{\textcolor{myblue}{flip}} map $\tau(V,W) : V \otimes W \rightarrow W \otimes V$
sending $v \otimes w$ to $w \otimes v$, for $v \in V$ and $w \in W$. 
If $\C = {}_{\Bbbk}\Vect$, we will omit the reference to the category $\C$ in the notation.  
\end{example}
%%%%%%%

In what follows, we will freely use the left and right unit isomorphisms $X \otimes_{\C} \mathbf{I}_{\C} \cong X \cong \mathbf{I}_{\C} \otimes_{\C} X$ as well as the associativity constraints $(X \otimes_{\C} Y) \otimes_{\C} Z \cong X \otimes_{\C} (Y \otimes_{\C} Z)$, for all objects $X$, $Y$ and $Z$ in $\C$. 

For completeness and to establish the notation we will use in the sequel, we recall the following definitions of (associative, commutative and Lie) algebras in the context of symmetric monoidal categories, as well as modules over them. 
They are just the usual definitions written in categorical terms (see \cite{MR1701597}, \S 1.3). 

%%%%%%%%%%%%%%%%%%%%%%%%%%%%%%%%%%%%%%%%%%%%%%%%%%%%%%%%%%%%%%%%%%%%%%
\subsubsection{Algebras}
\label{subseubsection:mono-alg}

Recall that a(n associative) \textbf{\textcolor{myblue}{algebra}} in the symmetric monoidal category $\C$ is a pair $(A,\mu)$, where $A$ is an object of $\C$ and $\mu : A \otimes_{\C} A \rightarrow A$ is a morphism of $\C$ such that 
\[    
\mu \circ (\id_{A} \otimes_{\C} \mu) = \mu \circ (\mu \otimes_{\C} \id_{A}).     
\label{eq:generic-algebra-def}
\]
We further say that it is \textbf{\textcolor{myblue}{commutative}} if we also have that $\mu \circ \tau^{\C}(A,A) = \mu \label{eq:generic-comm-algebra-category}$. 
We say that an algebra $(A,\mu)$ in $\C$ is \textbf{\textcolor{myblue}{unitary}} if there is a morphism $\eta : \mathbf{I}_{\C} \rightarrow A$ in $\C$ such that 
\[     \mu \circ (\id_{A} \otimes_{\C} \eta) = \id_{A} = \mu \circ (\eta \otimes_{\C} \id_{A}).     \]
In the case of vector spaces, we typically denote the image of $1_{\Bbbk} \in \Bbbk$ under the unit $\eta$ by $1_{A}$. 
A \textbf{\textcolor{myblue}{morphism of algebras}} is a morphism $f : A \rightarrow A'$ in $\C'$ such that $\mu' \circ (f \otimes_{C} f) = f \circ \mu$. 
If the algebras are unitary we further assume that $f \circ \eta = \eta'$. 
It is clear that the (resp., unitary) algebras in $\C$ form a category, denoted by $\Alg(\C) \label{eq:generic-category-of-algebras}$ (resp., $\uAlg(\C) \label{eq:generic-category-of-unitary-algebras} $), for the composition and identities induced by those of $\C$. 
Similarly, the (resp., unitary) commutative algebras in $\C$ form a category, denoted by $\CAlg(\C) \label{eq:generic-category-of-commutative-algebras}$ (resp., $\uCAlg(\C) \label{eq:generic-category-of-unitary-commutative-algebras} $), for the composition and identities induced by those of $\C$. 
Since we are more interested in algebras with unit and morphisms between them, we will refer to them simply as algebras and morphisms of algebras, respectively, whereas for algebras not necessarily endowed with a unit and morphisms between them, we shall refer to them as \textbf{\textcolor{myblue}{nonunitary algebras}} and \textbf{\textcolor{myblue}{morphisms of nonunitary algebras}}, respectively. 

%%%%%%%
\begin{example}
\label{example:tensor-alg}
Let $M$ be an object of $\C$. 
Recall that the \textbf{\textcolor{myblue}{tensor algebra}} on $M$ is the object $T_{\C}(M) = \oplus_{n \in \NN_{0}} M^{\otimes_{\C} n}$, where $M^{\otimes_{\C} 0} = \mathbf{I}_{\C}$ endowed with the unit $\eta\colon \mathbf{I}_{\C} \rightarrow T_{\C}(M)$ given by the canonical inclusion and the product $\mu : T_{\C}(M) \otimes_{\C} T_{\C}(M) \rightarrow T_{\C}(M)$, whose restriction to $M^{\otimes_{\C} n} \otimes M^{\otimes_{\C} m}$ is the canonical inclusion $M^{\otimes_{\C} (n+m)} \rightarrow T_{\C}(M)$. 
It is clear that $T_{\C}(M)$ is an algebra in $\C$, and it satisfies that, for any algebra $A$ in $\C$, the map 
\begin{equation}
\label{eq:tensor-alg}
\Hom_{\uAlg(\C)}\big(T_{\C}(M),A\big) \longrightarrow \Hom_{\C}(M,A)
\end{equation}
sending $f$ to $f \circ \iota_{M}^{\mathrm{te}}$ 
is a bijection, where $\iota_{M}^{\mathrm{te}} : M \rightarrow T_{\C}(M)$ is the canonical inclusion. 
Needless to say that the tensor algebra in $\C$ is uniquely determined up to isomorphism by \eqref{eq:tensor-alg}. 
\end{example}
%%%%%%%

%%%%%%%
\begin{example}
\label{example:sym-alg} 
Let $\catA$ be a full subcategory of the category of commutative algebras of $\C$. 
Motivated by the universal property \eqref{eq:tensor-alg}, given an object $M$ in $\C$, we define the \textbf{\textcolor{myblue}{symmetric algebra}} on $M$ as the commutative algebra $\Sym_{\catA}(M)$ 
in $\catA$
satisfying that there is natural bijection 
\begin{equation}
\label{eq:sym-alg}
\Hom_{\uAlg(\C)}
\big(\Sym_{\catA}(M),D\big) \longrightarrow \Hom_{\C}(M,D)
\end{equation}
for any commutative algebra $D$ in $\catA$. 
If it exists, it is unique up to isomorphism, by \eqref{eq:sym-alg}. 
We will also denote the canonical morphism $M \rightarrow \Sym_{\catA}(M)$ associated with the identity under \eqref{eq:sym-alg} by $\iota_{M}$. 
Note that \eqref{eq:sym-alg} sends $f \in \Hom_{\uAlg(\C)}(\Sym_{\catA}(M),D)$ 
to $f \circ \iota_{M}$.

If $\catA$ is the category of all commutative algebras in $\C$, then 
we can prove that the symmetric algebra $\Sym_{\catA}(M)$, which will be simply denoted by $\Sym_{\C}(M)$, exists. 
Indeed, given $n \in \NN_{0}$, let $\Sym_{\C}^{n}(M)$ be the space of coinvariants of $M^{\otimes_{\C} n}$ under the 
natural action of the symmetric group of $n$ letters $\SG_{n}$ induced by the symmetric braiding of $\C$, and set $\Sym_{\C}(M) = \oplus_{n \in \NN_{0}} \Sym_{\C}^{n}(M)$. 
The unit and product of $T_{\C}(M)$ induce a unit and a product on $\Sym_{\C}(M)$, 
such that the morphism $q : T_{\C}(M) \rightarrow \Sym_{\C}(M)$ whose restriction to $M^{\otimes_{\C} n}$ is the canonical projection $M^{\otimes_{\C} n} \rightarrow \Sym_{\C}^{n} (M)$ is a morphism of algebras. 
Then, the map \eqref{eq:sym-alg} sending $f \in \Hom_{\uAlg(\C)}
(\Sym_{\C}(M),C)$ to $f \circ \iota_{M}$ 
is a bijection, where $\iota_{M} : M \rightarrow \Sym_{\C}(M)$ is the canonical inclusion. 
\end{example}
%%%%%%%

Recall that, if $A$ is an algebra in $\C$ with product $\mu$, $A^{\op} \label{index:opposite-algebra}$ denotes the \textbf{\textcolor{myblue}{opposite algebra}}, \textit{i.e.} the underlying object of $A^{\op}$ is $A$ 
with product $\mu_{\op} = \mu \circ \tau^{\C}(A,A)$. 

Given an algebra $(A,\mu,\eta)$ in $\C$ and $n \in \NN_{0}$, define the morphism $\mu^{[n]} : A^{\otimes_{\C} n} \rightarrow A$ inductively by $\mu^{[0]} = \eta$, $\mu^{[1]} = \id_{A}$ and $\mu^{[n+1]} = \mu \circ (\mu^{[n]} \otimes_{\C} \id_{A})$ for all $n \in \NN$. 
Furthermore, given two algebras $(A,\mu_{A},\eta_{A})$ and $(B,\mu_{B},\eta_{B})$ in $\C$, recall that the tensor product $A \otimes_{\C} B$ has a natural structure of algebra in $\C$, called the 
\textbf{\textcolor{myblue}{tensor product algebra}}, with $\eta_{A \otimes B} = \eta_{A} \otimes_{\C} \eta_{B}$ and $\mu_{A \otimes B} = (\mu_{A} \otimes_{\C} \mu_{B}) \circ (\id_{A} \otimes_{\C} \tau^{\C}(B,A) \otimes_{\C} \id_{B})$. 
If $A$ and $B$ are commutative, then their tensor product is also commutative. 
Recall that $A^{\env} = A \otimes_{\C} A^{\op} \label{index:enveloping-alg}$ denotes the \textbf{\textcolor{myblue}{enveloping algebra}}, 
whose product will be sometimes denoted by $\cdot_{\env}$ in the case of vector spaces. 
Note that, if $\varphi : A \rightarrow B$ is a morphism of algebras, then the map 
\[
\varphi^{\env} : A^{\env} \longrightarrow B^{\env}
\]
defined as $\varphi \otimes_{\C} \varphi$ is also a morphism of algebras. 

If $A$ is an algebra in $\C$, we will denote by $A_{\cyc} \label{index:alg-cyc}$ (resp., $[A,A]$ \label{index:v-s-commutators}) the cokernel (resp., image) of the morphism $A \otimes_{\C} A \rightarrow A$ given by $\mu - \mu \circ \tau^{\C}(A,A)$ in $\C$. 
Note that if $\C$ is the category vector spaces, then $[A,A]$ is the vector subspace of $A$ spanned by $\{ a a' - a' a : a, a ' \in A \}$ and $A_{\cyc} = A/[A,A]$.
We remark that, if $\varphi : A \rightarrow B$ is a morphism of algebras, then it induces a unique morphism $\varphi_{\cyc} : A_{\cyc} \rightarrow B_{\cyc}$ in $\C$ 
such that $\varphi_{\cyc} \circ \pi_{A,\cyc} = \pi_{B,\cyc} \circ \varphi$, where $\label{eq:cokernel-morphism-pi}\pi_{A,\cyc} : A \rightarrow A_{\cyc}$ denotes the cokernel morphism for $A$ and similarly for $B$. 
This clearly gives a functor $\label{eq:cyc}(\place)_{\cyc} : \Alg(\C) \rightarrow \C$.

Since the natural map
\begin{equation}
\label{eq:c-alg-push-out}
\Hom_{\uAlg(\C)}
\big(A, C\big) \oplus 
\Hom_{\uAlg(\C)}
\big(B , C\big) 
\longrightarrow
\Hom_{\uAlg(\C)}
\big(A \otimes_{\C} B , C\big)
\end{equation}
sending $(f,g)$ to $f \otimes_{\C} g$ is a natural isomorphism for every triple of commutative algebras $A$, $B$ and $C$, the universal property \eqref{eq:sym-alg} gives the following result.

%%%%%%%
\begin{fact}
\label{example:sym-alg-2}
Let $\catA$ be a full subcategory of the category of commutative algebras of $\C$. 
Let $M$ and $N$ be objects of $\C$, and assume that the symmetric algebras $\Sym_{\catA}(M)$ and $\Sym_{\catA}(N)$ exist. 
Then, the universal property \eqref{eq:sym-alg} implies that the symmetric algebra $\Sym_{\catA}(M \oplus N)$ also exists, 
giving the isomorphim 
\begin{equation}
\label{eq:sym-alg-2}
\Sym_{\catA}(M) \otimes_{\C} \Sym_{\catA}(N) \overset{\cong}{\longrightarrow} \Sym_{\catA}(M \oplus N) 
\end{equation}
of algebras of $\C$.
\end{fact}
%%%%%%%

For instance, if we further assume that $\C$ is abelian, 
let $\iota'_{M} : \Sym_{\C}(M) \rightarrow \Sym_{\C}(M \oplus N)$ and 
$\iota'_{N} : \Sym_{\C}(N) \rightarrow \Sym_{\C}(M \oplus N)$ be the morphisms of algebras induced 
by the canonical inclusions $M \rightarrow M \oplus N \rightarrow \Sym_{\C}(M \oplus N)$ 
and $N \rightarrow M \oplus N \rightarrow \Sym_{\C}(M \oplus N)$, respectively, using the universal property \eqref{eq:sym-alg}. 
Then, the universal property \eqref{eq:sym-alg} implies that $\iota'_{M} \otimes_{\C} \iota'_{N}$ is the explicit isomorphism of algebras  \eqref{eq:sym-alg-2}. 

If $G$ is a group, we will denote its \textbf{\textcolor{myblue}{group algebra}} (in the category of vector spaces) by $\Bbbk G \label{index:group-alg}$, and if $\phi : G \rightarrow G'$ is a morphism of groups, we will denote the induced morphism of algebras by $\Bbbk \phi : \Bbbk G \rightarrow \Bbbk G'$. 
Similarly to the case of algebras, if $G$ is a group then $G^{\op} \label{index:opposite-group}$ denotes the \textbf{\textcolor{myblue}{opposite group}} whose product $g \cdot_{\op} g' = g' \cdot g$, for $g , g ' \in G$. 
Notice that the map $G \rightarrow G^{\op}$ 
sending $g \in G$ to $g^{-1}$ is an isomorphism of groups. 
Note also that the group algebra $\Bbbk G^{\op}$ of $G^{\op}$ coincides with the opposite algebra 
$(\Bbbk G)^{\op}$ of the group algebra of $G$, and if $G = G_{1} \times G_{2}$ for groups $G_{1}$
and $G_{2}$, then the map $\Bbbk G \rightarrow \Bbbk G_{1} \otimes \Bbbk G_{2}$ sending $(\sigma_{1},\sigma_{2})$ to $\sigma_{1} \otimes \sigma_{2}$ for $\sigma_{i} \in G_{i}$ and $i \in \{ 1 , 2 \}$ is an isomorphism of algebras. 
This implies that $\Bbbk G^{\env}$ is canonically isomorphic to $(\Bbbk G)^{\env}$, where $\label{eq:enveloping-group}G^{\env} = G \times G^{\op}$. 
We will also consider the morphism of groups 
\begin{equation}
\label{eq:group-diag}
    \double_{G} : G \longrightarrow G^{\env},\quad g\longmapsto (g,g^{-1}).
\end{equation}

We recall that a (nonnegative) \textbf{\textcolor{myblue}{grading}} of an algebra $A$ in $\C$ is a coproduct $A = \oplus_{n \in \NN_{0}} A_{n}$ in $\C$ (\textit{i.e.} $A_{n}$ is an object of $\C$ for all $n \in \NN_{0}$) such that the image of $\mu(A_{n} \otimes A_{m})$ is included in $A_{n+m}$ for all $n, m \in \NN_{0}$ and the image of $\eta$ is included in $A_{0}$. 
Given graded algebras $A$ and $B$, a morphism $f : A \rightarrow B$ of algebras is called \textbf{\textcolor{myblue}{degree preserving}} (or a \textbf{\textcolor{myblue}{morphism of graded algebras}}) if $f(A_{n}) \subseteq B_{n}$ for all $n \in \NN_{0}$. 
Given graded algebras $A$ and $B$ in $\C$, we will denote by $\label{eq:set-hom-graded-alg}\operatorname{Hom}_{\Alg(\C)}^{\gr}(A,B)$ the set of morphisms of graded algebras from $A$ to $B$. 
Similar definitions also hold for modules over graded algebras. 

%%%%%%%%
\begin{example} 
\label{example:tensor-grading}
Let $M$ be an object in $\C$. 
Then, the algebras $T_{\C}(M)$ and $\Sym_{\C}(M)$ have a natural grading given by setting $M^{\otimes_{\C} n}$ and $\Sym_{\C}^{n}(M)$ in degree $n$ for all $n \in \NN_{0}$. 
\end{example}
%%%%%%%%

%%%%%%%%
\begin{fact} 
\label{fact:sym-grading}
Let $M$ be an object in $\C$ and $\catA$ a full subcategory of the category of $\NN_{0}$-graded commutative algebras whose zeroth degree component is $\mathbf{I}_{\C}$ and the unit is the canonical inclusion of $\mathbf{I}_{\C}$ into the zeroth degree component. 
Assume further that $\catA$ contains all 
$\NN_{0}$-graded commutative algebras of the form $\mathbf{I}_{\C} \oplus X$ for all objects $X$ in $\C$, where $X$ sits in degree $1$, the restriction of the product to $X \otimes_{\C} X$ vanishes, and the restriction to $\mathbf{I}_{\C} \otimes_{\C} X'$ (resp., $X' \otimes_{\C} \mathbf{I}_{\C}$ is the left (resp., right) unit constraint of $\C$ for $X' \in \{ \mathbf{I}_{\C} , X\}$. 
Finally, assume that the symmetric algebra $\Sym_{\catA}(M)$ exists. 

Then, $\iota_{M} : M \rightarrow \Sym_{\catA}(M)$ is precisely the composition of an isomorphism $M \rightarrow \Sym_{\catA}^{1}(M)$ and the canonical inclusion $\Sym_{\catA}^{1}(M) \rightarrow \Sym_{\catA}(M)$. 
Moreover, the subalgebra of $\Sym_{\catA}(M)$ generated by $\Sym_{\catA}^{0}(M)$ and $\Sym_{\catA}^{1}(M)$ is $\Sym_{\catA}(M)$.
\end{fact}
%%%%%%%
\begin{proof}
Note first that \eqref{eq:sym-alg} for $D = \mathbf{I}_{\C} \oplus M$ tells us that there exists a  morphism $\Pi : \Sym_{\catA}(M) \rightarrow D$ of graded algebras such that $\Pi \circ \iota_{M}$ is the identity of $M$. 
Considering the restriction to the component of degree $1$, we conclude that 
there exists a morphism $\pi :\Sym_{\catA}^{1}(M) \rightarrow M$ such that $\pi \circ \iota_{M}$ is the identity $\id_{M}$ of $M$. 

Let $D' = \mathbf{I}_{\C} \oplus \Sym_{\catA}^{1}(M)$ and consider the unique homogeneous morphism $\Pi' : \Sym_{\catA}(M) \rightarrow D'$ (resp., $\Pi'' : \Sym_{\catA}(M) \rightarrow D'$)
of degree zero given as the direct sum of the identity of $\mathbf{I}_{\C}$ and $\id_{\Sym_{\catA}^{1}(M)} - \iota_{M} \circ \pi$ (resp., the zero morphism). 
Then, $\Pi'$ (resp., $\Pi''$) is a morphism of graded algebras. 
Since the associated morphisms $\Pi' \circ \iota_{M}$ and $\Pi'' \circ \iota_{M}$ coincide, as they both vanish, the universal property \eqref{eq:sym-alg} tells us that $\Pi' = \Pi''$, which in turn implies that $\id_{\Sym_{\catA}^{1}(M)} = \iota_{M} \circ \pi$, so $\iota_{M}$ and $\pi$ are inverse morphisms. 
This proves the claim. 
\end{proof}

%%%%%%%
\begin{remark}
\label{remark:sym-grading}
By Fact \ref{fact:sym-grading}, we can assume that the canonical morphism $\iota_{M} : M \rightarrow \Sym_{\catA}(M)$ of Example \ref{example:sym-alg} induces an isomorphism $M \rightarrow \Sym_{\catA}^{1}(M)$, so we will identify them.
\end{remark}
%%%%%%%

%%%%%%%%%%%%%%%%%%%%%%%%%%%%%%%%%%%%%%%%%%%%%%%%%%%%%%%%%%%
\subsubsection{Modules and bimodules} 
\label{subsubsection:mod-bimodules}

Given two nonunitary algebras $(A,\mu_{A},\eta_{A})$ and $(B,\mu_{B},\eta_{B})$ in $\C$, an \textbf{\textcolor{myblue}{$A$-$B$-bimodule}} in $\label{eq:A-B-bimodule}\C$ is a triple $(M,\rho_{\mathcalboondox{l}}, \rho_{\mathcalboondox{r}})$ with $M$ an object of $\C$, and $\rho_{\mathcalboondox{l}} : A \otimes_{\C} M \rightarrow M$ and $\rho_{\mathcalboondox{r}} : M \otimes_{\C} B \rightarrow M$ two morphisms satisfying 
\begin{equation}  
\begin{split}
\rho_{\mathcalboondox{l}} \circ (\id_{A} \otimes_{\C} \rho_{\mathcalboondox{l}}) &= \rho_{\mathcalboondox{l}} \circ (\mu \otimes_{\C} \id_{M}), \quad \rho_{\mathcalboondox{r}} \circ (\rho_{\mathcalboondox{r}} \otimes_{\C} \id_{B}) = \rho_{\mathcalboondox{r}} \circ (\id_{M} \otimes_{\C} \mu),     \\
&\text{ and }   \rho_{\mathcalboondox{r}} \circ (\rho_{\mathcalboondox{l}} \otimes_{\C} \id_{B}) = \rho_{\mathcalboondox{l}} \circ (\id_{A} \otimes_{\C}  \rho_{\mathcalboondox{r}}).
\label{eq:right-action-bimodules}
\end{split}
\end{equation}
If $A$ (resp., $B$) has a unit $\eta_{A}$ (resp., $\eta_{B}$), we further require that 
\begin{equation} 
\label{eq:right-action-bimodules-unit}
\rho_{\mathcalboondox{l}} \circ (\eta_{A} \otimes_{\C} \id_{M})= \id_{M} \text{ \big(resp., }\rho_{\mathcalboondox{r}} \circ (\id_{M} \otimes_{\C} \eta_{B})= \id_{M}\big).
\end{equation}
If $A = B$ we will simply say that $M$ is an \textbf{\textcolor{myblue}{$A$-bimodule}} in $\C$. 
For instance, the (resp., nonunitary) algebra $A$ endowed with $\rho_{\mathcalboondox{l}} = \rho_{\mathcalboondox{r}} = \mu$ is an $A$-bimodule, called the \textbf{\textcolor{myblue}{regular bimodule}}. 
Given $A$-$B$-bimodules $(M,\rho_{\mathcalboondox{l}}, \rho_{\mathcalboondox{r}})$ and $(M',\rho_{\mathcalboondox{l}}', \rho_{\mathcalboondox{r}}')$ in $\C$, a \textbf{\textcolor{myblue}{morphism}} of $A$-$B$-bimodules from $M$ to $M'$ is a morphism $f : M \rightarrow M'$ in $\C$ such that $f \circ \rho_{\mathcalboondox{l}} = \rho'_{\mathcalboondox{l}} \circ (\id_{A} \otimes_{\C} f)$ and $f \circ \rho_{\mathcalboondox{r}} = \rho'_{\mathcalboondox{r}} \circ (f \otimes_{\C} \id_{B})$. 
We will denote by $\label{eq:set-morphisms-A-B-bimodules} \Hom_{A\text{-}B}(M,M')$ the set of all morphisms of $A$-$B$-bimodules from $M$ to $M'$. 
The collection of $A$-$B$-bimodules together with the morphisms of $A$-$B$-bimodules thus forms a category ${}_{A}\Mod_{B}(\C) \label{eq:generic-categ-bimodules}$, whose composition and identities are induced by those of $\C$. 

A pair $(M,\rho_{\mathcalboondox{l}})$ (resp., $(M,\rho_{\mathcalboondox{r}})$) with 
$\rho_{\mathcalboondox{l}} : A \otimes_{\C} M \rightarrow M$ (resp., $\rho_{\mathcalboondox{r}} : M \otimes_{\C} B \rightarrow M$) satisfying the first (resp., second) identity of the first line of \eqref{eq:right-action-bimodules}, 
as well as \eqref{eq:right-action-bimodules-unit} is said a \textbf{\textcolor{myblue}{left $A$-module}} (resp., \textbf{\textcolor{myblue}{right $B$-module}}). 
We will write ${}_{A}\Mod(\C) \label{eq:generic-categ-left-modules}$ (resp., $\Mod_{B}(\C) \label{eq:generic-categ-right-modules}$) the category of left $A$-modules (resp., right $B$-modules), where the notions of morphisms are clear. 
We will denote by $\Hom_{A\text{-}}(M,N)\label{set-ofmorphisms-of-left-modules}$ (resp., $\Hom_{\,\text{-}B}(M,N)$) the set of all morphisms of left $A$-modules (resp., right $B$-modules) from $M$ to $N$.

Assume in this paragraph that the algebras $A$ and $B$ are unitary. 
Given an $A$-$B$-bimodule $(M,\rho_{\mathcalboondox{l}}, \rho_{\mathcalboondox{r}})$, define the morphism $\rho : A \otimes_{\C} M \otimes_{\C} B \rightarrow M$ by 
\begin{equation}
\label{eq:assoc-mor-bimod}
\rho = \rho_{\mathcalboondox{r}} \circ (\rho_{\mathcalboondox{l}} \otimes_{\C} \id_{B}) = \rho_{\mathcalboondox{l}} \circ (\id_{A} \otimes_{\C}  \rho_{\mathcalboondox{r}}).
\end{equation}
Then, $(M,\rho)$ satisfies that
\begin{equation}
\label{eq:assoc-mor-bimod-axioms}
\rho \circ (\id_{A} \otimes_{\C} \rho \otimes_{\C} \id_{B}) = \rho \circ (\mu_{A} \otimes_{\C} \id_{M} \otimes_{\C} \mu_{B}) \text{ $\phantom{x}$ and $\phantom{x}$ }  
\rho \circ (\eta_{A} \otimes_{\C} \id_{M} \otimes_{\C} \eta_{B})= \id_{M}.
\end{equation}
Conversely, given a pair $(M,\rho)$ with $\rho : A \otimes_{\C} M \otimes_{\C} B \rightarrow M$ satisfying \eqref{eq:assoc-mor-bimod-axioms}, define the morphisms $\rho_{\mathcalboondox{l}} : A \otimes_{\C} M \rightarrow M$ and $\rho_{\mathcalboondox{r}} : M \otimes_{\C} B \rightarrow M$ by $\rho_{\mathcalboondox{l}} = \rho \circ (\id_{A} \otimes_{\C} \id_{M} \otimes_{\C} \eta_{B})$ and 
$\rho_{\mathcalboondox{r}} = \rho \circ (\eta_{A} \otimes_{\C} \id_{M}  \otimes_{\C} \id_{B})$. 
This gives two equivalent definitions of $A$-$B$-bimodules.  
Then, $(M,\rho_{\mathcalboondox{l}}, \rho_{\mathcalboondox{r}})$ is an $A$-$B$-bimodule. 
Furthermore, there is an isomorphism between 
${}_{A}\Mod_{B}(\C)$ and $\Mod_{B \otimes_{\C} A^{\op}}(\C)$. 
Indeed, to an $A$-$B$-bimodule $(M,\rho_{\mathcalboondox{l}},\rho_{\mathcalboondox{r}})$ in ${}_{A}\Mod_{B}(\C)$ one associates the right $B \otimes A^{\op}$-module $(M, \rho'_{\mathcalboondox{r}})$ 
with $\rho'_{\mathcalboondox{r}} = \rho \circ (\tau^{\C}(M,A) \otimes_{\C} \id_{A}) \circ (\id_{M} \otimes_{\C} \tau^{\C}(B,A))$ and $\rho$ as in \eqref{eq:assoc-mor-bimod}. 
Conversely, to a right $B \otimes A^{\op}$-module $(M, \rho'_{\mathcalboondox{r}})$ one associates the unique $A$-$B$-bimodule $(M,\rho_{\mathcalboondox{l}},\rho_{\mathcalboondox{r}})$ satisfying that the associated morphism \eqref{eq:assoc-mor-bimod} is given by
$\rho = \rho'_{\mathcalboondox{r}} \circ (\id_{M} \otimes_{\C} \tau^{\C}(A,B)) \circ (\tau^{\C}(A,M) \otimes_{\C} \id_{A})$. 
It is clear that the previous mappings are inverse to each other and functorial. 
This applies in particular if $A = B$, giving an isofunctor between ${}_{A}\Mod_{A}(\C)$ and $\Mod_{A^{\env}}(\C)$. 
We will use in the sequel any of the previous equivalent definitions of bimodules over unitary algebras. 

Let $A$ be a unitary algebra and let $M$ be an $A$-bimodule $(M,\rho)$. 
We define $\label{index:naturalization}M_{\cycm}$ (resp., $\label{index:commutator-modules}[A,M]$) as the cokernel (resp., image) of the morphism $A \otimes_{\C} M \rightarrow M$ given by $\rho_{\mathcalboondox{l}} - \rho_{\mathcalboondox{l}} \circ \tau^{\C}(A,M)$. 
If $\C$ is the category of vector spaces, then $[A,M]$ is the the vector subspace of $M$ spanned by the set $\{ a m - m a : a \in A, m \in M \}$, and $M_{\cycm} = M/[A,M] \cong A\otimes_{A^e} M$. 
Notice that if $f : M \rightarrow M'$ is a morphism of $A$-bimodules, then it induces a unique morphism $f_{\cycm} : M_{\cycm} \rightarrow M'_{\cycm}$ in $\C$ 
such that $f_{\cycm} \circ \pi_{M,\cycm} = \pi_{M',\cycm} \circ f$, where $\label{eq:cokernel-morphism-pi-mod}\pi_{M,\cycm} : M \rightarrow M_{\cycm}$ denotes the cokernel morphism for $M$ and similarly for $M'$. 
This clearly gives a functor $\label{eq:cycm}(\place)_{\cycm} : {}_{A}\Mod_{A}(\C) \rightarrow \C$.

If the algebra $A$ is commutative, we say that an $A$-bimodule $(M,\rho_{\mathcalboondox{l}}, \rho_{\mathcalboondox{r}})$ is \textbf{\textcolor{myblue}{symmetric}} if $\rho_{\mathcalboondox{r}} = \rho_{\mathcalboondox{l}} \circ \tau^{\C}(M,A)$. 
We denote by ${}_{A}\Mod^{\s}_{A}(\C) \label{eq:generic-categ-symm-bimodules}$ the full subcategory of ${}_{A}\Mod_{A}(\C)$ formed by all symmetric $A$-bimodules. 
We also remark that the categories ${}_{A}\Mod(\C)$ and ${}_{A}\Mod^{\s}_{A}(\C)$ are isomorphic. 
Indeed, to a symmetric $A$-bimodule $(M,\rho_{\mathcalboondox{l}}, \rho_{\mathcalboondox{r}})$ we associate the left $A$-module $(M,\rho_{\mathcalboondox{l}})$, 
and conversely, given a left $A$-module $(M,\rho_{\mathcalboondox{l}})$, we define the symmetric $A$--bimodule $(M,\rho_{\mathcalboondox{l}}, \rho_{\mathcalboondox{r}})$ with $\rho_{\mathcalboondox{r}} = \rho_{\mathcalboondox{l}} \circ \tau^{\C}(M,A)$. 
It is clear that these correspondences are functorial and inverse to each other. 
The analogous result also holds if we replace left by right modules. 
We will thus identify (left or right) $A$-modules and symmetric $A$-bimodules over a commutative algebra $A$. 

As in the usual case of algebras, given an algebra $A$ in $\C$ and two $A$-bimodules $M$ and $M'$, 
$M \otimes_{A} M' \label{eq:tensor-prod-cokernel-wrt-A}$ is the cokernel of the morphism $M \otimes_{\C} A \otimes_{\C} M' \rightarrow M \otimes_{\C} M'$ given by $\rho_{\mathcalboondox{r}} \otimes_{\C} \id_{M'} - \id_{M} \otimes_{\C} \rho'_{\mathcalboondox{l}}$. 
It is easy to see that the map $\rho_{\mathcalboondox{l}} \otimes \rho'_{\mathcalboondox{r}}$ 
induces an $A$-bimodule structure on $M \otimes_{A} M'$. 
The category ${}_{A}\Mod_{A}(\C)$ of $A$-bimodules endowed with the tensor product $\otimes_{A}$ and the unit $A$ is a monoidal category. 
If $A$ is commutative, the subcategory ${}_{A}\Mod^{\s}_{A}(\C)$ of ${}_{A}\Mod_{A}(\C)$ (and thus ${}_{A}\Mod(\C)$) is also monoidal for the previous tensor product and unit. 
Moreover, the braid of $\C$ induces a symmetric braiding of ${}_{A}\Mod^{\s}_{A}(\C)$. 

If $M$ is a bimodule over an algebra $A$, to reduce notation we will denote simply by $T_{A}(M)$ or $T_{A}M$ the tensor algebra $T_{{}_{A}\Mod_{A}(\C)}(M) \label{eq:generic-tensor-alg}$. 
Similarly, if $M$ is a (symmetric bi)module over a commutative algebra $A$, to reduce notation we will denote simply by $\Sym_{A}(M)$ or $\Sym_{A}M$ the symmetric algebra $\Sym_{{}_{A}\Mod^{\s}_{A}(\C)}(M) \label{eq:generic-symmetric-alg}$. 

%%%%%%%%%%%%%%%%%%%%%%%%%%%%%%%%%%%%%%%%%%%%%%%%%%%%%%%%%%%
\subsubsection{Rings over algebras}

Let $A$ be an algebra in $\C$. 
An algebra $U$ in ${}_{A}\Mod_{A}(\C)$ will be called an \textbf{\textcolor{myblue}{$A$-ring}}. 
We will denote the category of $A$-rings by ${}_{A}\Ring(\C) \label{eq:generic-categ-A-rings}$. 
Equivalently, an $A$-ring $U$ is given by an algebra $U$ in $\C$ together with a morphism of algebras $\iota : A \rightarrow U$. 
Indeed, if $(U,\nu,\iota)$ is an $A$-ring with product $\nu : U \otimes_{A} U \rightarrow U$, set $U^{\flat} \label{eq:index-flat} $ as the algebra in $\C$ whose underlying object is $U$, whose product $\nu^{\flat}$ is the composition of the canonical epimorphism $U \otimes U \rightarrow U \otimes_{A} U$ and $\nu$, and whose unit $\iota^{\flat}$ is the composition of the unit of $A$ and $\iota$, the unit of $U$. 
Conversely, if $U^{\flat}$ is an algebra in $\C$
with product $\mu : U \otimes U \rightarrow U$ and morphism of algebras $\iota : A \rightarrow U$, let $U$ be the $A$-module $U^{\flat}$ with action $\rho = \mu^{[3]} \circ (\iota \otimes_{\C} \id_{U^{\flat}} \otimes_{\C} \iota)$, product $\nu : U \otimes_{A} U \rightarrow U$ induced by $\mu$ and unit $\iota$. 
If $A$ is commutative, we will say that an $A$-ring $U$ is \textbf{\textcolor{myblue}{symmetric}}, or an \textbf{\textcolor{myblue}{$A$-algebra}}, if $U$ is a symmetric $A$-bimodule, and we denote by ${}_{A}\Ring^{s}(\C) \label{index:categ-symm-rings}$ the full subcategory of ${}_{A}\Ring(\C)$ formed by symmetric $A$-rings. 
Note that ${}_{A}\Ring^{s}(\C)$ is canonically isomorphic to the category $\Alg({}_{A}\Mod_{A}^{\s}(\C))$. 

%%%%%%%
\begin{example} 
\label{example:sym-mod}
Let $A$ be a commutative algebra in $\C$ with product $\mu$. 
Given an object $V$ of $\C$, the \textbf{\textcolor{myblue}{free}} (symmetric bi)module over $A$ generated by $V$ is the object $A \otimes V$ endowed with the left module structure $\mu \otimes_{\C} \id_{V}$. 
Then, if the symmetric algebra $\Sym_{\C}(V)$ exists, the symmetric algebra $\Sym_{A}(A \otimes_{\C} V)$ also exists, since the universal property \eqref{eq:sym-alg} gives us the isomorphism 
\begin{equation} 
\label{eq:sym-mod}
    \Sym_{A}(A \otimes_{\C} V) \cong A \otimes_{\C} \Sym_{\C}(V)
\end{equation}
of $A$-rings. 
\end{example} 
%%%%%%%

An \textbf{\textcolor{myblue}{augmented $A$-ring}} is an $A$-ring $\iota : A \rightarrow U$ endowed with an extra morphism of algebras $\epsilon : U \rightarrow A$ such that $\epsilon \circ \iota = \id_{A}$. 
In this case we have the coproduct $U = A \oplus \Ker(\epsilon)$ in $\C$. 
Given two augmented $A$-rings $U$ and $V$, with augmentations $\epsilon_{U}$ and $\epsilon_{V}$, a \textbf{\textcolor{myblue}{morphism of augmented $A$-rings}} from $U$ to $V$ is a morphism of $A$-rings $f : U \rightarrow V$ such that $\epsilon_{V} \circ f = \epsilon_{U}$. 
We will denote the category of augmented $A$-rings by ${}_{A}\aRing(\C) \label{index:categ-augmented-rings}$, and, provided $A$ is commutative, by ${}_{A}\aRing^{s}(\C)$ (or $\aAlg({}_{A}\Mod_{A}^{\s}(\C)) \label{index:categ-augmented-symm-rings}$) its full subcategory formed by augmented symmetric $A$-rings. 

%%%%%%%
\begin{example} 
\label{example:graded-aug}
A graded algebra $A = \oplus_{n \in \NN_{0}} A_{n}$ 
is an augmented $A_{0}$-ring using the canonical inclusion $\iota : A_{0} \rightarrow A$ and the canonical projection $\epsilon : A \rightarrow A_{0}$. 
\end{example}
%%%%%%%

%%%%%%%
\begin{example}[Square-zero construction]
\label{example:zse} 
Let $B$ be an algebra in $\C$, with product $\mu$ and unit $\eta$, and let $M$ be a $B$-bimodule in $\C$, with action $\rho$.  
Define the augmented $B$-ring $\ZSE(B,M)$ as follows. 
The underlying object is $B \oplus M$, the unit given by the composition of $\eta$ and the canonical injection $B \rightarrow B \oplus M$, and product $\hat{\mu} : (B \oplus M)^{\otimes_{\C} 2} \rightarrow B \oplus M$ satisfies that its restriction to $M^{\otimes_{\C} 2}$ vanishes, its restriction to $B^{\otimes_{\C} 2}$ is the composition of $\mu$ and the canonical inclusion $B \rightarrow B \oplus M$, and the restriction to $B \otimes_{\C} M$ (resp., $M \otimes_{\C} B$) is the composition of $\rho_{\mathcalboondox{l}}$ (resp., $\rho_{\mathcalboondox{r}}$) and the canonical inclusion $M \rightarrow B \oplus M$. 
The $B$-ring structure is given by the canonical inclusion $B \rightarrow B \oplus M$, and the augmentation by the canonical projection $B \oplus M \rightarrow B$. 
Note that $\ZSE(B,M)$ is commutative if $B$ is commutative and $M$ is a (symmetric bi)module over $B$. 

We will further consider the unique grading on $\ZSE(B,M)$ where $\ZSE(B,M)_{0} = B$ and $\ZSE(B,M)_{1} = M$. 
In this case, the corresponding graded $B$-ring will be denoted by $\ZSE^{\gr}(B,M)$.
\end{example}
%%%%%%% 

The following result is immediate. 
%%%%%%
\begin{fact} 
\label{fact:mor-zse}
Let $M$ and $N$ be two $B$-bimodules. 
Then, the map 
\begin{equation}
\label{eq:zse}
     \Hom_{{}_{B}\Mod_{B}(\C)}(M,N) \longrightarrow \Hom_{{}_{B}\aRing(\C)} \big(\ZSE(B,M) , \ZSE(B,N)\big)
\end{equation}
sending a morphism $f : M \rightarrow N$ to $\id_{B} \oplus f : \ZSE(B,M) \rightarrow \ZSE(B,N)$ is well defined, and its image is exactly the set of degree preserving morphisms of the corresponding graded augmented $B$-rings. 
\end{fact}
%%%%%%

Given an algebra $B$ with product $\mu$ and a $B$-bimodule $N$ with action $\rho$ in $\C$, recall that a 
\textbf{\textcolor{myblue}{derivation}} of $B$ with values in $N$ is a morphism $d : B \rightarrow N$ in $\C$ such that 
\begin{equation}
\label{eq:der-b-n}
d \circ \mu = \rho_{\mathcalboondox{l}} \circ (\id_{B} \otimes_{\C} d) + \rho_{\mathcalboondox{r}} \circ (d  \otimes_{\C} \id_{B}),
\end{equation}
where $\rho_{\mathcalboondox{l}} : B \otimes_{\mathcal{C}} N \rightarrow N$ (resp., $\rho_{\mathcalboondox{r}} : N \otimes_{\mathcal{C}} B \rightarrow N$) is the induced left (resp., right) action on $N$. 
As usual, the set of all derivations on $B$ with values in $N$ will be denoted by $\Der(B,N) \label{index:vector-sp-derivations} $ and it clearly forms a vector space. 
If $N = B$ with the standard structure, we will write $\Der(B)$ instead of $\Der(B,B)$. 

We will also be interested in the intermediate category 
${}_{A}\aRng(\C) \label{index: intermediate-categ-fixed-morph}$ formed by all algebras $B$ in $\C$ 
endowed with a fixed morphism of algebras $\varepsilon_{B} : B \rightarrow A$, and whose set of morphisms from $B$ to $C$
is formed by all morphisms $f : B \rightarrow C$ of algebras in $\C$ such that $\varepsilon_{C} \circ f= \varepsilon_{B}$. 
Note that there is a canonical forgetful functor ${}_{A}\aRing(\C) \rightarrow {}_{A}\aRng(\C)$.

The following result can be considered as an improvement of Fact \ref{fact:mor-zse}, and its proof is also immediate. 
%%%%%%
\begin{fact} 
\label{fact:mor-zse-bis}
Let $B$ be an algebra and let $M$ and $N$ be two $B$-bimodules. 
Then, the map 
\begin{equation}
\label{eq:der-hom}
\Hom_{{}_{B}\aRng(\C)}\big(\ZSE(B,M), \ZSE(B,N)\big)\longrightarrow \Der(B,N)\oplus\Hom_{{}_{B}\Mod_{B}(\C)}(M,N)
\end{equation}
given by sending $F$ to $(\pi_N\circ F\vert_B, \pi_N\circ F\vert_M)$,
where $\pi_N \colon \ZSE(B,N)\to N$ is the canonical projection onto $N$, is an isomorphism. 
\end{fact}
%%%%%%

%%%%%%%%%%%%%%%%%%%%%%%%%%%%%%%%%%%%%%%%%%%%%%%%%%%%%%
\subsubsection{Lie and Poisson algebras}
\label{sec:Lie and Poisson algebras}

Given a symmetric monoidal category $\C$, a \textbf{\textcolor{myblue}{Lie algebra}} in $\C$ is a pair $(\g,[\hskip 0.6mm ,])$, where $\g$ is an object of $\C$ and $[\hskip 0.6mm ,] : \g \otimes_{\C} \g \rightarrow \g$ is a morphism of $\C$ such that 
\begin{align}
[\hskip 0.6mm ,] \circ \tau^{\C}(\g,\g) &= - [\hskip 0.6mm ,],
\label{eq:skew-symm}
\\
\sum_{i=0}^{2} [\hskip 0.6mm ,] \circ (\id_{\g} \otimes_{\C} [\hskip 0.6mm ,]) \circ \sigma^{i} &= 0,
\label{eq:jacobi}
\end{align}
 where $\sigma : \g^{\otimes_{\C} 3} \rightarrow \g^{\otimes_{\C} 3}$ is given by $\sigma = (\tau^{\C}(\g,\g) \otimes_{\C} \id_{\g}) \circ (\id_{\g} \otimes_{\C} \tau^{\C}(\g,\g))$. The second equation is usually called the \textbf{\textcolor{myblue}{ Jacobi identity}}.
If $A$ is an algebra in $\C$, $\Lie(A)\label{eq:Lie-alg-str-over-A}$ is the Lie algebra structure over $A$ with bracket $[\hskip 0.6mm ,] = \mu \circ (\id_{A^{\otimes_{\C} 2}} - \tau^{\C}(A,A))$. 
The notion of morphism of Lie algebras is analogous to that of algebras. 

Let $\g$ be a Lie algebra in $\C$. 
A \textbf{\textcolor{myblue}{module}} over $\g$ in $\C$ is a pair $(N, q)\label{index:moduleover-Lie-alg}$ with $N$ an object of $\C$ and $q : \g \otimes_{\C} N \rightarrow N$ a morphism  satisfying that
\begin{equation*}
  q \circ (\id_{\g} \otimes_{\C} q) \circ \Big(\big((\id_{\g} \otimes_{\C} \id_{\g}) - \tau^{\C}(\g,\g)\big) \otimes_{\C} \id_{N}\Big) = q \circ ([\hskip 0.6mm ,] \otimes_{\C} \id_{N}).    
\label{eq:eq-compatibility-anchor}
\end{equation*}

Let $A$ be an algebra in $\C$ and $\g$ be a Lie algebra in $\C$. 
A module structure on $A$ over $\g$ given by the morphism $q : \g \otimes_{\C} A \rightarrow A$ is said to be an \textbf{\textcolor{myblue}{action by derivations}}
if we further have that
\begin{equation}
     q \circ (\id_{\g} \otimes_{\C} \mu) = \mu \circ (q \otimes_{\C} \id_{A}) + \mu \circ (\id_{A} \otimes_{\C} q) \circ (\tau^{\C}(\g,A) \otimes_{\C} \id_{A}).
\end{equation}

Furthermore, given a commutative algebra $(A, \mu,\eta)$ in $\C$, we say that a morphism $[ \hskip 0.6mm , ] : A \otimes_{\C} A \rightarrow A$ is a \textbf{\textcolor{myblue}{bracket}} if it satisfies 
\begin{align}
\label{eq:axioms-bracket}
[ \hskip 0.6mm , ] \circ \tau^{\C}(A,A) &= -[ \hskip 0.6mm , ] ,
\\
[ \hskip 0.6mm ,] \circ (\id_{A} \otimes_{\C} \mu) 
&= \mu \circ \big([ \hskip 0.6mm ,] \otimes_{\C} \id_{A}\big) + \mu \circ \big(\id_{A} \otimes_{\C} [ \hskip 0.6mm ,]\big) \circ (\tau^{\C}(A,A) \otimes_{\C} \id_{A}).
\label{eq:leibniz}
\end{align}
The second equation is usually called the
\textbf{\textcolor{myblue}{Leibniz identity}}. 
Finally, a \textbf{\textcolor{myblue}{Poisson algebra}} in $\C$ is a commutative algebra $(A,\mu,\eta)$ in $\C$ endowed with a bracket $[ \hskip 0.6mm ,] : A \otimes_{\C} A \rightarrow A$ in $\C$ such that 
$(A , [ \hskip 0.6mm ,])$ is a Lie algebra in $\C$. 
A \textbf{\textcolor{myblue}{morphism of Poisson algebras}} is both a morphism of algebras and a morphism of Lie algebras.  

%%%%%%%
\begin{example} 
\label{example:sym-poisson} 
Assume that $\C$ is abelian. 
Let $M$ be an object in $\C$ and let $\Sym_{\C}(M)$ be the symmetric algebra on $M$ recalled in Example \ref{example:sym-alg}. 
Then, it is easy to verify that the map 
\begin{equation}
\begin{split}
 \big\{ &\beta : \Sym_{\C}(M)^{\otimes_{\C} 2} \rightarrow \Sym_{\C}(M) \mid \beta \text{ is a bracket} \big\} 
 \\
& \longrightarrow \big\{ g : M^{\otimes_{\C} 2} \rightarrow \Sym_{\C}(M) \mid g \circ \tau^{\C}(M,M) = - g  \big\},
\end{split}
\end{equation}
sending $\beta$ to its restriction to $M^{\otimes_{\C} 2} \subseteq \Sym_{\C}(M)^{\otimes_{\C} 2}$, is a bijection. 
\end{example}
%%%%%%%

%%%%%%%%%%%%%%%%%%%%%%%%%%%%%%%%%%%%%%%%%%%%%%%%%%%%%%%%%%%%%%%%%%%%%%
\subsection{Restriction and induction} 
\label{subsection:res-ind}

In this subsection we recall basic facts about the restriction and induction functors, since we will use them in the sequel. 
We will present them for right modules, but the analogous results also hold for left modules. 

Recall that if $\varphi : A \rightarrow A'$ is a morphism of algebras in $\C$, we can define a functor 
\begin{equation}
	\label{eq:res}
	\Res_{\varphi} : \Mod_{A'}(\C) \longrightarrow \Mod_{A}(\C)
\end{equation}	
that sends an $A'$-module $(M,\rho)$ to the $A$-module $\big(M,\rho \circ (\id_{M} \otimes_{\C} \varphi)\big)$, 
and it is the identity on morphisms. 
A standard computation shows that $\Res_{\varphi}$ has the left adjoint 
\begin{equation}
	\label{eq:ind}
	\Ind_{\varphi} = (-) \otimes_{A} A'  : \Mod_{A}(\C) \longrightarrow \Mod_{A'}(\C),
\end{equation}	
where $A'$ has the $A'$-$A$-bimodule structure morphism $\mu_{A'}^{[3]} \circ (\varphi \otimes_{\C} \id_{A'} \otimes_{\C} \id_{A'})$. 
Indeed, given a right $A$-module $M$ and a right $A'$-module $N$ with action $\rho_{N}$, the morphism 
\begin{equation}
	\label{eq:ind-res} 
	\Phi_{\varphi}(M,N) : \Hom_{\text{-} A}\big(M,\Res_{\varphi}(N)\big) \longrightarrow \Hom_{\text{-}A'}\big(\Ind_{\varphi}(M),N\big)
\end{equation}		
sending $f \in \Hom_{\text{-}A}(M,\Res_{\varphi}(N))$ to the unique map $\Phi_{\varphi}(M,N)(f) : M \otimes_{A} A' \rightarrow N$ such that the composition of the canonical projection $q' : M \otimes_{\C} A' \rightarrow M \otimes_{A} A'$ with $\Phi_{\varphi}(M,N)(f)$ coincides with 
$\rho_{N} \circ (f \otimes_{\C} \id_{A'})$. 
Its inverse is given by 
\[     
\Phi_{\varphi}(M,N)^{-1}(f') = f' \circ (\id_{M} \otimes_{A} \varphi) \circ q \circ  (\id_{M} \otimes_{\C} \eta_{A})     
\]
for $f' \in \Hom_{\text{-}A'}(\Ind_{\varphi}(M),N)$, where $q : M \otimes_{\C} A \rightarrow M \otimes_{A} A$ is the canonical projection. 
The naturality of $\Phi_{\varphi}(M,N)$ for $M \in \Mod_{A}(\C)$ and $N \in \Mod_{A'}(\C)$ is immediate. 
If $\C$ is the category of vector spaces, $f$ is a monomorphism and $A'$ is a free right $A$-module via $f$, we will simply denote the canonical monomorphism $M \rightarrow \Ind_{\varphi}(M)$ 
given by $\id_{M} \otimes_{A} f$ 
as a mere inclusion.

\begin{fact} 
\label{fact:nat-res-ind}
Given a pair of morphisms of algebras, $\varphi : A' \rightarrow A$ and $\psi : A \rightarrow A''$, there is a unique natural isomorphism 
$\phi^{\operatorname{I}}(\psi,\varphi) : \Ind_{\psi} \circ \Ind_{\varphi} \rightarrow \Ind_{\psi \circ \varphi}$ of functors coming from the fact that both are left adjoints of $\Res_{\varphi} \circ \Res_{\psi} = \Res_{\psi \circ \varphi}$. 
\end{fact}

Let $\phi : H \rightarrow G$ be a morphism of finite groups, and let $\varphi = \Bbbk \phi : \Bbbk H \rightarrow \Bbbk G$ be the induced morphism of algebras. 
As noted in the previous paragraph, this induces the pair of adjoint functors
\[
\Ind_{\varphi} : \Mod_{\Bbbk H} \longrightarrow \Mod_{\Bbbk G}\quad \text{ and } \quad \Res_{\varphi} : \Mod_{\Bbbk G} \longrightarrow \Mod_{\Bbbk H}.
\]
Moreover, as it is well-known, the left adjoint functor $\Ind_{\varphi} : \Mod_{\Bbbk H} \rightarrow \Mod_{\Bbbk G}$ is also right adjoint to $\Res_{\varphi} : \Mod_{\Bbbk G} \rightarrow \Mod_{\Bbbk H}$. 
To show this, fix a set $\texttt{S} \subseteq G$ such that the composition of the inclusion and the canonical projection $G \rightarrow G/H$ is a bijection, and, given a right $\Bbbk H$-module $M$, we consider the map 
\begin{equation}
	\label{eq:p-M} 
	     \mathfrak{p}_{M} : \Res_{\varphi}\big(\Ind_{\varphi}(M)\big) \longrightarrow M      
\end{equation}		
such that $\mathfrak{p}_{M}(m \otimes_{\Bbbk H} g) = m g$ if $g \in H$ and $\mathfrak{p}_{M}(m \otimes_{\Bbbk H} g) = 0$ if $g \in G \setminus H$. 
It is clear that $\mathfrak{p}_{M}$ is a morphism of right $\Bbbk H$-modules.
Now, given a right $\Bbbk H$-module $M$ and a right $\Bbbk G$-module $N$, the morphism 
\begin{equation}
	\label{eq:res-ind} 
	\Psi_{\varphi}(N,M) : \Hom_{\text{-}\Bbbk H}\big(\Res_{\varphi}(N),M\big) \longrightarrow \Hom_{\text{-}\Bbbk G}\big(N,\Ind_{\varphi}(M)\big)
\end{equation}		
given by 
\[     \Psi_{\varphi}(N,M)(f)(n) = \sum_{g \in \texttt{S}} f(n g) \otimes_{\Bbbk H} g^{-1}     \] 
for $f \in \Hom_{\text{-}\Bbbk H}(\Res_{\varphi}(N),M)$ and $n \in N$ is independent of $\texttt{S}$ and bijective, 
since its inverse is 
\[
\Psi_{\varphi}(N,M)^{-1}(f')(n) =  \mathfrak{p}_{M}(f'(n))
\]
for $f' \in \Hom_{\text{-}\Bbbk H}(N,\Ind_{\varphi}(M))$ and $n \in N$. 
The naturality of $\Psi_{\varphi}(N,M)$ for $N \in \Mod_{\Bbbk H}$ and $M \in \Mod_{\Bbbk G}$ is immediate. 

%%%%%%%%%%%%%%%%%%%%%%%%%%%%%%%%%%%%%%%%%%%%%%%%%%%%%%%%%%%%%%%%%%%%%%
\subsection{Basic constructions for symmetric groups}
\label{subsection:basic-gr}

If $n \in \NN_{0}$ and $G = \SG_{n}$, we will denote the map \eqref{eq:group-diag} 
simply by $\double_{n} : \SG_{n} \rightarrow \SG_{n}^{\env} \label{eq:group-diag-symm-gp}$. 
This induces the morphism of algebras 
\begin{equation} 
	\label{eq:diag-n} 
	\Delta_{n} : \Bbbk \SG_{n} \longrightarrow \Bbbk \SG_{n}^{\env} \cong (\Bbbk \SG_{n})^{\env}
\end{equation} 
given as composition of $\Bbbk \double_{n}$ and the isomorphism 
recalled in the penultimate paragraph of Subsubsection \ref{subseubsection:mono-alg}.
More explicitly, $\Delta_{n}$ sends $\sigma \in \SG_{n}$ to $\sigma \otimes \sigma^{-1}$. 
We will also consider the morphism of algebras 
\begin{equation} 
	\label{eq:diag-n,m} 
	\Delta_{n,m} : \Bbbk \SG_{n} \otimes \Bbbk \SG_{m} \longrightarrow (\Bbbk \SG_{n} \otimes \Bbbk \SG_{m}) \otimes (\Bbbk \SG_{n})^{\op} \otimes (\Bbbk \SG_{m})^{\op} = (\Bbbk \SG_{n} \otimes \Bbbk \SG_{m})^{\env}
\end{equation} 
given by 
\[     \Delta_{n,m} = \Big(\id_{\Bbbk \SG_{n}} \otimes \tau\big((\Bbbk \SG_{n})^{\op},\Bbbk \SG_{m}\big) \otimes \id_{(\Bbbk \SG_{m})^{\op}}\Big) \circ (\Delta_{n} \otimes \Delta_{m}).     
\label{diagonal-morphism-n-m}
\]

Recall that, given integers $1\leq k\leq n$ and an injective map $\iota : \llbracket 1, k \rrbracket \rightarrow \llbracket 1, n \rrbracket$, 
the \textbf{\textcolor{myblue}{cycle}} associated with $\iota$ is the unique permutation $(\iota(1) \ \dots \ \iota(k)) \in \SG_{n}$ that is the identity on $\llbracket 1, n \rrbracket \setminus \Img(\iota)$ and sends $\iota(i)$ to $\iota(i+1)$ for all $i \in \llbracket 1, k - 1 \rrbracket$. 
In particular, given integers $1\leq i \leq j \leq n$, 
we will denote by $(i \ \dots \ j) \in \SG_{n}\label{cycle}$ the cycle associated with the injective map $\iota : \llbracket 1, j-i+1 \rrbracket \rightarrow \llbracket 1, n \rrbracket$ sending $k$ to $k+i-1$. 
The following result is immediate and its proof is left to the reader. 
%%%%%%%
\begin{fact}
\label{fact:cases}
Let $n \in \NN$. 
Given positive integers $i \leq n$ and $k \leq n + 1$. 
Then,
\begin{enumerate}[label=(C.\arabic*)]
    \item\label{item:C1} if $i < k$, we have
        $(n \ n+1) = (i \ \dots \ n)^{-1} (k \ \dots \ n+1)^{-1} (i \ \dots \ n+1) (k-1 \ \dots \ n)$;
    \item\label{item:C2} and if $i \geq k$, 
        $(n \ n+1) = (i \ \dots \ n)^{-1} (k \ \dots \ n+1)^{-1} (i + 1 \ \dots \ n+1) (k \ \dots \ n)$.
\end{enumerate}
\end{fact}
%%%%%%%

To simplify, from now on, given $\bar{n} = (n_{1}, \dots,n_{\ell}) \in \NN_{0}^{\ell}\label{tuple}$ for $\ell \in \NN$, we will usually denote by $\Bbbk \SG_{\bar{n}}\label{tesnor-product-group-algebras}$ the tensor product of group algebras $\Bbbk \SG_{n_{1}} \otimes \dots \otimes \Bbbk \SG_{n_{\ell}}$. 
Let $n \in \NN$ and $\bar{m} = (m_{1}, \dots,m_{n}) \in \NN_{0}^{n}$ with $m = m_{1} + \dots + m_{n}$. 
Recall the map
\begin{equation} 
	\label{eq:inc-1} 
	\sump_{m_{1}, \dots,m_{n}} : \SG_{m_{1}} \times \dots \times \SG_{m_{n}} \longrightarrow \SG_{m}
\end{equation}
sending $(\sigma_{1}, \dots,\sigma_{n}) \in \SG_{m_{1}} \times \dots \times \SG_{m_{n}}$ to the unique permutation $\sigma \in \SG_{m}$ given by
\begin{equation} 
	\label{eq:inc-1-secpart} 
	\sigma\bigg(k + \sum_{i=1}^{j-1} m_{i}\bigg) = \sigma_{j}(k) + \sum_{i=1}^{j-1} m_{i}, 
\end{equation}
for all $j \in \llbracket 1, n \rrbracket$ and $k \in \llbracket 1, m_{j} \rrbracket$. 
The permutation $\sigma$ is called the \textbf{\textcolor{myblue}{ordered sum of the permutations $\sigma_{1}, \dots,\sigma_{n}$}}. 
It is easy to see that $\sump_{m_{1}, \dots,m_{n}}$ is a 
morphism of groups, so it induces a morphism of $\Bbbk$-algebras 
\begin{equation} 
	\label{eq:inc-1-alg} 
	\Bbbk \sump_{m_{1}, \dots,m_{n}} : \Bbbk\SG_{\bar{m}} \longrightarrow \Bbbk\SG_{m},\quad \sigma_{1} \otimes \dots \otimes \sigma_{n}\longmapsto \sump_{m_{1}, \dots,m_{n}}(\sigma_{1}, \dots,\sigma_{n}).
\end{equation}

Let $n \in \NN$ and $\bar{m} = (m_{1}, \dots,m_{n}) \in \NN^{n}_{0}$ with $m = m_{1} + \dots + m_{n}$. 
We consider the map 
\begin{equation} 
	\label{eq:inc-2} 
	\block_{m_{1}, \dots,m_{n}} : \SG_{n} \longrightarrow \SG_{m}
\end{equation}
sending $\tau \in \SG_{n}$ to the unique permutation $\sigma \in \SG_{m}$ given by
\begin{equation} 
	\label{eq:inc-2-secpart} 
	\sigma\bigg(k + \sum_{i=1}^{j-1} m_{i}\bigg) = k + \sum_{i=1}^{\tau(j)-1} m_{\tau^{-1}(i)}, 
\end{equation}
for all $j \in \llbracket 1, n \rrbracket$ and $k \in \llbracket 1, m_{j} \rrbracket$. 
The permutation $\sigma$ is called the \textbf{\textcolor{myblue}{block permutation of $\tau$ with respect to $(m_{1}, \dots, m_{n}) \in \NN^{n}_{0}$}}. 

%%%%%%%
\begin{remark}
	\label{remark:rem-inc-2}
	The block permutation can be concretely described as follows. Given totally ordered finite nonempty sets $I_{1}, \dots, I_{n}$, denote by $I_{1} < \dots < I_{n}$ the disjoint union $I = \sqcup_{i=1}^{n} I_{i}$ endowed with the (total) order such that the inclusions $I_{i} \rightarrow I$ are order-preserving for all $i \in \llbracket 1, n \rrbracket$ and $x < y$ if $x \in I_{i}$ and $y \in I_{j}$ with $i < j$. 
	The block permutation $\block_{m_{1}, \dots,m_{n}}(\tau)$ is the unique order-preserving map from $I_{\tau^{-1}(1)} < \dots < I_{\tau^{-1}(n)}$ to $\llbracket 1, \sum_{j=1}^{n} m_{j} \rrbracket$, where $I_{i} = \llbracket 1+ \sum_{j=1}^{i-1} m_{j}, \sum_{j=1}^{i} m_{j} \rrbracket$ is endowed with the order induced by that of $\NN$. 
	One usually says that the block permutation of $\tau$ is the permutation obtained by permuting the blocks $I_{1}, \dots, I_{n}$ according to $\tau$, although this is somewhat ambiguous. 
 \end{remark} 

 \begin{remark}
	Note that $\block_{m_{1}, m_{2}}(1 \ 2)$ 
	is denoted by $(1 \ 2)^{m_{2},m_{1}}$ in \cite{MR2734329}. 
\end{remark} 
%%%%%%%

Although the maps $\block_{m_{1}, \dots,m_{n}}$ are not morphisms of groups in general, they satisfy the following nice property, whose proof directly follows from the definition:
\begin{equation} 
	\label{eq:prop-2} 
	\block_{m_{\tau^{-1}(1)}, \dots,m_{\tau^{-1}(n)}}(\tau') \circ \block_{m_{1}, \dots,m_{n}}(\tau) = 
	\block_{m_{1}, \dots,m_{n}}(\tau' \circ \tau),
\end{equation}
for all $\tau, \tau' \in \SG_{n}$, $(m_{1}, \dots,m_{n}) \in \NN_{0}^{n}$ and $n \in \NN$. 
Note in particular that the previous identity tells us that 
$\block_{m_{1}, \dots,m_{n}}\big(\id_{\llbracket 1 , n \rrbracket}\big) = \id_{\llbracket 1 , m \rrbracket}$. 
Moreover, the reader can verify the following elementary identity 
\begin{equation} 
	\label{eq:prop-1-2} 
	\begin{split}
		\sump_{m_{\sigma^{-1}(1)}, \dots,m_{\sigma^{-1}(n)}}&(\tau_{\sigma^{-1}(1)},\dots, \tau_{\sigma^{-1}(n)}) \circ \block_{m_{1}, \dots,m_{n}}(\sigma) 
		\\ 
		&=  \block_{m_{1}, \dots,m_{n}}(\sigma) \circ 
		\sump_{m_{1}, \dots,m_{n}}(\tau_{1},\dots, \tau_{n}), 
	\end{split}
\end{equation}
for all $\tau_{i} \in \SG_{m_{i}}$, $i \in \llbracket 1 , n \rrbracket$, $\sigma \in \SG_{n}$ and $n \in \NN$. 

Given $n \in \NN_{0}$, a set $\texttt{S}\label{index:set-S}$, and a vector space $V$, we also recall that the natural left and right actions of $\SG_{n}$ on $\texttt{S}^{n}$ and $V^{\otimes n}$ are given by 
\begin{align} 
	\label{eq:act-set} 
	     \sigma \cdot (s_{1}, \dots, s_{n}) &=  (s_{\sigma^{-1}(1)}, \dots, s_{\sigma^{-1}(n)}) = (s_{1}, \dots, s_{n}) \cdot \sigma^{-1},  
      \\
	\label{eq:act-tens} 
	     \sigma \cdot (v_{1} \otimes \dots \otimes v_{n}) &=  v_{\sigma^{-1}(1)} \otimes \dots \otimes v_{\sigma^{-1}(n)} = (v_{1} \otimes \dots \otimes v_{n}) \cdot \sigma^{-1},     
\end{align}
for $\sigma \in \SG_{n}$, $s_{1}, \dots, s_{n} \in \texttt{S}$ and $v_{1}, \dots, v_{n} \in V$. 

%%%%%%%%%%%%%%%%%%%%%%%%%%%%%%%%%%%%%%%%%%%%%%%%%%%%%%%%%%%%%%%%%%%%%%
\section{\texorpdfstring{$\SG$-modules and diagonal $\SG$-bimodules}{S-modules and diagonal S-bimodules}}
\label{subsection:s-mod}
%%%%%%%%%%%%%%%%%%%%%%%%%%%%%%%%%%%%%%%%%%%%%%%%%%%%%%%%%%%%%%%
\subsection{Generalities}
\label{subsec:generalities-s-mod}

\subsubsection{\texorpdfstring{Definitions of $\SG$-modules and diagonal $\SG$-bimodules. Basic properties}{Definitions of S-modules and diagonal S-bimodules, and basic properties}}

We recall that an \textbf{\textcolor{myblue}{$\SG$-module}} $S$ is a sequence $( S(n) )_{n \in \NN_{0}}$ such that $S(n)$ is a right $\Bbbk\SG_{n}$-module 
for all $n \in \NN_{0}$. 
Moreover, a \textbf{\textcolor{myblue}{morphism of $\SG$-modules}} $S \rightarrow S'$ is a sequence $ f = ( f(n) )_{n \in \NN_{0}}$ such that $f(n) : S(n) \rightarrow S'(n)$ is a morphism of right $\Bbbk\SG_{n}$-modules 
for all $n \in \NN_{0}$. 
We will denote by $\label{eq: srt-morphisms-Mod-S}\Hom_{\SG}(S,S')$ 
the set of morphisms of $\SG$-modules from $S$ to $S'$. 
We define the \textbf{\textcolor{myblue}{composition}} of two morphisms $f : S \rightarrow S'$ and $g : S' \rightarrow S''$ of $\SG$-modules as the sequence $( g(n) \circ f(n) )_{n \in \NN_{0}}$; we will simply denote this composition by $g \circ f$. 
It is clear that $\SG$-modules with the previous morphisms and compositions form a category, that we will denote by $\SMod \label{def-categ-S-mod}$, which is abelian. 

%%%%%%%
\begin{example}
\label{example:ten-1}
Let $V$ be a vector space. 
Given $n \in \NN_{0}$, define the right $\Bbbk\SG_{n}$-module $\TT_{V}(n) = V^{\otimes n}$, where we recall that $V^{\otimes 0} = \Bbbk$, with action given by the second equality of \eqref{eq:act-tens}. 
Therefore, $\TT_{V} = ( \TT_{V}(n) )_{n \in \NN_{0}}$ is an $\SG$-module. 
\end{example}
%%%%%%%

%%%%%%%
\begin{example} 
\label{example:end-1}
Let $V$ be a vector space. 
Given $n \in \NN_{0}$, define the right $\Bbbk\SG_{n}$-module $\mathbb{E}_{V}(n) = \Hom(V^{\otimes n}, V)$ with action given by 
\[    (f \cdot \sigma)(v_{1} \otimes \dots \otimes v_{n}) = f  (v_{\sigma^{-1}(1)} \otimes \dots \otimes v_{\sigma^{-1}(n)})    \] 
for $f \in \Hom(V^{\otimes n}, V)$, $\sigma \in \SG_{n}$ and $v_{1}, \dots, v_{n} \in V$. 
Then, $\mathbb{E}_{V} = ( \mathbb{E}_{V}(n) )_{n \in \NN_{0}}$ is an $\SG$-module. 
\end{example}
%%%%%%% 

Similarly, a \textbf{\textcolor{myblue}{diagonal $\SG$-bimodule}} $S$ is a sequence $( S(n) )_{n \in \NN_{0}}$ such that $S(n)$ is a $\Bbbk\SG_{n}$-bimodule 
for all $n \in \NN_{0}$. 
Moreover, a \textbf{\textcolor{myblue}{morphism of diagonal $\SG$-bimodules}} $S \rightarrow S'$ is a sequence $f = ( f(n) )_{n \in \NN_{0}}$ such that $f(n) : S(n) \rightarrow S'(n)$ is a morphism of $\Bbbk\SG_{n}$-bimodules 
for all $n \in \NN_{0}$. 
We will denote by $\label{eq:set-morphims-DMod-S}\Hom_{\SG^{\env}}(S,S')$ 
the set of morphisms of diagonal $\SG$-modules from $S$ to $S'$. 
As in the case of $\SG$-modules, we define the \textbf{\textcolor{myblue}{composition}} of two morphisms $f : S \rightarrow S'$ and $g : S' \rightarrow S''$ of diagonal $\SG$-bimodules as the sequence $( g(n) \circ f(n) )_{n \in \NN_{0}}$, which we will simply denote as $g \circ f$. 
It is clear that diagonal $\SG$-bimodules with the previous morphisms and compositions form a category, that we will denote by $\DMod \label{def-categ-S-bimod}$. 
Moreover, it is easy to verify that $\DMod$ is abelian. 

%%%%%%%
\begin{example} 
\label{example:ten-2}
	Let $V$ be a vector space. 
	Given $n \in \NN_{0}$, let $\SG_{V}^{\env}(n) = \Ind_{\Delta_{n}} (\mathbb{T}_{V}(n))$, where $\mathbb{T}_{V}(n)$ was defined in Example \ref{example:ten-1} and $\Delta_{n}$ is the morphism of algebras given in \eqref{eq:diag-n}. 
	Then, $\SG_{V}^{\env} = ( \SG_{V}^{\env}(n) )_{n \in \NN_{0}}$ is a diagonal $\SG$-bimodule.
\end{example} 
%%%%%%%

%%%%%%%
\begin{example} 
\label{example:end-2}
	Let $V$ be a vector space. 
	Given $n \in \NN_{0}$, define the $\Bbbk\SG_{n}$-bimodule $\EE_{V}^{\env}(n) = \Hom(V^{\otimes n}, V^{\otimes n})$ with action given by 
	\[    (\sigma' \cdot f \cdot \sigma)(v_{1} \otimes \dots \otimes v_{n}) = \sigma' \cdot \big(f  (v_{\sigma^{-1}(1)} \otimes \dots \otimes v_{\sigma^{-1}(n)})\big)    \] 
	for $f \in \Hom(V^{\otimes n}, V^{\otimes n})$, $\sigma, \sigma' \in \SG_{n}$ and $v_{1}, \dots, v_{n} \in V$, where the action of $\sigma'$ is given by \eqref{eq:act-tens}. 
	Then, $\EE_{V}^{\env} = ( \EE_{V}^{\env}(n) )_{n \in \NN_{0}}$ is a diagonal $\SG$-bimodule. 

    It is convenient to write the previous action explicitly in the case when $V$ is finite dimensional. 
    Let $\{ e_{\alpha} : \alpha \in \texttt{A} \}$ be a distinguished basis of $V$. 
    Then, $\{ e_{\bar{\alpha}} : \bar{\alpha} \in \texttt{A}^{n} \}$ is a basis of $V^{\otimes n}$, where $e_{\bar{\alpha}} = e_{\alpha_{1}} \otimes \dots \otimes e_{\alpha_{n}} \in V^{\otimes n}$ for $\bar{\alpha} = (\alpha_{1},\dots,\alpha_{n}) \in \texttt{A}^{n}$.
    Given $\bar{\alpha}, \bar{\beta} \in \texttt{A}^{n}$ we define $E_{\bar{\alpha}, \bar{\beta}} \in \End(V^{\otimes n})$ as the unique linear map sending $e_{\bar{\beta}}$ to $e_{\bar{\alpha}}$, and $e_{\bar{\beta}'}$ to zero for all $\bar{\beta}' \in \texttt{A}^{n} \setminus \{ \bar{\beta} \}$. 
    Then, it is easy to see that 
	\begin{equation}
 \label{eq:act-end}
     \sigma' \cdot E_{\bar{\alpha}, \bar{\beta}} \cdot \sigma = E_{\sigma' \cdot \bar{\alpha}, \bar{\beta} \cdot \sigma},    
    \end{equation}  
for all $\sigma, \sigma' \in \SG_{n}$ and $\bar{\alpha}, \bar{\beta} \in \texttt{A}^{n}$, where we used the notation from \eqref{eq:act-set}. 
\end{example} 
%%%%%%%

%%%%%%%
\begin{example} 
\label{example:ten-3}	
    Let $A$ be an algebra over $\Bbbk$. 
	Given $n \in \NN_{0}$ consider the map 
	\[     \mathbbl{i}_{A}'(n) : \mathbb{T}_{A}(n) \longrightarrow \Res_{\Delta_{n}}\big(\mathbb{E}_{A}^{\env}(n)\big)      \]
    given by 
    \[     \mathbbl{i}_{A}'(n)(a_{1} \otimes \dots \otimes a_{n})(b_{1} \otimes \dots \otimes b_{n}) = 
    a_{1} b_{1} \otimes \dots \otimes a_{n} b_{n}     \]
    for all $a_{1}, b_{1}, \dots, a_{n}, b_{n} \in A$. 
    Then, $\mathbbl{i}_{A}'(n)$ is a morphism of $\Bbbk \SG_{n}$-modules for all $n \in \NN_{0}$, since 
    \begin{align*}
        \big(\sigma^{-1} \cdot \mathbbl{i}_{A}'(n)&(a_{1} \otimes \dots \otimes a_{n}) \cdot \sigma \big)(b_{1} \otimes \dots \otimes b_{n}) 
        = a_{\sigma(1)} b_{1} \otimes \dots \otimes a_{\sigma(n)} b_{n} 
        \\ 
        &= \mathbbl{i}_{A}'(n)\big((a_{1} \otimes \dots \otimes a_{n}) \cdot \sigma\big)(b_{1} \otimes \dots \otimes b_{n})
    \end{align*}
	for all $a_{1}, b_{1}, \dots, a_{n}, b_{n} \in A$ and 
	$\sigma \in \SG_{n}$. 
	By the adjunction recalled in Subsection \ref{subsection:res-ind}, $\mathbbl{i}_{A}'(n)$ induces a morphism of $\Bbbk \SG_{n}$-bimodules 
	\[     \mathbbl{i}_{A}(n) : \SG_{A}^{\env}(n) = \Ind_{\Delta_{n}}\big(\mathbb{T}_{A}(n)\big) \longrightarrow \mathbb{E}_{A}^{\env}(n)      \]
	for all $n \in \NN_{0}$, so we have a morphism of diagonal $\SG$-bimodules $\mathbbl{i}_{A} : \SG_{A}^{\env} \rightarrow \mathbb{E}_{A}^{\env}$. 
	A simple verification shows that $\mathbbl{i}_{A}(n)$ is injective for all $n \in \NN_{0}$, so we can regard $\SG_{A}^{\env}$ as a diagonal sub-$\SG$-bimodule of 
	$\mathbb{E}_{A}^{\env}$. 
\end{example}

The categories $\SMod$ and $\DMod$ have direct products and coproducts.
Indeed, given an arbitrary family $\{ S_{i} : i \in \texttt{I} \}$ of $\SG$-modules (resp., diagonal $\SG$-bimodules), the product $\label{eq:categorical-product}\prod_{i \in \texttt{I}} S_{i}$ is the 
$\SG$-module (resp., diagonal $\SG$-bimodule) with $n$-th component given by the right $\Bbbk\SG_{n}$-module (resp., $\Bbbk\SG_{n}$-bimodule)
\[     \bigg(\prod_{i \in \texttt{I}} S_{i}\bigg)(n) = \prod_{i \in \texttt{I}} S_{i}(n),     \]
for all $n \in \NN_{0}$, and the coproduct $\label{de-coproduct} \coprod_{i \in \texttt{I}} S_{i}$ is the 
$\SG$-module (resp., diagonal $\SG$-bimodule) whose $n$-th component is the right $\Bbbk\SG_{n}$-module (resp., $\Bbbk\SG_{n}$-bimodule)
\[     \bigg(\coprod_{i \in \texttt{I}} S_{i}\bigg)(n) = \bigoplus_{i \in \texttt{I}} S_{i}(n),     \]
for all $n \in \NN_{0}$. 

For convenience we recall the notion of  \textbf{\textcolor{myblue}{intersection}} \label{eq:intersection-s-mod-def} (in the sense of \cite{MR0202787}, I.8) for $\SG$-modules and diagonal $\SG$-bimodules. 
Let $\{ S^{i} = (S^{i}(n))_{n \in \NN_{0}} : i \in I \}$ be an arbitrary family of $\SG$-submodules (resp., diagonal $\SG$-subbimodules) of $S$, 
\textit{i.e.} $S^{i}(n) \subseteq S(n)$ is a $\Bbbk\mathbb{S}_{n}$-submodule (resp., $\Bbbk\mathbb{S}_{n}$-subbimodule) of $S(n)$ for all $n \in \NN_{0}$.
Then, the \textbf{\textcolor{myblue}{intersection}} $\cap_{i \in I} S^{i}$ of the previous family 
is the $\SG$-submodule (resp., diagonal $\SG$-subbimodule) of $S$ is given by 
\begin{equation}          
\label{eq:intersection-s-mod}
     \bigg(\bigcap_{i \in I} S^{i}\bigg)(n) = \bigcap_{i \in I} S^{i}(n)
\end{equation}
for all $n \in \NN_{0}$.

%%%%%%%%%%%%%%%%%%%%%%%%%%%%%%%%%%%%%%%%%%%%%%%%%%%%%%%%
\subsubsection{Categorical properties: symmetric monoidal structures}

Recall the map $\kk\sump_{m_1,m_2}$ defined in \eqref{eq:inc-1-alg}. 
Given $\SG$-modules $S_{1}$ and $S_{2}$, the
\textbf{\textcolor{myblue}{Cauchy tensor product}} $S_{1} \otimes_{\SG} S_{2}$ is the $\SG$-module given by 
\begin{equation} 
\label{eq:cauchy} 
\begin{split}
     \Big(S_{1} \otimes_{\SG} S_{2}\Big)(m) 
     &= \bigoplus_{\text{\begin{tiny}$\begin{matrix}(m_{1},m_{2}) \in \NN_{0}^{2} \\ m_{1} + m_{2} = m \end{matrix}$\end{tiny}}} \Ind_{\Bbbk \sump_{m_{1},m_{2}}} \big(S_{1}(m_{1}) \otimes S_{2}(m_{2})\big),     
\end{split}
\end{equation}
where $S_{1}(m_{1}) \otimes S_{2}(m_{2})$ is a right module over $\Bbbk\SG_{m_{1}} \otimes \Bbbk\SG_{m_{2}}$ with the usual action 
$(v_{1} \otimes v_{2}) \cdot (\sigma_{1} \otimes \sigma_{2}) = (v_{1} \cdot \sigma_{1}) \otimes (v_{2} \cdot \sigma_{2})$ for $v_{i} \in S_{i}(m_{i})$, $\sigma_{i} \in \SG_{m_{i}}$ and $i \in \{ 1, 2 \}$,
and the right module over $\Bbbk\SG_{m}$ on the right-hand side of \eqref{eq:cauchy} is given as the direct sum of each of the right modules over $\Bbbk\SG_{m}$ 
indexed by $(m_{1},m_{2}) \in \NN_{0}^{2}$ such that $m_{1} + m_{2} = m$. 
It is clear that the Cauchy tensor product of 
$\SG$-modules commutes with arbitrary coproducts on each side. 

Similarly, given diagonal $\SG$-bimodules $S_{1}$ and $S_{2}$, we can define the 
\textbf{\textcolor{myblue}{Cauchy tensor product}} $S_{1} \otimes_{\SG^{\env}} S_{2}$ as the diagonal $\SG$-bimodule given by 
\begin{equation} 
	\label{eq:cauchydiagonal} 
	\begin{split}
		\Big(S_{1} \otimes_{\SG^{\env}} S_{2}\Big)(m) 
		&= \bigoplus_{\text{\begin{tiny}$\begin{matrix}(m_{1},m_{2}) \in \NN_{0}^{2} \\ m_{1} + m_{2} = m \end{matrix}$\end{tiny}}} \Ind_{\Bbbk \sump_{m_{1},m_{2}}^{\env}} \big(S_{1}(m_{1}) \otimes S_{2}(m_{2})\big),     
	\end{split}
\end{equation}
where $S_{1}(m_{1}) \otimes S_{2}(m_{2})$ is a bimodule over $\Bbbk\SG_{m_{1}} \otimes \Bbbk\SG_{m_{2}}$ with the usual action 
\begin{equation}
(\tau_{1} \otimes \tau_{2}) \cdot (v_{1} \otimes v_{2}) \cdot (\sigma_{1} \otimes \sigma_{2}) = (\tau_{1} \cdot v_{1} \cdot \sigma_{1}) \otimes (\tau_{2} \cdot v_{2} \cdot \sigma_{2})
\end{equation}
for $v_{i} \in S_{i}(m_{i})$, $\sigma_{i}, \tau_{i} \in \SG_{m_{i}}$ and $i \in \{ 1, 2 \}$,
and the bimodule structure over $\Bbbk\SG_{m}$ on the right-hand side of \eqref{eq:cauchydiagonal} is given as the direct sum of each of the bimodules over $\Bbbk\SG_{m}$ 
indexed by $(m_{1},m_{2}) \in \NN_{0}^{2}$ such that $m_{1} + m_{2} = m$. 
It is also clear that the Cauchy tensor product of diagonal $\SG$-bimodules commutes with arbitrary coproducts on each side. 

%%%%%%%
\begin{remark} 
\label{remark:elements}
Since we will deal with their elements often, let us emphasize that an element of $\Ind_{\Bbbk \sump_{m_{1},m_{2}}} (S_{1}(m_{1}) \otimes S_{2}(m_{2}))$ will be a linear combination of 
elements of the form 
\[
(v \otimes w) \otimes_{(\Bbbk\SG_{m_{1}} \otimes \Bbbk\SG_{m_{2}})} \sigma
\]
for $v \in S_{1}(m_{1})$, $w \in S_{2}(m_{2})$ and $\sigma \in \SG_{m}$. 
In a similar way, a general element of the space $\Ind_{\Bbbk \sump_{m_{1},m_{2}}^{\env}} (S_{1}(m_{1}) \otimes S_{2}(m_{2}))$ 
will be a linear combination of 
elements of the form 
\[
(v \otimes w) \otimes_{(\Bbbk\SG_{m_{1}} \otimes \Bbbk\SG_{m_{2}})^{\env}} (\sigma \otimes \sigma')
\]
for $v \in S_{1}(m_{1})$, 
$w \in S_{2}(m_{2})$ and $\sigma, \sigma' \in \SG_{m}$.
\end{remark}
%%%%%%%

The following result is now a simple exercise that we leave to the reader. 
%%%%%%%
\begin{lemma}
	\label{lemma:cauchy}
	The Cauchy tensor product $\otimes_{\SG}$ given in \eqref{eq:cauchy} defines a monoidal product on the category $\SG$-modules, whose unit is the unique $\SG$-module $\mathbf{1}_{\SG}$ such that $\mathbf{1}_{\SG}(n) = 0$ for all $n \in \NN$ and $\mathbf{1}_{\SG}(0)$ is $\Bbbk$. 
	Similarly, the Cauchy tensor product $\otimes_{\SG^{\env}}$ given in \eqref{eq:cauchydiagonal} defines a monoidal product on the category diagonal $\SG$-bimodules, whose unit is the unique diagonal $\SG$-bimodule $\mathbf{1}_{\SG^{\env}}$ such that $\mathbf{1}_{\SG^{\env}}(n) = 0$ for all $n \in \NN$ and $\mathbf{1}_{\SG^{\env}}(0)$ is $\Bbbk$.
\end{lemma} 
%%%%%%%

We also note the following trivial result, which is a direct consequence of the isomorphism \eqref{eq:ind-res} and which we will use in the sequel.  
%%%%%%%
\begin{fact}
\label{fact:morph-tensor-s-bimod}
Let $R$, $S$ and $T$ be three $\SG$-modules (resp., diagonal $\SG$-bimodules), and $\Bi_{\SG}(R,S;T)$ (resp., $\Bi_{\SG^{\env}}(R,S;T)$) be the set formed by the sequences of maps $\{ f_{n,m} : R(n) \otimes S(m) \rightarrow \Res_{\Bbbk \sump_{n,m}^{\env}}(T(n+m)) \}_{n,m \in \NN_{0}}$ such that $f_{n,m}$ is a morphism of right $(\Bbbk \SG_{n} \otimes \Bbbk \SG_{m})$-modules (resp., $(\Bbbk \SG_{n} \otimes \Bbbk \SG_{m})$-bimodules). 
Consider the map 
\begin{equation}
\label{eq:morph-tensor-s-bimod}
\begin{split}
    &\Hom_{\SG} (R \otimes_{\SG} S , T) \longrightarrow \Bi_{\SG}(R,S;T)
    \\
    &\text{ \bigg(resp., } \Hom_{\SG^{\env}} (R \otimes_{\SG^{\env}} S , T) \longrightarrow \Bi_{\SG^{\env}}(R,S;T) \text{ \bigg),} 
\end{split}
\end{equation}
which sends a morphism $(f(N))_{N \in \NN_{0}}$ to the sequence $f = \{ f_{n,m} \}_{n,m \in \NN_{0}}$, which is 
given by $f_{n,m}(r \otimes s) = f(n+m)((r \otimes s) \otimes_{\Bbbk \SG_{n} \otimes \Bbbk \SG_{m}} \id_{\llbracket 1 , n+m \rrbracket})$ (resp., $f_{n,m}(r \otimes s) = f(n+m)((r \otimes s) \otimes_{(\Bbbk \SG_{n} \otimes \Bbbk \SG_{m})^{\env}} (\id_{\llbracket 1 , n+m \rrbracket} \otimes \id_{\llbracket 1 , n+m \rrbracket}))$) for $r \in R(n)$ and $s \in S(m)$.
Then the map \eqref{eq:morph-tensor-s-bimod} is a bijection. 
\end{fact} 
%%%%%%%

We recall that both the category of $\SG$-modules and diagonal $\SG$-bimodules are endowed with symmetric braidings $\tau^{\SG}$ and $\tau^{\SG^{\env}}$, respectively, which can be defined as follows. 
Given $\SG$-modules $S$ and $T$, let 
\begin{equation} 
	\label{eq:tau-s} 
	\tau^{\SG}(S,T) : S \otimes_{\SG} T \longrightarrow T \otimes_{\SG} S
\end{equation} 
be the morphism of $\SG$-modules whose restriction to $\Ind_{\Bbbk \sump_{m_{1},m_{2}}}(S(m_{1}) \otimes T(m_{2}))$ sends $(v \otimes w) \otimes_{\Bbbk \SG_{m_{1}} \otimes \Bbbk \SG_{m_{2}}} 1_{\Bbbk \SG_{m_{1}+m_{2}}}$ to 
$(w \otimes v) \otimes_{\Bbbk \SG_{m_{2}} \otimes \Bbbk \SG_{m_{1}}} \block_{m_{1},m_{2}}(1 \, 2)$ for $v \in S(m_{1})$ and $w \in T(m_{2})$. 
This map is well defined due to \eqref{eq:prop-1-2}.
Similarly, given diagonal $\SG$-bimodules $S$ and $T$, let 
\begin{equation} 
	\label{eq:tau-s-e} 
	\tau^{\SG^{\env}}(S,T) : S \otimes_{\SG^{\env}} T \longrightarrow T \otimes_{\SG^{\env}} S
\end{equation} 
be the morphism of diagonal $\SG$-bimodules whose restriction to the direct summand  $\Ind_{\Bbbk \sump_{m_{1},m_{2}}^{\env}}(S(m_{1}) \otimes T(m_{2}))$ sends $(v \otimes w) \otimes_{(\Bbbk \SG_{m_{1}} \otimes \Bbbk \SG_{m_{2}})^{\env}} 1_{\Bbbk \SG_{m_{1}+m_{2}}^{\env}}$ to 
the element $(w \otimes v) \otimes_{(\Bbbk \SG_{m_{2}} \otimes \Bbbk \SG_{m_{1}})^{\env}} (\block_{m_{1},m_{2}}(1 \, 2) \otimes \block_{m_{2},m_{1}}(1 \, 2))$ for $v \in S(m_{1})$ and $w \in T(m_{2})$. 
Again, \eqref{eq:prop-1-2} ensures that this map is well defined.

We collect the previous statements in the following result, whose proof is direct. 
%%%%%%%
\begin{lemma}
	\label{lemma:cauchy-2}
	The braiding \eqref{eq:tau-s} endows the monoidal category of $\SG$-modules in Lemma \ref{lemma:cauchy} with a symmetric monoidal structure. 
	Similarly, the braiding \eqref{eq:tau-s-e} endows the monoidal category of diagonal $\SG$-bimodules in Lemma \ref{lemma:cauchy} with a symmetric monoidal structure. 
\end{lemma} 
%%%%%%% 

We will now compare diagonal $\SG$-bimodules and $\SG$-modules. 
Recall the morphism of algebras $\Delta_{n}$ 
introduced in \eqref{eq:diag-n}.
Given a $\Bbbk \SG_{n}$-bimodule $M$ and using the identification between $\Bbbk \SG_{n}$-bimodules and 
right $(\Bbbk \SG_{n})^{\env}$-modules, $\Res_{\Delta_{n}}(M)$ is a right $\Bbbk \SG_{n}$-module, whose action is given explicitly by 
$m \cdot \sigma = \sigma^{-1} . m . \sigma$ for $m \in M$ and $\sigma \in \SG_{n}$. 
We can then define the functor 
\begin{equation} 
\label{eq:diag}
\Res : \DMod \longrightarrow \SMod
\end{equation} 
that sends a diagonal $\SG$-bimodule $S = (S(n))_{n \in \NN_{0}}$ to the $\SG$-module given by $\Res(S) = \big(\Res_{\Delta_{n}}(S)(n)\big)_{n \in \NN_{0}}$, 
and is the identity on the morphisms. 

%%%%%%%
\begin{lemma}
\label{lemma:res}
The functor $\Res$ defined in \eqref{eq:diag} is braided lax monoidal. 
\end{lemma} 
%%%%%%%
\begin{proof} 
Consider the structure morphisms
\begin{equation} 
\label{eq:diag-str}
		\varphi_{2}(S,T) : \Res(S) \otimes_{\SG} \Res(T) \longrightarrow \Res (S \otimes_{\SG^{\env}} T)  \text{ and } \varphi_{0} : \mathbf{1}_{\SG} \longrightarrow \Res(\mathbf{1}_{\SG^{\env}})   
\end{equation} 
where $\varphi_{0}$ is given by the identity of $\mathbf{1}_{\SG}$ and the morphism $\varphi_{2}(S,T)$ is given as follows. 
Consider $R_{1} = \Res_{\Delta_{m_{1}}}(S(m_{1}))$, $R_{2} = \Res_{\Delta_{m_{2}}}(T(m_{2}))$ and $R = \Res_{\Delta_{m_{1},m_{2}}}(S(m_{1}) \otimes T(m_{2}))$ 
for $m_{1}, m_{2} \in \NN_{0}$. 
Define the restriction of $\varphi_{2}(S,T)$ to $M = \Ind_{\Bbbk \sump_{m_{1},m_{2}}} (R_{1} \otimes R_{2})$ as 
$\Phi_{\Delta_{m_{1}+m_{2}}}(M,N)^{-1} (\Ind_{\Bbbk \sump^{\env}_{m_{1},m_{2}}}(\bar{\rho}))$, 
where the morphism $\bar{\rho} : \Ind_{\Delta_{m_{1},m_{2}}} (R_{1} \otimes R_{2}) \rightarrow S(m_{1}) \otimes T(m_{2})$ is 
$\Phi_{\Delta_{m_{1},m_{2}}}(R_{1} \otimes R_{2} , S(m_{1}) \otimes T(m_{2}))(\id)$, $\id : R_{1} \otimes R_{2} \rightarrow R$ is the morphism whose underlying set-theoretic map is the identity and $N = \Ind_{\Bbbk \sump_{m_{1},m_{2}}^{\env}} (S(m_{1}) \otimes T(m_{2}))$. 
We used that $\Delta_{m_{1} + m_{2}} \circ \Bbbk\sump_{m_{1},m_{2}} = \Bbbk\sump_{m_{1},m_{2}}^{\env} \circ \Delta_{m_{1},m_{2}}$ together with Fact \ref{fact:nat-res-ind}.
Taking into account our description of the elements of $M$ in Remark \ref{remark:elements}, the restriction of $\varphi_{2}(S,T)$ to the previous direct summand $M$ can be written more explicitly as the morphism of right $\Bbbk\SG_{m_{1}+ m_{2}} $-modules given by 
\[     \varphi_{2}\Big((v \otimes w) \otimes_{\Bbbk\SG_{m_{1}} \otimes \Bbbk\SG_{m_{2}}} \sigma \Big) = (v \otimes w) \otimes_{(\Bbbk\SG_{m_{1}} \otimes \Bbbk\SG_{m_{2}})^{\env}} (\sigma \otimes \sigma^{-1})      \] 
for all $v \in S(m_{1})$, $w \in T(m_{2})$ and $\sigma \in \SG_{m_{1}+m_{2}}$. 
The reader can verify that the axioms of a braided lax monoidal functor are satisfied (see \cite{MR2724388}, Def. 3.1 and 3.11).
\end{proof}

Given $i \in \{ 0 , 1\}$, consider the functor
\begin{equation}
\label{eq:vect-s-se} 
     \I_{\SG,i} : \Vect \longrightarrow \SMod \text{  $\Big($resp., } \I_{\SG^{\env},i} : \Vect \longrightarrow \DMod \text{$\Big)$}
\end{equation}
given by sending a vector space $V$ to the unique $\SG$-module (resp., diagonal $\SG$-bimodule) $S$ such that $S(n) = 0$ 
for $n \in \NN_{0} \setminus \{ i \}$ and $S(i) = V$, and a linear map $f : V \rightarrow V'$ to the unique morphism $f : S \rightarrow S'$ such that $f(i) = f$.

Recall the functor $\Res$ introduced in \eqref{eq:diag}.
The following result is immediate. 
%%%%%%%
\begin{fact}
\label{fact:vect-s-se} 
The functors $\I_{\SG,0}$ and 
$\I_{\SG^{\env},0}$ 
defined in \eqref{eq:vect-s-se} are symmetric strong monoidal, and $\Res \circ \I_{\SG^{\env},0} = \I_{\SG,0}$.
In particular, if $B$ is an (resp., commutative) algebra in $\Vect$, then 
$\I_{\SG^{\env},0}(B)$ is an (resp., commutative) algebra in $\DMod$. 
\end{fact} 
%%%%%%%

%%%%%%%%%%%%%%%%%%%%%%%%%%%%%
\subsubsection{\texorpdfstring{Shifted $\SG$-modules and diagonal $\SG$-bimodules}{Shifted S-modules and diagonal S-bimodules}}
\label{sec:shifted-SG-modules}
%%%%%%%%%%%%%%%%%%%%%%%%%%%%%%
Let 
\begin{equation}
     \linc_{N,N+k}, \rinc_{N,N+k} : \SG_{N} \longrightarrow \SG_{k+N}
     \label{eq:def-rinc}
\end{equation}        
be the morphisms of groups given, for $\sigma \in \SG_{N}$, by 
\[    \linc_{N,N+k}(\sigma)(i) = \begin{cases} 
\sigma(i), &\text{if $i \in \llbracket 1 , N \rrbracket$,}
\\
i, &\text{if $i \in \llbracket N + 1 , N + k \rrbracket$,}
\end{cases}
\]
and
\[    \rinc_{N,N+k}(\sigma)(i) = \begin{cases} 
i, &\text{if $i \in \llbracket 1 , k \rrbracket$,}
\\
\sigma(i-k) + k, &\text{if $i \in \llbracket k + 1 , k + N \rrbracket$.}
\end{cases}
\]

Given a $\SG$-module (resp., diagonal $\SG$-bimodule) $S = (S(n))_{n \in \NN_{0}}$ and $k \in \NN_{0}$, define the \textbf{\textcolor{myblue}{(right) shifted}} $\SG$-module (resp., diagonal $\SG$-bimodule) $\sh_{k}(S)$ (resp., $\sh_{k}^{\env}(S)$) by 
\[
\sh_{k}(S)(N) = \Res_{\Bbbk \rinc_{N,N+k}}(S(N+k))\quad \bigg(\text{resp., } \sh_{k}^{\env}(S)(N) = \Res_{\Bbbk \rinc_{N,N+k}^{\env}}(S(N+k)) \bigg)
\label{def-shifted-module-bimod}
\]
for $N \in \NN_{0}$. 
Similarly, given a morphism $f : S \rightarrow S'$ of $\SG$-modules (resp., diagonal $\SG$-bimodules), define the \textbf{\textcolor{myblue}{(right) shifted}} morphism of $\SG$-modules (resp., diagonal $\SG$-bimodules) $\sh_{k}(f) : \sh_{k}(S) \rightarrow \sh_{k}(S')$ (resp., $\sh_{k}^{\env}(f) : \sh_{k}^{\env}(S) \rightarrow \sh_{k}^{\env}(S')$) by $\sh_{k}(f)(N) = f(N+k)$ (resp., $\sh_{k}^{\env}(f)(N) = f(N+k)$) for $N \in \NN_{0}$. 
It is easy to see that this defines a functor 
\begin{equation}
\label{eq:shifted-s-mod}
    \sh_{k} : \SMod \longrightarrow \SMod \quad  \text{ \bigg(resp., $\sh_{k}^{\env} : \DMod \longrightarrow \DMod$ \bigg).}
\end{equation}

Given two $\SG$-modules (resp., diagonal $\SG$-bimodules) $S = (S(n))_{n \in \NN_{0}}$ and $S' = (S'(n))_{n \in \NN_{0}}$, we define 
the $\SG$-module (resp., diagonal $\SG$-bimodule) $\IHom_{\SG}(S,S')$ (resp., $\IHom_{\SG^{\env}}(S,S')$) satisfying that
\begin{equation}
\label{eq:int-hom-s-mod}
\begin{split}
    \IHom_{\SG}(S,S')(n) 
    &= \Hom_{\SG}\big(S,\sh_{n}(S')\big) 
     \\
     &= \prod_{m \in \NN_{0}} \Hom_{\Bbbk \SG_{m}}\Big(S(m),\Res_{\Bbbk \rinc_{m,n+m}}\big(S(n+m)\big)\Big)
     \\
     \text{\bigg(resp., }
     \IHom_{\SG^\env}(S,S')(n) 
     &= \Hom_{\SG^\env}\big(S,\sh_{n}^{\env}(S')\big) 
     \\
     &= \prod_{m \in \NN_{0}} \Hom_{\Bbbk \SG_{m}^{\env}}\Big(S(m),\Res_{\Bbbk \rinc_{m,n+m}^{\env}}\big(S(n+m)\big)\Big) \text{\bigg),}
\end{split}
\end{equation} 
where the  right action of $\Bbbk \SG_{n}$ (resp., $\Bbbk \SG_{n}^{\env}$) on \eqref{eq:int-hom-s-mod} is given as follows. 
Given $f$ in the $m$-th factor of \eqref{eq:int-hom-s-mod}, $v \in S(m)$ and $\sigma \in \SG_{n}$ (resp., $\sigma, \sigma' \in \SG_{n}$), we set $(f \cdot \sigma)(v)= f(v) \cdot \linc_{n,m+n}(\sigma)$ (resp., $(\sigma' \cdot f \cdot \sigma)(v)= \linc_{n,m+n}(\sigma') \cdot f(v) \cdot \linc_{n,m+n}(\sigma)$). 
Then, \eqref{eq:int-hom-s-mod} has the product action. 
Now, the following result is immediate. 
%%%%%%%
\begin{fact}
\label{fact:internal-hom-s-mod}
Given three $\SG$-modules (resp., diagonal $\SG$-bimodules) $S$, $S'$ and $S''$, there is a natural linear isomorphism 
\begin{align*}
&\Hom_{\SG}(S \otimes_{\SG} S', S'') \longrightarrow \Hom_{\SG}\big(S , \IHom_{\SG}(S', S'') \big)      
\\
&\bigg(\text{resp., } \Hom_{\SG^\env}(S \otimes_{\SG^{\env}} S', S'') \longrightarrow \Hom_{\SG^\env}\big(S , \IHom_{\SG^\env}(S', S'') \big) \text{\bigg)} 
\end{align*}
sending a morphism $f : S \otimes_{\SG} S \rightarrow S''$ of $\SG$-modules (resp., diagonal $\SG$-bimodules) with components $f_{m,n} : S(m) \otimes S'(n) \rightarrow S''(m+n)$ for $m, n \in \NN_{0}$ to the morphism $\tilde{f} : S \rightarrow \IHom_{\SG}(S', S'')$ of $\SG$-modules (resp., diagonal $\SG$-bimodules) satisfying that, if we write $\tilde{f}(m)(v) = (g_{n})_{n \in \NN_{0}} \in \IHom_{\SG}(S', S'')(m)$ for $n \in \NN_{0}$ and $v \in S(m)$, then
$g_{n}(w) = f_{m,n}(v \otimes w)$ for all $m \in \NN_{0}$ and $w \in S''(n)$. 
\end{fact}
%%%%%%%

%%%%%%%
\begin{remark}
There is an analogous definition of left shifted 
$\SG$-module (resp., diagonal $\SG$-bimodule), using the morphisms $\linc_{N,N+k}$ (resp., $\linc_{N,N+k}^{\env}$) instead of $\rinc_{N,N+k}$ (resp., $\rinc_{N,N+k}^{\env}$). 
However, this would also force us to change the definition of $\SG$-module (resp., diagonal $\SG$-bimodule) structure of \eqref{eq:int-hom-s-mod}, 
using the maps $\rinc_{n,n+m}$ (resp., $\rinc_{n,n+m}^{\env}$) instead of $\linc_{n,n+m}$ (resp., $\linc_{n,n+m}^{\env}$), as well as the isomorphisms of Fact \ref{fact:internal-hom-s-mod}.
\end{remark}
%%%%%%%

%%%%%%%%%%%%%%%%%%%%%%%%%%%%%%%%%%%%%%%%%%%%%%%%%%%%%%%%%%%%%%%
\subsection{\texorpdfstring{Algebraic structures}{Algebraic structures}}
\label{sec:algebraic-structures-diagonal-bimodules}

The following result will be useful in the sequel.
Its proof is immediate. 
%%%%%%%
\begin{fact}
\label{fact:algebra-dmod}
Let $S = (S(n))_{n \in \NN_{0}}$ be a $\SG$-module (resp., diagonal $\SG$-bimodule). 
A morphism $\mu : S \otimes_{\SG} S \rightarrow S$ (resp., $\mu : S \otimes_{\SG^{\env}} S \rightarrow S$) of $\SG$-modules (resp., diagonal $\SG$-bimodules) gives a(n unitary and associative) algebra in the monoidal category $\SMod$ (resp., $\DMod$) if and only if the corresponding sequence of maps $\{ \mu_{n,m} : S(n) \otimes S(m) \rightarrow S(n+m) \}_{n,m \in \NN_{0}}$ given in Fact \ref{fact:morph-tensor-s-bimod} satisfies that 
\begin{equation} 
\label{eq:assoc-wh}
    \mu_{n+m,p} \circ (\mu_{n,m} \otimes \id_{S(p)}) = \mu_{n,m+p} \circ (\id_{S(n)} \otimes \mu_{m,p})
\end{equation}
for all $n,m,p \in \NN_{0}$, and there is a linear map $\eta : \Bbbk \rightarrow S(0)$ such that 
\begin{equation} 
\label{eq:unit-wh}
    \mu_{0,n} \circ (\eta \otimes \id_{S(n)}) = \id_{S(n)} = \mu_{n,0} \circ (\id_{S(n)} \otimes \eta)
\end{equation}
for all $n \in \NN_{0}$. 
Moreover, $S$ endowed with $\mu$ is a commutative algebra in the symmetric monoidal category $\SMod$ (resp., $\DMod$)
if and only if we further have that 
\begin{equation} 
\label{eq:comm-wh}
\begin{split}
    &\mu_{n,m} (v,w) = \mu_{m,n}(w,v) \cdot \block_{n,m}(1 \ 2)
    \\
    &\text{\bigg(resp., } \;\mu_{n,m} (v,w) = \block_{m,n}(1 \ 2) \cdot \mu_{m,n}(w,v) \cdot \block_{n,m}(1 \ 2) \text{\bigg)}
\end{split}    
\end{equation}
for all $n,m \in \NN_{0}$, $v \in S(n)$ and $w \in S(m)$. 
\end{fact}
%%%%%%%

%%%%%%%
\begin{example}
\label{example:symm-dmod}
Given a vector space $V$, let $\SG^{\env}_{V} = (\SG^{\env}_{V}(n))_{n \in \NN_{0}}$ be the diagonal $\SG$-bimodule defined in Example \ref{example:ten-2}. 
For $n , m \in \NN_{0}$, define the map $\mu_{n,m} : \SG^{\env}_{V}(n) \otimes \SG^{\env}_{V}(m) \rightarrow \SG^{\env}_{V}(n+m)$ by
\begin{equation}
\label{eq:prod-sym-d-mod}
\begin{split}
   &\mu_{n,m}\bigg( \big((v_{1} \otimes \dots\otimes  v_{n}) \otimes_{\Bbbk \SG_{n}} \!(\sigma \otimes \tau)\big)  \otimes \big((w_{1} \otimes \dots\otimes  w_{m}) \otimes_{\Bbbk \SG_{m}}\!\! (\sigma' \otimes \tau')\big) \bigg) 
   \\
   &= (v_{1} \otimes \dots \otimes v_{n} \otimes w_{1} \otimes \dots \otimes w_{m}) \otimes_{\Bbbk \SG_{n+m}} \big(\sump_{n,m}(\sigma,\sigma') \otimes \sump_{n,m}(\tau,\tau') \big) 
 \end{split}  
\end{equation} 
for all $v_{1}, \dots, v_{n}, w_{1}, \dots, w_{m} \in V$, $\sigma, \tau \in \SG_{n}$ and 
$\sigma', \tau' \in \SG_{m}$. 
Moreover, let $\eta : \Bbbk \rightarrow \SG^{\env}_{V}(0)$ be the identity.
It is easy to see that the maps $\{ \mu_{n,m} \}_{n,m \in \NN_{0}}$ satisfy all conditions of 
Fact \ref{fact:morph-tensor-s-bimod} and they endow $\SG^{\env}_{V}$ with a commutative algebra structure in  $\DMod$. 
\end{example}
%%%%%%%

%%%%%%%
\begin{example}
\label{example:emod} 
Given a vector space $V$, let $\EE^{\env}_{V} = (\EE^{\env}_{V}(n))_{n \in \NN_{0}}$ be the diagonal $\SG$-bimodule defined in Example \ref{example:end-2}. 
For $n , m \in \NN_{0}$, define the map $\mu_{n,m} : \EE^{\env}_{V}(n) \otimes \EE^{\env}_{V}(m) \rightarrow \EE^{\env}_{V}(n+m)$ by
\begin{equation}
\label{eq:prod-emod}
\begin{split}
\mu_{n,m}( f  \otimes g)(v_{1} \otimes \dots\otimes  v_{n+m})
= f(v_{1} \otimes \dots \otimes v_{n}) \otimes g(v_{n+1} \otimes \dots \otimes v_{n+m})
 \end{split}  
\end{equation} 
for all $v_{1}, \dots, v_{n+m} \in V$, $f \in \EE^{\env}_{V}(n)$ and $g \in \EE^{\env}_{V}(m)$. 
Moreover, let $\eta : \Bbbk \rightarrow \EE^{\env}_{V}(0)$ be the identity.
It is easy to see that the maps $\{ \mu_{n,m} \}_{n,m \in \NN_{0}}$ satisfy all conditions of 
Fact \ref{fact:morph-tensor-s-bimod} and they endow $\EE^{\env}_{V}$ with a commutative algebra structure in $\DMod$. 
\end{example}
%%%%%%%

%%%%%%%
\begin{example} 
\label{example:end-alg-s-mopd}
Let $S=(S(n))_{n \in \NN_{0}}$ be an $\SG$-module (resp., a diagonal $\SG$-bimodule). 
Then, the $\SG$-module $\IHom_{\SG}(S,S)$ (resp., the diagonal $\SG$-bimodule $\IHom_{\SG^\env}(S,S)$) has a natural structure of algebra in the monoidal category $\SMod$ (resp., $\DMod$) for the product 
\begin{align*}
    &\mu_{n,m} : \IHom_{\SG}(S,S)(n) \otimes \IHom_{\SG}(S,S)(m) \rightarrow \IHom_{\SG}(S,S)(n+m)
    \\ 
    &\bigg(\text{resp., } \mu_{n,m} : \IHom_{\SG^\env}(S,S)(n) \otimes \IHom_{\SG^\env}(S,S)(m) \rightarrow \IHom_{\SG^\env}(S,S)(n+m) \bigg)     
\end{align*}  
\hskip -1.2mm given by 
$\mu_{n,m}(f,g) = \sh_{m}(f) \circ g$ (resp., $\mu_{n,m}(f,g) = \sh_{m}^{\env}(f) \circ g$) together with the unit $\eta : \Bbbk \rightarrow \IHom_{\SG}(S,S)(0)$ (resp., $\eta : \Bbbk \rightarrow \IHom_{\SG^\env}(S,S)(0)$) given by $\eta(1) = \id_{S}$.
\end{example}
%%%%%%%

Given an $\SG$-module (resp., diagonal $\SG$-bimodule) $S$, we will denote by $\Sym_{\SG}(S) \label{eq-symm-alg-MOD}$ (resp., $\Sym_{\SG^{\env}}(S) \label{eq-symm-alg-DMOD}$) the symmetric algebra of $S$ in the symmetric monoidal category $\SMod$ (resp., $\DMod$) instead of $\Sym_{\SMod}(S)$ (resp., $\Sym_{\DMod}(S)$).
Recall the functor $\I_{\SG^{\env},1}(V)$ defined in \eqref{eq:vect-s-se}. Then, the algebra in Example \ref{example:symm-dmod} is actually a symmetric algebra, as the following result shows. 
%%%%%%%
\begin{lemma}
\label{lemma:symm-dmod}
Given a vector space $V$, let $\SG^{\env}_{V} = (\SG^{\env}_{V}(n))_{n \in \NN_{0}}$ be the diagonal $\SG$-bimodule defined in Example \ref{example:ten-2} endowed with the algebra structure of Example \ref{example:symm-dmod}. 
Then, the algebras $\SG^{\env}_{V}$ and $\Sym_{\SG^{\env}}(\I_{\SG^{\env},1}(V))$ in the category $\DMod$ are isomorphic.
\end{lemma}
%%%%%%%
\begin{proof}
It suffices to show that $\SG^{\env}_{V}$ satisfies the universal property \eqref{eq:sym-alg} of the symmetric algebra on $\I_{\SG^{\env},1}(V)$ in the category $\DMod$. 
Let $A$ be a commutative algebra in $\DMod$ with product $\mu$. 
It is clear that the map in \eqref{eq:sym-alg}
\begin{equation}
\Hom_{\Alg(\DMod)}
(\SG^{\env}_{V},A) \longrightarrow \Hom_{\SG^\env}(\I_{\SG^{\env},1}(V),A)
\end{equation}
sending a morphism $f$ to its restriction to $\I_{\SG^{\env},1}(V)$ is injective. 
Let us show that it is surjective. 
Given a linear map $g : V \rightarrow A(1)$,
which is tantamount to a morphism $g : \I_{\SG^{\env},1}(V) \rightarrow A$ of diagonal $\SG$-modules, then, for $n \in \NN_{0}$ define the map 
$f'(n) : V^{\otimes n} \rightarrow \Res_{\Bbbk \Delta_{n}}(A(n))$ given as the composition of $g^{\otimes n} : V^{\otimes n} \rightarrow A(1)^{\otimes n}$ and the restriction of $\mu^{[n]} : A^{\otimes_{\SG^{\env}} n} \rightarrow A$ to $A(1)^{\otimes n}$. 
Since $A$ is commutative, this map is a morphism of right $\Bbbk \SG_{n}$-modules, where $V^{\otimes n}$ has the usual right $\Bbbk \SG_{n}$-module structure given by \eqref{eq:act-tens}. 
Using the adjunction \eqref{eq:res-ind}, it induces a morphism 
$f(n) : \Res_{\Bbbk \Delta_{n}}(V^{\otimes n}) \rightarrow A(n)$ of $\Bbbk \SG_{n}$-bimodules for $n \in \NN_{0}$. 
It is easy to see that $f = (f(n))_{n \in \NN_{0}}$ is a morphism of algebras from 
$\SG^{\env}_{V}$ to $A$ whose restriction to $\I_{\SG^{\env},1}(V)$ is $g$, as we wished.  
\end{proof}

We have the analogous result to Fact \ref{fact:algebra-dmod} for Poisson algebras in the category of diagonal $\SG$-bimodules. 
We leave to the reader the analogous statement for the category of $\SG$-modules, which we will not use in this work. 
%%%%%%%
\begin{fact} 
\label{fact:s-bimod-poisson-algebra-ex}
Let $S = (S(n))_{n \in \NN_{0}}$ be a diagonal $\SG$-bimodule. 
Assume further that $S$ is endowed with a morphism $\mu : S \otimes_{\SG^{\env}} S \rightarrow S$ of diagonal $\SG$-bimodules giving a(n unitary and associative) commutative algebra in $\DMod$. 
Let $\{ \hskip 0.6mm , \} : S \otimes_{\SG^{\env}} S \rightarrow S$ be a morphism of diagonal $\SG$-bimodules,  equivalently characterized by the family of morphisms $\{ \hskip 0.6mm , \}_{n,m} : S(n) \otimes S(m) \rightarrow \Res_{\Bbbk \sump_{n,m}^{\env}}(S(n+m))$ of $(\Bbbk \SG_{n} \otimes \Bbbk \SG_{m})$-bimodules satisfying condition \eqref{eq:morph-tensor-s-bimod} of Fact \ref{fact:morph-tensor-s-bimod}. 
Then, $\{ \hskip 0.6mm , \}$ is a bracket if and only if we have that 
\begin{equation}
     \label{eq:poisson-se-1}
         \{ v,w \}_{n,m} = - \block_{m,n}(1 \ 2) \cdot \{ w , v \}_{m,n} \cdot \block_{n,m}(1 \ 2)
\end{equation}
and  
     \begin{equation}
     \label{eq:poisson-se-2}
     \begin{split}
         \{ u , \mu_{n,m}(v,w) \}_{p,n+m} &= 
         \mu_{p+n,m}\big(\{ u, v \}_{p,n} , w \big) 
         \\
         &\phantom{=}+ \block_{n,p,m}(1 \ 2) \cdot \mu_{n,p+m} \big(v, \{ u, w \}_{p,m} \big) \cdot \block_{p,n,m}(1 \ 2)
         \end{split}
\end{equation}
for all $p, n, m \in \NN_{0}$, $u \in S(p)$, $v \in S(n)$ and $w \in S(m)$.
Moreover, this bracket gives a Poisson algebra structure on $(S,\mu)$ in the category of diagonal diagonal $\SG$-bimodules if and only if we further have that
     \begin{equation}
     \label{eq:poisson-se-3}
          \begin{split}
         0 &= \{ u , \{ v,w \}_{n,m} \}_{p,n+m} 
         + \block_{n,m,p}(1 \ 2 \ 3) \cdot \{ v, \{  w , u \}_{m,p} \}_{n,m+p} \cdot \block_{p,n,m}(1 \ 3 \ 2) 
         \\ 
         &\phantom{=}+ \block_{m,p,n}(1 \ 3 \ 2) \cdot \{ w, \{ u, v \}_{p,n} \}_{m,p+n} \cdot \block_{p,n,m}(1 \ 2 \ 3)
         \end{split}
\end{equation}
for all $p, n, m \in \NN_{0}$, $u \in S(p)$, $v \in S(n)$ and $w \in S(m)$.
\end{fact}
%%%%%%%
\begin{proof}
It is direct to verify that \eqref{eq:poisson-se-1} is equivalent to $\{ \hskip 0.6mm , \} \circ \tau^{\SG^{\env}}(S,S) = -\{ \hskip 0.6mm , \}$, whereas \eqref{eq:poisson-se-2} and \eqref{eq:poisson-se-3}
are tantamount to the Leibniz and the Jacobi identities for $\{ \hskip 0.6mm , \}$ in the symmetric monoidal category of diagonal $\SG$-bimodules, respectively. 
\end{proof}

%%%%%%%
\begin{remark}
Note that, assuming \eqref{eq:poisson-se-1}, \eqref{eq:poisson-se-2} is tantamount to 
     \begin{equation}
     \label{eq:poisson-se-2-equiv}
     \begin{split}
         \{ \mu_{n,m}(v,w) , u \}_{n+m,p} &= 
         \block_{n,p,m}(2 \ 3) \cdot \mu_{p+n,m}\big(\{ v, u \}_{n,p} , w \big) \cdot \block_{n,m,p}(2 \ 3)
         \\
         &\phantom{=}+ \mu_{n,m+p} \big(v, \{ w, u \}_{p,m} \big)
         \end{split}
\end{equation}
for all $p, n, m \in \NN_{0}$, $u \in S(p)$, $v \in S(n)$ and $w \in S(m)$.
\end{remark}
%%%%%%%

For completeness, we also present the following result, analogous to Fact \ref{fact:algebra-dmod} for bimodules in the category of diagonal $\SG$-bimodules. 
We leave to the reader the analogous statement for the category of $\SG$-modules, which we will not use in this article. 
%%%%%%%
\begin{fact}
\label{fact:bimod-dmod}
Let $S = (S(n))_{n \in \NN_{0}}$ be an algebra in $\DMod$ 
with product $\mu : S \otimes_{\SG^{\env}} S \rightarrow S$ and unit $\eta : \Bbbk \rightarrow S(0)$, 
and let $T = (T(n))_{n \in \NN_{0}}$ be a diagonal $\SG$-bimodule. 
A morphism $\rho : S \otimes_{\SG} T \otimes_{\SG} S \rightarrow T$ of diagonal $\SG$-bimodules gives a structure of $S$-bimodule if and only if the corresponding sequence of maps $\{ \rho_{n,m,p} : S(n) \otimes T(m) \otimes S(p) \rightarrow T(n+m+p) \}_{n,m, p \in \NN_{0}}$ given in Fact \ref{fact:morph-tensor-s-bimod} satisfies that 
\begin{equation} 
\label{eq:assoc-wh-bimod}
    \rho_{n+m,p,q+r} \circ (\mu_{n,m} \otimes \id_{T(p)} \otimes \mu_{q,r}) = \rho_{n,m+p+q,r} \circ (\id_{S(n)} \otimes \mu_{m,p,q} \otimes \id_{S(r)})
\end{equation}
for all $n,m,p, q, r \in \NN_{0}$ and $\rho_{0,n,0} \circ (\eta \otimes \id_{T(n)} \otimes \eta) = \id_{T(n)}$
for all $n \in \NN_{0}$. 
Moreover, if $S$ endowed with $\mu$ is a commutative algebra, the $S$-bimodule $T$ is symmetric if and only if
\begin{equation} 
\label{eq:comm-wh-bimod}
    \rho_{n,m,0} \big(v,w, \eta(1_{\Bbbk})\big) = \block_{m,n}(1 \ 2) \cdot \rho_{0,m,n}\big(\eta(1_{\Bbbk}),w,v\big) \cdot \block_{n,m}(1 \ 2)   
\end{equation}
for all $n,m \in \NN_{0}$, $u \in S(n)$, $v \in T(m)$ and $u \in W(p)$. 
\end{fact}
%%%%%%%

%%%%%%%%%%%%%%%%%%%%%%%%%%%%%%%%%%%%%%%%%%%%%%%%%%%%%%%%%%%%%%%%%%%%%%
\section{\texorpdfstring{Generalized wheelspaces and wheelspaces}{Generalized wheelspaces and wheelspaces}}
\label{section:wheelspaces}

%%%%%%%%%%%%%%%%%%%%%%%%%%%%%%%%%%%%%%%%%%%%%%%%%%%%%%%%%%%%%%%%%%%%%%
\subsection{Basic definitions} 
\label{subsection:wh-basic}

Given $n, m \in \NN_{0}$ with $n \leq m$, recall the injective morphism of groups
\begin{equation}
	\label{eq:e-n}
	\linc_{n,m} : \SG_{n} \longrightarrow \SG_{m}
\end{equation} 
defined in \eqref{eq:def-rinc}. 
To simplify, we will write $\linc_{n}$ instead of $\linc_{n,n+1}$.

Given $n \in \NN$ and $i \in \llbracket 1 , n \rrbracket$, we define the maps 
\begin{equation}
		\label{eq:l-r-n}
		\mathbbl{l}^n_{i}, \mathbbl{r}_{i}^{n} : \SG_{n} \longrightarrow \SG_{n-1}
	\end{equation} 
by 
\begin{equation}    
		\label{eq:r-n-def}
\mathbbl{r}_{i}^{n}(\sigma)(k) = \begin{cases}
	\sigma(k), &\text{if $k \in \llbracket 1 , i-1 \rrbracket$ and $\sigma(k) \in \llbracket 1 , \sigma(i)-1 \rrbracket$,}
	\\
	\sigma(k)-1, &\text{if $k \in \llbracket 1 , i-1 \rrbracket$ and $\sigma(k) \in \llbracket \sigma(i)+1,n \rrbracket$,}
	\\
	\sigma(k+1), &\text{if $k \in \llbracket i , n-1 \rrbracket$ and $\sigma(k+1) \in \llbracket 1 , \sigma(i)-1 \rrbracket$,}
	\\ 
	\sigma(k+1)-1, &\text{if $k \in \llbracket i , n-1 \rrbracket$ and $\sigma(k+1) \in \llbracket \sigma(i)+1,n \rrbracket$,}
	\end{cases} 
\end{equation} 
and $\mathbbl{l}_{i}^{n}(\sigma) = (\mathbbl{r}_{i}^{n}(\sigma^{-1}))^{-1}$ for all $\sigma \in \SG_{n}$ (\textit{cf.} \cite{MR4507248}). 
Note that neither $\mathbbl{l}_{i}^{n}$ nor $\mathbbl{r}_{i}^{n}$ are morphisms of groups in general. 

%%%%%%%
\begin{remark}
\label{remark:e-l-r} 
Note that $\mathbbl{l}_{n}^{n} \circ \linc_{n-1} = \id_{\SG_{n-1}} = \mathbbl{r}_{n}^{n} \circ \linc_{n-1}$ for all $n \in \NN$. 
\end{remark}
%%%%%%%

The following result will be useful in the sequel. 

%%%%%%%
\begin{fact}
\label{fact:e-l-r}
Given $n \in \NN$ and $\sigma \in \SG_{n}$, 
we have that 
\begin{align}
\linc_{n-1}\big(\mathbbl{l}_{i}^{n}(\sigma)\big) &= (i \ \dots \ n)^{-1} \sigma (\sigma^{-1}(i) \ \dots \ n),
\label{eq:l-id}
\\
\linc_{n-1}\big(\mathbbl{r}_{i}^{n}(\sigma)\big) &= (\sigma(i) \ \dots \ n)^{-1} \sigma (i \ \dots \ n). 
\label{eq:r-id}
\end{align}
In particular, 
\begin{equation}
\label{eq:l-r-id}
\mathbbl{l}_{i}^{n}(i \ \dots \ n)  =  \id_{\llbracket 1 , n \rrbracket} = \mathbbl{r}_{i}^{n}\big((i \ \dots \ n)^{-1}\big). 
\end{equation}
\end{fact} 
%%%%%%%
\begin{proof}
Note first that \eqref{eq:l-id} follows from \eqref{eq:r-id}, since 
\begin{align*}     \linc_{n-1}\big(\mathbbl{l}_{i}^{n}(\sigma)\big) 
&= \linc_{n-1}\big(\mathbbl{r}_{i}^{n}(\sigma^{-1})^{-1}\big) 
= \linc_{n-1}\big(\mathbbl{r}_{i}^{n}(\sigma^{-1})\big)^{-1} 
\\ 
&= \Big( (\sigma^{-1}(i) \ \dots \ n)^{-1} \sigma^{-1} (i \ \dots \ n) \Big)^{-1} 
\\ 
&= 
(i \ \dots \ n)^{-1} \sigma (\sigma^{-1}(i) \ \dots \ n).
\end{align*}
The proof of \eqref{eq:r-id} now follows from a straightforward computation using \eqref{eq:r-n-def}. 

Finally, \eqref{eq:l-r-id} is derived from 
\eqref{eq:l-id} and \eqref{eq:r-id} together with the fact that $\linc_{n-1}$ is injective. 
Indeed, if $\sigma = (i \ \dots \ n)^{-1}$, then $\sigma(i) = n$, which implies that the first factor of the right member of \eqref{eq:r-id} is the identity, which thus yields that $\linc_{n-1}(\mathbbl{r}_{i}^{n}(\sigma))$ is the identity, and in turn implies that $\mathbbl{r}_{i}^{n}(\sigma)$ is the identity, by the injectivity of $\linc_{n-1}$. 
The proof of the first identity of \eqref{eq:l-r-id} 
is similar.
\end{proof}

We now introduce one of the main objects of this work: generalized wheelspaces, which are a generalization of the definition of wheelspaces given in \cite{MR2734329}, Def. 3.1.1 (\textit{cf.} \cite{MR4507248}, Def. 1.1.1). 
The reason for considering these generalized objects will be discussed in Subsection \ref{subsection:contradiction} (see in particular Proposition \ref{proposition:contradiction}).
%%%%%%%
\begin{definition} 
\label{definition:gwh}
A \textbf{\textcolor{myblue}{generalized wheelspace}} is a diagonal $\SG$-bimodule $S = (S(n))_{n \in \NN_{0}}$, together with a distinguished decomposition 
\begin{equation}
\label{eq:decomp-gwh} 
     S(n) = \bigoplus_{\bar{n} \in \Part(n)} \Ind_{\Bbbk\sump_{\bar{n}}^{\env}}\big(S(n)_{\bar{n}}\big)
\end{equation}
of $\SG_{n}$-bimodules for all $n \in \NN_{0}$, 
where $S(n)_{\bar{n}}$ is a $\Bbbk \SG_{\bar{n}}$-bimodule for all 
$\bar{n} = (n_{1},\dots,n_{\ell}) \in \Part(n)$, together with a family of maps 
\[     \Big\{ {}^{\bar{n}}_{j}t_{i} : \bar{n} \in \NN_{0}^{(\NN)}, i, j \in \llbracket 1 , |\bar{n}| \rrbracket, \ord_{\bar{n}}(i) = \ord_{\bar{n}}(j)  \Big\},     \] 
called \textbf{\textcolor{myblue}{contractions}}, of the form
\begin{equation}
\label{ew:gwh-part-trace} 
     {}^{\bar{n}}_{j}t_{i} : S\big(|\bar{n}|\big)_{\bar{n}} \longrightarrow 
     S\big(|\bar{n}|-1\big)_{\bar{n}-\bar{e}^{p}}
\end{equation} 
with $p = \ord_{\bar{n}}(i)$, satisfying that  
\begin{enumerate}[label = (GW.\arabic*)]
	\item\label{item:GW1}
	${}_{j}^{\bar{n}}t_{i}(\sigma \cdot v \cdot \tau) = \sigma' \cdot {}_{\sigma^{-1}_{p}(j)}^{\phantom{xxxx}\bar{n}}t_{\tau_{p}(i)} (v) \cdot \tau'$, 
	for all $\bar{n} = (n_{1},\dots,n_{\ell}) \in \Part(n)$ partition of $n \in \NN_{0}$, $\sigma_{k}, \tau_{k} \in \SG_{n_{k}}$ for $k \in \llbracket 1 , \ell \rrbracket$, $i, j \in \llbracket 1 , n \rrbracket$ such that $\ord_{\bar{n}}(i) = \ord_{\bar{n}}(j)$ and $v \in S(n)_{\bar{n}}$, where $\sigma = \sigma_{1} \otimes \dots \otimes \sigma_{\ell}$, $\tau = \tau_{1} \otimes \dots \otimes \tau_{\ell}$, $p = \ord_{\bar{n}}(i)$, and
	\begin{equation}
	\begin{split}
	\sigma' &= \sigma_{1} \otimes \dots \otimes \sigma_{p-1} \otimes \mathbbl{l}_{j}^{n_{p}}(\sigma_{p}) \otimes \sigma_{p+1} \otimes \dots \otimes \sigma_{\ell}, 
	\\
	\tau' &= \tau_{1} \otimes \dots \otimes \tau_{p-1} \otimes \mathbbl{r}_{i}^{n_{p}}(\tau_{p}) \otimes \tau_{p+1} \otimes \dots \otimes \tau_{\ell};
	\end{split}
	\end{equation}
	
	\item\label{item:GW2}
	${}_{\phantom{xxx}j}^{\bar{n}-\bar{e}^{p}}t_{i} \circ {}_{h}^{\bar{n}}t_{k} = {}_{\phantom{xxx}j'}^{\bar{n}-\bar{e}^{p}}t_{i'} \circ {}_{h'}^{\bar{n}}t_{k'}$ for all $k, h \in \llbracket 1 , |\bar{n}| \rrbracket$ and $i, j \in \llbracket 1 , |\bar{n}|-1 \rrbracket$ such that $\ord_{\bar{n}}(k) = \ord_{\bar{n}}(h)$, that we write simply as $p$, and $\ord_{\bar{n} - \bar{e}^{p}}(i) = \ord_{\bar{n} - \bar{e}^{p}}(j) = p$, where 
	$(i',k') = (k-1,i)$ if $i < k$ and $(i',k') = (k,i+1)$ if $k \leq i$, and similarly for the definition of $(j', h')$; 

	\item\label{item:GW2bis} ${}_{\phantom{xxx}j}^{\bar{n}-\bar{e}^{p}}t_{i} \circ {}_{h}^{\bar{n}}t_{k} = {}_{\phantom{xxx}h'}^{\bar{n}-\bar{e}^{q}}t_{k'} \circ {}_{j'}^{\bar{n}}t_{i'}$ for all $k, h \in \llbracket 1 , |\bar{n}| \rrbracket$ such that $\ord_{\bar{n}}(k) = \ord_{\bar{n}}(h)$ and $i, j \in \llbracket 1 , |\bar{n}|-1 \rrbracket$ such that $\ord_{\bar{n} -\bar{e}^{p}}(i) = \ord_{\bar{n} -\bar{e}^{p}}(j)$ with $p = \ord_{\bar{n}}(k) \neq q = \ord_{\bar{n}-\bar{e}^{p}}(i)$, where 
	$(i',j') = (i,j)$ and $(k',h') = (k-1,h-1)$ if $q < p$, and $(i',j') = (i+1,j+1)$ and $(k',h') = (k,h)$ if $q > p$.
\end{enumerate} 

Let $S$ and $S'$ be two generalized wheelspaces. A \textbf{\textcolor{myblue}{morphism of generalized wheelspaces}} 
is a morphism of diagonal $\SG$-bimodules 
$f : S \rightarrow S'$ such that there exists an agglutination map $\varphi_{f}(n) : \Part(n) \rightarrow \Part(n)$ for all $n \in \NN_{0}$ 
satisfying that 
\begin{equation}
\label{eq:morph-gwh-1}
f(n)\big(S(n)_{\bar{n}}\big) \subseteq S'(n)_{\varphi_{f}(n)(\bar{n})}
\end{equation}
and 
\begin{equation}
\label{eq:morph-gwh-2}
{}^{\varphi_{f}(n)(\bar{n})}_{\phantom{xxxxxx}j}t'_{i} \circ f(n)|_{S(n)_{\bar{n}}} = f(n-1) \circ {}^{\bar{n}}_{j}t_{i}
\end{equation}
for all $n \in \NN_{0}$, $\bar{n} \in \Part(n)$, $i,j \in  \llbracket 1 , n \rrbracket$ with $\ord_{\bar{n}}(i) = \ord_{\bar{n}}(j)$, where ${}^{\bar{n}}_{j}t_{i}$ (resp.,  ${}^{\bar{n}}_{j}t'_{i}$) denotes the contraction of $S$ (resp., of $S'$). 
\end{definition}
%%%%%%% 

%%%%%%%
\begin{remark}
Note that the maps $\{ \varphi_{f}(n) \}_{n \in \NN_{0}}$ in the previous definition are uniquely determined by $f$, and that the identity  \eqref{eq:morph-gwh-2} makes sense because $\varphi_{f}(n)$ is an agglutination map. 
\end{remark}
%%%%%%%

Let $f : S \rightarrow S'$ and $g : S' \rightarrow S''$ be two morphisms of generalized wheelspaces. 
It is easy to see that the composition $g \circ f$ of the underlying diagonal $\SG$-bimodules of $f$ and $g$ is also a morphism of generalized wheelspaces from $S$ to $S''$, \textit{i.e.} $g \circ f$ verifies \eqref{eq:morph-gwh-2}.  
We then obtain the category $\GWMod \label{eq:def-gen-wheelsp-notation}$ of generalized wheelspaces. 

We now recover the notion of wheelspaces introduced in \cite{MR2734329}, Def. 3.1.1, as a particular kind of generalized wheelspaces. 
%%%%%%%
\begin{definition} 
\label{definition:wh}
A \textbf{\textcolor{myblue}{wheelspace}} is a generalized wheelspace $S = (S(n))_{n \in \NN_{0}}$, satisfying that $S(|\bar{n}|)_{\bar{n}} = 0$ for all $\bar{n} \in \cup_{\ell \geq 2} \NN_{0}^{\ell}$, \textit{i.e.} $S(n) = S(n)_{n}$ for all $n \in \NN_{0}$. 
The collection of wheelspaces is a full subcategory 
$\WMod$ of the category $\GWMod$ of generalized wheelspaces.
\end{definition}
%%%%%%%

%%%%%%%
\begin{remark} 
\label{rem:def-wheelspaces-W1-W2}
By restricting the notions introduced in Definition \ref{definition:gwh}
to the case of Definition \ref{definition:wh}, we can equivalently state that a wheelspace is a 
diagonal $\SG$-bimodule $S = (S(n))_{n \in \NN_{0}}$ together with a family of linear maps
\begin{equation}
\label{ew:part-trace} 
     {}_{j}^{n}t_{i} : S(n) \longrightarrow S(n-1)
\end{equation} 
for $i,j \in \llbracket 1 , n \rrbracket$ and $n \in \NN$, 
called \textbf{\textcolor{myblue}{contractions}}, satisfying that 
\begin{enumerate}[label = (W.\arabic*)]
	\item\label{item:W1} ${}_{j}^{n}t_{i}(\sigma \cdot v \cdot \tau) = \mathbbl{l}_{j}^{n}(\sigma) \cdot {}_{\sigma^{-1}(j)}^{\phantom{xxxx}n}t_{\tau(i)} (v) \cdot \mathbbl{r}_{i}^{n}(\tau)$, 
	for $\sigma, \tau \in \SG_{n}$, $i, j \in \llbracket 1, n \rrbracket$ and $v \in S(n)$; 
	\item\label{item:W2} ${}_{j}^{n}t_{i} \circ {}_{\phantom{xx}\ell}^{n+1}t_{k} = {}_{j'}^{n}t_{i'} \circ {}_{\phantom{xx}\ell'}^{n+1}t_{k'}$, where 
	$(i',k') = (k-1,i)$ if $i < k$ and $(i',k') = (k,i+1)$ if $k \leq i$, and similarly for the definition of $(j', \ell')$. 
\end{enumerate} 

Similarly, given two wheelspaces $S$ and $S'$, a \textbf{\textcolor{myblue}{morphism of wheelspaces}} 
is a morphism of diagonal $\SG$-bimodules 
$f : S \rightarrow S'$ satisfying that 
\begin{equation}
\label{eq:morph-wh}
{}_{j}^{n}t'_{i} \circ f(n) = f(n-1) \circ {}_{j}^{n}t_{i}
\end{equation}
for all $i,j \in \llbracket 1 , n \rrbracket$ and $n \in \NN$, where ${}_{j}^{n}t_{i}$ (resp.,  ${}_{j}^{n}t'_{i}$) denotes the contraction of $S$ (resp., of $S'$). 
We will sometimes denote a wheelspace by $(S, {}_{\bullet}^{\bullet}t_{\bullet})\label{index:wheelsp-notation-with-bulles}$, or even by the underlying diagonal $\mathbb{S}$-bimodule $S$ if the contractions are clear from the context.
\end{remark}
%%%%%%% 

%%%%%%%
\begin{remark}
It is not difficult to see that the category $\GWMod$ has arbitrary coproducts and finite limits. 
Furthermore, finite limits and colimits of diagrams in the additive subcategory $\WMod$ of $\GWMod$ exist and belong to $\WMod$, so the latter is also pre-abelian. 
Since isomorphisms (resp., monomorphisms, epimorphisms) of wheelspaces are the same as morphisms $f = (f(n))_{n \in \NN_{0}}$ such that $f(n)$ is bijective (resp., injective, surjective) for all $n \in \NN_{0}$, we see that $\WMod$ is even abelian. 
However, the category $\GWMod$ is \emph{not} abelian, since the morphism $\varphi_{2}(S,S')$ that we will introduce in \eqref{eq:monomorphismo-phi-2} below is a monomorphism and epimorphism, but it is not an isomorphism in general.
\end{remark}
%%%%%%%

%%%%%%%
\begin{example}
\label{example:wh0}
Let $V$ be a vector space. 
Then, the diagonal $\SG$-bimodule $\I_{\SG^{\env},0}(V) \in \DMod$ defined in \eqref{eq:vect-s-se} has a unique structure of wheelspace, where all the contractions \textit{a fortiori} vanish. 
We will denote this wheelspace also by $\I_{\SG^{\env},0}(V)$. 
Conversely, given any wheelspace $S = (S(n))_{n \in \NN_{0}}$ satisfying that $S(n)$ vanishes for all $n \in \NN$, we have 
$S = \I_{\SG^{\env},0}(S(0))$. 
\end{example}
%%%%%%% 

%%%%%%%
\begin{example}
\label{example:end-3}
Let $V$ be a finite dimensional vector space, and let $\mathbb{E}_{V}^{\env} = (\mathbb{E}_{V}^{\env}(n))_{n \in \NN_{0}}$ be the associated diagonal $\SG$-bimodule of Example \ref{example:end-2}. 
Let $\{ e_{\alpha} : \alpha \in \mathtt{A} \}$ be a finite basis of $V$ and $\{ e_{\alpha}^{*} : \alpha \in \mathtt{A} \}$ be the dual basis of $V^{*}$. 
Given $n \in \NN$ and $i,j \in \llbracket 1 , n \rrbracket$, we define ${}_{j}^{n}t_{i} : \mathbb{E}_{V}^{\env}(n) \to \mathbb{E}_{V}^{\env}(n-1)$ by 
\begin{equation}
\label{eq:contraction-end}
\begin{split}
&{}_{j}^{n}t_{i} (f) (v_{1} \otimes \dots \otimes v_{n-1}) 
\\ 
&= \sum_{\alpha \in \mathtt{A}} \big(\id_{V}^{\otimes (j-1)} \otimes e_{\alpha}^{*} \otimes \id_{V}^{\otimes (n-j)}\big) \big(f(v_{1} \otimes \dots \otimes v_{i-1} \otimes e_{\alpha} \otimes v_{i} \otimes \dots \otimes v_{n-1})\big) 
\end{split}
\end{equation}
for all $f \in S(n)$ and $v_{1}, \dots, v_{n-1} \in V$. 
It is easy to check that the previous identity is independent of the choice of the basis $\{ e_{\alpha} : \alpha \in \mathtt{A} \}$ of $V$ 
and that the diagonal $\SG$-bimodule $\mathbb{E}_{V}^{\env} = (\mathbb{E}_{V}^{\env}(n))_{n \in \NN_{0}}$ endowed with the previous maps is a wheelspace. 

It will be convenient to write \eqref{eq:contraction-end} using the notation introduced at the end of Example \ref{example:end-2}. 
Given $\bar{\alpha} \in \mathtt{A}^{n}$ and $i \in \llbracket 1 , n \rrbracket$, let $\omicron_{i}(\bar{\alpha}) = (\alpha_{1}, \dots,\alpha_{i-1},\alpha_{i+1},\dots,\alpha_{n}) \in \mathtt{A}^{n-1}$.  
Then, we see that 
\[     {}_{j}^{n}t_{i} (E_{\bar{\alpha},\bar{\beta}}) = \delta_{\alpha_{j},\beta_{i}} E_{\omicron_{j}(\bar{\alpha}), \omicron_{i}(\bar{\beta})},     \] for all $\bar{\alpha}, \bar{\beta} \in \mathtt{A}^{n}$ and $i, j \in \llbracket 1 , n \rrbracket$. 
Given $\beta \in \mathtt{A}$, $\bar{\alpha} \in \mathtt{A}^{n}$ and $j \in \llbracket 1 , n \rrbracket$, we will denote by $\incl_{j,\beta}(\bar{\alpha}) \in \mathtt{A}^{n+1}$ the $(n+1)$-tuple $(\alpha_{1},\dots,\alpha_{j-1},\beta,\alpha_{j},\dots,\alpha_{n})$.
\end{example}
%%%%%%% 

%%%%%%%
\begin{definition}
\label{definition:strict-gen-subwheel}
Let $S = (S(n))_{n \in \NN_{0}}$ be a generalized wheelspace with set of contractions $\{ {}^{\bar{n}}_{j}t_{i} : \bar{n} \in \NN_{0}^{(\NN)}, i, j \in \llbracket 1 , |\bar{n}| \rrbracket, \ord_{\bar{n}}(i) = \ord_{\bar{n}}(j)  \}$. 
Recall that we have by definition a distiguished decomposition 
\begin{equation}
\label{eq:decomp-gwh-s} 
     S(n) = \bigoplus_{\bar{n} \in \Part(n)} \Ind_{\Bbbk\sump_{\bar{n}}^{\env}}\big(S(n)_{\bar{n}}\big)
\end{equation}
of $\SG_{n}$-bimodules for all $n \in \NN_{0}$, 
where $S(n)_{\bar{n}}$ is a $\Bbbk \SG_{\bar{n}}$-bimodule for all 
$\bar{n} = (n_{1},\dots,n_{\ell}) \in \Part(n)$. 
A diagonal $\SG$-subbimodule $S' = (S'(n))_{n \in \NN_{0}}$ of $S$ 
is said to be a \textbf{\textcolor{myblue}{generalized (strict) subwheelspace}} of $S$ if there is a family 
$\{ S'(n)_{\bar{n}} : \bar{n} = (n_{1},\dots,n_{\ell}) \in \Part(n), n \in \NN_{0} \}$ where $S'(n)_{\bar{n}}$ is a $\Bbbk \SG_{\bar{n}}$-subbimodule of 
$S(n)_{\bar{n}}$ for all 
$\bar{n} = (n_{1},\dots,n_{\ell}) \in \Part(n)$ and $n \in \NN_{0}$, we have 
\begin{equation}
\label{eq:decomp-gwh-sp} 
     S'(n) = \bigoplus_{\bar{n} \in \Part(n)} \Ind_{\Bbbk\sump_{\bar{n}}^{\env}}\big(S'(n)_{\bar{n}}\big)
\end{equation}
for all $n \in \NN_{0}$, and ${}^{\bar{n}}_{j}t_{i}(S'(|\bar{n}|)_{\bar{n}}) \subseteq S'(|\bar{n}|)_{\bar{n} - \bar{e}^{p}}$, for all possible $i$, $j$ and $\bar{n}$, where $p = \ord_{\bar{n}}(i)$. 
If $S$ is a wheelspace, then $S'$ is also a wheelspace, which we will simply call a \textbf{\textcolor{myblue}{subwheelspace}} of $S$. 

Let 
\begin{equation}
\label{eq:forgetful-gwh-dmod}
\gFo : \GWMod \longrightarrow \DMod
\end{equation}
be the forgetful functor 
given by sending a generalized wheelspace to its underlying diagonal $\SG$-bimodule and that is the identity on morphisms. 
Then, given a generalized subwheelspace $S'$ of a generalized wheelspace $S$, described as before, the quotient 
$\gFo(S)/\gFo(S')$ in $\DMod$ then has a structure of 
generalized wheelspace, which we will denote by $S/S'$ and call the \textbf{\textcolor{myblue}{generalized wheelquotient}}, with 
\begin{equation}
\label{eq:decomp-gwh-q} 
     (S/S')(n)_{\bar{n}} = S(n)_{\bar{n}}/S'(n)_{\bar{n}}
\end{equation}
for all $n \in \NN_{0}$ and $\bar{n} = (n_{1},\dots,n_{\ell}) \in \Part(n)$, and 
set of contractions $\{ {}^{\bar{n}}_{j}\bar{t}_{i} : \bar{n} \in \NN_{0}^{(\NN)}, i, j \in \llbracket 1 , |\bar{n}| \rrbracket, \ord_{\bar{n}}(i) = \ord_{\bar{n}}(j)  \}$, 
with ${}^{\bar{n}}_{j}\bar{t}_{i} \circ \pi_{\bar{n}} = \pi_{\bar{n}} \circ {}^{\bar{n}}_{j}t_{i}$ for all possible $i$, $j$ and $\bar{n}$, where the map $\pi_{\bar{n}} : S(n)_{\bar{n}} \rightarrow S(n)_{\bar{n}}/S'(n)_{\bar{n}}$ denotes
 the canonical projection. 
\end{definition}
%%%%%%%

%%%%%%%
\begin{remark}
\label{remark:intersection-wheel}
Let $S = (S(n))_{n \in \NN_{0}}$ be a generalized wheelspace with set of contractions as in Definition \ref{definition:strict-gen-subwheel}, 
and let $\{ S^{i} = (S^{i}(n))_{n \in \NN_{0}} : i \in \mathtt{I} \}$ be an arbitrary family of generalized subwheelspaces of $S$. 
Then, the intersection $\bar{S} = \cap_{i \in \mathtt{I}} \gFo(S^{i})$ 
of the family $\{ \gFo(S^{i}) : i \in \mathtt{I} \}$ of diagonal $\SG$-bimodules recalled in \eqref{eq:intersection-s-mod} is also a generalized subwheelspace of $S$, such that 
\begin{equation}
\label{eq:decomp-int} 
     \bar{S}(n)_{\bar{n}} = \bigcap_{i \in \mathtt{I}} S^{i}(n)_{\bar{n}}
\end{equation}
for all $n \in \NN_{0}$ and $\bar{n} \in \Part(n)$.
\end{remark}
%%%%%%%

Let $T$ be a diagonal $\SG$-subbimodule of $\gFo(S)$, where $S$ is a generalized wheelspace. 
Define the \textbf{\textcolor{myblue}{generalized subwheelspace of $S$ generated by $T$}} as the intersection of the nonempty family 
$\{ S^{i} = (S^{i}(n))_{n \in \NN_{0}} : i \in \mathtt{I} \}$ formed by all generalized subwheelspaces of $S$ containing $T$. 
We will denote this generalized subwheelspace of $S$ by $\langle T \rangle\label{index:subwheelsp-generated-by-T}$. 

%%%%%%%%%%%%%%%%%%%%%%%%%%%%%%%%%%%%%%%%%%%%%%%%%%%%%%%%%%%%%%%%%%%%%%
\subsection{A simpler definition of wheelspaces} 
\label{subsection:wh-basic-2}

Although the definition of wheelspace is natural, it embodies a considerable amount of structure.
It will be convenient to reduce it with the following equivalent definition.

%%%%%%%
\begin{definition}
\label{def:partial-wheelspace}
A \textbf{\textcolor{myblue}{partial wheelspace}} is a diagonal $\SG$-bimodule $S$ together with a family of morphisms 
\begin{equation}
\label{ew:part-trace-n} 
     {}^{n}\mathcalboondox{t} : \Res_{\Bbbk \linc_{n-1}^{\env}} \big(S(n)\big) \longrightarrow S(n-1)
\end{equation} 
of $\Bbbk \SG_{n-1}$-bimodules for $n \in \NN$, 
called \textbf{\textcolor{myblue}{partial contractions}}, satisfying that 
\begin{equation}
\label{eq:part-trace-comm} 
     \big({}^{n}\mathcalboondox{t} \circ {}^{n+1}\mathcalboondox{t}\big) \big((n \ n+1) \cdot v \cdot (n \ n+1)\big) = \big({}^{n}\mathcalboondox{t} \circ {}^{n+1}\mathcalboondox{t}\big) (v)
\end{equation} 
for all $v \in S(n+1)$ and all $n \in \NN$. 
Given two partial wheelspaces $S$ and $S'$, a \textbf{\textcolor{myblue}{morphism of partial wheelspaces}} 
is a morphism of diagonal $\SG$-bimodules 
$f : S \rightarrow S'$ satisfying that 
\begin{equation} 
\label{eq:morph-wh-red}
{}^{n}\mathcalboondox{t}' \circ f(n) = f(n-1) \circ {}^{n}\mathcalboondox{t}
\end{equation}
for all $n \in \NN$, where ${}^{n}\mathcalboondox{t}$ (resp., ${}^{n}\mathcalboondox{t}'$) denotes the partial contraction of $S$ (resp., of $S'$).
The composition of morphisms is the same as that of morphisms of diagonal $\SG$-bimodules, giving the category $\pWMod$ of partial wheelspaces.
\end{definition}
%%%%%%% 

%%%%%%%
\begin{lemma}
\label{lema:equivalencia-wheel-partial}
The categories $\WMod$ of wheelspaces and $\pWMod$ of partial wheelspaces are isomorphic.
\end{lemma} 
%%%%%%%
\begin{proof} 
Consider the functor 
\begin{equation}
\label{ew:wmod-pwmod} 
     \operatorname{F} : \WMod \longrightarrow \pWMod
\end{equation} 
sending a wheelspace $S$ with contractions 
$\{ {}_{j}^{n}t_{i} : n \in \NN \text{ and } i,j \in \llbracket 1 , n \rrbracket \}$ to the partial wheelspace with the same underlying diagonal $\SG$-bimodule and family of partial contractions $\{ {}^{n}\mathcalboondox{t} : n \in \NN \}$ with ${}^{n}\mathcalboondox{t} = {}_{n}^{n}t_{n}$ 
for all $n \in \NN$. 
This functor is well defined. 
Indeed, \ref{item:W2} together with Remark \ref{remark:e-l-r} imply that ${}^{n}\mathcalboondox{t}$ is a morphism of $\Bbbk \SG_{n-1}$-bimodules. 
Moreover, \ref{item:W2} for $\sigma = (j \ \dots \ n)$ and $\tau = (i \ \dots \ n)^{-1}$ together with \eqref{eq:l-r-id} imply that 
\begin{equation}
\label{eq:t-i-j} 
     {}_{j}^{n}t_{i}(v) = {}_{n}^{n}t_{n}\big( (j \ \dots \ n)^{-1} \cdot v \cdot (i \ \dots \ n) \big)
\end{equation} 
for all $n \in \NN$, $i,j \in \llbracket 1 , n \rrbracket$ and $v \in S(n)$. 
Finally, \ref{item:W1} tells us in particular that 
${}_{n}^{n}t_{n} \circ {}_{n+1}^{n+1}t_{n+1} = {}_{n}^{n}t_{n} \circ {}_{\phantom{xx}n}^{n+1}t_{n}$, 
which together with \eqref{eq:t-i-j} gives us precisely \eqref{eq:part-trace-comm}. 
Hence, $F$ is well defined on objects and morphisms. 
The fact that $F$ is full and faithful immediately follows from \eqref{eq:t-i-j}. 

We will finally prove that $F$ is an isomorphism. 
Let $S$ be a partial wheelspace with family of partial contractions $\{ {}^{n}\mathcalboondox{t} : n \in \NN \}$. 
Let us define a family of maps $\{ {}_{j}^{n}t_{i} : n \in \NN \text{ and } i,j \in \llbracket 1 , n \rrbracket \}$ by ${}_{n}^{n}t_{n} = {}^{n}\mathcalboondox{t}$ and \eqref{eq:t-i-j}. 
It suffices to show that they satisfy \ref{item:W1} and \ref{item:W2}. 

To prove \ref{item:W1}, we firstly assume that $i<k$ and $j<\ell$. Then, for an arbitrary $v\in S(n)$, using \eqref{eq:t-i-j}, Fact \ref{fact:cases} and \eqref{eq:part-trace-comm}, we have
\allowdisplaybreaks
\begin{align*}
\big({}_j^nt_i & \circ {}_\ell^{n+1}t_k\big)(v)
\\
&={}_j^nt_i\Big({}_{n+1}^{n+1}t_{n+1}\big( (\ell \ \dots \ n+1)^{-1}\cdot v\cdot (k\ \dots \ n+1)\big)\Big)
\\
&= {}_j^nt_i\Big( {}^{n+1}\mathcalboondox{t}\big((\ell \ \dots\ n+1)^{-1}\cdot v\cdot (k\ \dots \ n+1)\big)\Big)
\\
&={}^n \mathcalboondox{t}\Big((j\ \dots \ n)^{-1}\cdot {}^{n+1}\mathcalboondox{t}\big( (\ell\ \dots \ n+1)^{-1}\cdot v\cdot (k\ \dots \ n+1)\big) \cdot (i\ \dots \ n)\Big)
\\
&={}^n \mathcalboondox{t}\circ {}^{n+1} \mathcalboondox{t} \Big((j\ \dots \ n)^{-1} (\ell \ \dots \ n+1)^{-1}\cdot v\cdot (k\ \dots \ n+1) (i\ \dots \ n)\Big)
\\
&={}^n \mathcalboondox{t}\circ {}^{n+1} \mathcalboondox{t} \Big((n\ n+1)(\ell -1\ \dots \ n)^{-1} (j\ \dots \ n+1)^{-1}\cdot v
\\
&\qquad \qquad \cdot (i\ \dots \ n+1) (k-1\ \dots \ n)(n \ n+1)\Big)
\\
&={}^n \mathcalboondox{t}\circ {}^{n+1} \mathcalboondox{t} \Big((\ell-1\ \dots \ n)^{-1} (j\ \dots \ n+1)^{-1}\cdot v \cdot (i\ \dots \ n+1) (k-1\ \dots \ n)\Big)
\\
&={}^n \mathcalboondox{t} \Big((\ell-1\ \dots \ n)^{-1} \cdot {}^{n+1} \mathcalboondox{t} \big( (j\ \dots \ n+1)^{-1}\cdot v \cdot (i\ \dots \ n+1)\big)\cdot  (k-1\ \dots \ n)\Big)
\\
&=\big({}^n_{\ell-1}t_{k-1}\circ {}^{n+1}_j t_i\big)(v);
\end{align*}
note that we used in the third and seventh identities that ${}^{n+1} \mathcalboondox{t}$ is a morphism of $\Bbbk\SG_n$-bimodules for all $n \in \NN$. The remaining three cases are similar and they are left to the reader.

Next, for all $\sigma,\tau\in\SG_n$ and $v\in S(n)$, the condition \ref{item:W2} follows from 
\eqref{eq:t-i-j} together with \eqref{eq:l-id} and \eqref{eq:r-id}:
\begin{align*}
    \mathbbl{l}^n_{j}(\sigma)\cdot  {}_{\sigma^{-1}(j)}^n & t_{\tau(i)}(v)\cdot \mathbbl{r}^n_{i}(\tau)
    \\
    &=   \mathbbl{l}^n_{j}(\sigma)\cdot {}^n_nt_n\Big( (\sigma^{-1}(j)\ \dots \ n)^{-1}\cdot v \cdot (\tau(i)\ \dots \ n)\Big)\cdot \mathbbl{r}^n_{i}(\tau)
    \\
        &=   \mathbbl{l}^n_{j}(\sigma)\cdot {}^n\mathcalboondox{t}\Big( (\sigma^{-1}(j)\ \dots \ n)^{-1}\cdot v \cdot (\tau(i)\ \dots \ n)\Big)\cdot \mathbbl{r}^n_{i}(\tau)
    \\
        &=   {}^n\mathcalboondox{t}\Big( \linc_{n-1} \big( \mathbbl{l}^n_{j}(\sigma)\big) (\sigma^{-1}(j)\ \dots \ n)^{-1}\cdot v \cdot (\tau(i)\ \dots \ n) \linc_{n-1} \big( \mathbbl{r}^n_{i}(\tau)\big)\Big)
         \\
        &=   {}^n\mathcalboondox{t}\Big( (j\ \dots \ n)^{-1}\sigma (\sigma^{-1}(j)\ \dots \ n)^{-1} (\sigma^{-1}(j)\ \dots \ n)^{-1}\cdot v
        \\
        &\qquad \qquad \cdot (\tau(i)\ \dots \ n)(\tau(i)\ \dots \ n)^{-1}\tau (i\ \dots \ n)\Big)
        \\
        &=   {}^n\mathcalboondox{t}\Big(( j\ \dots \ n)
^{-1}\sigma \cdot v\cdot \tau (i\ \dots \ n)\Big)= {}^n_jt_i(v),        
\end{align*}
where we also used that ${}^n\mathcalboondox{t}$ is a morphism of $\Bbbk\SG_{n-1}$-bimodules for all $n \in \NN$.
\end{proof} 

Since the definitions of wheelspace and partial wheelspace are equivalent, we shall henceforth drop the prefix partial and use either of the two descriptions of wheelspaces (given in Definitions \ref{definition:wh} and \ref{def:partial-wheelspace}) interchangeably, depending on the situation and our purposes. 

%%%%%%%%%%%%%%%%%%%%%%%%%%%%%%%%%%%%%%%%%%%%%%%%%%%%%%%%%%%%%%%%%%%%%% 
\subsection{A symmetric monoidal structure on generalized wheelspaces}
\label{subsection:gwh-mono}

We consider two generalized wheelspaces $S$ and $S'$, endowed with contractions $\{  {}_{j}^{\bar{n}}t_{i} : \bar{n} \in \NN_{0}^{(\NN)}, i, j \in \llbracket 1 , |\bar{n}| \rrbracket, \ord_{\bar{n}}(i) = \ord_{\bar{n}}(j) \}$ and $\{  {}_{j}^{\bar{n}}t'_{i} : \bar{n} \in \NN_{0}^{(\NN)}, i, j \in \llbracket 1 , |\bar{n}| \rrbracket, \ord_{\bar{n}}(i) = \ord_{\bar{n}}(j) \}$, respectively. 
Given $N \in \NN_{0}$ and $\bar{N} \in \Part(N)$, we define the $\Bbbk\SG_{\bar{N}}$-bimodule
\begin{equation}
\label{eq:tensor-decomp} 
     \big( S \otimes_{\gwh} S' \big)(N)_{\bar{N}} = 
     \bigoplus_{\text{\begin{tiny}$\begin{matrix} \bar{n}, \bar{n}' \in \NN_{0}^{(\NN)} \\ \bar{n} \sqcup \bar{n}' = \bar{N}\end{matrix}$\end{tiny}}} S\big(|\bar{n}|\big)_{\bar{n}} \otimes_{} S'\big(|\bar{n}'|\big)_{\bar{n}'}.
\end{equation} 
By definition of the tensor product of diagonal $\SG$-bimodules and Fact \ref{fact:nat-res-ind}, we then have the decomposition  
\[     \big( S \otimes_{\SG^{\env}} S' \big)(N) = \bigoplus_{\bar{N} \in \Part(N)} \Ind_{\Bbbk \sump_{\bar{N}}}\Big(\big( S \otimes_{\gwh} S' \big)(N)_{\bar{N}} \Big)    \]
of $\SG_{N}$-bimodules for all $N \in \NN_{0}$. 
Moreover, let $\bar{n}, \bar{n}', \bar{N} \in \NN_{0}^{(\NN)}$ such that $\bar{n} \sqcup \bar{n}' = \bar{N}$.  
If $i, j \in \llbracket 1 , |\bar{N}| \rrbracket$ such that $\ord_{\bar{N}}(i) = \ord_{\bar{N}}(j) = p$, we define 
\begin{equation}
\label{eq:T-gwh}   
{}^{\bar{N}}_{\phantom{i}j}T_{i} : \big( S \otimes_{\gwh} S' \big)(N)_{\bar{N}} \longrightarrow \big( S \otimes_{\gwh} S' \big)(N-1)_{\bar{N}-\bar{e}^{p}}      
\end{equation} 
as follows. 
The restriction of \eqref{eq:T-gwh} to $S(|\bar{n}|)_{\bar{n}} \otimes_{} S'(|\bar{n}'|)_{\bar{n}'}$ in \eqref{eq:tensor-decomp} is given by 
\begin{equation}
 \label{eq:T-gwh-2}    {}^{\bar{N}}_{\phantom{i}j}T_{i}|_{S(|\bar{n}|)_{\bar{n}} \otimes_{} S'(|\bar{n}'|)_{\bar{n}'}} = \begin{cases} 
    {}^{\bar{n}}_{j}t_{i} \otimes \id_{S'(|\bar{n}'|)}, 
    &\text{if $i, j \in \llbracket 1 , |\bar{n}| \rrbracket$,}
    \\
    \id_{S(|\bar{n}|)} \otimes {}^{\phantom{xx,}\bar{n}'}_{j-|\bar{n}|}t'_{i-|\bar{n}|}, &\text{if $i, j \in \llbracket |\bar{n}| + 1 , |\bar{N}| \rrbracket$.}
    \end{cases}
\end{equation} 

%%%%%%%
\begin{lemma-definition}
Let $S$ and $S'$ be two generalized wheelspaces. 
Then, the diagonal $\SG$-bimodule $S \otimes_{\SG^{\env}} S'$ endowed with the 
contractions \eqref{eq:T-gwh} is a generalized wheelspace, which we will denote by $S \otimes_{\gwh} S'$. 
\label{lem-def:gwh-tensor-product}
\end{lemma-definition}
%%%%%%%
\begin{proof}
Conditions \ref{item:GW1}, \ref{item:GW2} and \ref{item:GW2bis} for \eqref{eq:T-gwh}  follows immediately from the corresponding properties for the contractions of $S$ and $S'$ and \eqref{eq:T-gwh-2}. 
\end{proof}

%%%%%%%
\begin{proposition} 
\label{proposition:gfo}
The category $\GWMod$ endowed with the tensor product $\otimes_{\gwh}$ 
and the same unit as the category of diagonal $\SG$-bimodules is a symmetric monoidal category such that the forgetful functor 
\begin{equation*}
\gFo : \GWMod \longrightarrow \DMod,
\end{equation*}
introduced in \eqref{eq:forgetful-gwh-dmod}, is braided strong monoidal. 
\end{proposition} 
%%%%%%%
\begin{proof} 
The fact that the tensor product $\otimes_{\gwh}$ satisfies the axioms of a monoidal category and that the functor $\gFo$ is strong monoidal are immediate, since the associativity and unit isomorphisms of the tensor product $\otimes_{\SG^{\env}}$ applied to generalized wheelspaces clearly commute with the corresponding partial contractions. 
It thus suffices to show that the braiding $\tau^{\SG^{e}}(S,T)$ given in \eqref{eq:tau-s-e} is a morphism of generalized wheelspaces for any pair of generalized wheelspaces $S$ and $T$, \textit{i.e.} that 
$\tau^{\SG^{e}}(S,T)$ commutes with the respective contractions. 
This follows immediately from \eqref{eq:T-gwh}. 
\end{proof} 

We will also consider the functor 
\begin{equation}
\label{eq:forgetful-wh-dmod}
\Fo : \WMod \longrightarrow \DMod 
\end{equation}
given by the composition of the canonical inclusion $\WMod \rightarrow \GWMod$ and the forgetful functor $\gFo$ defined in \eqref{eq:forgetful-gwh-dmod}. 

The reason for introducing generalized wheelspaces is the following result, which does not hold for usual wheelspaces (see Proposition \ref{proposition:contradiction} below). 
It can be regarded as the natural extension of Fact \ref{fact:morph-tensor-s-bimod} to the setting of generalized wheelspaces.
%%%%%%%
\begin{proposition}
\label{proposition:morph-tensor-gwh}
Let $S = \big(S(n)\big)_{n \in \NN_{0}}$, $S' = \big(S'(n)\big)_{n \in \NN_{0}}$ and $S'' = \big(S''(n)\big)_{n \in \NN_{0}}$ be generalized wheelspaces, with contractions $({}^{\bar{n}}_{i}t_{j})_{\bar{n} \in \NN_{0}^{(\NN)}, i, j \in \llbracket 1 , |\bar{n}| \rrbracket}$, 
$({}^{\bar{n}}_{i}t'_{j})_{\bar{n} \in \NN_{0}^{(\NN)}, i, j \in \llbracket 1 , |\bar{n}| \rrbracket}$ and $({}^{\bar{n}}_{i}t''_{j})_{\bar{n} \in \NN_{0}^{(\NN)}, i, j \in \llbracket 1 , |\bar{n}| \rrbracket}$, respectively. 
Let $\Bi_{\gwh}(S,S';S'')$ be the set formed by the families of linear maps 
\[    
\Big\{ f_{\bar{n},\bar{m}} : S(|\bar{n}|)_{\bar{n}} \otimes S'(|\bar{m}|)_{\bar{m}} \longrightarrow \Res_{\Bbbk \sump_{\bar{n} \sqcup \bar{m}}^{\env}}\big(S''(|\bar{n}|+|\bar{m}|)\big) \Big\}_{\bar{n}, \bar{m} \in \NN_{0}^{(\NN)}}    
\] 
such that $f_{\bar{n},\bar{m}}$ is a morphism of $\Bbbk \SG_{\bar{n} \sqcup \bar{m}}$-bimodules satisfying that its image is included in $S''(|\bar{n}| + |\bar{m}|)_{\varphi_{f}(\bar{n},\bar{m})}$ for an agglutination $\varphi_{f}(\bar{n},\bar{m})$ of $\bar{n} \sqcup \bar{m}$, and 
\begin{equation}
\label{eq:morph-tensor-wh-def}   
{}^{\varphi_{f}(\bar{n},\bar{m})}_{\phantom{xxxxxx}i}t''_{j} \circ f_{\bar{n},\bar{m}} = 
\begin{cases}
f_{\bar{n}-\bar{e}^{p},\bar{m}} &\hskip -4mm\circ \hskip 1mm ({}^{\bar{n}}_{i}t_{j} \otimes \id_{S'(|\bar{m}|)_{\bar{m}}}), 
\\
&\text{if $i, j \in \llbracket 1 , |\bar{n}| \rrbracket$ and $\ord_{\bar{n}}(i) = \ord_{\bar{n}}(j)$,}
\\
f_{\bar{n},\bar{m}-\bar{e}^{p'}} &\hskip -4mm\circ \hskip 1mm (\id_{S(|\bar{n}|)_{\bar{n}}} \otimes {}^{\phantom{x,}\bar{m}}_{i-n}t'_{j-n}), 
\\
&\text{if $i, j \in \llbracket |\bar{n}| , |\bar{n}| + |\bar{m}| \rrbracket$ and $\ord_{\bar{m}}(i) = \ord_{\bar{m}}(j)$,}
\end{cases} 
\end{equation}
for all $\bar{n}, \bar{m} \in \NN_{0}^{(\NN)}$, where $\ell = \len(\bar{n})$ and $\ell' = \len(\bar{m})$. 
Then, the map 
\begin{equation}
\label{eq:morph-tensor-gwh}
    \Hom_{\GWMod} (S \otimes_{\gwh} S' , S'') \longrightarrow \Bi_{\gwh}(S,S';S'')
\end{equation}
that sends a morphism $(f(N))_{N \in \NN_{0}}$ to the sequence $f = \{ f_{\bar{n},\bar{m}} \}_{\bar{n}, \bar{m} \in \NN_{0}^{(\NN)}}$, which is given by $f_{\bar{n},\bar{m}}(s \otimes s') = f(|\bar{n}|+|\bar{m}|)((s \otimes s') \otimes_{\Bbbk \SG_{\bar{n} \sqcup \bar{m}}^{\env}} (\id_{\llbracket 1 , |\bar{n}|+\bar{m}| \rrbracket} \otimes \id_{\llbracket 1 , |\bar{n}|+\bar{m}| \rrbracket}))$ for $s \in S(|\bar{n}|)_{\bar{n}}$ and $s' \in S'(|\bar{m}|)_{\bar{m}}$, is a bijection. 
\end{proposition}
%%%%%%%
\begin{proof}
The map \eqref{eq:morph-tensor-gwh} is clearly well defined and injective, by the definition of morphism of generalized wheelspaces. 
The surjectivity is also immediate, by the definition of morphisms of generalized wheelspaces.  
\end{proof}

%%%%%%%%%%%%%%%%%%%%%%%%%%%%%%%%%%%%%%%%%%%%%%%%%%%%%%%%%%%%%%%%%%%%
\subsection{Why generalized wheelspaces?}
\label{subsection:contradiction}

Let us consider three wheelspaces, $S$, $S'$ and $S''$, with contractions 
$({}^{n}_{i}t_{j})_{n \in \NN, i, j \in \llbracket 1 , n \rrbracket}$,
$({}^{n}_{i}t'_{j})_{n \in \NN, i, j \in \llbracket 1 , n \rrbracket}$ and $({}^{n}_{i}t''_{j})_{n \in \NN, i, j \in \llbracket 1 , n \rrbracket}$, respectively. 
Then $\Bil_{\gwh}(S,S';S'')$ is the set of all sequences of linear maps $\{ f_{n,m} : S(n) \otimes S'(m) \rightarrow \Res_{\Bbbk \sump_{n,m}^{\env}}(S''(n+m)) \}_{n,m \in \NN_{0}}$ satisfying that $f_{n,m}$ is a morphism of $(\Bbbk \SG_{n} \otimes \Bbbk \SG_{m})$-bimodules and 
\begin{equation}
{}^{n+m}_{\phantom{xxx}i}t''_{j} \circ f_{n,m} = 
\begin{cases}
f_{n-1,m} \circ ({}^{n}_{i}t_{j} \otimes \id_{S'(m)}), &\text{if $i, j \in \llbracket 1 , n\rrbracket$,}
\\
f_{n,m-1} \circ (\id_{S(n)} \otimes {}^{\phantom{x,}m}_{i-n}t'_{j-n}), &\text{if $i, j \in \llbracket n+1 , n+m\rrbracket,$}
\end{cases} 
\end{equation}
for all $n, m \in \NN_{0}$. 
If $\Set \label{index:categ-sets}$ denotes the category of sets, one would like the functor 
\begin{equation}
\Bil_{\wh}(S,S';-) = \Bil_{\gwh}(S,S';-)|_{\WMod} : \WMod \longrightarrow \Set 
\label{index:functor-restriction-gen-wheelsp-wheelsp}
\end{equation}
to be representable for any pair of wheelspaces $S$ and $S'$, and the representative should be a reasonable tensor product $S \otimes_{\WMod} S'$ of $S$ and $S'$ in $\WMod$. 
However, we will prove in Proposition \ref{proposition:contradiction} below that this functor is not representable, which forces us to introduce generalized wheelspaces, since the previous functor
$\Bil_{\gwh}(S,S';-)$ is representable in $\GWMod$ by $S \otimes_{\gwh} S'$, as shown in
Proposition \ref{proposition:morph-tensor-gwh}. 

\begin{remark}
\label{remark:monoidal-wh-1}
Note that the tensor product that we defined in \S\ref{subsection:gwh-mono} does not actually appear in \cite{MR2734329}, since the authors do not deal with generalized wheelspaces. 
On pp. 690--691 of that article, the authors claim that ``The category of wheelspaces forms a symmetric monoidal category in such a way that the forgetful functor from wheelspaces to $\SG$-modules, is symmetric monoidal (\dots)''. 
As we will prove in this subsection, there exists no monoidal structure on the category of wheelspaces representing the functor $\Bil_{\gwh}(S,S';-)|_{\WMod}$ introduced in \eqref{index:functor-restriction-gen-wheelsp-wheelsp}, which seems to be what the authors of \cite{MR2734329} had in mind (see \textit{e.g.} \cite{MR2734329}, Def. 3.1.6), especially due to their remarks about wheeled PROPs (see \textit{e.g.} \cite{MR2734329}, Rk. 3.2.5). 
\end{remark}
%%%%%%%%

Firstly, we need to state some prerequisites. 
Given $N \in \NN$, we define the functor 
\begin{equation}
\tau_{\leq N} : \GWMod \longrightarrow  \GWMod    
\label{eq:truncation-functor-gwh}
\end{equation}
sending a generalized wheelspace $R = (R(n))_{n \in \NN_{0}}$ with contractions $\{ {}^{\bar{n}}_{j}t_{i} \}$ 
to the generalized wheelspace $\tau_{\leq N}(R) = (\tau_{\leq N}(R)(n))_{n \in \NN_{0}}$ with contractions $\{ {}^{\bar{n}}_{j}\bar{t}_{i} \}$  such that $\tau_{\leq N}(R)(n) = R(n)$ 
if $n \leq N$ and $\tau_{\leq N}(R)(n) = 0$ 
if $n > N$, endowed with the same decomposition, and the contractions given by ${}^{\bar{n}}_{j}\bar{t}_{i} = {}^{\bar{n}}_{j}t_{i}$ if $|\bar{n}| \leq N$ (resp., to morphism $f = (f(n))_{n \in \NN_{0}}$ of generalized wheelspaces). 
If $f = (f(n))_{n \in \NN_{0}}$ is a a morphism of generalized wheelspaces from $S$ to $S'$, $\tau_{\leq N}(f) = (\tau_{\leq N}(f)(n))_{n \in \NN_{0}}$ is the unique morphism of generalized wheelspaces from $\tau_{\leq N}(S)$ to $\tau_{\leq N}(S')$ such that $\tau_{\leq N}(f)(n) = f(n)$ if $n \leq N$.
Notice that $\tau_{\leq N}(\WMod) \subseteq \WMod$. 
It is also clear that 
\begin{equation}
\label{eq:truncation-functor-property}
\Hom_{\GWMod}(R , T) \cong \Hom_{\GWMod}(R , \tau_{\leq N}(T))     
\end{equation} 
for all $R, T \in \GWMod$ if $R(n) = 0$ for all $n > N$. 

More generally, let $\D \subseteq \C$ be a strictly full subcategory of a category $\C$, \textit{i.e.} $\D$ is a full subcategory, and if $X \cong Y$ is an isomorphism in $\C$ with $Y$ in $\D$, then $X \in \D$. 
We denote by $\mathcalboondox{i} : \D \rightarrow \C\label{inclusion-index}$ the inclusion functor. 
Recall that an object $X \in \C$ has a \textbf{\textcolor{ultramarine}{reflector}} if the functor $\Hom_{\C}(X , \mathcalboondox{i}(-)) : \D \rightarrow \Set$ is representable, \textit{i.e.}, there exists $Y \in \D$ and a natural isomorphism $\Hom_{\C}(X , \mathcalboondox{i}(-)) \cong \Hom_{\D}(Y , -)$. 
As usual, the object $Y \in \D$ is uniquely determined up to (canonical) isomorphism, and we will write in this case $\label{reflector-index}Y = \mathcalboondox{p}(X)$. 
By picking $Z = \mathcalboondox{p}(X)$ and $\id_{\mathcalboondox{p}(X)} \in \Hom_{\D}(\mathcalboondox{p}(X) , \mathcalboondox{p}(X)) \cong \Hom_{\C}(X , \mathcalboondox{i}(\mathcalboondox{p}(X)))$, we see that there is a morphism $p_{X} : X \rightarrow \mathcalboondox{i}(\mathcalboondox{p}(X))$ in $\C$ such that, given $Z \in \D$,  
the isomorphism 
\begin{equation}
\label{eq:iso-d-c}
     \Hom_{\D}(\mathcalboondox{p}(X) , Z) \longrightarrow \Hom_{\C}(X , \mathcalboondox{i}(Z))
\end{equation}
is given by $g \mapsto \mathcalboondox{i}(g) \circ p_{X}$. 
Assume further that $\C$ and $\D$ are linear categories and that $\Hom_{\C}(Y , X) = 0$ for all $Y \in \D$ and $X \in \C \setminus \D$. 
This hypothesis together with the isomorphism \eqref{eq:iso-d-c} implies in particular that $p_{X}$ is an epimorphism in $\C$. 

Let $S$ and $S'$ be two wheelspaces with contractions $({}^{n}\mathcalboondox{t})_{n \in \NN}$ and
$({}^{n}\mathcalboondox{t}')_{n \in \NN}$, respectively. 
We now note that $\WMod$ is a stricly full linear subcategory of $\GWMod$, as well as $\Hom_{\GWMod}(R, T) = 0$ for all $R \in \WMod$ and $T \in \GWMod \setminus \WMod$.
Since $\Bi_{\gwh}(S,S';-)$ is representable 
by $S \otimes_{\gwh} S'$, if the functor given by $\Bil_{\wh}(S,S'; -) = \Bi_{\gwh}(S,S';-)|_{\WMod}$ was representable, there is a morphism $p_{S \otimes_{\gwh} S'} : S \otimes_{\gwh} S' \rightarrow \mathcalboondox{p}(S \otimes_{\gwh} S')$
such that the map \eqref{eq:iso-d-c} is an isomorphism for $X = S \otimes_{\gwh} S'$ and any $Z$ in $\WMod$. 
Moreover, by the argument in the previous paragraph, the map $p_{S \otimes_{\gwh} S'}$ 
is an epimorphism in $\GWMod$. 

Assume there exists $n_{0} \in \NN$ such that $S(n) = S'(n) = 0$ for all integers $n \geq n_{0}$. 
Then, there exists $n,n' \in \NN$ such that 
$S(n) \otimes S'(n') \neq 0$ but $(S \otimes_{\gwh} S')(N) = 0$ for $N > n+n'$. 
We first claim that this implies that $\mathcalboondox{p}(S \otimes_{\gwh} S')(N) = 0$ for all $N > n+n'$. 
Indeed, \eqref{eq:truncation-functor-property} together with the universal property of $\mathcalboondox{p}(S \otimes_{\gwh} S')$ tells us that 
$\mathcalboondox{p}(S \otimes_{\gwh} S') = \tau_{\leq n+n'}(\mathcalboondox{p}(S \otimes_{\gwh} S'))$, \textit{i.e.} $\mathcalboondox{p}(S \otimes_{\gwh} S')(N) = 0$ for all integers $N > n+n'$, as was to be shown.

After these preparations, let $V_{1}, W_{1}, V_{0}, W_{0}$ be vector spaces and let $f : V_{1} \rightarrow V_{0}$ and $g : W_{1} \rightarrow W_{0}$ be nonzero linear maps. 
Define the unique wheelspaces $S$ and $S'$ such that $S(n) = V_{n}$ and $S'(n) = W_{n}$ for $n \in \{ 0, 1\}$, $S(n) = S'(n) = 0$ for all integers $n \geq 2$, the unique nonzero contraction ${}^{1}\mathcalboondox{t}$ of $S$ is given by $f$ and the unique nonzero contraction ${}^{1}\mathcalboondox{t}'$ of $S'$ is given by $g$. 
It is easy to see that these are indeed wheelspaces. 
Consider now the diagonal $\SG$-bimodule $S''$ 
given by $S''(2) = \Ind_{\Bbbk \sump_{1,1}^{\env}}(V_{1} \otimes W_{1})$, 
$S''(1) = (V_{1} \otimes W_{0}) \oplus (V_{0} \otimes W_{1})$ and 
$S''(0) = V_{0} \otimes W_{0}$.
Let $h : V_{1} \otimes W_{1} \rightarrow S''(1)$ be any linear map. 
Define the map ${}^{1}\mathcalboondox{t}_{h} : S''(1) \rightarrow S''(0)$ given by $f \otimes \id_{V_{0}} + \id_{W_{0}} \otimes g$, as well as ${}^{2}\mathcalboondox{t}_{h} : S''(2) \rightarrow S''(1)$ by 
\begin{equation}
   {}^{2}\mathcalboondox{t}_{h}\big((v \otimes w)  \otimes_{\Bbbk \SG_{2}^{\env}} (\sigma \otimes \tau) \big) 
   = \begin{cases}
   v \otimes g(w), \;\text{ if $\sigma = \tau = \id_{\{ 1,2\}}$,}
   \\
   f(v) \otimes w, \;\text{ if $\sigma = \tau = ( 1 \ 2)$,}
   \\
   h(v \otimes w), \;\text{ if $\sigma, \tau \in \SG_{2}$ and $\sigma \neq \tau$.}
   \end{cases}
\end{equation}
It is then easy to see that, for any linear map $h : V_{1} \otimes W_{1} \rightarrow V_{1} \otimes W_{0} \oplus V_{0} \otimes W_{1}$, the diagonal $\SG$-bimodule $S''$ with the contractions ${}^{1}\mathcalboondox{t}_{h}$ and ${}^{2}\mathcalboondox{t}_{h}$ becomes a wheelspace, since it verifies condition \eqref{eq:part-trace-comm}. 
We will denote it $S \otimes_{h} S'$.
Furthermore, using $Z = S \otimes_{h} S'$ with $h = 0$ in \eqref{eq:iso-d-c} and the monomorphism 
\begin{equation}
\varphi_{2}(S,S') : S \otimes_{\gwh} S' \longrightarrow S \otimes_{h} S',
\label{eq:monomorphismo-phi-2}
\end{equation}
whose underlying morphism of diagonal $\SG$-bimodules is the identity, we conclude that $p_{S \otimes_{\gwh} S'}$ is also a monomorphism. 

%%%%%%%
\begin{proposition}
\label{proposition:contradiction}
Let $S = (S(n))_{n \in \NN_{0}}$ and $S' = (S'(n))_{n \in \NN_{0}}$ be the wheelspaces defined in the previous paragraph. 
Then the functor $\Bil_{\wh}(S,S'; -)$ introduced in \eqref{index:functor-restriction-gen-wheelsp-wheelsp} is not representable. 
\end{proposition}
%%%%%%% 
\begin{proof}
Assume that $\Bil_{\wh}(S,S'; -)$ is representable and let $\mathcalboondox{p}(S \otimes_{\gwh} S')$ be a representative, following the notation of 
the previous paragraphs. 
For every linear map $h : V_{1} \otimes W_{1} \rightarrow V_{1} \otimes W_{0} \oplus V_{0} \otimes W_{1}$, there exists a unique morphism of generalized wheelspaces $i_{h} : S \otimes_{\gwh} S' \rightarrow S \otimes_{h} S'$ whose underlying morphism of diagonal $\SG$-bimodules is the identity. 
Let $j_{h} : \mathcalboondox{p}(S \otimes_{\gwh} S') \rightarrow S \otimes_{h} S'$ be the unique morphism of wheelspaces such that $j_{h} \circ p_{S \otimes_{\gwh} S'} = i_{h}$. 
Since $j_{h}$ is a morphism of wheelspaces, then ${}^{2}\mathcalboondox{t}_{h} \circ j_{h}(2) = j_{h}(1) \circ {}^{2}\mathcalboondox{T}$, where $\{ {}^{n}\mathcalboondox{T} \}_{n \in \NN}$ denote the partial contractions of $\mathcalboondox{p}(S \otimes_{\gwh} S')$. 
Since $i_{h}$ is an epimorphism in $\GWMod$, $j_{h}$ is an epimorphism of $\WMod$, \textit{i.e.} $j_{h}(n)$ is surjective for all $n \in \NN_{0}$. 
Furthermore, we also note that $j_{h}(n)$ is even bijective for all $n \in \{ 0, 1 \}$, since the map 
\[     \Hom_{\WMod}\big(\tau_{\leq 1}(S \otimes_{h} S'),T(n)\big) \longrightarrow \Hom_{\WMod}\Big(\tau_{\leq 1}\big(\mathcalboondox{p}(S \otimes_{\gwh} S')\big),T(n)\Big)     \]
sending $g$ to $g \circ \tau_{\leq 1}(j_{h})$
is clearly an isomorphism for all $T \in \WMod$ such that $\tau_{\leq 1}(T) = T$. 
We can thus assume without loss of generality that $\mathcalboondox{p}(S \otimes_{\gwh} S')(n) = (S \otimes_{h} S')(n)$ for $n \in \{ 0, 1\}$ and ${}^{1}\mathcalboondox{T} = {}^{n}\mathcalboondox{t}_{h}$, and 
$j_{h}(n)$ is the identity map for $n \in \{ 0, 1\}$. 
Hence, ${}^{2}\mathcalboondox{T} = {}^{1}\mathcalboondox{t}_{h} \circ j_{h}(2)$, which \textit{a fortiori} implies the identity of the images $\Img({}^{2}\mathcalboondox{T}) = \Img({}^{1}\mathcalboondox{t}_{h})$, 
since $j_{h}(2)$ is surjective. 
However, since the image $\Img({}^{1}\mathcalboondox{t}_{h})$ depends on the choice of $h$, we get a contradiction, which in turn implies that the representative $\mathcalboondox{p}(S \otimes_{\gwh} S')$ does not exist. 
\end{proof}

Although in general the tensor product of wheelspaces is not a wheelspace but a generalized wheelspace, in some interesting cases, we do obtain wheelspaces as the following simple result shows. 
%%%%%%%
\begin{fact}
\label{fact:monoidal-wh+0}
Let $S$ be a wheelspace with contractions $({}^{n}_{i}t_{j})_{n \in \NN, i, j \in \llbracket 1 , n \rrbracket}$, and $V$ be a vector space. 
Then, $S \otimes_{\gwh} \I_{\SG^{\env},0}(V)$ is also a wheelspace, where $\I_{\SG^{\env},0}(V)$ has the trivial wheelspace structure indicated in Example \ref{example:wh0}. 
\end{fact}
%%%%%%%
\begin{proof}
It is easy to check that $(S \otimes_{\gwh} \I_{\SG^{\env},0}(V))(n) = S(n) \otimes V$ for all $n \in \NN_{0}$, and its contractions are given by
$({}^{n}_{i}t_{j} \otimes \id_{V})_{n \in \NN, i, j \in \llbracket 1 , n \rrbracket}$ for all $n \in \NN_{0}$. 
\end{proof}

From now on, we will implicitly use the notation of Proposition \ref{proposition:morph-tensor-gwh} and 
this subsection: given wheelspaces $S$, $S'$ and $S''$, we will  identify a morphism of generalized wheelspaces $f : S \otimes_{\gwh} S' \rightarrow S''$ with the corresponding collection of maps 
$\{ f_{n,m} : S(n) \otimes S'(m) \rightarrow S''(n+m) \}$.

%%%%%%%%%%%%%%%%%%%%%%%%%%%%%%%%%%%%%%%%%%%%%%%%%%%%%%%%%%%%%%%%%%%%%%
\subsection{Internal homomorphism spaces of wheelspaces}
\label{subsection:internal-homos}

Given a wheelspace $S = (S(n))_{n \in \NN_{0}}$ with contractions 
$({}^{N}_{\phantom{i}j}t_{i})_{N \in \NN, i, j \in \llbracket 1, N \rrbracket}$
and $k \in \NN_{0}$, we define the \textbf{\textcolor{myblue}{shifted}} 
wheelspace $\sh_{k}^{\wh}(S)\label{eq:shifted-wheelsp}$ whose underlying diagonal $\SG$-bimodule is given by $\sh_{k}^{\env}(\Fo(S))$, \textit{i.e.}
\[
\Fo\Big(\sh_{k}^{\wh}(S)\Big)(N) = \sh_{k}^{\env}\big(\Fo(S)\big)(N) = \Res_{\Bbbk \rinc_{N,N+k}^{\env}}\big(S(N+k)\big)
\]
for $N \in \NN_{0}$ (see Subsection \ref{sec:shifted-SG-modules}), 
with contractions ${}^{N}_{\phantom{i}j}\sh_{k}(t)_{i} = {}^{N+k}_{\phantom{i}j+k}t_{i+k}$ for $N \in \NN$ and $i, j \in \llbracket 1, N \rrbracket$. 
Similarly, given a morphism $f : S \rightarrow S'$ of wheelspaces, we define the \textbf{\textcolor{myblue}{shifted}} morphism of wheelspaces $\sh_{k}^{\wh}(f) : \sh_{k}^{\wh}(S) \rightarrow \sh_{k}^{\wh}(S')$ given by $\sh_{k}^{\wh}(f)(N) = f(N+k)$ for $N \in \NN_{0}$. 
It is easy to see that this defines a functor 
\begin{equation}
\label{eq:shifted-w-mod}
    \sh_{k}^{\wh} : \WMod \longrightarrow \WMod.
\end{equation}

Given wheelspaces $S = (S(n))_{n \in \NN_{0}}$ and $S' = (S'(n))_{n \in \NN_{0}}$ with sets of contractions 
$({}^{N}_{\phantom{i}j}t_{i})_{N \in \NN, i, j \in \llbracket 1, N \rrbracket}$ and $({}^{N}_{\phantom{i}j}t'_{i})_{N \in \NN, i, j \in \llbracket 1, N \rrbracket}$, respectively, we define 
the wheelspace $\IHom_{\WMod}(S,S')$ satisfying that
\begin{equation}
\label{eq:int-hom-w-mod}
\begin{split}
     \IHom_{\WMod}(S,S')(n) = \Hom_{\WMod}\big(S,\sh_{n}^{\wh}(S')\big) 
     \subseteq \IHom_{\SG^\env}(\Fo(S),\Fo(S'))(n)  
\end{split}
\end{equation} 
for all $n \in \NN_{0}$ in the following way.
It is immediate to see that the action of $\Bbbk \SG_{n}^{\env}$ on 
$\IHom_{\SG^\env}(\Fo(S),\Fo(S'))(n)$ described at in Subsection \ref{sec:shifted-SG-modules} preserves the subspace $\IHom_{\WMod}(S,S')(n)$, so it is naturally a $\Bbbk \SG_{n}^{\env}$-module. 
Moreover, given $n \in \NN$ and $i, j \in \llbracket 1, n \rrbracket$, we define the contraction ${}^{n}_{j}T_{i}\label{eq: contractions-internal-wheelsp-hom}$ on 
$f  = (f(m))_{m \in \NN_{0}} \in \IHom_{\WMod}(S,S')(n)$ by ${}^{n}_{j}T_{i}(f) = (g(m))_{m \in \NN_{0}}$ with $g(m) = {}^{n+m}_{\phantom{xxx}j}t'_{i} \circ f(m)$ for all $m \in \NN_{0}$.
It is clear that $\IHom_{\WMod}(S,S')$ is indeed a wheelspace. 

The following result is now direct. 
%%%%%%%
\begin{fact}
\label{fact:internal-hom-w-mod}
Given wheelspaces $S$, $S'$ and $S''$, then the natural linear isomorphism of Fact \ref{fact:internal-hom-s-mod} induces a natural linear isomorphism
\begin{align*}
\Hom_{\GWMod}(S \otimes_{\gwh} S', S'') &\overset{\cong}{\longrightarrow} \Hom_{\WMod}\big(S , \IHom_{\WMod}(S', S'') \big).       
\end{align*}
\end{fact}
%%%%%%%

%%%%%%%%%%%%%%%%%%%%%%%%%%%%%%%%%%%%%%%%%%%%%%%%%%%%%%%%%%%%%%%%%%%%%%%%%%%%%%%%%%%%%%%%%%%%%%%%%%%%%%
\subsection{Algebraic structures}
\label{sec:alg-str}
%%%%%%%%%%%%%%%%%%%%%%%%%%%%%%%%%%%%%%%%%%%%%%%%%%%%%%%%%%%%%%%%%%%%%%%%%%%%%%%%%%%%%%%%%%%%%%%%%%%%%%%%%

\subsubsection{Associative, Lie and Poisson wheelgebras and wheelbimodules}

The following definition is then natural, once we have the symmetric monoidal structure on $\GWMod$ (\textit{cf.} \cite{MR2734329}, Def. 3.1.6). 
%%%%%%%
\begin{definition} 
Let $S$ be a wheelspace. 
It is called a(n unitary and associative) \textbf{\textcolor{myblue}{wheelgebra}} if it 
is a(n unitary and associative) algebra in the symmetric monoidal category of generalized wheelspaces $(\GWMod,\otimes_{\gwh},\mathbf{1}_{\SG^{\env}}, \tau^{\SG^{\env}})$. 
It is further said to be \textbf{\textcolor{myblue}{commutative}} 
if it is a commutative algebra in such a symmetric monoidal category. 
A (resp., \textbf{\textcolor{myblue}{symmetric}}) \textbf{\textcolor{myblue}{generalized wheelbimodule}} over a (resp., commutative) wheelgebra $\wA$ is a generalized wheelspace $T$ that is also a 
(resp., symmetric) $\wA$-bimodule
in the symmetric monoidal category of generalized wheelspaces. 
If a (resp., symmetric) generalized wheelbimodule over a (resp., commutative) wheelgebra $\wA$ is a wheelspace we will say it is a (resp., \textbf{\textcolor{myblue}{symmetric}}) \textbf{\textcolor{myblue}{generalized wheelbimodule}} over $\wA$.
Similarly, a \textbf{\textcolor{myblue}{Lie wheelgebra}} (resp., \textbf{\textcolor{myblue}{Poisson wheelgebra}}) 
is a wheelspace endowed with a Lie algebra (resp., a Poisson algebra) structure in the symmetric monoidal category of generalized wheelspaces. 
\label{def:Poisson-Lie-wheeled}
\end{definition}
%%%%%%%

We will denote by $\WAlg \label{eq:categ-wheelalg-def}$ the full subcategory of $\uAlg(\GWMod)$ formed by wheelgebras. 
Given a (resp., commutative) wheelgebra $\wA$, we will denote by 
${}_{\wA}\GWhMod_{\wA} \label{index:categ-generalize-wheelbimod-def}$ (resp., ${}_{\wA}\GWhMod^{\s}_{\wA} \label{index:categ-sym-generalize-wheelbimod-def}$) the category ${}_{\wA}\Mod_{\wA}(\GWMod)$ (resp., ${}_{\wA}\Mod^{\s}_{\wA}(\GWMod)$) formed by all (resp., symmetric) generalized wheel\-bi\-mod\-ules over $\wA$, and by ${}_{\wA}\WhMod_{\wA} \label{index:categ-wheelbimod-def}$ (resp., ${}_{\wA}\WhMod_{\wA}^{\s} \label{index:categ-wheelbimod-sym-def}$) its full subcategory formed by all (resp., symmetric) $\wA$-wheel\-bi\-mod\-ules. 
As recalled in Subsubsection \ref{subsubsection:mod-bimodules}, instead of symmetric (resp., generalized) whelbimodules over a commutative wheelgebra $\wA$, we will often work with the isomorphic category of 
(resp., generalized) whelmodules over $\wA$, whose definition is immediate, and which will be denoted by 
${}_{\wA}\WhMod \label{index:categ-wheelmod-def}$ (resp., ${}_{\wA}\GWhMod \label{index:categ--generalize-wheelmod-def}$). 

Recall the forgetful functor $\Fo$ given in \eqref{eq:forgetful-wh-dmod}. 
We have the following explicit description of a (resp., commutative) wheelgebra and wheelbimodule over it, whose proof is immediate.  
%%%%%%%
\begin{fact} 
\label{fact:gwh-algebra-ex}
Let $S$ be a wheelspace and $S'$ a generalized wheelspace. 
\begin{enumerate}[label=(\arabic*)]
\item \label{item:4.21-1}
A morphism $\mu : S \otimes_{\gwh} S \rightarrow S$ of generalized wheelspaces gives a(n unitary and associative) wheelgebra if and only if  $\Fo(S)$ is a(n unitary and associative) algebra in the monoidal category of diagonal $\SG$-bimodules for the product $\Fo(\mu) : \Fo(S) \otimes_{\SG^{\env}} \Fo(S) \rightarrow \Fo(S)$. 
\item \label{item:4.21-2}
Moreover, $S$ endowed with $\mu$ is commutative if and only if $\Fo(S)$ endowed with 
$\Fo(\mu)$ is a commutative algebra 
in the symmetric monoidal category of diagonal $\SG$-bimodules.
\item Assume that $S$ is a (resp., commutative) wheelgebra for the product $\mu : S \otimes_{\gwh} S \rightarrow S$. 
A morphism $\rho : S \otimes_{\gwh} S' \otimes_{\gwh} S \rightarrow S'$ of generalized wheelspaces gives a (resp., symmetric) generalized wheelbimodule structure on $S'$ over $S$ if and only if $\Fo(\rho) : \Fo(S) \otimes_{\SG^{\env}} \Fo(S') \otimes_{\SG^{\env}} \Fo(S) \rightarrow \Fo(S')$ gives a (resp., symmetric) bimodule structure on $\Fo(S')$ over the algebra $\Fo(S)$ with product $\Fo(\mu)$ in the monoidal category of diagonal $\SG$-bimodules. 
\end{enumerate}
\end{fact}
%%%%%%%

%%%%%%%
\begin{example}
\label{example:wh0-2} 
Let $B$ be a (resp., commutative) algebra in $\Vect$. 
Then, Fact \ref{fact:vect-s-se}, Example \ref{example:wh0} and Fact \ref{fact:gwh-algebra-ex} implies that $\I_{\SG^{\env},0}(B)$ is a (resp., commutative) wheelgebra. 
\end{example}
%%%%%%%

%%%%%%%
\begin{remark}
Given a wheeled PROP (see \cites{MR2648322,MR2483835, MR2641194}), or equivalently a TRAP (as defined in \cite{MR4507248}), $(P(m,n))_{m,n \in \NN_{0}}$, its diagonal part 
$(P(n,n))_{n \in \NN_{0}}$ is precisely a commutative wheelgebra. 
This follows from the fact that the tensor product $\otimes_{\gwh}$ in the monoidal category of generalized wheelspaces represents the functor $\Bil_{\gwh}(\place,\place;\place)$, which appears in the definition of PROP when restricted to the diagonal part. 
On the other hand, working with not necessarily diagonal $\SG$-bimodules instead of diagonal ones, one can easily extend the results in Subsection \ref{subsection:wh-basic} to show that a wheeled PROP is exactly a commutative algebra in a symmetric monoidal category analogous to our category of generalized wheelspaces, but formed of (nondiagonal) $\SG$-bimodules. 
\end{remark}
%%%%%%%

%%%%%%%
\begin{example} 
\label{example:end-alg-w-mopd}
Let $S=(S(n))_{n \in \NN_{0}}$ be a wheelspace. 
Then, the wheelspace $\IHom_{\WMod}(S,S)$ has a natural structure of a wheelgebra for the product 
\[     \mu_{n,m} : \IHom_{\WMod}(S,S)(n) \otimes \IHom_{\WMod}(S,S)(m) \longrightarrow \IHom_{\WMod}(S,S)(n+m)     \] 
given by 
$\mu_{n,m}(f,g) = \sh_{m}^{\wh}(f) \circ g$ for $n, m \in \NN_{0}$ and the unit $\eta : \Bbbk \rightarrow \IHom_{\WMod}(S,S)(0)$ given by $\eta(1) = \id_{S}$.
\end{example}
%%%%%%%

The following result can be regarded as an extension 
of Fact \ref{fact:gwh-algebra-ex}, and its proof is immediate.  
%%%%%%%
\begin{lemma} 
\label{fact:wh-poisson-algebra-ex}
Let $S$ be a wheelspace.
We further assume that
\begin{enumerate}[label=(\arabic*)]
    \item \label{fact:wh-poisson-algebra-ex-1}
$S$ is endowed with a morphism $\mu : S \otimes_{\gwh} S \rightarrow S$ of generalized wheelspaces giving a(n unitary and associative) commutative wheelgebra, and 
\item \label{fact:wh-poisson-algebra-ex-2}
$\{ \hskip 0.6mm , \} : S \otimes_{\gwh} S \rightarrow S$ is a morphism of generalized wheelspaces such that $\Fo(S)$ endowed with $\Fo(\{ \hskip 0.6mm , \})$ 
is a Poisson algebra in the symmetric monoidal category $\DMod$ of diagonal $\SG$-bimodules. 
\end{enumerate}
Then, $(S,\mu)$ endowed with $\{ \hskip 0.6mm , \}$ is a Poisson wheelgebra. 
\end{lemma}
%%%%%%%

%%%%%%%%%%%%%%%%%%%%%%%%%%%%%%%%%%%%%%%%%%%%%%%%%%%%%%%%%%%%%%%%%%
\subsubsection{More on generalized wheelbimodules}

Let $\wA$ be a wheelgebra and let $\wM$ be a generalized wheelbimodule over $\wA$. 
Consider a family $\{ \wM^{i} : i \in \mathtt{I} \}$ of $\wA$-subwheelbimodules of $\wM$. 
Then, the generalized wheelspace given by the intersection of the underlying generalized wheelspaces of $\wM^{i}$ for $i \in I$ recalled in Remark \ref{remark:intersection-wheel} is naturally a $\wA$-subwheelbimodule of $\wM$.
Let $T$ be a diagonal $\SG$-subbimodule of $\gFo(\wM)$, where we also denote by $\wM$ the underlying generalized wheelspace of $\wM$. 
Define the \textbf{\textcolor{myblue}{generalized $\wA$-subwheelbimodule of $\wM$ generated by $T$}} as the intersection of the nonempty family 
$\{ \wM^{i} = (\wM^{i}(n))_{n \in \NN_{0}} : i \in \mathtt{I} \}$ formed by all generalized $\wA$-subwheelbimodules of $\wM$ containing $T$. 
We will denote this generalized $\wA$-subwheelbimodule of $\wM$ by $\langle T \rangle \label{index:gen-wheelbimodule-generated-by}$.

Given a commutative wheelgebra $\wC$ and two generalized $\wC$-wheelmodules $(\wM, \rho)$ and $(\wM',\rho')$, we define the \textbf{\textcolor{myblue}{tensor product}}  $\wM \otimes_{\wC} \wM'$ as the generalized wheelquotient of 
the generalized wheelspace
$\wM \otimes_{\gwh} \wM'$ 
by the generalized subwheelspace generated by 
\begin{equation}
\label{eq:tensor-prod-c-wheelmod}
\Big\{ \hskip -0.6mm \rho_{\bar{n},p}(v,c) \otimes w - v \otimes \rho'_{p,\bar{m}}(c,w) \hskip -0.8mm : \hskip -0.8mm v \in \wM(|\bar{n}|)_{\bar{n}}, c \in \wC(p), w \in \wN(|\bar{m}|)_{\bar{m}}, \bar{n}, \bar{m} \in \NN_{0}^{(\NN)}, p \in \NN_{0} \hskip -0.6mm \Big\}\hskip -0.6mm .     \end{equation}
The following result is easily verified. 
%%%%%%
\begin{fact} 
\label{fact:symm-mon-wheelmod}
Given a commutative wheelgebra $\wC$ and two generalized $\wC$-wheelmodules $(\wM, \rho)$ and $(\wM',\rho')$, then $\wM \otimes_{\wC} \wM'$ is a generalized $\wC$-wheelmodule. 
Moreover, this tensor product gives naturally a symmetric monoidal structure on ${}_{\wC}\GWhMod$ with unit object $\wC$ and 
symmetric braiding induced by that of $\GWMod$. 
\end{fact}
%%%%%%

As one might expect, the shift operation of wheelspaces in Subsection \ref{subsection:internal-homos} preserves bimodule structures, as the following example shows. 
%%%%%%%
\begin{example} 
\label{example:shift-w-mod}
Let $\wC$ be a commutative wheelgebra and let $(\wM, \rho)$ be a resp., symmetric $\wC$-wheelbimodule with right action $\rho : \wM \otimes_{\gwh} \wC \rightarrow \wM$. 
Then, given $k \in \NN_{0}$, the shifted space $\sh_{k}^{\wh}(\wM)$ of the underlying wheelspace of $\wM$
defined in Subsection \ref{subsection:internal-homos} has a natural structure of 
(resp., symmetric) $\wC$-biwheelmodule given by the right action $\sh_{k}(\rho) : \sh_{k}^{\wh}(\wM) \otimes_{\gwh} \wC \rightarrow \sh_{k}^{\wh}(\wM)$ such that
$\sh_{k}(\rho)_{N,m} = \rho_{k+N,m}$, for all $N,m \in \NN_{0}$. 
We will still denote the corresponding symmetric $\wC$-wheelbimodule by $\sh_{k}^{\wh}(\wM)$. 
\end{example}
%%%%%%%

Assume now that $(\wM, \rho)$ and $(\wM',\rho')$ are two $\wC$-wheelmodules over a commutative wheelgebra $\wC$. 
Given $f = (f(m))_{m \in \NN_{0}} \in \IHom_{\WMod}(\wM,\wM')(n)$ and $c \in \wC(p)$, define $\wad_{c}(f) \in \IHom_{\WMod}(\wM,\wM')(n+p)$ by 
\begin{equation}   
\label{eq:wad}
\wad_{c}(f)(m)(x) = \rho'_{p,m+n} \big(c, f(m)(x)\big) - \block_{n,p,m}(1 \, 2) \cdot f(m+p)\big(\rho_{p,m}(c,x)\big) \cdot\block_{p,n,m}(1 \, 2)    
\end{equation}
for $m \in \NN_{0}$ and $x \in \wM(m)$. 
It is then direct to check that, given $c \in \wC(p)$, \eqref{eq:wad} defines a morphism of wheelspaces 
\begin{equation}   
\label{eq:wad-bis}
\wad_{c} :  \IHom_{\WMod}(\wM,\wM') \rightarrow \sh_{p}^{\wh}\Big(\IHom_{\WMod}(\wM,\wM')\Big).
\end{equation}

Given $n \in \NN_{0}$, we define the $\Bbbk \SG_{n}^{\env}$-subspace $\IHom_{\wC}(\wM,\wM')(n)$ of $\IHom_{\WMod}(\wM,\wM')(n)$
by 
\begin{equation}
\label{eq:hom-w-mod-c}
\IHom_{\wC}(\wM,\wM')(n) = \big\{ f \in \IHom_{\WMod}(\wM, \wM')(n) : \wad_{c}(f) = 0 \text{ for all } c \in \wC(p), p \in \NN_{0} \big\}.     
\end{equation}
Since \eqref{eq:wad-bis} is a morphism of wheelspaces, the collection 
$\{ \IHom_{\wC}(\wM,\wM')(n) : n \in \NN_{0} \}$
is a subwheelspace of $\IHom_{\WMod}(\wM,\wM')$, that we denote by $\IHom_{\wC}(\wM,\wM')$.

With the previous assumptions, define an action 
\begin{equation}
\label{eq:wmod-int-c-pre}
\bar{\rho} : \wC \otimes_{\gwh} \IHom_{\wC}(\wM,\wM') \rightarrow \IHom_{\wC}(\wM,\wM')     
\end{equation}
 by $\bar{\rho}_{p,n}(c,f) = (g(m))_{m \in \NN_{0}} \in \IHom_{\WMod}(\wM,\wM')(p+n)$ with 
\begin{equation}
\label{eq:wmod-int-c}
g(m)(x) = \rho'_{p,n+m}\big(c,f(m)(x)\big) 
\end{equation}
for all $p, m,n \in \NN_{0}$, $c \in \wC(p)$,  
$f = (f(m))_{m \in \NN_{0}} \in \IHom_{\wC}(\wM,\wM')(n)$, and $x \in \wM(m)$. 
Using Fact \ref{fact:gwh-algebra-ex} and the previous definitions we directly obtain the following result. 

%%%%%%%
\begin{fact}
\label{fact:wheelmod-tensor-hom}
Assume now that $(\wM, \rho)$ and $(\wM',\rho')$ are two $\wC$-wheelmodules over a commutative wheelgebra $\wC$. Then, the wheelspace $\IHom_{\wC}(\wM,\wM')$ defined by \eqref{eq:hom-w-mod-c} is $\wC$-wheelmodule
for the action given by \eqref{eq:wmod-int-c-pre}. 
Moreover, the adjunction given in Fact \ref{fact:internal-hom-w-mod} induces a natural linear isomorphism
\begin{align*}
\Hom_{\wC}(\wM \otimes_{\wC} \wM', \wM'') &\cong \Hom_{\wC}\big(\wM , \IHom_{\wC}(\wM', \wM'') \big),    
\end{align*}
for all $\wC$-modules $\wM$, $\wM'$ and $\wM''$. 
\end{fact}
%%%%%%%

%%%%%%%
\begin{remark}
\label{remark:hom-ihom-wheelmod}
Assume now that $(\wM, \rho)$ and $(\wM',\rho')$ are two $\wC$-wheelmodules over a commutative wheelgebra $\wC$. 
It is immediate from the previous definitions that 
\begin{equation}   
\label{eq:hom-wheelhom}
\Hom_{\wC}(\wM , \sh_{k}^{\wh}(\wM')) = \IHom_{\wC}(\wM , \wM')(k).   
\end{equation}
for all $k \in \NN_{0}$. 
\end{remark} 
%%%%%%%

%%%%%%%%%%%%%%%%%%%%%%%%%%%%%%%%%%%%%%%%%%%%%%%%%%%%%%%%%%%%%%%
\subsubsection{Admissibility of wheelgebras and wheelbimodules}

The following property of wheelgebras did not appear in \cite{MR2734329}, but as we will see it plays a key role in establishing a relation between algebraic constructions in the category $\WMod$ and noncommutative geometry. 
%%%%%%%
\begin{definition}
\label{definition:admissible-wh}
Let $S = (S(n))_{n \in \NN_{0}}$ be a wheelgebra with product $\mu : S \otimes_{\gwh} S \rightarrow S$ and 
contractions $({}^{n}_{i}t_{j})_{n \in \NN, i, j \in \llbracket 1 , n \rrbracket}$. 
We will say that $S$ is \textbf{\textcolor{myblue}{admissible}} 
if the map 
\begin{equation} 
\label{eq:prod-S1}
{}^{2}_{2}t_{1} \circ \mu_{1,1} : S(1) \otimes S(1) \longrightarrow S(1)
\end{equation}
turns $S(1)$ into a nonunitary algebra. 
\end{definition}
%%%%%%%

%%%%%%%
\begin{remark}
\label{remark:sub-wheel-adm}
Note that a subwheelgebra of an admissible wheelgebra is clearly admissible. 
\end{remark}
%%%%%%%

Before proceeding further let us give some simple examples of the admissibility property. 

%%%%%%%
\begin{example}
\label{example:end-4} 
Combining Examples \ref{example:end-2}, \ref{example:emod} and \ref{example:end-3} with 
Fact \ref{fact:gwh-algebra-ex}, we immediately obtain that $\EE^{\env}_{V}$ 
is a commutative wheelgebra for any finite dimensional vector space $V$. 
Moreover, since $\EE^{\env}_{V}(1) = \End(V)$ and the corresponding map \eqref{eq:prod-S1} is the usual composition product of 
$\End(V)$, we conclude that $\EE^{\env}_{V}$ is an admissible commutative wheelgebra. 

More generally, given also a (resp., commutative) algebra $B$ in $\Vect$, taking also into account Example \ref{example:wh0-2} and Fact \ref{fact:monoidal-wh+0}, we get that the (resp., commutative) wheelgebra $S = \EE^{\env}_{V} \otimes_{\gwh} \I_{\SG^{\env},0}(B)$ is also admissible, since in this case $S(1) = \End(V) \otimes B$ and the corresponding map \eqref{eq:prod-S1} is the usual composition product of a tensor product of algebras in $\Vect$. 
\end{example}
%%%%%%%

The following result provides even more examples of admissible wheelgebras. 
%%%%%%%
\begin{lemma}
\label{lemma:w-mod-end} 
Let $S=(S(n))_{n \in \NN_{0}}$ be a wheelspace with contractions 
$({}^{n}_{i}t_{j})_{n \in \NN, i, j \in \llbracket 1 , n \rrbracket}$. 
Then, the wheelgebra $\IHom_{\WMod}(S,S)$ of Example \ref{example:end-alg-w-mopd} is admissible. 
\end{lemma}
%%%%%%%
\begin{proof}
First of all, note that the associated product \eqref{eq:prod-S1} on $\IHom_{\WMod}(S,S)(1)$ is given by $(f \cdot g)(i) = {}^{i+2}_{\phantom{xx}2}t_{1} \circ g(i+1) \circ f(i)$, for all $i \in \NN_{0}$
and $f, g \in \IHom_{\WMod}(S,S)(1)$.
Let $f, g, h \in \IHom_{\WMod}(S,S)(1)$. 
It suffices to show that $(f \cdot g) \cdot h = f \cdot (g \cdot h)$, \textit{i.e.} 
\[     {}^{i+2}_{\phantom{xx}2}t_{1} \circ {}^{i+3}_{\phantom{xx}2}t_{1} \circ f(i+2) \circ g(i+1) \circ h(i) 
= {}^{i+2}_{\phantom{xx}2}t_{1} \circ f(i+1) \circ {}^{i+2}_{\phantom{xx}2}t_{1} \circ g(i+1) \circ h(i).     \]
Now, since $S$ is a wheelspace, \ref{item:W2} tells us that 
${}^{i+2}_{\phantom{xx}2}t_{1} \circ {}^{i+3}_{\phantom{xx}2}t_{1} = {}^{i+2}_{\phantom{xx}2}t_{1} \circ {}^{i+3}_{\phantom{xx}3}t_{2}$, so 
\begin{align*}     
{}^{i+2}_{\phantom{xx}2}t_{1} \circ {}^{i+3}_{\phantom{xx}2}t_{1} \circ f(i+2) \circ g(i+1) \circ h(i) 
&= {}^{i+2}_{\phantom{xx}2}t_{1} \circ {}^{i+3}_{\phantom{xx}3}t_{2} \circ f(i+2) \circ g(i+1) \circ h(i)
\\
&={}^{i+2}_{\phantom{xx}2}t_{1} \circ f(i+1) \circ {}^{i+2}_{\phantom{xx}2}t_{1} \circ g(i+1) \circ h(i),
\end{align*}
where in the last identity we used that ${}^{i+3}_{\phantom{xx}2}t_{3} \circ f(i+2) = f(i+1) \circ {}^{i+2}_{\phantom{xx}2}t_{1}$, since $f \in \IHom_{\WMod}(S,S)(1)$. 
\end{proof}

We finally show that in fact all wheelgebras are admissible. 
%%%%%%%
\begin{lemma}
\label{lemma:wheel-adm} 
Let $\wA$ be a wheelgebra with product $\mu$. 
By Fact \ref{fact:internal-hom-w-mod}, $\mu$ induces a morphism
\[     \bar{\mu} : \wA \rightarrow \IHom_{\WMod}(\wA,\wA)     \]
of wheelspaces. 
Moreover, $\bar{\mu}$ is a monomorphism of wheelspaces and a morphism of wheelgebras, where $\IHom_{\WMod}(\wA,\wA)$ has the wheelgebra structure of Example \ref{example:end-alg-w-mopd}. 
In consequence, $\wA$ is admissible. 
\end{lemma}
%%%%%%%
\begin{proof}
In explicit terms, given $n \in \NN_{0}$ and $a \in \wA(n)$, 
\[     \bar{\mu}(n)(a) = \big(g(m)\big)_{m \in \NN_{0}} \in \IHom_{\WMod}(\wA,\wA) \subseteq \prod_{m \in \NN_{0}} \Hom\big(\wA(m),\wA(n+m)\big)     \] 
is given by 
\[    g(m)(b) = \mu_{n,m}(a,b)     \] 
for all $m \in \NN_{0}$ and $b \in \wA(m)$. 
The unitarity condition on the wheelgebra $\wA$ then implies that the linear map $\bar{\mu}(n)$ is injective for all $n \in \NN_{0}$. 
Indeed, if we denote by $1_{\wA} \in \wA(0)$ the image of $1_{\Bbbk}$ under the unit map $\eta : \Bbbk \rightarrow \wA(0)$ of $\wA$ and if $\bar{\mu}(n)(a) =  \big(g(m)\big)_{m \in \NN_{0}}$ vanishes for some $a \in \wA(n)$, then $g(0)$ vanishes, 
which in turn implies that $g(0)(1_{\wA}) = \mu_{n,0}(a,1_{\wA}) = a$ vanishes. 
Similarly, the associativity of $\mu$ implies that $\bar{\mu}$ is a morphism of wheelgebras. 
The admissibility of $\wA$ now follows from Remark \ref{remark:sub-wheel-adm} and Lemma \ref{lemma:w-mod-end}. 
\end{proof}

It is immediate to see that the assignment
\begin{equation} 
\label{eq:wheel-alg}
   \alg : \Adm \longrightarrow \Alg 
\end{equation}
sending a wheelgebra $(S,\mu,\eta)$ 
to the nonunitary algebra $(S(1), {}^{2}_{2}t_{1} \circ \mu_{1,1})$ and a morphism $f = (f(n))_{n \in \NN_{0}}$ between wheelgebras to the map $f(1)$ is a functor. 
Note that this functor is well defined due to Lemma \ref{lemma:wheel-adm}. 

The following result is a simple observation. 
%%%%%%%
\begin{fact}
\label{fact:alg-vanishes-comm}
Let $S = (S(n))_{n \in \NN_{0}}$ be a commutative wheelgebra with product $\mu : S \otimes_{\gwh} S \rightarrow S$ and 
contractions $({}^{n}_{i}t_{j})_{n \in \NN, i, j \in \llbracket 1 , n \rrbracket}$. 
Then, the map ${}^{1}_{1}t_{1} : S(1) \rightarrow S(0)$ 
vanishes on $[\alg(S),\alg(S)] \subseteq S(1)$. 
\end{fact}
%%%%%%%
\begin{proof}
This follows from 
\begin{equation}
\begin{split}
    ({}^{1}_{1}t_{1} \circ {}^{2}_{2}t_{1})\big(\mu_{1,1}(a,b)\big) 
    &= ({}^{1}_{1}t_{1} \circ {}^{2}_{1}t_{2})\big(\mu_{1,1}(a,b)\big)
    = {}^{1}_{1}t_{1} \Big( {}^{2}_{2}t_{1} \big( (1 \ 2) \cdot \mu_{1,1}(a,b) \cdot (1 \ 2) \big)\Big) 
    \\
    &= ({}^{1}_{1}t_{1} \circ {}^{2}_{2}t_{1})\big(\mu_{1,1}(b,a)\big), 
\end{split}    
\end{equation}
for all $a, b \in S(1)$, where in the first identity we used \ref{item:W2}, in the second identity we used \ref{item:W1}, and in the third we used \eqref{eq:comm-wh}. 
\end{proof}

Let $\wA$ be a wheelgebra. 
A wheelbimodule $(\wN, \rho)$ over $\wA$ is said to be \textbf{\textcolor{myblue}{admissible}} if $\wN(1)$ 
endowed with the morphisms
${}^{2}_{2}t_{1} \circ (\rho_{\mathcalboondox{l}})_{1,1} : \wA(1) \otimes \wN(1) \rightarrow \wN(1)$ and ${}^{2}_{2}t_{1} \circ (\rho_{\mathcalboondox{r}})_{1,1} : \wN(1) \otimes \wA(1)  \rightarrow \wN(1)$ is an $\alg(\wA)$-bimodule.

Similarly to Lemma \ref{lemma:wheel-adm}, all wheelbimodules are admissible. 
%%%%%%%
\begin{lemma}
\label{lemma:wheelbimod-adm} 
Let $\wA$ be a wheelgebra  
and let $\wN$ be a wheelbimodule over $\wA$. 
Then, $\wN$ is admissible. 
\end{lemma}
%%%%%%%
\begin{proof}
It is immediate to see that $\wN$ is admissible if and only if the associated wheelgebra $\ZSE(\wA,\wN)$ recalled in Example \ref{example:zse} is admissible. 
The result then follows from Lemma \ref{lemma:wheel-adm}. 
\end{proof}

It is easy to see that the assignment 
\begin{equation} 
\label{eq:wheel-mod}
   \bmod : {}_{\wA}\AMod_{\wA} \longrightarrow {}_{\alg(\wA)}\Mod_{\alg(\wA)} 
\end{equation}
sending an $\wA$-wheelbimodule $(\wN, \rho)$ with contractions $\{ {}_{j}^{n}t_{i} \}_{n \in \NN, i,j \in \llbracket 1 , n \rrbracket}$
to the $\alg(\wA)$-bimodule $\wN(1)$ with left (resp., right) action given by ${}^{2}_{2}t_{1} \circ (\rho_{\mathcalboondox{l}})_{1,1}$ (resp., ${}^{2}_{2}t_{1} \circ (\rho_{\mathcalboondox{r}})_{1,1}$) and a morphism $f = (f(n))_{n \in \NN_{0}}$ between $\wA$-wheelbimodules to the map $f(1)$ is a functor. 
The functor is well defined due to Lemma 
\ref{lemma:wheelbimod-adm}. 
In fact, the functor \eqref{eq:wheel-mod} is uniquely characterized by the following property, whose proof is immediate. 
%%%%%%%
\begin{fact} 
\label{fact:zse-adm} 
Let $\wA$ be an wheelgebra. 
Then, there is a natural isomorphism 
\begin{equation} 
\label{eq:wheel-mod-2}
   \ZSE\big(\alg(\wA),\bmod(\wN)\big) \cong \alg\Big(\ZSE\big( \wA , \wN \big)\Big) 
\end{equation}
for $\wN$ in ${}_{\wA}\AMod_{\wA}$. 
\end{fact}
%%%%%%%

%%%%%%%%%%%%%%%%%%%%%%%%%%%%%%%%%%%%%%%%%%%%%%%%%%%%%%
\subsection{Kähler wheeldifferential forms and symplectic wheelgebras}
\label{subsec:wheeldif-forms} 

%%%%%%%%%%%%%%%%%%%%%%%%%%%%%%%%%%%%%%%%%%%%%%%%%%%%%%%%%%%%%%%
\subsubsection{Wheelderivations and wheeldifferential operators}

The following examples are obtained by applying the usual definitions of derivations and of differential operators to the case of the internal space of morphisms introduced in Subsection \ref{subsection:internal-homos}. 

Let $\wC$ be a commutative wheelgebra with product $\mu$, and let $\wN$ be a wheelmodule over $\wC$ with left action $\rho$, which is seen as a symmetric wheelbimodule. 
Recall the definition of $\wad_{c}(f)$ given in \eqref{eq:wad} for $\wM = \wC$
 and $\wM' = \wN$.

%%%%%%%
\begin{example} 
\label{example:wheelder}
Let $\IDer(\wC,\wN)\label{eq:wheelder-def}$ be the subwheelspace of $\IHom_{\WMod}(\wC,\wN)$ whose component $\IDer(\wC,\wN)(n)$ for $n \in \NN_{0}$ is formed by all $f = (f(m))_{m \in \NN_{0}} \in \IHom_{\WMod}(\wC,\wN)(n)$ such that 
\begin{equation}
\label{eq:wheelleibniz}
        \wad_{c}(f)(c') = - \block_{m,n,p}(1 \, 3) \cdot \rho_{m,p+n}\big(c',f(p)(c)\big) \cdot\block_{p,n,m}(1 \, 3)
\end{equation} 
for all $c \in \wC(p)$ and $c' \in \wC(m)$. 
It is direct to check that $\IDer(\wC,\wN)$ is a subwheelspace of 
$\IHom_{\WMod}(\wC,\wN)$, called the wheelspace of \textbf{\textcolor{myblue}{wheel\-der\-i\-va\-tions}} of $\wC$ with values in $\wN$, and 
it even has a natural structure of $\wC$-wheelmodule for the structure map \eqref{eq:wmod-int-c-pre}. 

Moreover, $\IDer(\wC,\wC)$ (which we simply denote as $\IDer(\wC)$) is even a Lie subwheelgebra of the associated Lie wheelgebra to the wheelgebra $\IHom_{\WMod}(\wC,\wC)$ introduced in Example \ref{example:end-alg-w-mopd}, called the Lie wheelgebra of \textbf{\textcolor{myblue}{wheel\-der\-i\-va\-tions}} of $\wC$. 
Note that $\IDer(\wC,\wN)$ does not appear explicitly in \cite{MR2734329}.
\end{example}
%%%%%%%

%%%%%%%
\begin{remark}
\label{remark:der-wheelder}
Recall that a \textbf{\textcolor{myblue}{derivation}} of a commutative wheelgebra $\wC$ with product $\mu$ with values in a wheelmodule $\wN$ with action $\rho$ 
is a map of wheelspaces $f\colon\wC\to\wN$ satisfying the Leibniz rule 
\begin{equation}   
\label{eq:der-wheel-alg}
f(n+m)\big(\mu_{n,m}(c,c')\big) = \rho_{n,m} \big(c, f(m)(c')\big) + \block_{m,n}(1 \, 2) \cdot \rho_{m,n}\big(f(c),c')\big) \cdot\block_{n,m}(1 \, 2)    
\end{equation}
for $n, m \in \NN_{0}$, $c \in \wC(n)$ and $c' \in \wC(m)$. 
The vector space formed by all derivations of $\wC$ with values in $\wN$ will be denoted by $\Der(\wC , \wN)$. 
Then, it is immediate from the previous definitions that 
\begin{equation}   
\label{eq:der-wheelder}
\Der\big(\wC , \sh_{k}^{\wh}(\wN)\big) = \IDer(\wC , \wN)(k), 
\end{equation}
for all $k \in \NN_{0}$.
\end{remark} 
%%%%%%%

%%%%%%%
\begin{example} 
\label{example:wheeldiff}
Define the wheelspace $\WDiff(\wC,\wN)\label{eq:wheeldiff-def}$ as follows. 
Set $\WDiff^{-1}(\wC,\wN)$ as the zero wheelspace.
Next, given $\WDiff^{\ell}(\wC,\wN)$ for $\ell \in \NN_{0}$, we define 
$\WDiff^{\ell+1}(\wC,\wN)$ as the subwheelspace of $\IHom_{\WMod}(\wC,\wN)$ whose component $\WDiff^{\ell+1}(\wC,\wN)(n)$ for $n \in \NN_{0}$
is formed by all $f = (f(m))_{m \in \NN_{0}} \in \IHom_{\WMod}(\wC,\wN)(n)$ such that 
$\wad_{c}(f) \in \WDiff^{\ell}(\wC,\wN)(n+p)$ for all $c \in \wC(p)$. 
Define now $\WDiff(\wC,\wN)(n) = \cup_{\ell \in \NN_{0}} \WDiff^{\ell}(\wC,\wN)(n)$ for every $n \in \NN_{0}$. 
It is direct to check that $\WDiff(\wC,\wN)$ is a subwheelspace of 
$\IHom_{\WMod}(\wC,\wN)$, called the wheelspace of \textbf{\textcolor{myblue}{wheeldifferential operators}} of $\wC$ with values in $\wN$, 
and if $\wN = \wC$, it is even a subwheelgebra of the wheelgebra $\IHom_{\WMod}(\wC,\wC)$ introduced in Example \ref{example:end-alg-w-mopd} and it is then called the wheelgebra of \textbf{\textcolor{myblue}{wheeldifferential operators}} of $\wC$. 
\end{example}
%%%%%%%

%%%%%%%%%%%%%%%%%%%%%%%%%%%%%%%%%%%%%%%%%%%%%%%%%%%%%%%%%%%%%%%%%%%%%%
\subsubsection{Cartan calculus for commutative wheelgebras}
\label{sec:Cartan calculus for commutative wheelgebras}

Let $\wC$ be a commutative wheelgebra.
We say that $\wC$ admits a \textbf{\textcolor{myblue}{wheelmodule of Kähler wheeldifferential forms}} if there exists a (unique up to canonical isomorphism) $\wC$-wheelmodule that represents the functor 
\begin{equation}
\label{eq:wheel-der-func}
    \IDer(\wC, -) :{}_{\wC}\AMod \longrightarrow {}_{\wC}\AMod
\end{equation}
sending a $\wC$-wheelmodule $\wN$ to 
$\IDer(\wC, \wN)$, and a morphism $f : \wN \rightarrow \wN'$ 
of $\wC$-wheelmodules to the morphism sending $d = (d(n))_{n \in \NN_{0}} \in \IDer(\wC, \wN)(m)$ to the element $d' = (d'(n))_{n \in \NN_{0}} \in \IDer(\wC, \wN')(m)$ given by $d'(n) = f(n+m) \circ d(n)$ for all $n, m \in \NN_{0}$. 
We will denote the representing 
$\wC$-wheelmodule
by $\Omega^{1}_{\wh}\wC$; in other words, $\Omega^{1}_{\wh}\wC$ satisfies that there exists a natural isomorphism of wheelspaces 
\begin{equation}
\label{eq:wheel-der-rep}
    \IDer(\wC, \wN) \overset{\cong}{\longrightarrow} \IHom_{\wC}(\Omega^{1}_{\wh}\wC,\wN)
\end{equation}
for all $\wC$-wheelmodules $\wN$, provided it exists. 
In this case, there exists the \textbf{\textcolor{myblue}{universal derivation}} $\derdifw_{\wC} \in \Der(\wC , \Omega_{\wh}^{1}\wC)\label{index:univ-derivation-symp-wheel}$ associated with the identity of $\Omega_{\wh}^{1}\wC$ under the isomorphism \eqref{eq:wheel-der-rep} for $\wN = \Omega_{\wh}^{1}\wC$. 

We have the following interesting result concerning $\Omega_{\wh}^{1}\wC$. 
%%%%%%%
\begin{fact}
\label{fact:omega1-gen}
Let $\wC$ be a commutative wheelgebra admitting a wheelmodule of Kähler wheeldifferential forms $\Omega_{\wh}^{1}\wC$, and let $\derdifw_{\wC} \in \Der(\wC , \Omega_{\wh}^{1}\wC)$ be the universal derivation. 
Then, $\Omega_{\wh}^{1}\wC$ coincides with the $\wC$-wheelsubmodule of $\Omega_{\wh}^{1}\wC$ generated by $\{ \derdifw_{\wC}(n)(c) : c \in \wC(n) \text{ and } n \in \NN_{0} \}$.
\end{fact}
%%%%%%%
\begin{proof}
Let $\bar{\Omega}_{\wh}^{1}\wC$ be the $\wC$-wheelsubmodule of $\Omega_{\wh}^{1}\wC$ generated by $\{ \derdifw_{\wC}(n)(c) : c \in \wC(n) \text{ and } n \in \NN_{0} \}$. 
Denote by $j : \bar{\Omega}_{\wh}^{1}\wC \rightarrow \Omega_{\wh}^{1}\wC$ the inclusion. 
Then, the derivation $\derdifw_{\wC} \in \Der(\wC , \Omega_{\wh}^{1}\wC)$ corestricts to a derivation $\bar{d}_{\wC} \in \Der(\wC , \bar{\Omega}_{\wh}^{1}\wC)$ 
such that $j \circ \bar{d}_{\wC} = \derdifw_{\wC}$. 
The isomorphism \eqref{eq:wheel-der-rep} for $\wN = \bar{\Omega}_{\wh}^{1}\wC$ sends $\bar{d}_{\wC}$ to a morphism of $\wC$-wheelmodules $p : \Omega^{1}_{\wh}\wC \rightarrow \bar{\Omega}^{1}_{\wh}\wC$. 
By naturality of \eqref{eq:wheel-der-rep} with respect to $\wN$, $j \circ \bar{d}_{\wC} = \derdifw_{\wC}$ tells us that $j \circ p = \operatorname{id}_{\Omega^{1}_{\wh}\wC}$. 
Since $j$ is a monomorphism, $j \circ p \circ j = \operatorname{id}_{\Omega^{1}_{\wh}\wC} \circ j = j \circ \operatorname{id}_{\bar{\Omega}^{1}_{\wh}\wC}$ tells us that $p \circ j = \operatorname{id}_{\bar{\Omega}^{1}_{\wh}\wC}$, 
so $\bar{\Omega}^{1}_{\wh}\wC = \Omega^{1}_{\wh}\wC$, as was to be shown. 
\end{proof}

We will say that the commutative wheelgebra $\wC$ admits a \textbf{\textcolor{myblue}{(graded) wheelgebra of Kähler wheeldifferential forms}} if it admits a wheelmodule of Kähler wheeldifferential forms $\Omega^{1}_{\wh}\wC$, which we view as a graded $\wC$-wheelmodule concentrated in degree $1$ for $\wC$  considered as a graded wheelgebra concentrated in degree zero, and the symmetric algebra $\Sym_{\wcatA}(\Omega^{1}_{\wh}\wC)$ exists, where $\wcatA \label{index:categ-N-0-grad}$ denotes the category of $\NN_{0}$-graded $\wC$-wheelgebras whose zeroeth component is precisely $\wC$ and the $\wC$-wheelgebra structure morphism is the canonical inclusion of $\wC$ in the degree zero component (we regard $\wcatA$ as a full subcategory of the category of commutative algebras in the symmetric monoidal category of $\NN_{0}$-graded $\wC$-wheelmodules with the usual signed flip following the Koszul sign rule). 
If it exists, we will denote $\Sym_{\wcatA}(\Omega^{1}_{\wh}\wC)$ by $\Omega_{\wC}^{\bullet}\label{index:wheel-diff-forms-comm}$ or simply $\Omega_{\wC}$. 
As usual, the homogeneous component of degree $n \in \NN_{0}$ of $\Omega_{\wC}^{\bullet}$ will be denoted by $\Omega_{\wC}^{n}$. 
As noted in Remark \ref{remark:sym-grading}, we can assume that the canonical morphism $\iota_{\Omega^{1}_{\wh}\wC} : \Omega^{1}_{\wh}\wC \rightarrow \Omega_{\wC}^{\bullet} \label{index:wheel-contraction-Ex2.3}$ of Example \ref{example:sym-alg} induces an isomorphism $\Omega^{1}_{\wh}\wC \rightarrow \Omega_{\wC}^{1}$, so we will identify these wheelmodules. 

We have the following interesting result concerning $\Omega_{\wC}^{\bullet}$. 
%%%%%%%
\begin{fact}
\label{fact:omegabullet-gen}
Let $\wC$ be a commutative wheelgebra admitting a (graded) wheelgebra of Kähler wheeldifferential forms $\Omega_{\wC}^{\bullet}$. 
Then, $\Omega_{\wC}^{\bullet}$ coincides with the (graded) subwheelgebra of $\Omega_{\wC}^{\bullet}$ generated by the homogeneous components of degree $0$ and $1$. 
In particular, the canonical morphism of wheelspaces
\begin{equation}
\label{eq:res-der-omega}
        \IDer(\Omega_{\wC}^{\bullet}) \longrightarrow \IHom_{\WMod}\big(\wC \oplus \Omega_{\wh}^{1}\wC , \Omega_{\wC}^{\bullet}\big)     
\end{equation}         
sending $d$ to $d|_{\wC \oplus \Omega_{\wh}^{1}\wC}$ is a monomorphism, 
which in turn induces an injection from the set $\{ d \in \Der(\Omega_{\wC}^{\bullet}) : d \circ d = 0\}$ to $\Hom_{\WMod}\big(\wC  , \Omega_{\wC}^{\bullet}\big)$. 
\end{fact}
%%%%%%%
\begin{proof}
Let $\bar{\Omega}_{\wC}^{\bullet}$ be the graded subwheelgebra of $\Omega_{\wC}^{\bullet}$ generated by the homogeneous components of degree $0$ and $1$, 
\textit{i.e.} $\wC$ and $\Omega_{\wh}^{1}\wC$. 
Denote by $j : \wC \oplus \Omega_{\wh}^{1}\wC \rightarrow \Omega_{\wh}^{\bullet}\wC$ the inclusion. 
Note that $\bar{\Omega}_{\wC}^{\bullet}$ is a graded $\wC$-subwheelgebra of $\Omega_{\wC}^{\bullet}$. 
We will denote the corestriction of the morphism $\iota_{\Omega^{1}_{\wh}\wC} : \Omega^{1}_{\wh}\wC \rightarrow \Omega_{\wC}^{\bullet}$ 
to $\bar{\Omega}_{\wC}^{\bullet}$ by 
$\bar{\iota}_{\Omega^{1}_{\wh}\wC} : \Omega^{1}_{\wh}\wC \rightarrow \bar{\Omega}_{\wC}^{\bullet}$. 
The isomorphism \eqref{eq:sym-alg} for $D = \bar{\Omega}_{\wC}^{\bullet}$ and $\bar{\iota}_{\Omega^{1}_{\wh}\wC} : \Omega^{1}_{\wh}\wC \rightarrow \bar{\Omega}_{\wC}^{\bullet}$, tells us that there exists a morphism of graded $\wC$-wheelgebras $p : \Omega^{\bullet}_{\wC} \rightarrow \bar{\Omega}^{\bullet}_{\wC}$ corresponding to $\bar{\iota}_{\Omega^{1}_{\wh}\wC}$. 
Moreover, the naturality of \eqref{eq:sym-alg} with respect to $D$ and $j \circ \bar{\iota}_{\Omega^{1}_{\wh}\wC} = \iota_{\Omega^{1}_{\wh}\wC}$ tell us that $j \circ p = \operatorname{id}_{\Omega^{\bullet}_{\wC}}$. 
Since $j$ is a monomorphism, $j \circ p \circ j = \operatorname{id}_{\Omega^{\bullet}_{\wC}} \circ j = j \circ \operatorname{id}_{\bar{\Omega}^{\bullet}_{\wC}}$ implies that $p \circ j = \operatorname{id}_{\bar{\Omega}^{\bullet}_{\wC}}$, 
so $\bar{\Omega}^{\bullet}_{\wC} = \Omega^{\bullet}_{\wC}$, as was to be shown. 

The fact that the map \eqref{eq:res-der-omega} is a monomorphism follows immediately from the first part of the statement and the Leibniz identity \eqref{eq:wheelleibniz}. 
Finally, the last part is a direct consequence of the fact that \eqref{eq:res-der-omega} is a monomorphism and Fact \ref{fact:omega1-gen}. 
\end{proof}

We will say that $\wC$ admits a \textbf{\textcolor{myblue}{dg wheelgebra of Kähler wheeldifferential forms}} if it admits a graded wheelgebra of Kähler wheeldifferential forms $\Omega_{\wC}^{\bullet}$ and there exists a homogeneous derivation $\derdifwbul_{\wC} \in \Der(\Omega_{\wC}^{\bullet}) \label{index:diffferential-full-diff-forms}$ of
degree $1$ such that $\derdifwbul_{\wC} \circ \derdifwbul_{\wC} = 0$ and the restriction of $\derdifwbul_{\wC}$ to $\Omega_{\wC}^{0}$ coincides with the composition of the universal derivation $\derdifw_{\wC}$ and the canonical morphism of $\wC$-wheelmodules $\iota_{\Omega_{\wh}^{1}\wC} : \Omega_{\wh}^{1}\wC \rightarrow \Omega_{\wC}^{\bullet}$. 
The derivation $\derdifwbul_{\wC}$ is necessarily unique, if it exists, due to Fact \ref{fact:omegabullet-gen}. 
We will denote the previous dg wheelgebra also by $\Omega_{\wC}^{\bullet}$ or simply by $\Omega_{\wC}$. 
Examples of wheelgebras that admit a dg wheelgebra of Kähler wheeldifferential forms will be presented in Lemma \ref{lemma:dgwheel-fock}. 

Let $\wC$ be a commutative wheelgebra admiting a (graded) wheelgebra of Kähler wheeldifferential forms. 
Denote by $\IDer^{-1}(\Omega_{\wC})\label{index:subwheelsp-degree-1}$
the subwheelspace of $\IDer(\Omega_{\wC})$ such that, given $n \in \NN_{0}$, $\IDer^{-1}(\Omega_{\wC})(n)$ is formed by all homogeneous $d \in \IDer(\Omega_{\wC})(n)$ of degree $-1$. 
Then, given $d \in \IDer^{-1}(\Omega_{\wC})(n)$, $d \circ \iota_{\Omega_{\wh}^{1}\wC}$ factors through the canonical inclusion  $i_{\wC} : \wC \rightarrow \Omega_{\wC}^{\bullet}$, \textit{i.e.} there exists a (necessarily unique) $f_{d} \in \IHom_{\wC}(\Omega_{\wh}^{1}\wC, \wC)(n)$ such that $i_{\wC} \circ f_{d} = d \circ \iota_{\Omega_{\wh}^{1}\wC}$.
We will say that $\wC$ \textbf{\textcolor{myblue}{has wheelcontractions}} if the well-defined morphism of wheelspaces
\begin{equation}
\label{eq:der-star}
    \IDer^{-1}(\Omega_{\wC}^{\bullet}) \longrightarrow \IHom_{\wC}(\Omega_{\wh}^{1}\wC,\wC)
\end{equation}
sending $d$ to $f_{d}$ is an isomorphism. 
Note that \eqref{eq:der-star} is always a monomorphism by the last part of Fact \ref{fact:omegabullet-gen}, so $\wC$ has wheelcontractions if \eqref{eq:der-star} is an epimorphism. 
We will show examples of algebras that have wheelcontractions in Lemma \ref{lemma:cont-wheel-fock}. 

Assume that $\wC$ is a commutative wheelgebra having wheelcontractions. 
Define the map 
\begin{equation}
\label{eq:der-ext}
    \ext : \IDer(\wC) \longrightarrow \IDer(\Omega_{\wC}^{\bullet})
\end{equation}
given as the composition of \eqref{eq:wheel-der-rep} for $\wN = \wC$, the inverse of 
\eqref{eq:der-star} and the inclusion of 
$\IDer^{-1}(\Omega_{\wC}^{\bullet})$
in $\IDer(\Omega_{\wC}^{\bullet})$. 

%%%%%%%
\begin{remark}
Note that the map \eqref{eq:der-ext} is the analogous version for wheelgebras of the contraction map $X \mapsto \iota_{X}$ in usual geometry sending a vector field $X \in \Der_{\Bbbk}(A)$ to the associated contraction $\iota_{X} : \Omega_{A/\Bbbk}^{\bullet} \rightarrow \Omega_{A/\Bbbk}^{\bullet}$, for a commutative $\Bbbk$-algebra $A$. 
\end{remark}
%%%%%%%

Let $\wC$ be a commutative wheelgebra having wheelcontractions. 
By the previous remark, given $X \in \Der(\wC)$, we will denote by $\iotaw_{X} \in \Der(\Omega_{\wC}^{\bullet})$ the derivation defined as $\iotaw_{X} = \ext(0)(X)$, where we are using identity \eqref{eq:der-wheelder}. 
We will call $\iotaw_{X}\label{index:contraction-assoc-wheelderivat}$ the \textbf{\textcolor{myblue}{contraction}} associated with the derivation $X$. 

Let $\wC$ be a commutative wheelgebra admiting a dg wheelgebra of Kähler wheeldifferential forms and having wheelcontractions. 
We denote as before the corresponding derivation of $\Omega_{\wC}^{\bullet}$ by $\derdifwbul_{\wC}$. 
Given $X \in \Der(\wC)$, define  the \textbf{\textcolor{myblue}{Lie derivative}} $\Liew_{X} \in \Der(\Omega_{\wC}^{\bullet}) \label{index:Lie-derivative-wheel}$ associated with the derivation $X$ by $\Liew_{X} = \derdifwbul_{\wC} \circ \iotaw_{X} + \iotaw_{X} \circ \derdifwbul_{\wC}$. 

%%%%%%%
\begin{proposition}
\label{lemma:cartan-calculus}
Let $\wC$ be a commutative wheelgebra admiting a dg wheelgebra of Kähler wheeldifferential forms and having wheelcontractions. 
We denote as usual the corresponding derivation of $\Omega_{\wC}^{\bullet}$ by $\derdifwbul_{\wC}$. 
Given $X, Y \in \Der(\wC)$, we have the usual \textbf{\textcolor{myblue}{Cartan identities}} for $\Omega_{\wC}^{\bullet}$ given by
\begin{equation}
   \label{eq:cartanw}
   \begin{split}
   &\derdifwbul_{\wC} \circ  \derdifwbul_{\wC} = 0, \hskip 0.75cm   \Liew_{X} = \derdifwbul_{\wC} \circ \iotaw_{X} + \iotaw_{X} \circ \derdifwbul_{\wC}, \hskip 1cm
   [\derdifwbul_{\wC},\Liew_{X}] = 0,
   \\
   &[\iotaw_{X},\iotaw_{Y}] = 0, \hskip 1cm [\Liew_{X},\iotaw_{Y}] = \iotaw_{[X,Y]}, 
   \hskip 2.47cm
   [\Liew_{X},\Liew_{Y}] = \Liew_{[X,Y]},  
   \end{split}
\end{equation}
where the bracket appearing in the left member of the equalities denotes the usual graded commutator of homogeneous elements of $\Hom_{\WMod}(\Omega_{\wC}^{\bullet}, \Omega_{\wC}^{\bullet})$, and the one appearing in the right member is the usual commutator of derivations of $\wC$. 
\end{proposition}
%%%%%%%
\begin{proof}
The first two identities are a consequence of the definition, and the last one follows easily from the previous ones. 
The third identity follows from the Leibniz identity for the graded commutator together with the first identity. 
It thus suffices to prove the fourth and fifth identities. 
Since the graded commutator of derivations is a derivation, Fact \ref{fact:omegabullet-gen} tells us that it suffices to prove 
\begin{equation}  
\label{eq:cartanw2}
[\iotaw_{X},\iotaw_{Y}]|_{\wC \oplus \Omega_{\wh}^{1}\wC} = 0, \text{ and } [\Liew_{X},\iotaw_{Y}]|_{\wC \oplus \Omega_{\wh}^{1}\wC} = \iotaw_{[X,Y]}|_{\wC \oplus \Omega_{\wh}^{1}\wC}.
\end{equation}
For the first of these identities, note that the vanishing of $[\iotaw_{X},\iotaw_{Y}]|_{\wC \oplus \Omega_{\wh}^{1}\wC}$ follows from degree reasons, since the right member is a homogeneous derivation of $\Omega_{\wC}$ of degree $-2$. 
Moreover, $[\Liew_{X},\iotaw_{Y}]|_{\wC} = 0 = \iotaw_{[X,Y]}|_{\wC}$ by degree reasons, since $[\Liew_{X},\iotaw_{Y}]$ and $\iotaw_{[X,Y]}$ are homogeneous derivations of $\Omega_{\wC}$ of degree $-1$.
It remains to show that $[\Liew_{X},\iotaw_{Y}]|_{\Omega_{\wh}^{1}\wC} = \iotaw_{[X,Y]}|_{\Omega_{\wh}^{1}\wC}$. 
Fact \ref{fact:omega1-gen} and the Leibniz rule tell us that the previous identity follows from $[\Liew_{X},\iotaw_{Y}] \circ \derdifw_{\wC} = \iotaw_{[X,Y]} \circ \derdifw_{\wC}$, so it suffices to prove the latter equality, which we get from 
\begin{align*}
    [\Liew_{X},\iotaw_{Y}] \circ \derdifw_{\wC} 
    &= \Liew_{X} \circ \iotaw_{Y} \circ \derdifw_{\wC} - \iotaw_{Y} \circ \Liew_{X} \circ \derdifw_{\wC}
    \\
    &= \Liew_{X} \circ \iotaw_{Y} \circ \derdifw_{\wC} - \iotaw_{Y} \circ \derdifwbul_{\wC} \circ \Liew_{X}|_{\wC} 
    \\
    &= \derdifwbul_{\wC} \circ \iotaw_{X} \circ \iotaw_{Y} \circ \derdifw_{\wC} + \iotaw_{X} \circ \derdifwbul_{\wC} 
    \circ \iotaw_{Y} \circ \derdifw_{\wC} 
    \\
    &\quad - \iotaw_{Y} \circ \derdifwbul_{\wC} \circ \derdifwbul_{\wC} \circ \iotaw_{X}|_{\wC} 
     - \iotaw_{Y} \circ \derdifwbul_{\wC} \circ \iotaw_{X} \circ \derdifw_{\wC}
    \\
    &= X \circ Y - Y \circ X = [X,Y] = \iotaw_{[X,Y]} \circ \derdifw_{\wC}, 
\end{align*}
where we used the third identity of \eqref{eq:cartanw} in the second equality, the second identity of \eqref{eq:cartanw} in the third equality, the first identity of \eqref{eq:cartanw} and the fact that $\iotaw_{Z}|_{\wC}$ vanishes by degree reasons for $Z \in \Der(\wC)$ in the fourth equality, and that $\iotaw_{Z} \circ \derdifw_{\wC} = Z$ for $Z \in \Der(\wC)$ in the fourth and last identity. 
The proposition is thus proved. 
\end{proof}

%%%%%%%%%%%%%%%%%%%%%%%%%%%%%%%%%%%%%%%%%%%%%%%%%%%%%%%%%%%%%%%
\subsubsection{Symplectic wheelgebras}

At this point, we can state the main definition of this section.
%%%%%%%
\begin{definition}
\label{def:symplectic-wheelgebra-proto}
Let $\wC$ be a commutative wheelgebra admiting a dg wheelgebra of Kähler wheeldifferential forms and having wheelcontractions. 
An element $\varpi\in \Omega^{2}_{\wC}(0)$ is said to be \textbf{\textcolor{myblue}{closed}} if $\derdifwbul_{\wC}(0)(\varpi) \in \Omega^{3}_{\wC}(0)$ vanishes. 
Furthermore, we say that $\varpi\in \Omega^{2}_{\wC}(0)$ is  \textbf{\textcolor{myblue}{non-degenerate}}
if the morphism of wheelspaces 
\begin{equation}
\label{eq:wheelcont}
    \cano_{\varpi} : \IDer(\wC) \longrightarrow \Omega_{\wh}^{1}\wC
\end{equation}
sending $d \in \IDer(\wC)(n)$ to $\ext(n)(d)(0)(\varpi) \in \Omega_{\wC}^{1}(n) = \Omega_{\wh}^{1}\wC(n)$ for all $n \in \NN_{0}$ is an isomorphism. 
An element $\varpi\in \Omega^{2}_{\wC}(0)$ is called a \textbf{\textcolor{myblue}{wheelsymplectic form}} if it is closed and nondegenerate. 
A \textbf{\textcolor{myblue}{symplectic wheelgebra}} is a pair $(\wC,\varpi)$, where $\varpi$ is a wheelsymplectic form.
\end{definition}
%%%%%%%

%%%%%%%
\begin{remark}
\label{remark:pairing-antisymmetric}
The fact that $\varpi \in \Omega^{\bullet}_{\wC}(0) = \Sym_{\wcatA}(\Omega_{\wh}^{1}\wC)(0)$ (see the paragraph before Fact \ref{fact:omegabullet-gen}), implies that if $(\wC, \varpi)$ is a symplectic wheelgebra, then the isomorphism $\cano_{\varpi}$ given in 
\eqref{eq:wheelcont} is antisymmetric, \textit{i.e.} 
the canonical morphism 
\begin{equation}
\label{eq:antisym-symp-wheel}
\Omega_{\wh}^{1}\wC  \otimes_{\gwh} \Omega_{\wh}^{1}\wC  \hskip -0.6mm\overset{\text{\begin{tiny}$\cano_{\varpi}^{-1} \otimes \operatorname{id}  $\end{tiny}}}{\longrightarrow} \hskip -0.6mm\IDer(\wC) \otimes_{\gwh} \Omega_{\wh}^{1}\wC 
\overset{\cong}{\rightarrow} 
\IHom_{\wC}(\Omega^{1}_{\wh}\wC,\wC) \otimes_{\gwh} \Omega^{1}_{\wh}\wC
\rightarrow \wC
\end{equation}
of wheelmodules is antisymmetric, where the second isomorphism is induced by \eqref{eq:wheel-der-rep} and the last morphism is the usual evaluation. 
Morever, the composition of $\derdifw_{\wC} \otimes \derdifw_{\wC} : \wC \otimes_{\gwh} \wC \rightarrow \Omega_{\wh}^{1}\wC  \otimes_{\gwh} \Omega_{\wh}^{1}\wC$ with \eqref{eq:antisym-symp-wheel} is a Poisson bracket turning $\wC$ into a Poisson wheelgebra. 
\end{remark}
%%%%%%%

%%%%%%%%%%%%%%%%%%%%%%%%%%%%%%%%%%%%%%%%%%%%%%%%%%%%%%%%%%%%%%%%%%%%%%
\section{The Fock wheelgebra} 
\label{section:main-ex-Fock-alg}
%%%%%%%%%%%%%%%%%%%%%%%%%%%%%%%%%%%%%%%%%%%%%%%%%%%%%%%%%%%%%%%%%%%%%%
\subsection{\texorpdfstring{A symmetric algebra in $\DMod$}{A symmetric algebra in DMod}}
\label{sec:a-symmetric-algebra-diagonal-bimodules}

Given vector spaces $V$ and $W$, 
we define the commutative algebra $\mathcal{F}(V,W)$ as the symmetric algebra
\[  
\Sym_{\SG^\env}\big(\I_{\SG^{\env},1}(V) \oplus \I_{\SG^{\env},0}(W)\big)    
\label{eq:symmetric-alg-DMod}
\] 
in the symmetric monoidal category $\DMod$ of diagonal $\SG$-bimodules. 
Using the isomorphism of Fact \ref{example:sym-alg-2} and Lemma \ref{lemma:symm-dmod}, we get the isomorphism of algebras 
\[  
\mathcal{F}(V,W) = \Sym_{\SG^\env}\big(\I_{\SG^{\env},1}(V) \oplus \I_{\SG^{\env},0}(W)\big) 
\cong \SG_{V}^{\env} \otimes_{\SG^{\env}} \Sym_{\SG^{\env}}(\I_{\SG^{\env},0}(W))   
\label{eq:def-symm-F-caligrafica}
\] 
in $\DMod$, where the last member is the tensor product of (commutative) algebras, $\SG_{V}^{\env}$ was defined in Example \ref{example:ten-2} and $\I_{\SG^{\env},0}$ was defined in \eqref{eq:vect-s-se}. 
Since the functor $\I_{\SG^{\env}}$ is braided strong monoidal, we have the canonical isomorphism 
\[     
\Sym_{\SG^\env}(\I_{\SG^{\env},0}(W)) \cong \I_{\SG^{\env},0}(\Sym(W))     \]
of diagonal $\SG$-bimodules, and the latter space thus has a natural structure of algebra in $\DMod$. 
Explicitly, 
\[    
\mathcal{F}(V,W)(n) = \Ind_{\Delta_{n}} (V^{\otimes n}) \otimes \Sym(W),    
\]
for all $n \in \NN_{0}$, where $\Sym(W)$ is considered as a $\Bbbk \SG_{n}$-bimodule with the trivial left and right actions, and the 
$\Bbbk \SG_{n}$-bimodule structure of the previous tensor product is the diagonal one. 
Moreover, its product $\mu$ is the unique one whose associated collection of maps 
\[
\Big\{ \mu_{n,m} : \mathcal{F}(V,W)(n) \otimes \mathcal{F}(V,W)(m) \longrightarrow \mathcal{F}(V,W)(n+m) \Big\}_{n,m \in \NN_{0}}
\label{eq:product-F-caligrafic}
\]
given in Fact \ref{fact:morph-tensor-s-bimod} is defined as 
\begin{small}
\begin{equation}
\label{eq:prod-fock-s-e}
\begin{split}
   &\mu_{n,m}\bigg( \Big(\big((v_{1} \otimes \dots\otimes  v_{n}) \otimes_{\Bbbk \SG_{n}} \!(\sigma \otimes \tau)\big) \otimes \omega \Big) \otimes \Big(\big((v'_{1} \otimes \dots\otimes  v'_{n'}) \otimes_{\Bbbk \SG_{n'}}\!\! (\sigma' \otimes \tau')\big) \otimes \omega' \Big) \bigg) 
   \\
   &= \Big((v_{1} \otimes \dots \otimes v_{n} \otimes v'_{1} \otimes \dots \otimes v'_{n'}) \otimes_{\Bbbk \SG_{n+n'}} \big(\sump_{n,n'}(\sigma,\sigma') \otimes \sump_{n,n'}(\tau,\tau') \big)\Big) \otimes \omega \omega'
 \end{split}  
\end{equation} 
\end{small}
\hskip -0.8mm for all $v_{1}, \dots, v_{n}, v'_{1}, \dots, v'_{n'} \in V$, $\sigma, \tau \in \SG_{n}$,
$\sigma', \tau' \in \SG_{n'}$ and 
$\omega, \omega' \in \Sym(W)$. 
The unit of $\mathcal{F}(V,W)$ is given by the morphism $\label{eq:unit-F-calig} i : \mathbf{1}_{\SG^{\env}} \rightarrow \mathcal{F}(V,W)$ whose component of degree zero $i(0)$ is the unit of $\Bbbk \otimes \Sym(W)$. 
Hence, $\mathcal{F}(V,W)$ is a(n unitary and associative) commutative algebra in the symmetric monoidal category $\DMod$. 

Moreover, given linear maps $\varphi : V \rightarrow V'$ and $\psi : W \rightarrow W'$, the universal property \eqref{eq:sym-alg} implies that there exists a unique morphism 
\begin{equation}
\label{eq:morph-sym-V-W}
\mathcal{F}(\varphi,\psi) : \mathcal{F}(V,W) \longrightarrow \mathcal{F}(V',W') 
\end{equation}
of algebras in $\DMod$ whose restriction to $\I_{\SG^{\env},1}(V) \oplus \I_{\SG^{\env},0}(W)$ is given by 
the composition of
$\I_{\SG^{\env},1}(f) \oplus \I_{\SG^{\env},0}(g)$ and the canonical inclusion of $\I_{\SG^{\env},1}(V') \oplus \I_{\SG^{\env},0}(W')$ in $\mathcal{F}(V',W')$. 
More explicitly, $\mathcal{F}(\varphi,\psi)$ is 
the sequence of maps 
\[
\Big\{ \mathcal{F}(\varphi,\psi)(n) : \mathcal{F}(V,W)(n) \longrightarrow \mathcal{F}(V',W')(n) \Big\}_{n \in \NN_{0}} \]
given by
\begin{equation}
\label{eq:prod-fock-s-e-morph}
\begin{split}
   \mathcal{F}(\varphi,\psi)(n) &\Big(\big((v_{1} \otimes \dots\otimes  v_{n}) \otimes_{\Bbbk \SG_{n}} (\sigma \otimes \tau) \big) \otimes \omega \Big)
   \\ 
   &= \Big(\big(\varphi(v_{1}) \otimes \dots \otimes \varphi(v_{n})\big) \otimes_{\Bbbk \SG_{n}} (\sigma \otimes \tau) \Big) \otimes \Sym(\psi)(\omega)
 \end{split}  
\end{equation} 
for all $v_{1}, \dots, v_{n} \in V$, $\sigma, \tau \in \SG_{n}$, and 
$\omega \in \Sym(W)$, where $\Sym(\psi) : \Sym(W) \rightarrow \Sym(W')$ is the morphism of algebras induced by $\psi$. 

%%%%%%%%%%%%%%%%%%%%%%%%%%%%%%%%%%%%%%%%%%%%%%%%%%%%%%%%%%%%%%%%%%%%%%
\subsection{The Fock wheelgebra associated with a nonunitary algebra}
\label{sec:Fock-algebra-associated-algebra}

Let $A$ be a nonunitary algebra. 
Following \cite{MR2734329}, we will now introduce the \textbf{\textcolor{myblue}{Fock wheelgebra}} $\Fock(A) \label{index:Fock-wheelgebra}$, which is a commutative wheelgebra
such that 
\[
\Fo(\Fock(A)) = \mathcal{F}(A,A_{\cyc}) = \Sym_{\SG^\env}\big(\I_{\SG^{\env},1}(A) \oplus \I_{\SG^{\env},0}(A_{\cyc})\big).
\]
Here, $\Fo$ is the forgetful functor defined in \eqref{eq:forgetful-wh-dmod} that is given by sending a wheelspace to its underlying diagonal $\SG$-bimodule, and that is the identity on morphisms. 
To simplify our notation, let us define the diagonal $\SG$-module
\begin{equation}
\label{eq:gen-s}  
     \Gen(A) = \I_{\SG^{\env},1}(A) \oplus \I_{\SG^{\env},0}(A_{\cyc}). 
\end{equation}
By definition, the underlying diagonal $\SG$-bimodule of $\Fock(A)$ is $\mathcal{F}(A,A_{\cyc})$, and its product and unit are those of $\mathcal{F}(A,A_{\cyc})$. 
It remains to define the partial contractions of 
$\Fock(A)$. 
Given $i,j \in \llbracket 1 , n \rrbracket$, we define 
\[
{}^{n}\mathcalboondox{t} : \Fock(A)(n) \longrightarrow \Fock(A)(n-1)
\]
by 
\begin{equation} 
\label{eq:tij-fock}
\begin{split}
&{}^{n}\mathcalboondox{t}\bigg(\Big((a_{1} \otimes \dots \otimes a_{n}) \otimes_{\Bbbk \SG_{n}} \big((i \ \dots \ n) \cdot \sigma \otimes (j \ \dots \ n)^{-1} \cdot_{\op} \tau\big)\Big) \otimes \omega \bigg) 
\\ 
&= 
\begin{cases}
&\Big((a_{1} \otimes \dots \otimes a_{i-1} \otimes a_{i} a_{j} \otimes a_{i+1} \otimes \dots \otimes a_{j-1} \otimes a_{j+1} \otimes \dots \otimes a_{n}) 
\\
&\phantom{=}\otimes_{\Bbbk \SG_{n-1}} \big((i \ \dots \ j-1) \cdot \sigma \otimes \tau\big)\Big) \otimes \omega, 
\text{\hskip 1cm if $i < j$,}
\\
&\Big((a_{1} \otimes \dots \otimes a_{i-1} \otimes a_{i+1} \otimes \dots \otimes a_{n}) 
\otimes_{\Bbbk \SG_{n-1}} \big(\sigma \otimes \tau\big)\Big) \otimes \pi(a_{i}).\omega, 
\\
&\text{\hskip 6.8cm if $i = j$,}
\\
&\Big((a_{1} \otimes \dots \otimes a_{j-1} \otimes a_{j+1} \otimes \dots \otimes a_{i-1} \otimes a_{i} a_{j} \otimes a_{i+1} \otimes \dots \otimes a_{n})
\\ 
&\phantom{=}\otimes_{\Bbbk \SG_{n-1}} \big((j \ \dots \ i-1)^{-1} \cdot \sigma \otimes \tau\big)\Big) \otimes \omega,
\text{\hskip 0.65 cm if $i > j$,}
\end{cases}
\end{split}
\end{equation}     
for all $a_{1}, \dots, a_{n} \in A$, 
$\sigma, \tau \in \Img(\linc_{n-1})$ and 
$\omega \in \Sym(A_{\cyc})$, where $\pi : A \rightarrow A_{\cyc}$ is the canonical projection.  

%%%%%%%
\begin{remark}
\label{remark:t-ij-fock-2}
It is easy to see that \eqref{eq:tij-fock} is equivalent to 
\begin{equation} 
\label{eq:tij-fock-2}
\begin{split}
&{}^{n}\mathcalboondox{t}\bigg(\Big((a_{1} \otimes \dots \otimes a_{n}) \otimes_{\Bbbk \SG_{n}} \big((i \ \dots \ n) \cdot \sigma \otimes (j \ \dots \ n)^{-1} \cdot_{\op} \tau\big)\Big) \otimes \omega \bigg) 
\\ 
&= 
\begin{cases}
&\Big((a_{i} a_{j} \otimes a_{1} \otimes \dots \otimes a_{i-1} \otimes a_{i+1} \otimes \dots \otimes a_{j-1} \otimes a_{j+1} \otimes \dots \otimes a_{n}) 
\\
&\phantom{=}\otimes_{\Bbbk \SG_{n-1}} \big((1 \ \dots \ j-1) \cdot \sigma \otimes (1 \ \dots \ i)^{-1} \cdot_{\op} \tau\big)\Big) \otimes \omega, 
\text{\hskip 1cm if $i < j$,}
\\
&\Big((a_{1} \otimes \dots \otimes a_{i-1} \otimes a_{i+1} \otimes \dots \otimes a_{n}) 
\otimes_{\Bbbk \SG_{n-1}} \big(\sigma \otimes \tau\big)\Big) \otimes \pi(a_{i}).\omega, 
\\
&\text{\hskip 9cm if $i = j$,}
\\
&\Big((a_{i} a_{j} \otimes a_{1} \otimes \dots \otimes a_{j-1} \otimes a_{j+1} \otimes \dots \otimes a_{i-1} \otimes a_{i+1} \otimes \dots \otimes a_{n})
\\ 
&\phantom{=}\otimes_{\Bbbk \SG_{n-1}} \big((1 \ \dots \ j) \cdot \sigma \otimes (1 \ \dots \ i-1)^{-1} \cdot_{\op} \tau\big)\Big) \otimes \omega,
\text{\hskip 0.95 cm if $i > j$,}
\end{cases}
\end{split}
\end{equation}     
for all $a_{1}, \dots, a_{n} \in A$, 
$\sigma, \tau \in \Img(\linc_{n-1})$ and 
$\omega \in \Sym(A_{\cyc})$. 
Moreover, if we restrict $\sigma = \tau = \id_{\llbracket 1 , n \rrbracket}$ in \eqref{eq:tij-fock-2}, the latter coincides with the value of the contraction ${}^{n}_{j}t_{i}$ of $\Fock(A)$ at
\[          (a_{1} \otimes \dots \otimes a_{n}) \otimes_{\Bbbk \SG_{n}} (\id_{\llbracket 1 , n \rrbracket} \otimes \id_{\llbracket 1 , n \rrbracket}) \otimes \omega,     \]
by \eqref{eq:t-i-j}. 
\end{remark}
%%%%%%%

%%%%%%%
\begin{notation}
\label{notation:elements}
Since we will need to manipulate the elements of $\Fock(A)$ to save space, we will typically replace several tensor symbols in a typical element of $\Fock(A)(n)$ by bars and simply write 
\[       
(a_{1} | \dots | a_{n}) \otimes_{\Bbbk \SG_{n}} (\sigma | \tau) \otimes \omega,  
\]
where $a_{1}, \dots, a_{n} \in A$, $\sigma, \tau \in \SG_{n}$ and $\omega \in \Sym(A_{\cyc})$.
\end{notation}
%%%%%%%

%%%%%%%
\begin{lemma} 
\label{lemma:fock}
Let $A$ be a nonunitary algebra. 
The diagonal $\SG$-bimodule $\mathcal{F}(A,A_{\cyc})$ endowed with the partial contraction \eqref{eq:tij-fock}
is a wheelspace that we denote by $\Fock(A)$.
Moreover, $\Fock(A)$ endowed with the product $\mu$ given by \eqref{eq:prod-fock-s-e} 
for $V = A$ and $W = A_{\cyc}$ 
is a commutative wheelgebra. 
\end{lemma}
%%%%%%%
\begin{proof}
Based on Lemma \ref{lema:equivalencia-wheel-partial} and Definition \ref{def:partial-wheelspace}, to see that $\Fock(A)$ is a wheelspace, we need to prove that \eqref{eq:part-trace-comm} holds for the partial contractions defined in \eqref{eq:tij-fock}. 
As the reader quickly would be aware, there are a myriad of cases, depending on the order relations between $i,j,k,\ell$. 
Since all the cases are conceptually pretty similar, we will content ourselves to present here the case $i<k<\ell$ and $i<j<\ell$. 
Note that at this point there is no relation between $j$ and $k$; but to apply \eqref{eq:part-trace-comm} for ${}^{n}\mathcalboondox{t}$ we will need to distinguish the different cases.
Also, it is useful to point out that since ${}^{n}\mathcalboondox{t}$ (resp.,  ${}^{n+1}\mathcalboondox{t}$) is a map of $\Bbbk\SG_{n-1}$-bimodules (resp., of $\Bbbk\SG_{n}$-bimodules), it is enough to prove the identity
\begin{equation}
 \big({}^{n}\mathcalboondox{t} \circ {}^{n+1}\mathcalboondox{t}\big) \big((n \ n+1) \cdot x \cdot (n \ n+1)\big) = \big({}^{n}\mathcalboondox{t} \circ {}^{n+1}\mathcalboondox{t}\big) (x),
 \label{eq: 3.5-fock-wheelalgebras-inicio}
 \end{equation}
where 
\[
x=(a_1 | \dots |  a_{n+1}) \otimes_{\Bbbk\SG_{n+1}}\big( (k \ \dots \ n+1)(i \ \dots \ n) |  (\ell \ \dots \ n+1)^{-1}(j \ \dots \ n)^{-1}\big)\otimes \omega.
\]
On the right-hand side of \eqref{eq: 3.5-fock-wheelalgebras-inicio}, by \eqref{eq:part-trace-comm} and since ${}^{n}\mathcalboondox{t}$ is a map of $\Bbbk\SG_{n-1}$-bimodules, we have
\begin{small}
\begin{equation}
 \label{eq: 3.5-fock-wheelalgebras-inicio-1}
\begin{aligned}
& \big({}^{n}\mathcalboondox{t} \circ {}^{n+1}\mathcalboondox{t}\big) (x)
\\
&= {}^{n}\mathcalboondox{t}\Big(a_1|\cdots | a_{k-1}|a_k a_{\ell}| a_{k+1}|\cdots | \widehat{a}_{\ell}| \cdots | a_{n+1}
\\
&\qquad \qquad \otimes_{\Bbbk\SG_n}\big((k\dots \ell-1)(i\dots n) | (j\cdots n)^{-1}\big)\otimes \omega\Big)
\\
&= {}^{n}\mathcalboondox{t}\Big(a_1|\cdots | a_{k-1}| a_k a_{\ell}|a_{k+1}|\cdots | \widehat{a}_{\ell} | \cdots |a_{n+1}
\\
&\qquad \qquad \otimes_{\Bbbk\SG_n}\big( (i\dots n) | (j\dots n)^{-1}\big) \otimes \omega\Big)\cdot (k-1\dots \ell-2)
\\
&= \begin{cases}
a_1|\cdots|a_{i-1}| a_ia_j| a_{i+1}|\cdots |\widehat{a}_{j}|\cdots | a_{k-1}|a_ka_{\ell} |a_{k+1}| \cdots | \widehat{a}_{\ell} | \cdots | a_{n+1}
\\
  \hskip 2.32cm \otimes_{\kk\SG_{n-1}}\big((i\dots j-1)(k-1\dots \ell-2) |  1\big)\otimes \omega, \qquad \text{if $j<k$,} 
     \\
    a_1 | \cdots |  a_{i-1} | a_ia_ja_{\ell} | a_{i+1} |\cdots |\widehat{a}_k| \cdots | \widehat{a}_{\ell} |\dots | a_{n+1}
    \\
     \hskip 2.32cm \otimes_{\kk\SG_{n-1}}\big((i\dots k-1)(k-1\dots \ell-2) |  1\big)\otimes \omega, \qquad \text{if $j=k$,}
         \\
a_1 | \cdots | a_{i-1} | a_ia_j | a_{i+1}| \cdots  | a_{k-1} | a_ka_{\ell} | a_{k+1} | \cdots | \widehat{a}_{j} | \cdots | \widehat{a}_\ell | \cdots a_{n+1}  
   \\ 
     \hskip 2.32cm \otimes_{\kk\SG_{n-1}}\big((i\dots j-1)(k-1\dots \ell-2) | 1\big)\otimes \omega, \qquad \text{if $j>k$.}
     \end{cases}
\end{aligned} 
\end{equation}
\end{small}

On the left-hand side of \eqref{eq: 3.5-fock-wheelalgebras-inicio},  since $ {}^{n+1}\mathcalboondox{t}$ is a map of $\kk\SG_n$-bimodules, as well as by Fact \ref{fact:cases} and \eqref{eq:part-trace-comm}, we have 
\begin{small}
\allowdisplaybreaks
\begin{align*}
& \big({}^{n}\mathcalboondox{t} \circ {}^{n+1}\mathcalboondox{t}\big) \big((n \ n+1) \cdot x \cdot (n \ n+1)\big)
\\
&= \hskip -0.6mm \big({}^{n}\mathcalboondox{t} \circ {}^{n+1}\mathcalboondox{t}\big)\Big(a_1 | \cdots | a_{n+1}
\otimes_{\kk\SG_{n+1}}\hskip -0.8mm\big((i\dots n+1)(k-1\dots n) | (j\dots n+1)^{-1}(\ell-1\dots n)^{-1}\big)\otimes\omega\Big)
\\
&=\hskip -0.6mm {}^{n}\mathcalboondox{t}\Big( (\ell-1\dots n)^{-1} \hskip -0.4mm \cdot \hskip -0.4mm {}^{n+1}\mathcalboondox{t}\big(a_1| \cdots | a_{n+1}
\otimes_{\kk\SG_{n+1}}\hskip -1mm\big((i\dots n+1) | (j\dots n+1)^{-1}\big)\otimes\omega\big) \hskip -0.4mm \cdot \hskip -0.4mm (k-1\dots n)\Big)
\\
&=\hskip -0.6mm {}^{n}\mathcalboondox{t}\big(a_1 | \cdots | a_{i-1}|a_{i}a_j | a_{i+1} | \cdots | \widehat{a}_j | \cdots | a_{n+1}
 \otimes_{\kk\SG_{n}}\hskip -0.8mm\big((i\dots j-1)(k-1\dots n) |  (\ell-1\dots n)^{-1} \big)\otimes\omega\Big)
\\
\stepcounter{equation}\tag{\theequation}\label{eq: 3.5-fock-wheelalgebras-inicio-2}
&= \begin{cases}
 {}^{n}\mathcalboondox{t}\Big(a_1 | \cdots | a_{i-1} | a_{i}a_j | a_{i+1} |\cdots | \widehat{a}_j | \cdots | a_{n+1}
 \\
  \hskip 2.32cm \otimes_{\kk\SG_n}\big((k-1\dots n) | (\ell-1\dots n)^{-1}\big) \otimes \omega\Big)\cdot (i\dots j-1), \qquad \text{if $j<k$,} 
  \\
   {}^{n}\mathcalboondox{t}\Big(a_1 | \cdots | a_{i-1} | a_{i}a_j | a_{i+1} |\cdots |  \widehat{a}_j | \cdots | a_{n+1}
 \\
  \hskip 2.32cm \otimes_{\kk\SG_n}\big((i\dots n)| (\ell-1\dots n)^{-1}\big)\otimes \omega\Big),   \qquad \text{\hskip 2.45cm if $j=k$,} 
  \\
   {}^{n}\mathcalboondox{t}\Big(a_1 | \cdots | a_{i-1} | a_{i}a_j | a_{i+1} | \cdots | \widehat{a}_j | \cdots | a_{n+1}
 \\
  \hskip 2.32cm \otimes_{\kk\SG_n}\big((k\dots n) | (\ell-1\dots n)^{-1}\big)\otimes \omega\Big)\cdot ( i\dots j-2), \qquad \text{\hskip 0.6cm if $j>k$,} 
   \end{cases}
   \\
&= \begin{cases}
a_1| \cdots | a_{i-1} | a_{i}a_j | a_{i+1} |\cdots | \widehat{a}_j | \cdots a_{k-1} |a_ka_{\ell} | a_{k+1} | \cdots | \widehat{a}_\ell |\cdots | a_{n+1}
 \\
  \hskip 2.32cm \otimes_{\kk\SG_{n-1}}\big((k-1\dots \ell-2)(i\dots j-1) | 1 \big)\otimes \omega\Big), \qquad \text{\hskip 1.2cm if $j<k$,} 
  \\
a_1 | \cdots | a_{i-1} | a_{i}a_ja_{\ell} | a_{i+1} | \cdots | \widehat{a}_j | \cdots | \widehat{a}_{\ell} | \cdots  | a_{n+1}
 \\
  \hskip 2.32cm \otimes_{\kk\SG_{n-1}}\big((i\dots \ell-2)  | 1\big)\otimes \omega\Big), \qquad \text{\hskip 3.4cm if $j=k$,} 
  \\
a_1 | \cdots | a_{i-1} | a_{i}a_j |a_{i+1} |  \cdots |a_{k-1} | a_ka_\ell | a_{k+1} | \cdots | \widehat{a}_j | \cdots | \widehat{a}_{\ell} | \cdots | a_{n+1}
 \\
  \hskip 2.32cm \otimes_{\kk\SG_{n-1}}\big((k\dots \ell-2)(i\dots j-2) | 1\big) \otimes \omega\Big), \qquad \text{\hskip 1.8cm if $j>k$.} 
   \end{cases}
\end{align*}
\end{small}
It is straightforward to check that the permutations involved in each case in \eqref{eq: 3.5-fock-wheelalgebras-inicio-1} agree with the permutations written in the corresponding case of \eqref{eq: 3.5-fock-wheelalgebras-inicio-2}.

It remains to show that $\Fock(A)$ 
is a commutative wheelgebra. 
To prove it, note first that Proposition \ref{proposition:morph-tensor-gwh} tells us that $\mu$ is a morphism of generalized wheelspaces, since \eqref{eq:morph-tensor-wh-def} are trivially verified. 
Finally, since  $\mathcal{F}(A,A_{\cyc})$ 
is a(n unitary and associative) commutative algebra in the symmetric monoidal category $\DMod$, Fact \ref{fact:gwh-algebra-ex} implies that $\Fock(A)$
is a commutative wheelgebra. 
\end{proof}

%%%%%%%
\begin{remark}
\label{rem:alg-str-partial-contractions}
Note that, by definition, the only part where the algebra structure of $A$ plays a role in $\Fock(A)$ is in its (partial) contractions.  
\end{remark}
%%%%%%% 

%%%%%%%
\begin{remark}
It is interesting to explicitly determine $\alg(\Fock(A))$, and in particular verify the admissibility property. 
The latter follows from the fact that 
\[     ({}^{2}_{2}t_{1} \circ \mu_{1,1})(a \otimes \omega, a' \otimes \omega') = {}^{2}_{2}t_{1}\big((a | a') \otimes (\id_{\{ 1, 2 \}} | \id_{\{ 1, 2 \}}) \otimes \omega.\omega' \big) = a.a' \otimes \omega.\omega'     \]
for $a, a' \in A$ and $\omega, \omega' \in \Sym(A_{\cyc})$
tells us that ${}^{2}_{2}t_{1} \circ \mu_{1,1} : (A \otimes \Sym(A_{\cyc}))^{\otimes 2} \rightarrow A \otimes \Sym(A_{\cyc})$ is the usual product of the tensor product algebra of the algebras $A$ and $\Sym(A_{\cyc})$. 
\end{remark}
%%%%%%%

Finally we assume that $A$ is endowed with a grading of the form $A = \oplus_{n\in \NN_{0}} A_{n}$. 
It is clear that the vector space $A_{\cyc}$ has a unique grading such that the map $\pi : A \rightarrow A_{\cyc}$ is homogeneous of degree zero, and that the algebra $\Sym(A_{\cyc})$ has a unique grading such that the canonical inclusion $A_{\cyc} \rightarrow \Sym(A_{\cyc})$ preserves the degree. 
We denote the component of degree $n \in \NN_{0}$ of $\Sym(A_{\cyc})$ by $\Sym_{n}(A_{\cyc})\label{index:component-of-degree-n-of-Sym}$. 
Then, given $n \in \NN_{0}$, define the sub-wheelspace $\Fock(A)_{n}$ of $\Fock(A)$ by 
\begin{equation}
\label{eq:f-A-n}
\Fock(A)_{n}(m) = \bigoplus_{\text{\begin{tiny}$\begin{matrix}i, i' \in \NN_{0}\\ i + i' = n\end{matrix}$\end{tiny}}} \Ind_{\Delta_{m}} \bigg(\bigoplus_{\text{\begin{tiny}$\begin{matrix}(i_{1}, \dots, i_{m}) \in \NN_{0} \\ i_{1} + \dots + i_{m} = i\end{matrix}$\end{tiny}}} A_{i_{1}} \otimes \dots \otimes A_{i_{m}}\bigg) \otimes \Sym_{i'}(A_{\cyc})
\end{equation}
for $m \in \NN_{0}$. 
In the sequel, we will say that $\Fock(A)_{n}(m)$ has \textcolor{myblue}{\textbf{degree}} $n$ and \textcolor{myblue}{\textbf{weight}} $m$.
It is trivial to verify that $\Fock(A)_{n}$ is indeed a subwheelspace of $\Fock(A)$ for all $n \in \NN_{0}$. 
Moreover, by definition of grading, $\Fock(A)_{n}$ is naturally a symmetric wheelbimodule over $\Fock(A)_{0}$ for all $n \in \NN_{0}$. 
Note also that $\Fock(A)_{0} = \Fock(A_{0})$ and that 
\begin{equation}
\label{eq:f-A-1}
    \begin{split}
          \Fock(A)_{1}(m) &=   \Ind_{\Delta_{m}} \Big(\bigoplus_{i=0}^{m-1} A_{0}^{\otimes i} \otimes A_{1} \otimes A_{0}^{\otimes (m-i-1)}\Big) \otimes \Sym(A_{\cyc,0}) 
          \\
&\phantom{=} \oplus \Ind_{\Delta_{m}} \big(A_{0}^{\otimes m}\big) \otimes \Sym(A_{\cyc,0}).A_{\cyc,1} 
\end{split}
\end{equation}
for $m \in \NN_{0}$. 

We will use the following result to deal with the contractions of $\Fock(A)$. 
%%%%%%%
\begin{lemma} 
\label{lemma:contractions-fock}
Let $S' = (S'(n))_{n \in \NN_{0}}$ be a commutative wheelgebra with product $\mu' : S' \otimes_{\gwh} S' \rightarrow S'$ and contractions $({}^{n}_{j}t'_{i})_{n \in \NN_{0}, i, j \in \llbracket 1 , n \rrbracket}$, and let $f : \Fo(\Fock(A)) \rightarrow \Fo(S')$ be a morphism of algebras in $\DMod$. 
Given $n \in \NN_{0}$, let $\El_{n} \subseteq \Fock(A)(n)$ be the vector subspace spanned by the elements 
\begin{equation}
\label{eq:eln}
     (a_{1} | \dots |a_{n}) \otimes_{\Bbbk \SG_{n}} (\id_{\llbracket 1 , n \rrbracket}|\id_{\llbracket 1 , n \rrbracket}) \otimes 1_{\Sym(A_{\cyc})}
\end{equation}
for $a_{1}, \dots, a_{n} \in A$, that we will simply denote by $\overline{a_{1} | \dots |a_{n}}$. 
Then, $f = (f(n))_{n \in \NN_{0}}$ is a morphism of wheelgebras if and only if 
\begin{equation}
\label{eq:eln-bis}
\begin{aligned}
\big({}^{1}_{1}t'_{1} \circ f(1)\big)|_{\El_{1}} = \big(f(0) \circ {}^{1}_{1}t_{1}\big)|_{\El_{1}} \hskip 1.4mm \text{ and } \hskip 1.4mm
\big({}^{n}_{2}t'_{1} \circ f(n)\big)|_{\El_{n}} = \big(f(n-1) \circ {}^{n}_{2}t_{1}\big)|_{\El_{n}},
\end{aligned}
\end{equation}
for all integers $n \geq 2$, where we denote by $({}^{n}_{j}t_{i})_{n \in \NN_{0}, i, j \in \llbracket 1 , n \rrbracket}$ the contractions of $\Fock(A)$. 
\end{lemma}
%%%%%%%
\begin{proof} 
It is clear that if $f$ is a morphism of wheelgebras, then \eqref{eq:eln-bis} holds, since it follows from \eqref{eq:morph-wh}. 

Assume conversely that \eqref{eq:eln-bis} holds. 
Since $f$ is a morphism of algebras in $\DMod$, we only have to prove that \eqref{eq:morph-wh} is verified. 
We will first prove that \eqref{eq:eln-bis} implies that 
\begin{equation}
\label{eq:eln-bis-i-j}
\big({}^{n}_{j}t'_{i} \circ f(n)\big)|_{\El_{n}} = \big(f(n-1) \circ {}^{n}_{j}t_{i}\big)|_{\El_{n}}, 
\end{equation}
for all $n \in \NN$ and $i, j \in \llbracket 1 , n \rrbracket$.
We will prove this by induction on $n$. 
The identity \eqref{eq:eln-bis-i-j} for $i=j=1$ clearly holds, since  
\begin{equation}
\label{eq:eln-bis-i-j-case1}
\begin{split}
&\big({}^{n}_{1}t'_{1} \circ f(n)\big)(\overline{a_{1} | \dots | a_{n}}) 
= \big({}^{n}_{1}t'_{1} \circ f(n)\big)\big(\mu_{1,n-1}(\bar{a}_{1} , \overline{a_{2} | \dots | a_{n}})\big) 
\\
&= {}^{n}_{1}t'_{1}\Big(\mu'_{1,n-1}\big(f(1)(\bar{a}_{1}) , f(n-1)(\overline{a_{2} | \dots | a_{n}})\big)\Big)
\\
&=\mu'_{0,n-1}\Big(\big({}^{1}_{1}t'_{1} \circ f(1)\big)(\bar{a}_{1}) , f(n-1)(\overline{a_{2} | \dots | a_{n}})\Big)
\\
&=\mu'_{0,n-1}\Big(\big(f(0) \circ {}^{1}_{1}t_{1}\big)(\bar{a}_{1}) , f(n-1)(\overline{a_{2} | \dots | a_{n}})\Big)
\\
&=\big(f(n-1) \circ \mu_{0,n-1}\big)\big( {}^{1}_{1}t_{1}(\bar{a}_{1}) , \overline{a_{2} | \dots | a_{n}}\big)
\\
&=\big(f(n-1) \circ {}^{n}_{1}t_{1}\big) \Big(\mu_{1,n-1}\big( \bar{a}_{1} , \overline{a_{2} | \dots | a_{n}}\big)\Big)
=\big(f(n-1) \circ {}^{n}_{1}t_{1}\big) (\overline{a_{1} | \dots | a_{n}}),
\end{split}
\end{equation}
where we used in the second and fifth equalities that $f$ is a morphism of algebras in $\DMod$, in the third and sixth equalities we used that $\mu'$ and $\mu$ are morphisms of generalized wheelspaces, respectively, 
and the first identity of \eqref{eq:eln-bis} in the fourth equality. 
On the other hand, the identity \eqref{eq:eln-bis-i-j} for $i,j>1$ (and thus $n > 1$) also holds, since 
\begin{equation}
\label{eq:eln-bis-i-j-case2}
\begin{split}
&\big({}^{n}_{j}t'_{i} \circ f(n)\big)(\overline{a_{1} | \dots | a_{n}}) 
= \big({}^{n}_{j}t'_{i} \circ f(n)\big)\big(\mu_{1,n-1}(\bar{a}_{1} , \overline{a_{2} | \dots | a_{n}})\big) 
\\
&= {}^{n}_{j}t'_{i}\Big(\mu'_{1,n-1}\big(f(1)(\bar{a}_{1}) , f(n-1)(\overline{a_{2} | \dots | a_{n}})\big)\Big)
\\
&=\mu'_{1,n-2}\Big(f(1)(\bar{a}_{1}) , \big({}^{n-1}_{j-1}t_{i-1} \circ f(n-1)\big)(\overline{a_{2} | \dots | a_{n}})\Big)
\\
&=\mu'_{1,n-2}\Big(f(1)(\bar{a}_{1}) , \big(f(n-2) \circ {}^{n-1}_{j-1}t_{i-1} \big)(\overline{a_{2} | \dots | a_{n}})\Big)
\\
&=\big(f(n-1) \circ \mu_{1,n-2}\big)\big( \bar{a}_{1} , {}^{n-1}_{j-1}t_{i-1} (\overline{a_{2} | \dots | a_{n}})\big)
\\
&=\big(f(n-1) \circ {}^{n}_{j}t_{i}\big) \Big(\mu_{1,n-1}\big( \bar{a}_{1}, \overline{a_{2} | \dots | a_{n}}\big)\Big)
=\big(f(n-1) \circ {}^{n}_{j}t_{i}\big) (\overline{a_{1} | \dots | a_{n}}),
\end{split}
\end{equation}
where we used the same arguments as those of \eqref{eq:eln-bis-i-j-case1} in the second, third, fifth and sixth equalities,
and the inductive assumption in the fourth equality. 
It remains to show \eqref{eq:eln-bis-i-j} for $i = 1$ and $j>1$ (and thus $n > 1$), the case where $i>1$ and $j=1$ follows in the same way. 
Let us first note that, by definition of $\Fock(A)$, we have that 
\begin{equation}
\label{eq:aux-case3}
\overline{a_{1} | \dots | a_{n}} = 
(a_{1} | a_{j} | \alpha | \beta) \otimes_{\Bbbk \SG_{n}} \big( (2 \ \dots \ j) | (2 \ \dots \ j)^{-1} \big) \otimes 1_{\Sym(A_{\cyc})}, 
\end{equation}
where $\alpha = a_{2} | \dots | a_{j-1}$ and $\beta = a_{j+1} | \dots | a_{n}$. 
Then, 
\begin{equation}
\label{eq:eln-bis-i-j-case3}
\begin{split}
&\big({}^{n}_{j}t'_{1} \circ f(n)\big)(\overline{a_{1} | \dots | a_{n}}) 
\\
&= \big({}^{n}_{j}t'_{1} \circ f(n)\big)\Big((a_{1} | a_{j} | \alpha | \beta) \otimes_{\Bbbk \SG_{n}} \big( (2 \ \dots \ j) | (2 \ \dots \ j)^{-1} \big) \otimes 1_{\Sym(A_{\cyc})}\Big) 
\\
&= {}^{n}_{j}t'_{1}\Big((2 \ \dots \ j)^{-1} \cdot f(n)\big(\overline{a_{1} | a_{j} | \alpha | \beta} \big) \cdot (2 \ \dots \ j)\Big)
\\
&= \mathbbl{l}^{n}_{j}((2 \ \dots \ j)^{-1}) \cdot {}^{n}_{2}t'_{1}\Big(f(n)\big(\overline{a_{1} | a_{j} | \alpha | \beta}\big)\Big) \cdot \mathbbl{r}^{n}_{j}(2 \ \dots \ j)
\\
&= \mathbbl{l}^{n}_{j}((2 \ \dots \ j)^{-1}) \cdot f(n-1)\Big({}^{n}_{2}t_{1}\big(\overline{a_{1} | a_{j} | \alpha | \beta}\big)\Big) \cdot \mathbbl{r}^{n}_{j}(2 \ \dots \ j)
\\
&= f(n-1) \Big(\mathbbl{l}^{n}_{j}((2 \ \dots \ j)^{-1}) \cdot {}^{n}_{2}t_{1}\big(\overline{a_{1} | a_{j} | \alpha | \beta}\big) \cdot \mathbbl{r}^{n}_{j}(2 \ \dots \ j) \Big)
\\
&=\big(f(n-1) \circ {}^{n}_{j}t_{1}\big)\Big((a_{1} | a_{j} | \alpha | \beta) \otimes_{\Bbbk \SG_{n}} \big( (2 \ \dots \ j) | (2 \ \dots \ j)^{-1} \big) \otimes 1_{\Sym(A_{\cyc})}\Big) 
\\
&= \big(f(n-1) \circ {}^{n}_{j}t_{1}\big)(\overline{a_{1} | \dots | a_{n}}),
\end{split}
\end{equation}
where we used that $f(m)$ is a morphism of $\SG_{m}$-bimodules in the second and fifth equalities, 
\ref{item:W2} in the third and sixth equalities, and \eqref{eq:eln-bis} in the fourth equality. 
We have thus proved \eqref{eq:eln-bis-i-j}. 

Using Remark \ref{remark:t-ij-fock-2} with $\omega = 1_{\Sym(A_{\cyc})}$, \eqref{eq:eln-bis-i-j} implies that 
\begin{equation}
\label{eq:eln-bisbis}
{}^{n}\mathcalboondox{t}' \circ f(n)|_{\El_{n}^{\circ}} = f(n-1) \circ {}^{n}\mathcalboondox{t}|_{\El_{n}^{\circ}}
\end{equation}
holds for $n \in \NN$, where $\El_{n}^{\circ} = \Ind_{\Bbbk \Delta_{n}}(A^{\otimes n}) \otimes \Sym^{0}(A_{\cyc}) \subseteq \Fock(A)(n)$.
Finally, we claim that \eqref{eq:eln-bisbis} implies \eqref{eq:morph-wh-red}. 
Indeed, let $x \in \El_{n}^{\circ}$ be an element of the form $(a_{1} | \dots | a_{n}) \otimes_{\Bbbk \SG_{n}} (\sigma | \tau) \otimes 1_{\Sym(A_{\cyc})}$, for $a_{1}, \dots, a_{n} \in A$ and $\sigma, \tau \in \SG_{n}$.  
Then, given $\omega \in \Sym(A_{\cyc})$, 
\eqref{eq:eln-bisbis} tells us that 
\begin{equation} 
\label{eq:aux-1}
\begin{split}     
&\big({}^{n}\mathcalboondox{t}' \circ f(n)\big)\big((a_{1} | \dots | a_{n}) \otimes_{\Bbbk \SG_{n}} (\sigma | \tau) \otimes \omega\big) 
= \big({}^{n}\mathcalboondox{t}' \circ f(n)\big)\big(\mu_{0,n}(\omega,x)\big) 
\\
&= ({}^{n}\mathcalboondox{t}' \circ \mu'_{0,n})\big(f(0)(\omega),f(n)(x)\big)
= \mu'_{0,n-1}\Big(f(0)(\omega),\big({}^{n}\mathcalboondox{t}' \circ f(n)\big)(x)\Big)
\\
&= \mu'_{0,n-1}\Big(f(0)(\omega),\big(f(n-1) \circ {}^{n}\mathcalboondox{t} \big)(x)\Big) 
= \big(f(n-1) \circ \mu_{0,n-1}\big)\big(\omega,{}^{n}\mathcalboondox{t} (x)\big) 
\\
&= \big(f(n-1) \circ {}^{n}\mathcalboondox{t} \circ \mu_{0,n}\big)(\omega,x) 
\\
&= \big(f(n-1) \circ {}^{n}\mathcalboondox{t}\big)\big((a_{1} | \dots | a_{n}) \otimes_{\Bbbk \SG_{n}} (\sigma | \tau) \otimes \omega\big), 
\end{split}
\end{equation} 
where we used in the second and fifth identities that  
\[     f(k)\big(\mu_{0,k}(\omega,y)\big) = \mu'_{0,k}\big(f(0)(\omega),f(k)(y)\big)     \] 
for $y \in \Ind_{\Bbbk \Delta_{k}}(A^{\otimes k})$, since $f$ is a morphism of algebras in $\DMod$. 
In the third and sixth identities we used that $\mu'$ and $\mu$ are morphisms of generalized wheelspaces, respectively, 
and the assumption that \eqref{eq:morph-wh} restricted to $\Ind_{\Bbbk \Delta_{n}}(A^{\otimes n}) \otimes \Sym^{0}(A_{\cyc})$ holds for all $n \in \NN$ in the fourth identity.
Finally, since \eqref{eq:morph-wh-red} holds, Lemma \ref{lema:equivalencia-wheel-partial} tells us that $f$ is a morphism of wheelspaces.  
\end{proof}

The following result is an immediate consequence of the definition.
%%%%%%%
\begin{lemma} 
Let $\varphi : A \rightarrow A'$ be a morphism of nonunitary algebras. 
Then, the morphism $\Fock(\varphi) : \Fock(A) \rightarrow \Fock(A')$ given by $\mathcal{F}(\varphi,\varphi_{\cyc})$ in \eqref{eq:prod-fock-s-e-morph} is a morphism of commutative wheelgebras. 
Moreover, the assignment 
\begin{equation} 
\label{eq:functor-fock}
     \Fock : \Alg \longrightarrow \CAdm   
\end{equation}
sending a nonunitary algebra $A$ to the Fock wheelgebra $\Fock(A)$, and a morphism $\varphi : A \rightarrow A'$ of nonunitary algebras to the morphism $\Fock(\varphi) : \Fock(A) \rightarrow \Fock(A')$ of wheelgebras, is a functor.
\label{lem:functor-fock}
\end{lemma}
%%%%%%%
\begin{proof} 
We first note that \eqref{eq:prod-fock-s-e-morph} together with \eqref{eq:t-i-j} tell us that $\Fock(\varphi)$ is a morphism of wheelspaces. 
Since $\mathcal{F}(\varphi,\varphi_{\cyc})$ is a morphism of algebras in $\DMod$, $\Fock(\varphi)$ is a morphism of commutative wheelgebras. 
On the other hand, the expression \eqref{eq:prod-fock-s-e-morph} immediately implies that $\Fock(\id_{A}) = \id_{\Fock(A)}$ and $\Fock(\varphi' \circ \varphi) = \Fock(\varphi')  \circ \Fock(\varphi)$, for morphisms of algebras $\varphi : A \rightarrow A'$ and $\varphi : A' \rightarrow A''$, so \eqref{eq:functor-fock} is a functor.
\end{proof}

The Fock wheelgebra $\Fock(A)$ of a nonunitary algebra $A$ can be uniquely determined by the following universal property. 
%%%%%%%
\begin{theorem} 
\label{theorem:fock-adjoint}
The functor $\Fock : \Alg \rightarrow \CAdm$ in \eqref{eq:functor-fock} is left adjoint to the composition of the inclusion functor $\IA : \CAdm \rightarrow \Adm$ and the functor $\alg : \Adm \rightarrow \Alg$ introduced in \eqref{eq:wheel-alg}. 
Consequently, the functor $\IA \circ \Fock : \Alg \rightarrow \Adm$ is left adjoint to the functor $\alg : \Adm \rightarrow \Alg$. 
\end{theorem}
%%%%%%%
\begin{proof}
The second part of the theorem immediately follows from the first, as it is a particular case of the following elementary result on adjoint pairs: given two categories $\mathcal{C}$ and $\mathcal{D}$, a full subcategory $\mathcal{C}'$ 
of $\mathcal{C}$ with inclusion functor $\operatorname{inc} : \mathcal{C}' \rightarrow \mathcal{C}$, and functors $L : \mathcal{C} \rightarrow \mathcal{D}$ and $R : \mathcal{D} \rightarrow \mathcal{C}'$, if $L$ is left adjoint to $\operatorname{inc} \circ R$, then $L \circ \operatorname{inc}$ is left adjoint to $R$. 

We will thus prove the first part of the theorem. 
Let $(A, \mu)$ be a nonunitary algebra and $\Fock(A)$ the associated Fock wheelgebra. 
Given any commutative wheelgebra $S'$ with product $\mu' : S' \otimes_{\gwh} S' \rightarrow S'$ and contractions $\{ {}^{n}_{j}t_{i}' \}_{n \in \NN, i,j \in \llbracket 1,n \rrbracket }$, it is clear that the map
\begin{equation}
\label{eq:univ-fock} 
\Hom_{\WAlg} \big(\Fock(A), S'\big) \longrightarrow \Hom_{\Alg} \big( A , \alg(S')\big)
\end{equation}
sending a morphism $f = (f(n))_{n \in \NN_{0}}$ of wheelgebras from $\Fock(A)$ to $S'$ to the map $f(1)|_{A \otimes \Sym^{0}(A_{\cyc})}$ is natural in $S'$. 
It suffices to show that it is a bijection. 

We will first show \eqref{eq:univ-fock} is an injection. 
Assume that $f = (f(n))_{n \in \NN_{0}}$ and $f' = (f'(n))_{n \in \NN_{0}}$ are two morphisms of wheelgebras from $\Fock(A)$ to $S'$ satisfying that 
\[     f(1)|_{A \otimes \Sym^{0}(A_{\cyc})} = f'(1)|_{A \otimes \Sym^{0}(A_{\cyc})}.     \] 
Since $f$ and $f'$ are morphisms of wheel\-spaces, 
\begin{equation} 
\begin{split}
f(0)\big(\pi(a)\big) &= (f(0) \circ {}^{1}_{1}t_{1})(a \otimes 1_{\Sym(A_{\cyc})}) = \big({}^{1}_{1}t'_{1} \circ f(1)\big)(a \otimes 1_{\Sym(A_{\cyc})}) 
\\
&= \big({}^{1}_{1}t'_{1} \circ f'(1)\big)(a \otimes 1_{\Sym(A_{\cyc})}) = (f'(0) \circ {}^{1}_{1}t_{1})(a \otimes 1_{\Sym(A_{\cyc})}) 
\\
&= f'(0)\big(\pi(a)\big) 
\end{split}
\end{equation} 
for all $a \in A$, which implies that $f(0)|_{A_{\cyc}} = f'(0)|_{A_{\cyc}}$. 
In consequence, the morphisms $\Fo(f), \Fo(f') : \Fo(\Fock(A)) \rightarrow \Fo(S')$ of algebras in the category $\DMod$ satisfy that their restriction to $\Gen(A) = \I_{\SG^{\env},1}(A) \oplus \I_{\SG^{\env},A}(A_{\cyc})$ coincide, which implies that $\Fo(f) = \Fo(f')$, using that  $\Fo(\Fock(A)) \cong \Sym_{\SG^\env}(\Gen(A))$ and the universal property \eqref{eq:sym-alg}. 

Let us finally prove the surjectivity of \eqref{eq:univ-fock}. 
Let $g : A \rightarrow \alg(S')$ be a morphism of algebras. 
Let $h : A_{\cyc} \rightarrow S'(0)$ be the unique map satisfying that $\pi \circ h = {}^{1}_{1}t'_{1} \circ g$, 
which exists due to Fact \ref{fact:alg-vanishes-comm} and is unique by the surjectivity of $\pi$. 
By the universal property \eqref{eq:sym-alg} of $\Fo(\Fock(A)) \cong \Sym_{\SG^\env}(\Gen(A))$, we see that there exists a unique morphism 
$f : \Fo(\Fock(A)) \rightarrow \Fo(S')$ of algebras in the category $\DMod$ such that $f(1)|_{A \otimes \Sym^{0}(A_{\cyc})} = g$ and 
$f(0)|_{A_{\cyc}} = h$. 
Using Fact \ref{fact:gwh-algebra-ex}, it suffices to prove that $f$ is a morphism of wheelspaces, \textit{i.e.} it satisfies \eqref{eq:morph-wh}. 
To do that, we will use Lemma \ref{lemma:contractions-fock}: it thus suffices to prove that \eqref{eq:eln} holds. 
We will prove it by induction on the index $n$. 
The first identity of \eqref{eq:eln-bis} is immediately satisfied, by definition of $f$. 
We will now prove that the second identity of \eqref{eq:eln-bis} holds for all $n \geq 2$. 
This follows from 
\begin{equation}
\label{eq:eln-bis-i-j-cases}
\begin{split}
&\big({}^{n}_{2}t'_{1} \circ f(n)\big)(\overline{a_{1} | \dots | a_{n}}) 
= \big({}^{n}_{2}t'_{1} \circ f(n)\big)\Big(\mu_{2,n-2}\big(\mu_{1,1}(a_{1}, a_{2}), \overline{a_{3} | \dots | a_{n}}\big) \Big) 
\\
&= {}^{n}_{2}t'_{1} \Big(\mu'_{2,n-2}\Big(\mu'_{1,1}\big(f(1)(a_{1}), f(1)(a_{2})\big), f(n-2)(\overline{a_{3} | \dots | a_{n}})\Big) \Big) 
\\
&= \mu'_{1,n-2}\Big(({}^{2}_{2}t'_{1} \circ \mu'_{1,1})\big(f(1)(a_{1}), f(1)(a_{2})\big), f(n-2)(\overline{a_{3} | \dots | a_{n}})\Big) \Big) 
\\
&= \mu'_{1,n-2}\Big((f(1) \circ {}^{2}_{2}t_{1} \circ \mu_{1,1})(a_{1}, a_{2}), f(n-2)(\overline{a_{3} | \dots | a_{n}})\Big) \Big) 
\\
&= (f(n-1) \circ \mu'_{1,n-2})\Big(({}^{2}_{2}t_{1} \circ \mu_{1,1})(a_{1}, a_{2}), \overline{a_{3} | \dots | a_{n}}\Big) \Big) 
\\
&= (f(n-1) \circ {}^{n}_{2}t_{1} \circ \mu'_{2,n-2})\Big(\mu_{1,1}(a_{1}, a_{2}), \overline{a_{3} | \dots | a_{n}}\Big) \Big) 
\\
&= \big(f(n-1) \circ {}^{n}_{2}t_{1}\big)(\overline{a_{1} | \dots | a_{n}}),
\end{split}
\end{equation}
where we used that $f$ is a morphism of algebras in $\DMod$ in the 
second and fifth equalities, that $\mu'$ and $\mu$ are morphisms of generalized wheelspaces in the 
third and sixth equalities, respectively, and the fact that $g = f(1)|_{A \otimes \Sym^{0}(A_{\cyc})} : \Fock(A)(1) \rightarrow \alg(S')$ is a morphism of algebras in the fourth equality. 
The theorem is thus proved. 
\end{proof}

We have the following direct consequence of Theorem \ref{theorem:fock-adjoint}. 
%%%%%%%
\begin{corollary}
\label{corollary:fock-adjoint}
Let $A$ be a $B$-ring via $\iota : B \rightarrow A$, and let $\wC$ be a commutative wheelgebra that is also a $\Fock(B)$-algebra. 
Then, $\alg(\wC)$ has a natural structure of a $B$-ring by means of $B \rightarrow \alg(\wC)$ given by \eqref{eq:univ-fock}, and the map \eqref{eq:univ-fock} gives a natural isomorphism 
\begin{equation} 
\label{eq:univ-fock-2}
\Hom_{\Alg({}_{\Fock(B)}\GWhMod)}\big(\Fock(A),\wC\big) \longrightarrow \Hom_{{}_{B}\Ring}\big(A,\alg(\wC)\big), 
\end{equation} 
where $\Fock(A)$ is a $\Fock(B)$-algebra by means of $\Fock(\iota)$. 

Furthermore, if $A$ and $\wC$ are also assumed to be endowed with gradings such that $B = A_{0}$ and $\wC_{0} = \Fock(B)$, then \eqref{eq:univ-fock-2} further restricts to an isomorphism for the spaces of morphisms that also preserve the degree. 
Moreover, under the previous assumptions \eqref{eq:univ-fock-2} also restricts to an isomorphism 
\begin{equation} 
\label{eq:univ-fock-3}
\Hom_{\aAlg({}_{\Fock(B)}\GWhMod)}^{\gr}\big(\Fock(A),\wC\big) \longrightarrow \Hom_{{}_{B}\aRing}^{\gr}\big(A,\alg(\wC)\big),
\end{equation} 
where $\Fock(A)$ and $\wC$ are naturally augmented $\Fock(B)$-algebras, and $A$ and $\alg(\wC)$ are augmented $B$-rings following Example \ref{example:graded-aug}.
\end{corollary}
%%%%%%%

%%%%%%%
\begin{remark}
The isomorphism \eqref{eq:univ-fock-3} is no longer true if the superindices $\gr$ are removed. 
\end{remark}
%%%%%%%

Let $B$ be an associative algebra. 
We define a functor 
\begin{equation}
    \WW\colon {}_B \Mod_B\longrightarrow {}_{\Fock(B)}\AMod
    \label{index:functor-WW}
\end{equation}
as follows. 
Given a $B$-bimodule $M$, consider the graded tensor algebra $A=T_BM=\bigoplus_{n\geq 0}M^{\otimes_B n}$, where $M$ is assumed to be concentrated in internal degree $1$ and $B$ in degree $0$. 
As explained in the paragraph before Lemma \ref{lemma:contractions-fock}, the grading of $A$ induces a grading on $\Fock(A)$, explicitly indicated in \eqref{eq:f-A-n}. 
We then set $\WW (M)$ to be the 
$\Fock(B)$-wheelmodule $\Fock(T_{B}M)_1$. 
Similarly, given a morphism $f : M \rightarrow N$ of $B$-bimodules, $\WW(f) : \WW(M) \rightarrow \WW(N)$ is the morphism of $\Fock(B)$-wheelmodules $\Fock(T_{B}f)|_{\Fock(T_{B}M)_1} : \Fock(T_{B}M)_1 \rightarrow \Fock(T_{B}N)_1$. 
It is easy to verify that this defines indeed a functor. 
Also, it is straightforward to see that if $A'=\ZSE(B,M)$, as introduced in Example \ref{example:zse}, then we also have $\WW(M) = \Fock(A')_1$. 
To reduce the notation we will usually write $\WW\! M$ instead of $\WW(M)$. 

Recall the functor $\bmod$ defined in \eqref{eq:wheel-mod}. 
Then, we also have the following nice consequence, which gives a universal property describing the functor $\WW$.  
%%%%%%%
\begin{lemma} 
\label{lemma:fock-adjoint}
Let $B$ be a nonunitary algebra, $M$ a $B$-bimodule, and $\wN$ be a wheelmodule over $\Fock(B)$. 
Then, the map
\begin{equation} 
\label{eq:univ-fock-4}
\Hom_{{}_{\Fock(B)}\AMod}\big(\WW\!  M,\wN\big) \longrightarrow \Hom_{{}_{B}\Mod_{B}}\big(M,\bmod(\wN)\big)
\end{equation} 
sending $f = (f(n))_{n \in \NN_{0}}$ to $f(1)|_{M \otimes \Sym^{0}(A_{\cyc})}$ is a natural isomorphism in $\wN$. 
\end{lemma}
%%%%%%%
\begin{proof} 
Recall that $\WW\! M = \Fock(A)_{1}$, for $A = T_{B}M$, where $M$ is assumed to be concentrated in internal degree $1$ and $B$ in degree $0$. 
The result follows from the chain of natural isomorphisms 
\begin{equation} 
\begin{split}
\Hom_{{}_{\Fock(B)}\AMod}\big(\Fock(A)_{1},\wN\big) &\cong \Hom_{\aAlg({}_{\Fock(B)}\GWhMod)}^{\gr}\big(\Fock(A), \ZSE^{\gr}(\Fock(B),\wN)\big) 
\\ 
&\cong \Hom_{{}_{B}\aRing}^{\gr}\Big(A,\alg\big(\ZSE^{\gr}(\Fock(B),\wN)\big)\Big)
\\
&\cong \Hom_{{}_{B}\aRing}^{\gr}\bigg(A,\ZSE^{\gr}\Big(\alg\big(\Fock(B)\big),\bmod(\wN)\Big)\bigg) 
\\ 
&\cong \Hom_{{}_{B}\Mod_{B}}^{\gr}\Big(M,\alg\big(\Fock(B)\big)\oplus \bmod(\wN)\Big) 
\\
&\cong \Hom_{{}_{B}\Mod_{B}}\big(M,\bmod(\wN)\big), 
\end{split}
\end{equation}
where the first isomorphism follows from Fact \ref{fact:mor-zse} and a simple degree argument, the second isomorphism follows from \eqref{eq:univ-fock-3} in Corollary \ref{corollary:fock-adjoint}, 
the third isomorphism is given by Fact \ref{fact:zse-adm}, the fourth isomorphism follows from the universal property of the tensor algebra and the fifth follows from a direct degree argument. 
\end{proof}

Combining now \eqref{eq:univ-fock-2} in Corollary \ref{corollary:fock-adjoint} and Lemma \ref{lemma:fock-adjoint}, we get the following important result. 
%%%%%%%
\begin{corollary}
\label{corollary:fock-adjoint-2} 
Let $B$ be a nonunitary algebra, $M$ a $B$-bimodule, and $\wC$ a commutative wheelgebra that is further assumed to be a $\Fock(B)$-algebra. 
Then, the restriction map
\begin{equation} 
\label{eq:univ-fock-5}
\Hom_{\Alg({}_{\Fock(B)}\GWhMod)}\big(\Fock(A),\wC\big) \longrightarrow \Hom_{{}_{\Fock(B)}\AMod}\big(\Fock(A)_{1},\wC\big),
\end{equation} 
is a natural isomorphism in $\wC$,  
where $A = T_{B}M$ and $\wC$ is considered as a $\Fock(B)$-wheelmodule in the target space.
\end{corollary}
%%%%%%% 
\begin{proof} 
This follows from the chain of natural isomorphisms 
\begin{equation} 
\begin{split}
\Hom_{{}_{\Fock(B)}\AMod}\big(\Fock(A)_{1},\wC\big) &\cong \Hom_{{}_{B}\Mod_{B}}\big(M,\bmod(\wC)\big)
\cong \Hom_{{}_{B}\Ring}\big(A,\alg(\wC)\big)
\\
&\cong \Hom_{\Alg({}_{\Fock(B)}\GWhMod)}\big(\Fock(A),\wC\big)
\end{split}
\end{equation}
where the first isomorphism follows from Lemma \ref{lemma:fock-adjoint}, the second isomorphism follows from the universal property of the tensor algebra and the last isomorphism follows from \eqref{eq:univ-fock-2} in Corollary \ref{corollary:fock-adjoint}.
\end{proof}

%%%%%%%
\begin{remark}
\label{rem:symmetric-Fock-alg}
Following the terminology of Example \ref{example:sym-alg}, Corollary \ref{corollary:fock-adjoint-2} tells us that 
$\Fock(A) = \Sym_{\mathcalboondox{CA}}(\Fock(A)_{1})$, 
where $\mathcalboondox{CA}$ is the category of commutative $\Fock(B)$-algebras in $\WAlg$. 
This statement seems to have been vaguely mentioned (and without proof) in \cite{MR2734329} (see the second paragraph of their Subsection 3.6). 
\end{remark}
%%%%%%% 

An interesting property of the wheelmodule $\WW\! M$ is the following simple result. 
%%%%%%%
\begin{fact}
\label{fact:generators-wheel-w}
Let $B$ be a nonunitary algebra, and $M$ a $B$-bimodule. 
Then, the $\Fock(B)$-subwheel\-mod\-ule generated by the diagonal $\SG$-subbimodule $\I_{\SG^{\env},1}(M)$ introduced in \eqref{eq:vect-s-se} of the underlying diagonal $\SG$-bi\-mod\-ule of $\WW\! M$ is all $\WW\! M$.    
\end{fact}
%%%%%%%
\begin{proof}
It is clear that the diagonal $\SG$-bimodule $\I_{\SG^{\env},1}(M)$ is a diagonal $\SG$-subbimodule of the underlying diagonal $\SG$-bimodule of $\WW\! M$, since 
\[     
\I_{\SG^{\env},1}(M) \cong M \otimes \kk \subseteq M \otimes \Sym(B)_{\cyc} \subseteq M \otimes \Sym(B)_{\cyc} \oplus B \otimes \Sym(B)_{\cyc} M_{\cyc} 
= \WW\!M(1).     
\]
Note also that any $\Fock(B)$-subwheelsub\-mod\-ule of $\WW\!M$ containing $\I_{\SG^{\env},1}(M) \subseteq \WW\!M(1)$, also contains ${}^{1}_{1}t_{1}(\I_{\SG^{\env},1}(M)) = M_{\cyc} \subseteq \WW\!M(0)$. 
On the other hand, using the left and right action of $\Fock(B)$ on $\WW\!M$, we see that any $\Fock(B)$-subwheelsub\-mod\-ule of $\WW\!M$ containing $\I_{\SG^{\env},1}(M) \subseteq \WW\!M(1)$, also contains $(B^{\otimes i} \otimes M \otimes B^{\otimes (n-1-i)}) \otimes_{\kk \SG_{n}} \kk \SG_{n}^{\env} \otimes \Sym(B_{\cyc})$ for all $n \in \NN$ and $i \in \llbracket 0 , n \rrbracket$, whereas any $\Fock(B)$-subwheelsub\-mod\-ule of $\WW\!M$ containing $M_{\cyc} \subseteq \WW\!M(0)$, also contains $B^{\otimes n} \otimes_{\kk \SG_{n}} \kk \SG_{n}^{\env} \otimes \Sym(B_{\cyc}) M_{\cyc}$ for all $n \in \NN$. 
The result now follows from \eqref{eq:f-A-1}. 
\end{proof}

%%%%%%%%%%%%%%%%%%%%%%%%%%%%%%%%%%%%%%%%%%%%%%%%%%%%%%%%%%%%%%%%%%%%%%
\subsection{Graded version} 

All of the definitions and results in this and the previous sections also hold if we replace the symmetric monoidal category of vector spaces (with the usual tensor product and flip) by the symmetric monoidal category of graded vector spaces (with the usual tensor product and the signed flip following the Koszul sign rule). 
The verification is rather tedious, but straightforward, and is left to the reader. 

%%%%%%%%%%%%%%%%%%%%%%%%%%%%%%%%%%%%%%%%%%%%%%%%%%%%%%%%%%%%%%%%%%%%%%
\section{Fock wheelgebras and noncommutative geometry}
\label{section:fock-geometry}
%%%%%%%%%%%%%%%%%%%%%%%%%%%%%%%%%%%%%%%%%%%%%%%%%%%%%%%%%%%%%%%%%%%%%%

%%%%%%%%%%%%%%%%%%%%%%%%%%%%%%%%%%%%%%%%%%%%%%%%%%%%%%%%%%%%%%%%%%%%%%%
\subsection{Basics on noncommutative geometry. Bisymplectic algebras}
\label{subsec:bisymplectic-algebras}
%%%%%%%%%%%%%%%%%%%%%%%%%%%%%%%%%%%%%%%%%%%%%%%%%%%%%%%%%%%%%%%%%%%%%% 

In this section, we recall the basic definitions and results from noncommutative (simplectic) geometry, initiated in \cite{MR1303029}. 
Let $A$ be an associative (not necessarily commutative) algebra with product $\mu$, which we assume unitary, and $M$ be an $A$-bimodule. 
Following \cite{MR1303029}, \S2 (see also \cite{MR2294224}), we define the $A$-bimodule $\diff A$ of \textbf{\textcolor{myblue}{noncommutative K\"ahler differential $1$-forms}} as 
\begin{equation}
\label{eq:1-forms-Kahler-nc}
\diff A =\operatorname{Ker}\big(\mu\colon A\otimes A\longrightarrow A\big).
\end{equation}
It is endowed with a canonical derivation $\label{index:canonical-deriv} \du{}\colon A \to\diff A$, $a\mapsto a\otimes 1-1\otimes a$, which satisfies the following universal property: 
the canonical morphism of $A$-bimodules  
\begin{equation}
\Hom_{A^{\env}}(\diff A,M) \overset{\cong}{\longrightarrow} \Der(A,M)
\label{eq:prop-universal-CQ}
\end{equation}
sending an $A$-bimodule morphism $f : \diff A \rightarrow M$ 
to the derivation $f \circ \du{}$ is a bijection (see \cite{MR1303029}). 
Note also that \eqref{eq:prop-universal-CQ} is natural in $M$. 
Moreover, following \cite{MR2294224}, \S2.1, we shall denote the image of a derivation $\theta\colon A\to M$ under the inverse mapping to \eqref{eq:prop-universal-CQ} by $i_{\theta} \colon\diff A\to M$, so $i_\theta(a\du{}b)=a\theta (b)$, for all $a, b\in A$. 

Next, we define the \textbf{\textcolor{myblue}{algebra of noncommutative differential forms}} $\diffb A \label{index:algebra-nc-diff-forms}$ as the tensor algebra $T_A(\diffb A)$ of $\diff A$ over $A$; in particular, $\Omega^0_{\nc}A= A$ and $\diffn A= (\diff A)^{\otimes_An}$.
If $\alpha\in \diffn A$ with $\alpha\neq 0$, we will say that $\alpha$ has \textbf{\textcolor{myblue}{degree}} $n$ and denoted by $|\alpha|=n \label{index:form-degree}$; we also assume from now on that the elements under consideration are homogeneous. 
Since $(\diffb A,\du{})$ does not give rise to an interesting cohomology theory, we introduce the \textbf{\textcolor{myblue}{Karoubi--de Rham complex}} 
\begin{equation}
    \DR^\bullet A =\big(\diffb A\big)_{\cyc}=\diffb A/ [\diffb A, \diffb A],
    \label{eq:Karoubi-de-Rham-def}
\end{equation}
where $[\hskip 0.4mm,]$ denotes the linear span of all graded commutators. 
Observe that $ \DR^0 A=A_{\cyc}$ and $\DR^1 A=\operatorname{HH}_0(A,\diff A)$. 
The canonical differential $\du{}\colon \diffb A\to \diff A $ descends to a well-defined map $\label{index:differential-DR} \du_{\DR}\colon \DR^\bullet A\to \DR^{\bullet+1} A$, making $(\DR^\bullet A,\du_{\DR})$ into a differential graded vector space. 
We say that $\xi\in\DR^nA$ is \textbf{\textcolor{myblue}{closed}} if $\du_{\DR}\xi$ is 0 in $\DR^{n+1}A$.

The \textbf{\textcolor{myblue}{dual}} of an $A$-bimodule $M$ is defined as $\label{index:dual} M^\vee = \Hom_{A^{\env}}(M,(A\otimes A)_{\out})$, where $(A\otimes A)_{\out} \label{index:outer-str}$ denotes the usual \emph{outer} $A$-bimodule structure, which we will systematically use unless otherwise is stated. 
Note that $M^\vee$ is an $A$-bimodule with the surviving \emph{inner} $A$-bimodule structure, that is, $(afb)m = (f'm)b\otimes a(f''m))$, where $a,b\in A$, $m\in M$ and we used Sweedler's notation to drop the sum symbol (\textit{i.e.} $f'\otimes f''$ is a shorthand for $\sum_i f'_i\otimes f''_i$).
Taking $M=(A\otimes A)_{\out}$ in \eqref{eq:prop-universal-CQ}, we get 
\[
(\diff A)^\vee =\Hom_{A^{\env}}(\diff A, (A\otimes A)_{\out}) 
\cong \Der(A,(A\otimes A)_{\out}).
\]
Consequently, we can define the $A$-bimodule of \textbf{\textcolor{myblue}{double derivations}} of $A$ as $\label{index: double-der}\DDer A = \Der(A,(A\otimes A)_{\out})$, whose $A$-bimodule structure is induced by the inner bimodule structure of $A\otimes A$, \textit{i.e.} $(b_1\Theta b_2)(a)=\Theta'(a)b_2\otimes b_1 \Theta''(a)$, for all $a,b_1,b_2\in A$ and $\Theta\in\DDer A$. 

Following \cite{MR2294224} (based on \cite{MR1303029}), we say that an associative algebra $A$ is \textbf{\textcolor{myblue}{smooth}} (over $\kk$) if $\diff A$ is a finitely generated projective $A$-bimodule. 
In fact, we can prove that if $A$ is a smooth algebra, then $\DDer A$ is a finitely generated projective $A$-bimodule.
We will also use in Section \ref{sec:bisymplectic-and-wheelgebras} that if $A$ is a smooth algebra, there exists a linear isomorphism $\big(\DDer A\big)_{\cycm} \cong \Der(A)$ essentially given by the multiplication of $A$ (see \cite{MR2294224}, Prop. 2.3.2). 

In particular, given $\Theta\in\DDer A$, note that if $\alpha\in\diff A$, $i_\Theta\alpha=i'_\Theta\alpha\otimes i''_\Theta\alpha\in A\otimes A$. 
Since $\diffb A=T_A(\diff A)$, we can use the universal property of tensor algebras to extend via the Leibniz rule the map $i_\Theta$ in \eqref{eq:prop-universal-CQ} to $\diffb A$ to obtain the \textbf{\textcolor{myblue}{contraction}} (with respect to a double derivation)
\begin{equation}
i_\Theta\colon \Omega^\ell_{\nc}A\longrightarrow \bigoplus_{i+j=\ell-1}\Big(\Omega^i_{\nc}A\otimes \Omega^j_{\nc}A\Big).
\label{eq:contraction-extended-double-deriv-def}
\end{equation}
Note that this map has degree $-1$ and it is a double derivation of $\diffb A$, since the identity $i_\Theta(\alpha\beta)=(i_\Theta\alpha)\beta+(-1)^{|\alpha|}\alpha (i_\Theta\beta)$ holds for all $\alpha,\beta\in\diffb A$, as shown in \cite{MR2294224}, Lemma 2.6.3.

Given $\alpha=\alpha_1\otimes\alpha_2\in (\diffb A)^{\otimes 2}$, define $\label{index:circulito}{}^{\circ}\alpha =(-1)^{|\alpha_1||\alpha_2|}\alpha_2\alpha_1\in\diffb A$. 
More generally, given a linear map $\phi\colon \diffb A\to (\diffb A)^{\otimes 2}$, define ${}^\circ \phi\colon \diffb A\to\diffb A$ as ${}^\circ\phi(\beta) = {}^\circ(\phi (\beta))$, for all $\beta \in \diffb A$.
In particular, applying this definition to the contraction \eqref{eq:contraction-extended-double-deriv-def}, we obtain the \textbf{\textcolor{myblue}{reduced contraction}} by
\begin{equation}
\label{eq:reduced-contraction-double-der-def}
\iota_\Theta 
= {}^\circ( i_\Theta)\colon\Omega^\ell_{\nc}A\longrightarrow\Omega^{\ell-1}_{\nc}A.
\end{equation}

The following result collects interesting properties of the reduced contraction.
%%%%%%%
\begin{fact}
Let $\Theta,\Delta\in \DDer A$. Then
\begin{enumerate}[label=(\arabic*)]
    \item \label{item:RC1}
    $i_\Theta\iota_{\Delta}+\sigma_{(12)}i_{\Delta}\iota_{\Theta}=0$; 
      \item\label{item:RC2}
      for any $\beta\in\Omega^\ell_{\nc} A$, the element $\iota_\Theta\beta\in\Omega^{\ell-1}_{\nc}A$ only depends on the image of $\beta$ in $\DR^\ell A$, so the assignment $\beta\mapsto\iota_\Theta\beta$ descends to a well-defined map $\iota_\Theta\colon\DR^\ell A\to\Omega^{\ell-1}_{\nc}A$; 
       \item\label{item:RC3}
       for a fixed $\xi\in \DR^\ell A$, the map $\Theta\mapsto\iota_\Theta\xi$ gives an $A$-bimodule morphism $\iota(\xi)\colon\DDer A \to\Omega^{\ell-1}_{\nc}A$.
\end{enumerate}
\label{fact:hechos-interesantes-DR}
\end{fact}
%%%%%%%%
Note that while \ref{item:RC1} is precisely \cite{MR2425689}, (A.5), \ref{item:RC2} and \ref{item:RC3} follow from \cite{MR2294224}, Lemma 2.8.6.
Hence, taking $n=2$, we can define one of the main objects in noncommutative geometry based on double derivations.
%%%%%%%
\begin{definition}[\cite{MR2294224}, Def. 4.2.5; \cite{MR2425689}, Def. A.3.1]
\label{def:bisymplectic-algebras}
A $2$-form $\omega\in\DR^2\hskip-0.6mm A$ is \textbf{\textcolor{myblue}{bi-nondegenerate}} if the map of $A$-bimodules
\begin{equation}
\iota(\omega)\colon\DDer A \longrightarrow \Omega^1_{\nc}A,\quad \Theta\longmapsto \iota_\Theta\omega
\label{eq:bisymplectic-isom-def}
\end{equation}
is an isomorphism. 
Furthermore, if $\omega$ is also closed, we say that it is \textbf{\textcolor{myblue}{bisymplectic}}. A \textbf{\textcolor{myblue}{bisymplectic algebra}} is a pair $(A,\omega)$, where $A$ is an associative algebra and $\omega$ is a bisymplectic form.
\end{definition}
%%%%%%%

One of the main examples of a bisymplectic algebra is given by the following result. 
%%%%%%%
\begin{example}
\label{thm:Cbeg-Thm-5.1.1}
Assume that $B$ is a smooth algebra. Then, as proved in \cite{MR2294224}, Theorem 5.1.1, the algebra $A=T_B(\DDer B)$ is also smooth and it carries a canonical bisymplectic $2$-form $\omega\in\DR^2 A$. 
In fact, $\omega=\du{}_{\DR}\lambda$, where $\lambda\in\DR^1 A$ is a canonical $1$-form, called the \textbf{\textcolor{myblue}{Liouville $1$-form}}.
\end{example}
%%%%%%%

Finally, given a bi-nondegenerate form $\omega\in\DR^2 A$, and $a\in A$, we define the \textbf{\textcolor{myblue}{Hamiltonian double derivation}} $H_a\in\DDer A$ corresponding to $a$ by
\begin{equation}
\label{eq:Hamiltonian-double-deriv-def}
\iota_{H_a}\omega=\du{}a.
\end{equation}

%%%%%%%%%%%%%%%%%%%%%%%%%%%%%%%%%%%%%%%%%%%%%%%%%%%%%%
\subsection{Double Poisson algebras and Poisson wheelgebras}
\label{sec:double-Poisson-and-wheelgebras}

The reader might check that the following definition coincides with the one introduced in \cite{MR2425689}, Section 2.7. 
%%%%%%%
\begin{definition}
\label{definition:po}
Let $(A,\mu)$ be an algebra. 
A \textbf{\textcolor{myblue}{double bracket}} on $A$ is a linear map
\[     \lr{\hskip 0.4mm,}_{A} : A \otimes A \longrightarrow A \otimes A     \]
satisfying that 
\begin{enumerate}[label={(DB.\arabic*)}]
\setcounter{enumi}{0} 
\item\label{item:dpa1} $- \lr{\hskip 0.4mm,}_{A} \circ \tau_{A,A} = \tau_{A,A} \circ \lr{\hskip 0.4mm,}_{A}$;
\item\label{item:dpa2} for any $a \in A$, the map $\AD(a) : A \rightarrow A \otimes A$ given by $b \mapsto \lr{a,b}_{A}$ is a \emph{double derivation} of $A$, \textit{i.e.}
\[   
\AD(a) \circ \mu_{A} = (\mathrm{id}_{A} \otimes \mu_{A}) \circ \big(\AD(a) \otimes \mathrm{id}_{A}\big) + (\mu_{A} \otimes \mathrm{id}_{A}) \circ \big(\mathrm{id}_{A}  \otimes \AD(a)\big).   
\]
\end{enumerate}
We will typically write 
\[     \lr{a,b}_A=\lr{a,b}'\otimes \lr{a,b}''     \]
for $a, b \in A$. 
Moreover, we say that a double bracket is \textbf{\textcolor{myblue}{Poisson}} if it further satisfies that 
\begin{enumerate}[label={(DB.\arabic*)}]
\setcounter{enumi}{2} 
\item\label{item:dpa3} $\sum_{\sigma \in C_{3}} \tau_{A,3}(\sigma) \circ \lr{\hskip 0.4mm, \hskip 0.4mm,}_{A,L} \circ \tau_{A,3}(\sigma^{-1}) = 0$,
\end{enumerate}
where $C_{3} \subseteq \SG_{3}$ is the subgroup of cyclic permutations and $\lr{\hskip 0.4mm, \hskip 0.4mm,}_{A,L} : A^{\otimes 3} \rightarrow A^{\otimes 3}$ is the map 
$(\lr{\hskip 0.4mm,}_{A} \otimes \mathrm{id}_{A}) \circ (\mathrm{id}_{A} \otimes \lr{\hskip 0.4mm,}_{A})$. 
\end{definition}
%%%%%%%
The identity in \ref{item:dpa3} is known in the literature as the \textbf{\textcolor{myblue}{double Jacobi identity}}.
Also, for future use, we denote the left-hand side of \ref{item:dpa3} as $\label{index:double-Jacobiator}\Jac_{\lr{\hskip 0.4mm,}_A}(\hskip 0.4mm, \hskip 0.4mm,)$, called the \textbf{\textcolor{myblue}{double Jacobiator}}.

\begin{remark}
Given a bi-nondegenerate form $\omega\in\DR^2 A$, using the Hamiltonian double derivation defined in \eqref{eq:Hamiltonian-double-deriv-def}, Van den Bergh \cite{MR2425689}, \S A.3, defined an operation $\lr{a,b}_\omega =H_a(b)$. 
In fact, using \eqref{eq:Hamiltonian-double-deriv-def} and the fact that $H_a(b)=i_{H_a}(\du{} b)$, we can write
\begin{equation}
\lr{a,b}_\omega=i_{H_a}\iota_{H_b}\omega.
\label{eq:corchete-def-por-bisympl.1}
\end{equation}
It is easy to see that \eqref{eq:corchete-def-por-bisympl.1} is a double bracket (the skewsymmetry follows from Fact \ref{fact:hechos-interesantes-DR}\ref{item:RC1}). Furthermore, he was able to prove in \cite{MR2425689}, Prop. A.3.3, that if in addition $\omega$ is closed (that is, $\omega$ is bisymplectic), then $\lr{\hskip 0.4mm,}$ further satisfies \ref{item:dpa3}. Thus, \eqref{eq:corchete-def-por-bisympl.1} becomes a double Poisson bracket.
\label{erem:corchete-def-por-bisympl.2}
\end{remark}

Given a double bracket $\lr{\hskip 0.4mm,}_A$ on $A$, if we compose it with the natural multiplication map $\mu\colon A\otimes A\to A$, we obtain the \textbf{\textcolor{myblue}{associated bracket}}
\begin{equation*}
\label{associated-to-double-bracket}
\{\hskip 0.4mm,\}_{A} \colon A\otimes A \longrightarrow A     \end{equation*}
given by 
\[     \{a,b\}_{A} =\mu\circ \lr{a,b}_A=\lr{a,b}'\lr{a,b}'',
\]
for $a, b \in A$. 
It is easy to see that $\{\hskip 0.4mm,\}_{A}$ satisfies the Leibniz rule on its second argument, \textit{i.e.} $\{a,bc\}_{A}=b\{a,c\}_{A}+\{a,b\}_{A}c$ for $a,b,c \in A$. 
However, this identity does not hold in the first argument because $\lr{\hskip 0.4mm,a}_A\in\Der(A, (A\otimes A)_{\operatorname{inn}})$. 
Thus, the map $\{a,\hskip 0.4mm\}_{A}\colon A\to A$ is a derivation (in the usual sense). Also note that \ref{item:dpa1} in Definition \ref{definition:po} implies that $\{a,b\}_{A}=-\{b,a\}_{A}$ $\operatorname{mod} \,[A,A]$.

%%%%%%%
\begin{lemma}[\cite{MR2425689}, Prop. 2.4.2]
Let $\lr{\hskip 0.4mm,}_A$ be a double Poisson bracket on $A$. Then the following identity holds in $A\otimes A$:
\[
\{a,\lr{b,c}_A\}_{A}=\lr{\{a,b\}_{A},c}_A+\lr{b,\{a,c\}_{A}}_A,
\]
where $\{a,-\}_{A}$ acts on tensors $b_1\otimes b_2\in A\otimes A$ by $\{a,b_1\otimes b_2\}_{A}=\{a,b_1\}_{A}\otimes b_2+b_1\otimes\{ a,b_2\}_{A}$.
\label{fact:resultado-tecnico-prop-VdB-associated-brackets}
\end{lemma} 
%%%%%%%

As usual, let $\label{index:canonical-projection}\pi\colon A\to A_{\cyc}$ be the canonical projection. 
The following result (which corresponds to \cite{MR3366859}, Prop. 2.5) collects the main properties of $\{\hskip 0.4mm,\}_{A}$ that we will need in this article.
%%%%%%%
\begin{lemma}[\cite{MR2425689}, Lemma 2.4.1, Cor. 2.4.4 and 2.4.6]
\label{fact:VdB}
\hfill
\begin{enumerate}[label=(\roman*)]
\item
If $\lr{\hskip 0.4mm,}_A$ is a double bracket on $A$, then the associated bracket $\{ \hskip 0.4mm,\}_{A}$ induces well-defined maps (that we denote by the same symbol):
\begin{subequations}\label{eq:assoc-bracket-collection}
\begin{align}
\{\hskip 0.4mm,\}_{A}\colon A_{\cyc} \otimes A  &\longrightarrow A, 
\\
  \pi(a)\otimes b &\longmapsto \{\pi(a),b\}_A =\{a,b\}_A, \label{eq:assoc-bracket-collection.a}
\\
\{\hskip 0.4mm,\}_{A}\colon A_{\cyc} \otimes A_{\cyc}  &\longrightarrow A_{\cyc}, 
\\
 \pi(a) \otimes  \pi(b) &\longmapsto \{\pi(a),\pi(b)\}_A  = \pi\big(\{a,b\}_A\big). \label{eq:assoc-bracket-collection.b}
\end{align}
\end{subequations}
\item
If, in addition, $\lr{ \hskip 0.4mm, }_A$ is Poisson, then $A_{\cyc}$ is a Lie algebra with the bracket defined by \eqref{eq:assoc-bracket-collection.b}, and \eqref{eq:assoc-bracket-collection.a} defines a representation of this Lie algebra by derivations of $A$.
\end{enumerate}
\end{lemma}
%%%%%%%%%%%

A crucial feature of the theory is that any double bracket on an algebra $A$ induces a bracket on $\Fock(A)$. 
Recall that $\Fo(\Fock(A)) = \Sym_{\SG^\env}\big(\Gen(A)\big)$, where the diagonal $\SG$-bimodule $\Gen(A)$ was introduced in \eqref{eq:gen-s}. Consider the map 
\begin{equation}
\label{eq:beta-rest}
    \beta : \Gen(A)^{\otimes_{\SG^{\env}} 2} \longrightarrow \Sym_{\SG^\env}\big(\Gen(A)\big)
\end{equation}
given as follows. 
Its restriction 
\begin{equation}
    \beta|_{\I_{\SG^{\env},1}(A)^{\otimes_{\SG^{\env}} 2}} : \I_{\SG^{\env},1}(A)^{\otimes_{\SG^{\env}} 2} \longrightarrow \Sym_{\SG^\env}\big(\Gen(A)\big)
\end{equation}
is given by 
\begin{equation}
(a \otimes b) \otimes (\sigma \otimes \tau)\longmapsto  \lr{ a , b}_{A} \otimes \big((1 \ 2) \cdot \sigma \otimes \tau\big),
\label{eq:wheeled-double-poisson-bracket-def}
\end{equation}
for all $a, b \in A$ and $\sigma, \tau \in \SG_{2}$, whereas the restriction 
\begin{equation}
    \beta' = \beta|_{\I_{\SG^{\env},0}(A_{\cyc}) \otimes_{\SG^{\env}} \I_{\SG^{\env},1}(A)} : \I_{\SG^{\env},0}(A_{\cyc}) \otimes_{\SG^{\env}} \I_{\SG^{\env},1}(A) \longrightarrow \Sym_{\SG^\env}\big(\Gen(A)\big)
\end{equation}
is given by \eqref{eq:assoc-bracket-collection.a}. Also, we have 
\begin{equation}     \beta|_{\I_{\SG^{\env},1}(A) \otimes_{\SG^{\env}} \I_{\SG^{\env},0}(A_{\cyc})} 
= - \tau^{\SG^{\env}}\big(\I_{\SG^{\env},0}(A_{\cyc}),\I_{\SG^{\env},1}(A)\big) \circ \beta' \circ \tau^{\SG^{\env}}\big(\I_{\SG^{\env},1}(A),\I_{\SG^{\env},0}(A_{\cyc})\big),
\end{equation}
and 
\begin{equation}
    \beta|_{\I_{\SG^{\env},0}(A_{\cyc})^{\otimes_{\SG^{\env}} 2}} : \I_{\SG^{\env},0}(A_{\cyc})^{\otimes_{\SG^{\env}} 2} \longrightarrow \Sym_{\SG^\env}\big(\Gen(A)\big)
\end{equation}
is given by \eqref{eq:assoc-bracket-collection.b}. 
Using now Example \ref{example:sym-poisson}, we then conclude that there exists a morphism
\begin{equation}
\label{eq:bracket-from-double}
\{ \hskip 0.6mm , \} : \Fo\big(\Fock(A)\big) \otimes_{\SG^{\env}} \Fo(\Fock(A)\big) \longrightarrow \Fo\big(\Fock(A)\big)
\end{equation}
in $\DMod$ that is a bracket and such that its restriction to $\Gen(A)^{\otimes_{\SG^{\env}} 2}$ is $\beta$ as given in \eqref{eq:beta-rest}. 

We now present the following result, which was stated without proof in \cite{MR2734329}, Def. 3.5.17. 
%%%%%%%
\begin{proposition}
\label{prop:wheel-Poisson-double}
Let $\lr{\hskip 0.4mm,}_{A} : A \otimes A \rightarrow A \otimes A$ be a double Poisson bracket on an algebra $A$. 
Consider the bracket $\{ \hskip 0.6mm , \}$ defined on $\Fo(\Fock(A))$ 
given by \eqref{eq:bracket-from-double}. 
Then, $\{ \hskip 0.6mm , \}$ is indeed a morphism of generalized wheelspaces and gives $\Fock(A)$ the structure of a Poisson wheelgebra.
\end{proposition}
%%%%%%%%
\begin{proof} 
We will first check that \eqref{eq:bracket-from-double} satisfies the Jacobi identity, making $\Fo(\Fock(A))$ a Poisson algebra in the symmetric monoidal category $\DMod$. 
To prove this, it suffices to show that the Jacobi identity holds when its three arguments are in $\Gen(A)$, since $\Fo(\Fock(A)) = \Sym_{\SG^\env}(\Gen(A))$. 
If all entries of the Jacobi identity of $\Fo(\Fock(A))$ are in $\I_{\SG^{\env},0}(A_{\cyc})$, then this is precisely the Jacobi identity for the Lie algebra 
$A_{\cyc}$ with bracket \eqref{eq:assoc-bracket-collection.b}, which holds by Lemma \ref{fact:VdB}. 
Similarly, if two entries of the Jacobi identity of $\Fo(\Fock(A))$ are in $\I_{\SG^{\env},0}(A_{\cyc})$ (and thus the other is in $\I_{\SG^{\env},1}(A)$), then this is precisely the statement that the map \eqref{eq:assoc-bracket-collection.a} gives a Lie module structure on $A$ over the Lie algebra $A_{\cyc}$, which also holds by Lemma \ref{fact:VdB}.

Assume now that the three entries of the Jacobi identity for $\Fo(\Fock(A))$ are in $\I_{\SG^{\env},1}(A)$, given by $a,b,c \in A$. 
Consequently, taking $n=m=p=1$ as well as $u = a$, $v = b$ and $w = c$ in \eqref{eq:poisson-se-3}, we want to prove: 
\begin{equation}
          \begin{split}
         0 &= \{ a , \{ b,c \}_{1,1} \}_{1,2} 
         + \block_{1,1,1}(1 \ 2 \ 3) \cdot \{ b, \{  c , a \}_{1,1} \}_{1,2} \cdot \block_{1,1,1}(1 \ 3 \ 2) 
         \\ 
         &\phantom{=}+ \block_{1,1,1}(1 \ 3 \ 2) \cdot \{ c, \{ a, b \}_{1,1} \}_{1,2} \cdot \block_{1,1,1}(1 \ 2 \ 3).
         \end{split}
         \label{eq:double-Poisson-wheeled-Jacobi-aux}
\end{equation}

To do this, we first prove the following identity:
\[
\operatorname{RHS}\eqref{eq:double-Poisson-wheeled-Jacobi-aux}=\Jac_{\lr{\hskip 0.4mm ,}}(a,b,c)\cdot(1\ 2\ 3)- (2\ 3)\cdot \Jac_{\lr{\hskip 0.4mm ,}}(a,c,b)\cdot(1\ 2),
\]
where we simply denote the double Poisson bracket $\lr{\hskip 0.4mm,}_{A}$ by $\lr{\hskip 0.4mm,}$.

We note that we can write the first term of the right member of \eqref{eq:double-Poisson-wheeled-Jacobi-aux} as 
\begin{align*}
\big\{a,\{b,c\}_{1,1}&\big\}_{1,2}
=\Big\{a,\Big(\lr{b,c}\otimes_{\kk\SG_{2}}\big((1\ 2)\mid \id_{\llbracket 1 , 2 \rrbracket}\big)\Big)\Big\}_{1,2}
\\
&=\big\{a,\mu_{1,2}\big(\lr{b,c}',\lr{b,c}''\big)\}_{1,2}\cdot (2\ 3)
\\
&=\Big(\mu_{2,1}\big(\{a,\lr{b,c}'\},\lr{b,c}''\big)
\\
&\quad + \block_{1,1}(1\ 2)\cdot \mu_{1,2}\big(\lr{b,c}',\{a,\lr{b,c}''\}\big)\cdot \block_{1,1}(1\ 2)\Big)\cdot (2\ 3)
\\
&=\Big(\lr{a,\lr{b,c}'}\otimes\lr{b,c}''\otimes_{\kk\SG_{3}}\big( (1\ 2)\mid \id_{\llbracket 1 , 3 \rrbracket}\big)
\\
&\quad - (1\ 2)\cdot \lr{c,b}''\otimes\lr{a,\lr{c,b}'}\otimes_{\kk\SG_{3}}\big( (1\ 2) \mid \id_{\llbracket 1 , 3 \rrbracket}\big)\cdot (1\ 2)\Big)\cdot (2\ 3)
\\
&=\big(\lr{a,\lr{b,c}'}\otimes\lr{b,c}\big)\cdot (1\ 2)\cdot (2 \ 3) 
\\
&\quad - (1\ 2) \cdot \lr{c,b}''\otimes \lr{a,\lr{c,b}'}\cdot (2\ 3)\cdot (1\ 2)\cdot (2\ 3)
\\
&=\big(\lr{a,\lr{b,c}'}\otimes\lr{b,c}\big) \cdot (1\ 2 \ 3)
\\
&\quad - (1\ 2)\cdot (1\ 2\ 3) \cdot\big( \lr{a,\lr{c,b}'}\otimes \lr{c,b}''\big)\cdot (1\ 3\ 2)\cdot (1\ 3)
\\
&=\lr{a,\lr{b,c}}_L\cdot (1\ 2 \ 3) - (2\ 3) \cdot \lr{a,\lr{c,b}}_L \cdot (1\ 2),
\end{align*}
where we used \eqref{eq:wheeled-double-poisson-bracket-def}, the Leibniz identity \eqref{eq:poisson-se-2}, and the skewsymmetry \ref{item:dpa1} of the double Poisson bracket.
Similarly, we can rewrite the second term at the right-hand side of \eqref{eq:double-Poisson-wheeled-Jacobi-aux} as follows:
\begin{align*}
&\block_{1,1,1}(1\ 2 \ 3)\cdot \big\{b,\{c,a\}_{1,1}\big\}_{1,2}\cdot \block_{1,1,1}(1\ 3 \ 2)
\\
&=(1\ 2\ 3)\cdot \big\{b,\lr{c,a}\otimes_{\kk\SG_{2}}\big( (1\ 2) \mid \id_{\llbracket 1 , 2 \rrbracket}\big)\big\}_{1,2}\cdot(1\ 3\ 2)
\\
&= (1\ 2\ 3)\cdot \big\{b,\mu_{1,1}\big(\lr{c,a}',\lr{c,a}''\big)\big\}_{1,2}\cdot(2\ 3)\cdot (1\ 3\ 2)
\\
&=(1\ 2\ 3)\cdot\Big(\mu_{2,1}\big(\{b,\lr{c,a}'\},\lr{c,a}''\big),\lr{c,a}''\Big)
\\
&\quad +(1\ 2)\cdot \mu_{1,2}\big(\lr{c,a}',\{b,\lr{c,a}''\}\big)\cdot(1\ 2)\Big)\cdot (2\ 3)\cdot (1\ 3\ 2)
\\
&=(1\ 2\ 3)\cdot\Big(\lr{b,\lr{c,a}'}\otimes\lr{c,a}''\otimes_{\kk\SG_{3}}\big((1\ 2) \mid \id_{\llbracket 1 , 3 \rrbracket}\big)
\\
&\quad +(1\ 2)\cdot\big(\lr{c,a}'\otimes\lr{b,\lr{c,a}''}\otimes_{\kk\SG_{3}}\big((2\ 3)\mid \id_{\llbracket 1 , 3 \rrbracket}\big)\cdot(1\ 2)\Big)\cdot (2\ 3)\cdot (1\ 3\ 2)
\\
&=(1\ 2\ 3)\cdot \lr{b,\lr{c,a}'}\otimes\lr{c,a}''\cdot (1\ 3\ 2)\cdot (1 \ 2\ 3) 
\\
&\quad - (1\ 2\ 3)\cdot (1\ 2)\cdot\lr{a,c}''\otimes \lr{b,\lr{a,c}'}\cdot (1\ 3)\cdot (1\ 3\ 2)
\\
&=\big(\tau_{(123)}\lr{c,\lr{a,b}}_L\big)\cdot (1\ 2\ 3)
\\
&\quad -  (1\ 2\ 3)\cdot (1\ 2)\cdot (1\ 3\ 2)\cdot\big(\tau_{(132)}\lr{c,\lr{b,a}}_L\big)\cdot (1\ 2\ 3)\cdot (1\ 3)\cdot (1\ 3\ 2)
\\
&=\big(\tau_{(123)}\lr{c,\lr{a,b}}_L\big)\cdot (1\ 2\ 3)
- (2\ 3)\cdot\big(\tau_{(132)}\lr{c,\lr{b,a}}_L\big)\cdot (1\ 2).
\end{align*}
Finally, the third term at the right-hand side of \eqref{eq:poisson-se-3} is given by
\begin{equation}
\begin{split}
&\block_{1,1,1}(1\ 3 \ 2)\cdot \big\{c,\{a,b\}_{1,1}\big\}_{1,2}\cdot \block_{1,1,1}(1\ 2 \ 3)
\\
&=(1\ 3\ 2)\cdot \big\{c,\lr{a,b}\otimes_{\kk\SG_{2}}\big( (1\ 2) \mid \id_{\llbracket 1 , 2 \rrbracket}\big)\big\}_{1,2}\cdot(1\ 2\ 3)
\\
&= (1\ 3\ 2)\cdot \big\{c,\mu_{1,1}\big(\lr{a,b}',\lr{a,b}''\big)\big\}_{1,2}\cdot(2\ 3)\cdot (1\ 2\ 3)
\\
&=(1\ 3 \ 2)\cdot\Big(\mu_{2,1}\big(\{c,\lr{a,b}'\},\lr{a,b}''\big),\lr{a,b}''\Big)
\\
&\quad +(1\ 2)\cdot \mu_{1,2}\big(\lr{a,b}',\{c,\lr{a,b}''\}\big)\cdot(1\ 2)\Big)\cdot (2\ 3)\cdot (1\ 2\ 3)
\\
&=(1\ 3\ 2)\cdot\Big(\lr{c,\lr{a,b}'}\otimes\lr{a,b}''\otimes_{\kk\SG_{3}}\big((1\ 2) \mid \id_{\llbracket 1 , 3 \rrbracket}\big)
\\
&\quad +(1\ 2)\cdot\big(\lr{a,b}'\otimes\lr{c,\lr{a,b}''}\otimes_{\kk\SG_{3}}\big((2\ 3)\mid \id_{\llbracket 1 , 3 \rrbracket}\big)\cdot(1\ 2)\Big)\cdot (2\ 3)\cdot (1\ 2\ 3)
\\
&=(1\ 3\ 2)\cdot \lr{c,\lr{a,b}'}\otimes\lr{a,b}''\cdot (1 \ 3\ 2) 
\\
&\quad - (1\ 3\ 2)\cdot (1\ 2)\cdot\lr{b,a}''\otimes \lr{c,\lr{b,a}'}\cdot (1\ 2)
\\
&=\big(\tau_{(132)}\lr{c,\lr{a,b}}_L\big)\cdot (1\ 2\ 3)- (2\ 3)\cdot\big(\tau_{(123)}\lr{c,\lr{b,a}}_L\big)\cdot (1\ 2).
\end{split}
\end{equation}
Now, the double Jacobi identity for the double Poisson bracket $\lr{ \hskip 0.6mm , }$ implies the equalities $\Jac_{\lr{\hskip 0.3mm ,}}(a,b,c)=0=\Jac_{\lr{\hskip 0.3mm ,}}(a,c,b)$. So, we obtain the Jacobi identity for $\Fo(\Fock(A))$ when all the entries are in $\I_{\SG^{\env},1}(A)$. 

Finally, we consider the case where precisely one entry of the Jacobi identity of $\Fo(\Fock(A))$ is in $\I_{\SG^{\env},0}(A_{\cyc})$ (and thus the other two entries are in $\I_{\SG^{\env},1}(A)$). 
Let $a,b,c \in A$. 
We note that the first term of the right member of \eqref{eq:poisson-se-3} for $u = \pi(a)$, $v = b$ and $w = c$ is precisely 
\[     \{ a , \lr{b , c} \}_{A} \otimes_{\Bbbk \SG_{2}} \big( (1 \ 2) \otimes \id_{\llbracket 1 , 2 \rrbracket} \big) ,
\]
where $\{ a , x \otimes y \}_{A} = \{ a , x \}_{A} \otimes y + x \otimes \{ a , y \}_{A}$ for $x, y \in A$. 
Moreover, the second and third terms of the right member of \eqref{eq:poisson-se-3} for $u = \pi(a)$, $v = b$ and $w = c$ are given by 
\[     \lr{ b , \{c , \pi(a) \} } \otimes_{\Bbbk \SG_{2}} \big( (1 \ 2) \otimes \id_{\llbracket 1 , 2 \rrbracket} \big)  = - \lr{ b , \{a , c\}_{A} } \otimes_{\Bbbk \SG_{2}} \big( (1 \ 2) \otimes  \id_{\llbracket 1 , 2 \rrbracket} \big),     \]
and
\[     \lr{ c , \{ a , b \}_{A} } \otimes_{\Bbbk \SG_{2}} \big( \id_{\llbracket 1 , 2 \rrbracket} \otimes (1 \ 2) \big) = -\lr{ \{ a , b \}_{A} , c } \otimes_{\Bbbk \SG_{2}} \big( (1 \ 2) \otimes \id_{\llbracket 1 , 2 \rrbracket} \big),     \]
respectively. 
The Jacobi identity now follows from Fact \ref{fact:resultado-tecnico-prop-VdB-associated-brackets}. 

It remains to show that $\{ \hskip 0.6mm, \}$ is a morphism of generalized wheelspaces, since Lemma \ref{fact:wh-poisson-algebra-ex} then tells us that $\Fock(A)$ with the previous bracket is a Poisson wheelgebra. 
To prove it, recall the definition of the shifted wheelspace given in Subsection \ref{subsection:internal-homos}. 
Note now that \eqref{eq:poisson-se-1} tells us that $\{ \hskip 0.6mm, \}$ is a morphism of generalized wheelspaces if and only if the map $S \rightarrow \sh_{n}(S)$ defined as $v \mapsto \{ u , v\}$ for $v \in S$ 
is a morphism of wheelspaces for all $u \in S(n)$ and $n \in \NN_{0}$. 

Let us now show the intermediate result. 
%%%%%%%
\begin{lemma}
Let $S = (S(n))_{n \in \NN_{0}}$ be a commutative wheelgebra with product $\mu$ 
and let 
$\{ \hskip 0.6mm , \} : S \otimes_{\SG^{\env}} S \rightarrow S $ be a morphism of the underlying diagonal $\SG$-modules satisfying 
\eqref{eq:poisson-se-1} and \eqref{eq:poisson-se-2}. 
Let $v \in S(n)$ and $w \in S(m)$ such that the maps $S \rightarrow \sh_{n}(S)$ and $S \rightarrow \sh_{m}(S)$ 
given by $u \mapsto \{ v , u \}$ and $u \mapsto \{ w , u \}$ for $u \in S$, respectively, are morphisms of wheelspaces. 
Then, the map $S \rightarrow \sh_{n+m}(S)$ given by 
$u \mapsto \{ \mu_{n,m}(v,w) , u \}$ for $u \in S$ is also a morphism of wheelspaces.
\end{lemma}
%%%%%%%
\begin{proof}
Let $p \in \NN_{0}$. 
The identity \eqref{eq:poisson-se-2-equiv} tells us that the previous map $S \rightarrow \sh_{n+m}(S)$ applied to elements of $S(p)$ is precisely given by 
\begin{equation}
     \label{eq:poisson-se-2-bis}
     \begin{split}
         \{ \mu_{n,m}(v,w) , -  \}_{n+m, p} &= 
         \block_{n,p,m}(2 \ 3) \cdot \mu_{p+n,m}\big(\{ v, - \}_{p,n} , w \big) \cdot \block_{n,m,p}(2 \ 3)
         \\
         &\phantom{=}\quad + \mu_{n,p+m} \big(v, \{ -, w \}_{p,m} \big).
         \end{split}
\end{equation}
Now, by assumption, the second term of the right member of \eqref{eq:poisson-se-2-bis} is the composition of two morphisms of wheelspaces, so it is also a morphism of wheelspaces. 
Similarly, a direct computation using \eqref{eq:poisson-se-1} and \ref{item:W1} also shows that the first term of the right member of \eqref{eq:poisson-se-2-bis} is a morphism of wheelspaces. 
Since the sum of morphisms of wheelspaces is a morphism of wheelspaces, the result follows.
\end{proof}

By the previous lemma and a recursive argument on the number of factors on the first argument of $\{\hskip 0.6mm , \}$, it suffices to show that 
the map $\Fock(A) \rightarrow \Fock(A)$ and $\Fock(A) \rightarrow \sh_{1}(\Fock(A))$ defined by $u \mapsto \{ \pi(a) u \}$ and $u \mapsto \{ a , u \}$ for $u \in \Fock(A)$, respectively,  
are morphisms of wheelspaces for all $a \in A$. 

Consider first the map $\Fock(A) \rightarrow \Fock(A)$
given by $u \mapsto \{ \pi(a) , a \}$ for $u \in \Fock(A)$ and $a \in A$. 
It suffices to prove that ${}^{n}\mathcalboondox{t} (\{ \pi(a) , u \}_{0,n}) = \{ \pi(a) , {}^{n}\mathcalboondox{t}(u) \}_{0,n-1}$
for $a \in A$ and 
\begin{equation}
\label{eq:u}
u = (a_{1} | \dots | a_{n}) \otimes_{\Bbbk \SG_{n}} (\sigma | \tau) \otimes \pi(b_{1}) \dots \pi(b_{m}),     
\end{equation}
for $a_{1}, \dots, a_{n}, b_{1}, \dots, b_{m} \in A$, $\sigma, \tau \in \SG_{n}$. 
By \eqref{eq:poisson-se-1} and \eqref{eq:poisson-se-2-equiv} we have that 
\begin{small}
\begin{equation}
\label{eq:calculation}
\begin{split}
    &\{ \pi(a) , a \}_{n,0} = \sum_{i=1}^{n} (a_{1} | \dots | a_{i-1}| \{a,a_{i}\}_{A}|a_{i+1}| \dots a_{n}) \otimes_{\Bbbk \SG_{n}} (\sigma | \tau) \otimes \pi(b_{1}) \dots \pi(b_{m})
    \\
    &+ \sum_{j=1}^{m} (a_{1} | \dots | a_{n}) \otimes_{\Bbbk \SG_{n}} (\sigma | \tau) \otimes \pi(b_{1}) \dots \pi(b_{j-1}) \{ \pi(a) , \pi(b_{j}) \}_{A} \pi(b_{j+1}) \dots \pi(b_{m}).
\end{split}
\end{equation}
\end{small}
\hskip -1.2mm Then, using the expression \eqref{eq:tij-fock} for the contraction of the Fock algebra $\Fock(A)$ together with \eqref{eq:calculation}, it is easy to see that ${}^{n}\mathcalboondox{t} (\{ \pi(a) , u \}_{0,n}) = \{ \pi(a) ,  {}^{n}\mathcalboondox{t}(u) \}_{0,n-1}$, as was to be shown. 

Similarly, let $\Fock(A) \rightarrow \sh_{1}(\Fock(A))$ be the map 
given by $u \mapsto \{ a , u \}$ for $u \in \Fock(A)$ and $a \in A$. 
It suffices to prove that ${}^{n+1}\mathcalboondox{t} (\{ a , u \}_{1,n}) = \{  a , {}^{n}\mathcalboondox{t}(u) \}_{1,n-1}$
for $a \in A$ and $u$ given as in \eqref{eq:u}.  
By \eqref{eq:poisson-se-1} and \eqref{eq:poisson-se-2-equiv} we have that 
\begin{small}
\begin{equation}
\label{eq:calculation-bis}
\begin{split}
    &\{ a , u \}_{1,n} = \sum_{i=1}^{n} (a_{1} | \dots | a_{i-1}| \lr{a,a_{i}}_{A}|a_{i+1}| \dots a_{n}) 
    \\
    &\phantom{+\sum_{i=1}^{n}} \otimes_{\Bbbk \SG_{n+1}} \big((1 \ \dots (i + 1))^{-1} \cdot \rinc_{n,n+1}(\sigma) | (1 \ \dots i) \cdot_{\op} \rinc_{n,n+1}(\tau)\big) 
    \otimes \pi(b_{1}) \dots \pi(b_{m})
    \\
    &-\sum_{j=1}^{m} (a_{1} | \dots | a_{n}) \otimes_{\Bbbk \SG_{n}} (\sigma | \tau) \otimes \pi(b_{1}) \dots \pi(b_{j-1}) \{ \pi(b_{j}), a \}_{A} \pi(b_{j+1}) \dots \pi(b_{m}).
\end{split}
\end{equation}
\end{small}
\hskip -1.2mm Then, using the expression \eqref{eq:tij-fock} for the contraction of the Fock wheelgebra $\Fock(A)$ together with \eqref{eq:calculation-bis}, it is tedious but straightforward to show that ${}^{n+1}\mathcalboondox{t} (\{ a , u \}_{1,n}) = \{ a , {}^{n}\mathcalboondox{t}(u) \}_{1 , n-1}$. 
The proposition is thus proved. 
\end{proof}

%%%%%%%%%%%%%%%%%%%%%%%%%%%%%%%%%%%%%%%%%%%%%%%%%%%%%%
\subsection{Cartan calculus for Fock algebras}
\label{sec: Cartan calculus for Fock algebras}

Recall the definition of the functor $\WW$ introduced in \eqref{index:functor-WW}. 
Then, we can define
\begin{equation}
\WW\!\DDer B = \Fock\big(T_B\DDer B\big)_1,\quad \text{and}\quad \WW \diff B =\Fock\big(T_B\diff B\big)_1.
\label{index:WW-DDer-Omega}
\end{equation}
As explained in Section \ref{subsec:bisymplectic-algebras}, since by definition $(\diff B )_{\cyc}=\DR^1 B$ (see \eqref{eq:Karoubi-de-Rham-def}), and the map $(\DDer B)_{\cycm}\to \Der(B)$ given by the multiplication is an isomorphism if $B$ is smooth, 
using \eqref{eq:f-A-1} we can write more explicitly the component of weight $m$ of $\WW\!\DDer B$ and $\WW \Omega^1_{\nc} $ as follows:
\begin{equation}
\label{eq:description-WDer-WOmega1-weight-m}
\begin{aligned}
   (\WW\!\DDer B)(m)&= \bigoplus^{m-1}_{i=0}\big(B^{\otimes i}\otimes\DDer B\otimes B^{\otimes(m-i-1)}\otimes_{\kk\SG_m} \kk\SG^{\env}_m\big)\otimes\Sym(B_{\cyc}) 
   \\
   &\qquad \qquad\qquad  \oplus \big(B^{\otimes m}\otimes_{\kk\SG_m}\kk\SG^{\env}_m\big)  \otimes\Sym(B_{\cyc})(\Der B), 
   \\
     (\WW \Omega^1_{\nc}B )(m)&= \bigoplus^{m-1}_{i=0}\big(B^{\otimes i}\otimes\Omega^1_{\nc}B \otimes B^{\otimes(m-i-1)}\otimes_{\kk\SG_m}\kk\SG^{\env}_m\big)\otimes\Sym(B_{\cyc}) 
     \\
     &\qquad \qquad \qquad \oplus \big(B^{\otimes m}\otimes_{\kk\SG_m}\kk\SG^{\env}_m\big) \otimes\Sym(B_{\cyc})(\DR^1 B ). 
\end{aligned}
\end{equation}

%%%%%%%
\begin{remark}
Note that the wheelmodule $\WW\!\DDer B$ has already appeared in \cite{MR2734329}, eq. (3.6.1), even though our degree conventions are seemingly different; for example, compare the term in the middle of \emph{loc. cit.} with \eqref{eq:description-WDer-WOmega1-weight-m}.
\end{remark}
%%%%%%%

The next key result states that the $\Fock(B)$-wheelmodule $\WW \diff B$ represents the functor $\IDer(\Fock(B),-)$, \textit{i.e.} $\Fock(B)$ admits a wheelmodule of Kähler wheeldifferential forms with
$\Omega_{\wh}^{1}\Fock(B) = \WW \diff B$.
%%%%%%%%
\begin{proposition}
\label{prop:Fock-has-wheeldiff-forms}
Let $B$ be an algebra and $\wN$ be a $\Fock(B)$-wheelmodule.
Then, there exists a natural isomorphism
\begin{equation}
\label{eq:iso-wheelmod}
\IHom_{\Fock(B)}\big(\WW\diff B,\wN\big) 
\cong \IDer\big(\Fock(B),\wN\big),
\end{equation}
and in particular 
\begin{equation}
\label{eq:iso-wheelmod-bis}
\Hom_{\Fock(B)}\big(\WW\diff B,\wN\big) 
\cong \Der\big(\Fock(B),\wN\big). 
\end{equation}
In consequence, $\Fock(B)$ admits a wheelmodule of Kähler wheeldifferential forms $\Omega_{\wh}^{1}\Fock(B)$ given explicitly by
$\WW \diff B$.
\end{proposition}
%%%%%%%
\begin{proof}
We will first prove the second isomorphism. 
Recall the functor $\bmod$ introduced in \eqref{eq:wheel-mod}. 
Firstly, we shall prove that the map 
\begin{equation}
\label{eq:isom-evaluacion-em-1}
\Der\big(\Fock(B),\wN\big) \longrightarrow \Der\big(B,\bmod(\wN)\big)
\end{equation}
sending $f$ to $f(1)\vert_{B\otimes\kk}$ is an isomorphism. 
To prove this, we consider the commutative diagram 
\begin{equation}
\begin{tikzcd}[every matrix/.append style = {nodes={font=\small}}, every label/.append style = {font = \tiny}]
\Hom_{\Alg(\WMod)}\Big(\Fock(B),\ZSE\big(\Fock(B),\wN\big)\Big) 
\arrow[r, "\cong"]
&
[-5pt]
\Hom_{\Alg}\bigg(B,\ZSE\Big(\alg\big(\Fock(B)\big),\bmod(\wN)\Big)\bigg)
\\
\Hom_{{}_{\Fock(B)}\aRng(\WMod)}\Big(\Fock(B),\ZSE\big(\Fock(B),\wN\big)\Big) 
\arrow[u, phantom, sloped, "\subseteq"]
\arrow[r, dashed]
&
\Hom_{{}_{B}\aRng}\Big(B,\ZSE\big(B,\bmod(\wN)\big)\Big)
\arrow[u, phantom, sloped, "\subseteq"]
\\
\Der\big(\Fock(B),\wN\big)
\arrow[u, sloped, "\cong"]
\arrow[r]
&
\Der\big(B,\bmod(\wN)\big)
\arrow[u, sloped, "\cong"]
\end{tikzcd}
\end{equation}
where the upper horizontal map is obtained as the composition of \eqref{eq:univ-fock} and the map induced by \eqref{eq:wheel-mod-2}, 
the middle horizontal map is the restriction of the map above, and the lower horizontal map is \eqref{eq:isom-evaluacion-em-1}. On the other hand, the right vertical inclusion is induced by the canonical injection $B = B \otimes \Sym^{0}(B_{\cyc}) \subseteq \alg(\Fock(B))$ and the two lower vertical maps are given by Fact \ref{fact:mor-zse-bis}. 
A direct verification shows that the middle horizontal map is surjective, so an isomorphism, which in turn implies that the lower horizontal map is an isomorphism, as was to be shown.

Finally, we have the following sequence of isomorphisms
\begin{align*}
\Hom_{\Fock(B)}\big(\WW\diff B,\wN\big)
&\cong \Hom_{{}_{B}\Mod_{B}}\big(\diff B,\bmod(\wN)\big)
\\
&\cong \Der\big(B,\bmod(\wN)\big) 
\\
&\cong\Der\big(\Fock(B),\wN\big),
\end{align*}
where the first isomorphism is given by Lemma \ref{lemma:fock-adjoint}, the second is a consequence of the universal property of $\diff B$, and the last isomorphism is given by \eqref{eq:isom-evaluacion-em-1}. 
This proves the isomorphism \eqref{eq:iso-wheelmod-bis}. 
Finally, the isomorphism \eqref{eq:iso-wheelmod} follows from 
\eqref{eq:iso-wheelmod-bis}, together with the isomorphisms in Remarks \ref{remark:hom-ihom-wheelmod} and \ref{remark:der-wheelder}.
\end{proof}

Now, taking $\wN=\Fock(B)$ (regarded as a $\Fock(B)$-wheelmodule over itself), we have the linear bijection
\begin{equation}
\label{eq:prop-universal-wheelspaces-Kahler-diff}
\Der\big(\Fock(B)\big) \longrightarrow \Hom_{\Fock(B)}\big(\WW\diff B, \Fock(B)\big)
\end{equation}
sending $\theta$ to $\Wi_\theta$. 
The following results state that any Fock wheelgebra admits a dg wheelgebra of Kähler wheeldifferential forms and has wheelcontractions, so it comes equipped with a natural Cartan calculus.

%%%%%%%%
\begin{lemma}
\label{lemma:dgwheel-fock}
Let $B$ be an associative algebra, and let $\Fock(B)$ be the corresponding Fock wheelgebra. 
Then, $\Fock(B)$ admits a dg wheelgebra of Kähler wheeldifferential forms. 
Explicitly, $\Omega^\bullet_{\Fock(B)} = \Fock(\Omega^\bullet_{\nc}B)$ and the differential $\derdifwbul_{\Fock(B)} \in \Der(\Fock(\Omega^\bullet_{\nc}B))$ is the image of $\du{} \otimes 1_{\Sym((\Omega^\bullet_{\nc}B)_{\cyc})}$ under the composition of the canonical inclusion morphism 
\begin{equation}
\label{eq:derdiffbullet1}
\Der\big(\Omega^\bullet_{\nc}B\big) \otimes \Sym\big((\Omega^\bullet_{\nc}B)_{\cyc}\big)
    \longrightarrow \Der\Big(\Omega^\bullet_{\nc}B , \Omega^\bullet_{\nc}B \otimes \Sym\big((\Omega^\bullet_{\nc}B)_{\cyc}\big)\Big) 
\end{equation}
and the isomorphism \eqref{eq:isom-evaluacion-em-1} 
\begin{equation}
\label{eq:derdiffbullet2}
\Der\Big(\Omega^\bullet_{\nc}B , \bmod(\Fock(\Omega^\bullet_{\nc}B))\Big) \overset{\cong}{\longrightarrow} \Der(\Fock(\Omega^\bullet_{\nc}B)),
\end{equation}
where $\Omega^\bullet_{\nc}B \otimes \Sym\big((\Omega^\bullet_{\nc}B)_{\cyc}\big)$ has the natural structure of $\Omega^\bullet_{\nc}B$-bimodule via $\omega.(\omega' \otimes \rho).\omega'' = \omega \omega' \omega'' \otimes \rho$, for $\omega, \omega', \omega'' \in \Omega^\bullet_{\nc}B$ and $\rho \in \Sym\big((\Omega^\bullet_{\nc}B)_{\cyc}\big)$, the codomain of \eqref{eq:derdiffbullet1} clearly coincides with the domain of \eqref{eq:derdiffbullet2}, and $\du{} \in \Der(\Omega^\bullet_{\nc}B)$ denotes the universal differential of $B$. 
\end{lemma}
%%%%%%%
\begin{proof}
The existence of the symmetric algebra of 
$\WW\Omega^1_{\nc}B \cong \Omega_{\wh}^{1}\Fock(B)$ follows from
Corollary \ref{corollary:fock-adjoint-2} (see also Remark \ref{rem:symmetric-Fock-alg}), so $\Fock(B)$ admits a (graded) wheelgebra of Kähler wheeldifferential forms. 
To prove the second part, we only need to show that $\derdifwbul_{\Fock(B)}$ is a differential and that its restriction to $\Fock(B)$ coincides with the composition of the universal derivation $\derdifw_{\Fock(B)} : \Fock(B) \rightarrow \WW\Omega^1_{\nc}B$ and the canonical inclusion $\WW\Omega^1_{\nc}B \rightarrow \Fock(\Omega^\bullet_{\nc}B)$. 

The fact that $\derdifwbul_{\Fock(B)} \circ \derdifwbul_{\Fock(B)} = 0$ follows from the fact that $\derdifwbul_{\Fock(B)} \circ \derdifwbul_{\Fock(B)}$ is the image of $(\du{} \circ \du{}) \otimes 1_{\Sym((\Omega^\bullet_{\nc}B)_{\cyc})} = 0$ under the composition of \eqref{eq:derdiffbullet1} and \eqref{eq:derdiffbullet2}. 

Let us finally prove that the restriction of $\derdifwbul_{\Fock(B)}$ to $\Fock(B)$ coincides with the composition of the universal derivation $\derdifw_{\Fock(B)} : \Fock(B) \rightarrow \WW\Omega^1_{\nc}B$ and the canonical inclusion $\WW\Omega^1_{\nc}B \rightarrow \Fock(\Omega^\bullet_{\nc}B)$. 

Note first that we have the commutative diagram
\begin{equation}
\begin{tikzcd}[every matrix/.append style = {nodes={font=\small}}, every label/.append style = {font = \tiny}]
\Der\big(\Omega^\bullet_{\nc}B\big) \otimes \Sym\big((\Omega^\bullet_{\nc}B)_{\cyc}\big)
\arrow[r] 
\arrow[d,"\operatorname{res}"] 
&
\Der\Big(\Omega^\bullet_{\nc}B , \Omega^\bullet_{\nc}B \otimes \Sym\big((\Omega^\bullet_{\nc}B)_{\cyc}\big)\Big) 
\arrow[d,"\operatorname{res}"] 
\\
\Der\Big(B, \Omega^\bullet_{\nc}B \Big) \otimes \Sym\big((\Omega^\bullet_{\nc}B)_{\cyc}\big)\Big) 
\arrow[r] 
&
\Der\Big(B , \Omega^\bullet_{\nc}B \otimes \Sym\big((\Omega^\bullet_{\nc}B)_{\cyc}\big)\Big) 
\end{tikzcd}
\end{equation}
where the horizontal maps are given by the canonical inclusions, and the vertical arrows are given by the canonical restrictions. 
Note that the left vertical arrow of the previous diagram sends $\du{} \otimes 1_{\Sym((\Omega^\bullet_{\nc}B)_{\cyc})}$ to 
$d_{B} \otimes 1_{\Sym((\Omega^\bullet_{\nc}B)_{\cyc})}$, where $d_{B} : B \rightarrow \Omega^{\bullet}_{\nc}B$ is the composition of the universal derivation $B \rightarrow \Omega^{1}_{\nc}B$ and the canonical inclusion 
$\Omega^{1}_{\nc}B \rightarrow \Omega^{\bullet}_{\nc}B$.
Using this together with the previous commutative diagram and the naturality of \eqref{eq:isom-evaluacion-em-1}, we see that the composition of $\derdifwbul_{\Fock(B)}$with the canonical inclusion $\Fock(B) \rightarrow \Fock(\Omega_{\nc}^{\bullet}B)$, as was to be shown. 
\end{proof}

%%%%%%%%%%%%%
\begin{remark}
Note that, using \eqref{eq:f-A-n}, we immediately get that $\Omega^n_{\Fock(B)}(0) = \Sym_{n}(\DR^{\bullet} B)$ for all $n \in \NN_{0}$, and in particular we have that $\Omega^1_{\Fock(B)}(0) = \Sym(B_{\cyc}) \DR^{1} B$ and 
\[     \Omega^2_{\Fock(B)}(0) = \Sym(B_{\cyc}) \DR^{2} B \oplus \Sym(B_{\cyc}) \Sym^{2} (\DR^{1} B).     \] 
\end{remark}
%%%%%%%%%%%%% 

%%%%%%%%
\begin{lemma}
\label{lemma:cont-wheel-fock}
Let $B$ be an associative algebra, and let $\Fock(B)$ be the corresponding Fock wheelgebra. 
Then, $\Fock(B)$ has wheelcontractions. 
\end{lemma}
%%%%%%%%
\begin{proof}
It is easy to see that the canonical isomorphism \eqref{eq:isom-evaluacion-em-1} for $\Omega_{\nc}^{\bullet}B$ instead of $B$ is homogeneous of degree zero, so it induces an isomorphism
\begin{equation}
\label{eq:isom-evaluacion-em-1-omega}
\Der^{-1}\big(\Fock(\Omega_{\nc}^{\bullet}B)\big) \overset{\cong}{\longrightarrow} \Der^{-1}\Big(\Omega_{\nc}^{\bullet}B,\bmod\big(\Fock(\Omega_{\nc}^{\bullet}B)\big)\Big).
\end{equation}
Since $\bmod(\Fock(\Omega_{\nc}^{\bullet}B))$ is nonnegatively graded, 
the universal property of the graded algebra 
$\Omega_{\nc}^{\bullet}B = T_{B}\Omega_{\nc}^{1}B$ implies that the restriction map induces an isomorphism 
\begin{equation}
\label{eq:isom-evaluacion-em-1-omega-bis}
\Der^{-1}\Big(\Omega_{\nc}^{\bullet}B,\bmod\big(\Fock(\Omega_{\nc}^{\bullet}B)\big)\Big) \overset{\cong}{\longrightarrow} 
\Hom_{{}_{B}\Mod_{B}}^{-1}\Big(\Omega_{\nc}^{1}B,\bmod\big(\Fock(\Omega_{\nc}^{\bullet}B)\big)\Big),
\end{equation}
where the right member denotes the space of homogeneous morphisms of $B$-bi\-mod\-ules of degree $-1$. 
Finally, the isomorphism \eqref{eq:univ-fock-4} restricts to an isomorphism 
\begin{equation}
\label{eq:isom-evaluacion-em-1-omega-bis-bis}
\Hom_{{}_{B}\Mod_{B}}^{-1}\Big(\Omega_{\nc}^{1}B,\bmod\big(\Fock(\Omega_{\nc}^{\bullet}B)\big)\Big) 
\overset{\cong}{\longrightarrow} 
\Hom_{{}_{\Fock(B)}}^{-1}\big(\WW \Omega_{\nc}^{1}B,\Fock(\Omega_{\nc}^{\bullet}B)\big),
\end{equation}
where the right member is the space of homogeneous morphisms of $\Fock(B)$-wheel\-mod\-ules of degree $-1$. 
Note that, by degree reasons, 
\[     \Hom_{{}_{\Fock(B)}}^{-1}\big(\WW \Omega_{\nc}^{1}B,\Fock(\Omega_{\nc}^{\bullet}B)\big) = \Hom_{{}_{\Fock(B)}}\big(\WW \Omega_{\nc}^{1}B,\Fock(B)\big).     \]
It is direct to check that the composition of  \eqref{eq:isom-evaluacion-em-1-omega-bis} and \eqref{eq:isom-evaluacion-em-1-omega-bis-bis} is precisely the morphism \eqref{eq:der-star} for 
$\wC = \Fock(B)$, which in turn implies that 
$\Fock(B)$ has wheelcontractions, as was to be shown. 
\end{proof}

We consider now the map 
\begin{equation}
\label{eq:pre-wwderdoble-wder}
\begin{split}
\DDer(B) \longrightarrow \bmod\big(\IDer(\Fock(B))\big) 
&= \Der\Big(\Fock(B), \sh_{1}^{\wh}\big(\Fock(B)\big)\Big) 
\\
&\cong \Der\bigg(B,\bmod\Big(\sh_{1}^{\wh}\big(\Fock(B)\big)\Big)\bigg)
\end{split}
\end{equation}
given by sending $d : B \rightarrow B^{\env}$ to the derivation $\hat{d} \in \Der(B, \bmod(\sh_{1}^{\wh}(\Fock(B))))$ defined as $\hat{d}(b) = \overline{d(b)} = d(b) \otimes_{\Bbbk \SG_{2}} (\id_{\llbracket 1 , 2 \rrbracket}|\id_{\llbracket 1 , 2 \rrbracket}) \otimes 1_{\Sym(B_{\cyc})}$ for all $b \in B$, 
where we used the terminology of Lemma \ref{lemma:contractions-fock}, the equality \eqref{eq:der-wheelder} and the isomorphism \eqref{eq:isom-evaluacion-em-1}. 
It is long but straightforward to show that \eqref{eq:pre-wwderdoble-wder} is a well-defined morphism of $B$-bimodules. 
Then, by Lemma \ref{lemma:fock-adjoint}, the morphism \eqref{eq:pre-wwderdoble-wder} induces a unique morphism of $\Fock(B)$-wheelmodules 
\begin{equation}
\label{eq:wwderdoble-wder}
\WW\DDer(B) \longrightarrow \IDer(\Fock(B)).
\end{equation}

We then note the following interesting result. 
%%%%%%%%
\begin{lemma}
\label{lemma:wwderdoble-wder-iso}
Let $B$ be a smooth associative algebra. 
Then, the canonical morphism of $\Fock(B)$-wheelmodules 
\begin{equation}
\label{eq:wwderdoble-wder-bis}
\WW\DDer(B) \longrightarrow \IDer(\Fock(B))
\end{equation}
given in \eqref{eq:wwderdoble-wder}
is an isomorphism. 
\end{lemma}
%%%%%%%%
\begin{proof}
This follows from \cite{MR2734329}, Thm. 3.6.7, (ii). 
\end{proof}

%%%%%%%%%%%%%%%%%%%%%%%%%%%%%%%%%%%%%%%%%%%%%%%%%%%%%%
\subsection{Symplectic Fock wheelgebra of a bisymplectic algebra}
\label{sec:bisymplectic-and-wheelgebras}

To introduce appropriate symplectic structures on Fock wheelgebras, we need to define double and wheeled versions of the classical big bracket, originally introduced by B. Kostant and S. Sternberg. 
In this article, we follow the work of Y. Kosmann--Schwarzbach, who coined it (see \cite{MR1427124} and references therein).

Consider $\mathbb{E} =\DDer B\oplus\diff B$ \label{index:double-der-one-form} as a $B$-bimodule with the natural outer bimodule structure. 
We define a double bracket 
\begin{equation}
\lr{-,-}^{\imath}\colon \big(T_B\mathbb{E}\big)^{\otimes 2}\longrightarrow \big(T_B\mathbb{E}\big)^{\otimes 2}
\label{eq:double-big-bracket}
\end{equation}
of degree $-2$ as follows. 
We first set 
\begin{equation}
\label{eq:big-brackets-on-generators}
\begin{aligned}
    \lr{b_1,b_2}^{\imath}&= \lr{\Theta,b}^{\imath} =\lr{b,\Theta}^{\imath}=\lr{\alpha,b}^{\imath}=\lr{b,\alpha}^{\imath}=\lr{\Theta,\Delta}^{\imath}=\lr{\alpha,\beta}^{\imath}=0,
 \\ 
 \lr{\alpha,\Theta}^{\imath}&=i_\Theta\alpha,
  \hskip 0.9cm \lr{\Theta,\alpha}^{\imath} =\tau_{B,B}\big(i_\Theta\alpha\big),
 \end{aligned}
\end{equation}
for all $b,b_1,b_2\in B$, $\Theta,\Delta\in\DDer B$, and $\alpha,\beta\in\diff B$, where the contraction operator $i_\Theta$ was recalled in \eqref{eq:prop-universal-CQ} for $(B\otimes B)_{\out}$, \textit{i.e.}
$\lr{\alpha,\Theta}^{\imath}\in B\otimes B$. 
By the Leibniz identity, the definitions in \eqref{eq:big-brackets-on-generators} give rise to a unique 
double bracket 
\eqref{eq:double-big-bracket}. 
Since the double Jacobi identity trivially holds, $\lr{-,-}^{\imath}$ becomes a double Poisson bracket, which we call the \textbf{\textcolor{myblue}{double big bracket}}.

By Proposition \ref{prop:wheel-Poisson-double}, we can now define the \textbf{\textcolor{myblue}{wheeled big bracket}} as
\begin{equation}
\{-,-\}^{\imath}_{n,m}\colon\Fock(T_B\mathbb{E})(n)\otimes \Fock(T_B\mathbb{E})(m)\longrightarrow \Fock(T_B\mathbb{E})(n+m),
\label{eq:wheeled-big-bracket}
\end{equation}
satisfying that 
\begin{equation}
\label{eq:wheeled-big-bracket-bis}
\{ \chi,\xi\}^{\imath}_{1,1}=\lr{\chi, \xi}^{\imath}\otimes_{\kk \SG_{2}} \big((1\, 2)| 1\big)\otimes 1,
\end{equation}
for all $\chi, \xi \in \Fock(T_{B}\mathbb{E})(1)$. 
In particular, 
\begin{equation}
\label{eq:wheeled-big-bracket-bis-bis}
\{\du b \otimes 1,\Theta \otimes 1\}^{\imath}_{1,1}=\Theta(b)
\otimes \big((1\, 2)| 1\big)\otimes 1,
\end{equation}
for all $b \in B$ and $\Theta \in \DDer B$.

For the following result, recall that $\WW\!\DDer B(n),\Omega^1_{\Fock(B)}(n) \subseteq \Fock(T_B\mathbb{E})(n)$ for all $n \in \NN_{0}$ and $\Omega^2_{\Fock(B)}(0)\subseteq \Fock(T_B\mathbb{E})(0)$. 
%%%%%%%
\begin{fact}
\label{fact:symplectic-wheelgebra}
Let $B$ be a associative algebra and let $\varpi\in \Omega^{2}_{\Fock(B)}B(0)$ be a closed $2$-form. 
Then, the composition of the canonical morphisms of wheelmodules \eqref{eq:wwderdoble-wder} and \eqref{eq:wheelcont} is precisely the morphism of wheelmodules
\begin{equation}
\WW\!\DDer B \longrightarrow \WW\Omega^1_{\nc}B
\label{eq:wheel-symplectic-isom-def}
\end{equation}
sending $X \in \WW\!\DDer B(n)$ to $\{\varpi,X\}^{\imath}_{0,n}$. 
In consequence, if $B$ is smooth, $\Fock(B)$ endowed with $\varpi$ is a symplectic wheelgebra if and only if \eqref{eq:wheel-symplectic-isom-def} is an isomorphism. 
\end{fact}
%%%%%%%
\begin{proof}
This is a long but straightforward computation that follows directly from the definitions. 
\end{proof}

Building on Lemma \ref{lem:functor-fock} and Proposition \ref{prop:wheel-Poisson-double}, which establishes a link between noncommutative Poisson structures on associative algebras and the corresponding Poisson structures on Fock algebras, a natural question arises: can this connection be extended to the symplectic realm, encompassing  bisymplectic algebras (Definition \ref{def:bisymplectic-algebras}) and symplectic (Fock) wheelgebras?
The answer, as it turns out, is affirmative. 

%%%%%%%
\begin{theorem}
Let $B$ be a smooth associative algebra.
If $B$ is a bisymplectic algebra, then $\Fock(B)$ is a symplectic wheelgebra.
\label{theorem:bisymplectic-gives-wheel-symp}
\end{theorem}
%%%%%%%
\begin{proof}
We need to prove that \eqref{eq:wheel-symplectic-isom-def} is an isomorphism. 
The result then follows from Fact \ref{fact:symplectic-wheelgebra}. 
Let $\hat{\omega} \in \Omega^{2}_{\nc}B$ such that its associated class $\omega \in \DR^{2} B$ is a bisymplectic form on $B$. 
Hence, the reduced contraction $\iota(\omega) : \DDer B \to \Omega^1_{\nc}B$ is an isomorphism of $B$-bimodules. 
We define $\varpi \in \Omega^2_{\Fock(B)}(0) = \Sym(B_{\cyc}) \DR^{2} B \oplus \Sym(B_{\cyc}) \Sym^{2} (\DR^{1} B)$ as the element whose component in the first summand is just $\omega$ and whose component in the second summand is zero. 
By the definition of $\WW\!\DDer B$ and $\WW\Omega^1_{\nc}B$, since $\WW$ is a functor, we only need to show that 
$\{\varpi,X\}^{\imath}_{0,1}
= \WW(\iota(\omega))(X)$ for all $X \in \WW\!\DDer B(n)$ and all $n \in \NN_{0}$. 
We note that $\{\varpi,-\}^{\imath}_{0,1}$ is a morphism of $\Fock(B)$-modules, by the Leibniz rule for the big wheeled bracket. 
Moreover, $\WW(\iota(\omega))(-)$ is a morphism of $\Fock(B)$-modules by construction. 
Hence, by Fact \ref{fact:generators-wheel-w}, it suffices to show that $\{\varpi,X\}^{\imath}_{0,1}
= \WW(\iota(\omega))(X)$ for all $X \in \DDer B \otimes \kk \subseteq \WW\!\DDer B(1)$. 
Let $X=\Theta\otimes 1 \in \WW\!\DDer B(1)$, with $\Theta\in\DDer B$.  
We will write as usual $\Theta(b) = \Theta'(b) \otimes \Theta''(b) \in B \otimes B$, for $b \in B$. 
Then,
\begin{align*}
&\{\varpi,X\}^{\imath}_{0,1}
= \big\{ {}^{1}_{1}t_{1}(\hat{\omega} \otimes 1),X\big\}^{\imath}_{0,1}
= {}^{2}_{1}t_{1}\Big( \{ \hat{\omega}\otimes 1,X\}^{\imath}_{1,1}\Big)
\\
&= {}^{2}_{1}t_{1}\Big(i_\Theta\hat{\omega}\otimes_{\kk \SG_{2}}\big((1\, 2)| 1\big)\otimes 1 \Big)
= {}^{2}_{2}t_{1}\Big(i_\Theta\hat{\omega}\otimes_{\kk \SG_{2}}\big( 1 | 1\big)\otimes1\Big)
\\
&= {}^{2}_{2}t_{1}\Big(\big( a\Theta'(b)\otimes\Theta''(b)(\du c)-a(\du b)\Theta'(c)\otimes\Theta''(c)\big)\otimes_{\kk \SG_{2}} (1| 1)\otimes1\Big)
\\
&=\big(\Theta''(b)(\du c) a\theta'(b)-\Theta''(c)a(\du b)\Theta'(c)\big)\otimes_{\kk \SG_{2}} (1| 1)\otimes1 =\WW\big(\iota(\omega)\big)(X),
\end{align*}
where in the first equality we used that $\hat{\omega} \otimes 1 \in \Omega^{2}B \otimes \Sym(B_{\cyc}) \subseteq \Omega^{2}_{\Fock(B)}(1)$, in the third equality we used \eqref{eq:wheeled-big-bracket-bis}, 
and in the fourth equality we used that ${}^2_1t_2(w)={}^2_2t_2((1\,2) \cdot w)$, for any $w \in S(2)$ and any wheelspace $S = (S(n))_{n \in \NN_{0}}$ with contractions ${}^{\bullet}_{\bullet}t_{\bullet}$. 
The theorem is proved. 
\end{proof}

The following result follows from Example \ref{thm:Cbeg-Thm-5.1.1} and Theorem \ref{theorem:bisymplectic-gives-wheel-symp}.

%%%%%%%
\begin{corollary}
Let $B$ be a smooth associative algebra, and consider $A=T_B(\DDer B)$. 
Then $\Fock(A)$ is naturally endowed with a wheelsymplectic form.
\label{coro:Liouville-1-form}
\end{corollary}
%%%%%%%

%%%%%%%%%%%%%%%%%%%%%%%%%%%%%%%%%%%%%%%%%%%%%%%%%%%%%%%%%%%%%%%%%%%%%%%%%
\section{Non- and commutative representation schemes}
\label{section:kr}

%%%%%%%%%%%%%
\subsection{ The Kontsevich-Rosenberg principle (after Berest--Khachatryan--Ramadoss)}
\label{sec:KR-after-BKR}
%%%%%%%%%%%%%%

In this section we will briefly recall the main definitions and properties of representation schemes of algebras, and in particular for bisymplectic algebras. 

%%%%%%%%%%%%%%%%%%%%
\subsubsection{The representation scheme for algebras}
\label{sec:Rep}
%%%%%%%%%%%%%%%%%%%

Let $A$ be an associative algebra, and $V$ be a finite dimensional vector space. 
The representation scheme parametrizes the $n$-dimensional representations of $A$, and it is defined as the functor 
\begin{equation}
\Rep_V A\colon \CAlg\longrightarrow \Set
\label{eq:functor-representation}
\end{equation}
from the category of commutative algebras $\CAlg\label{index:categ-comm-alg}$ into $\Set$ that sends a commutative algebra $C$ to the set $\Hom_{\Alg}(A,\End (V)\otimes C)$. 

As explained in \cite{arXiv:1010.4901} (see also \cites{MR3084440,MR3204869}), based on the work of G. Bergman \cite{MR357503} and P. M. Cohn \cite{MR555546}, one proves the representability of \eqref{eq:functor-representation} by extending it from $\CAlg$ to the category $\Alg \label{index:categ-ass-alg}$ of all associative algebras, yielding the following diagram
\[
\xymatrix{
\CAlg \ar[rr]^-{\Rep_VA} \ar@{^{(}->}[d]_-{\inc} &&\Set
\\
\Alg \ar[rru]_-{\widetilde{\Rep}_VA}
}
\]
where the functor 
\begin{equation}
\widetilde{\Rep}_V A\colon \Alg\longrightarrow \Set
\label{eq:functor-representation-ext}
\end{equation}
sends an associative algebra $B$ to the set $\Hom_{\Alg}(A,\End (V)\otimes B)$, and
$\label{index:inclusion-functor}\inc\colon \colon \CAlg \hookrightarrow \Alg$ denotes the canonical inclusion. 
Recall also the following two functors
\begin{subequations}
\label{eq:root-VdB-def-BKR}
\begin{align}
\sqrt[V]{-}\colon \Alg &\longrightarrow \Alg,
\label{eq:root-VdB-def-BKR.a}
\\(-)_V\colon \Alg &\longrightarrow \CAlg
\label{eq:root-VdB-def-BKR.b}
\end{align}
\end{subequations}
sending an associative algebra $A$ to the sets $(\End(V)*A)^{\End(V)}$ and $(\sqrt[V]{A})_\ab$, respectively, 
where the \textbf{\textcolor{myblue}{abelianization functor}} 
\begin{equation}
\label{eq:abel}
(\place)_{\ab} : \Alg \longrightarrow \CAlg     
\end{equation}
sends an algebra $B$ to the quotient $B_{\ab}$ of $B$ by the ideal generated by $[B,B]$, and any morphism $f : B \rightarrow B'$ of algebras to the induced morphism $B_{\ab} \rightarrow B'_{\ab}$. 
If $\dim(V) = n$, Cohn called the algebra $\sqrt[V]{A}$ the \emph{$n$-matrix reduction of $A$}, where $n=\dim V$. 
The notation $\sqrt[V]{-}$ first appeared in \cite{MR1877866}. 

The following result shows that \eqref{eq:functor-representation-ext} and \eqref{eq:root-VdB-def-BKR.a} (resp., \eqref{eq:functor-representation} and \eqref{eq:root-VdB-def-BKR.b}) are adjoint functors.  

%%%%%%%
\begin{proposition}[\cite{MR3204869}, Prop. 2.1]
Let $V$ be a finite dimensional vector space. 
Then there exist natural bijections
\begin{subequations}
\label{BKR-representability}
\begin{align}
\Hom_{\Alg}\big(\sqrt[V]{A},B\big)&\cong \Hom_{\Alg}(A,\End(V)\otimes B\big),
\label{BKR-representability.a}
\\
\Hom_{\CAlg}\big(A_V,C\big)&\cong \Hom_{\Alg}\big(A,\End(V)\otimes C\big)
\label{BKR-representability.b}
\end{align}
\end{subequations}
for $A,B\in\Alg$ and $C\in\CAlg$. 
\label{prop:BKR-representability-Rep}
\end{proposition}
%%%%%%%

Hence, \eqref{BKR-representability.b} implies that the functor \eqref{eq:functor-representation} is representable by $A_V$, and we will denote the corresponding affine scheme by $\Rep(A,V) \label{index:representation-scheme}$ (\emph{i.e.} $\Rep(A,V)=\Spec(A_V)$, or $A_V=\Bbbk[\Rep(A,V)]$).
Moreover, there is an explicit algebraic description of the associative algebra $\sqrt[V]{A}$, and thus of $A_{V}$.  
If $n=\dim V$, it turns out that $\sqrt[V]{A}$ is isomorphic to the associative algebra on generators $\{a_{ij}\colon a\in A,\, 1\leq i,j\leq n\}$ subject to the relations
\begin{equation}
\label{eq:root-generators-relations}
(\lambda a)_{ij}=\lambda a_{ij},
\quad (a+b)_{ij}=a_{ij}+b_{ij},
\quad (ab)_{ij}=\sum_{t=1}^{n} a_{it}b_{tj},
\quad 1_{uv}a_{ij}=\delta_{iv} a_{ij},
\end{equation}
for all $\lambda\in\kk$, $a,b\in A$, $1\leq i,j,u,v\leq n$.

%%%%%%%%%%%%%%%%%%%%
\subsubsection{The representation scheme for bimodules over algebras and the Kontsevich--Rosenberg prin\-ci\-ple}
%%%%%%%%%%%%%%%%%%%%%%
The \textbf{\textcolor{myblue}{Kontsevich--Rosenberg principle}} \cite{MR1731635} states that every noncommutative algebro-geometric structure on an associative algebra $A$ should induce the corresponding usual commutative algebro-geometric structure on the representation scheme $\Rep(A,V)$.
It has been successfully applied to double (quasi-)Poisson algebras \cite{MR2425689}, bisymplectic algebras \cite{MR2294224}, double Lie algebroids \cite{MR2397630}, or double Courant--Dorfman algebras \cite{MR4640979}, to mention a few.
However, each instance may seem slightly `ad hoc'. To unify approaches, Van den Bergh introduced in \cite{MR2397630}, \S 3.3, an additive functor that extends Proposition \ref{prop:BKR-representability-Rep} from algebras to bimodules, as we now explain.

Given an algebra $B$, recall the \textbf{\textcolor{myblue}{abelianization functor}} 
\begin{equation}
\label{eq:abel-mod}
(\place)_{\ab}^{\mod} : {}_{B}\Mod_{B} \longrightarrow {}_{B_{\ab}}\Mod^{\s}_{B_{\ab}}    
\end{equation}
given by sending a $B$-bimodule $N$ to the quotient $N_{\ab}^{\mod}$ of $X = B_{\ab} \otimes_{B} N \otimes_{B} B_{\ab}$ by the $B_{\ab}$-subbimodule generated by $[X,B_{\ab}]$, and any morphism $f : N \rightarrow N'$ of $B$-bimodules to the induced morphism $N_{\ab}^{\mod} \rightarrow B_{\ab}^{\prime \mod}$. 
Then, $(\place)_{\ab}^{\mod}$ is the left adjoint to the inclusion functor ${}_{B_{\ab}}\Mod^{\s}_{B_{\ab}} \longrightarrow {}_{B}\Mod_{B}$ given by extension of scalars along the morphism $B \rightarrow B_{\ab}$.

Following \cite{MR3084440} (see also \cite{Khachatryan}), given a finite dimensional vector space $V$ and an associative algebra $A$, taking $B=\sqrt[V]{A}$ in \eqref{BKR-representability.a}, we get $\sqrt[V]{\pi}\colon A\to\End (V)\otimes\sqrt[V]{A}$. 
Note also that $\sqrt[V]{A}\otimes V$ is a left module over $\End (V)\otimes\sqrt[V]{A}$ and a right module over $\sqrt[V]{A}$. 
By restricting the left action by means of $\sqrt[V]{\pi}$, we obtain that $\sqrt[V]{A}\otimes V$ is a bimodule over $A$ and $\sqrt[V]{A}$. 
Similarly, $V^*\otimes\sqrt[V]{A}$ is a $\sqrt[V]{A}$-$A$-bimodule. Define now the counterpart to the functor \eqref{eq:root-VdB-def-BKR.a} for bimodules 
\begin{align}
\label{eq:VdB-root-functor-def.a}
\sqrt[V]{\place}\colon {}_A\Mod_{A}&\longrightarrow {}_{\sqrt[V]{A}}\Mod_{\sqrt[V]{A}},
\end{align}
by sending an $A$-bimodule $M$ to the $\sqrt[V]{A}$-bimodule 
$(V^*\otimes\sqrt[V]{A})\otimes_A M\otimes_A (\sqrt[V]{A}\otimes V)$.
The \textbf{\textcolor{myblue}{Van den Bergh functor}} 
\begin{equation}
\label{eq:VdB-functor-def}
(-)_V\colon {}_{A}\Mod_{A} \longrightarrow\Mod_{A_V} =  {}_{A_V}\Mod^{\s}_{A_V}
\end{equation}
arises by composing \eqref{eq:VdB-root-functor-def.a} and \eqref{eq:abel-mod} for $B = \sqrt[V]{A}$. 
More explicitly, it sends an $A$-bimodule $M$ to the (right) $A_{V}$-module $M\otimes_{A^\env} (\End(V)\otimes A_V)$.

The algebra morphism $A\to \End(V)\otimes \sqrt[V]{A}$ (which corresponds to the image of $\id_{\sqrt[V]{A}}$ under \eqref{BKR-representability.a} if $B=\sqrt[V]{A}$) tells us that any $(\End(V)\otimes \sqrt[V]{A})$-bimodule is naturally an $A$-bimodule. 
Similarly, the algebra morphism $A\to \End(V)\otimes A_{V}$ (which corresponds to the image of $\id_{A_V}$ under \eqref{BKR-representability.b} if $C=A_V$) tells us that any $(\End(V)\otimes A_V)$-bimodule is an $A$-bimodule. 
Similar to Proposition \ref{prop:BKR-representability-Rep}, both 
\eqref{eq:VdB-root-functor-def.a} and \eqref{eq:VdB-functor-def} can be understood in terms of adjunctions, as the following result shows. 

%%%%%%%
\begin{proposition}[\cite{MR3084440}, Lemma 5.1]
Let $V$ be a finite dimensional vector space. 
Then, we have the canonical isomorphisms
\begin{align*}
\Hom_{(\sqrt[V]{A})^\env}\big(\sqrt[V]{M},N\big)
&\cong \Hom_{A^\env}\big(M,\End(V)\otimes N\big),
\\
\Hom_{A_V}\big(M_V,L\big)
&\cong \Hom_{A^\env}\big(M,\End(V)\otimes L\big)
\end{align*}
\label{prop:adjunction-functor-VdB}
for $M\in{}_A\Mod_A$, $N\in {}_{\sqrt[V]{A}}\Mod_{\sqrt[V]{A}}$, and $L\in\Mod_{A_V} =  {}_{A_V}\Mod^{\s}_{A_V}$. 
\label{prop:BKR-adjunctions-bimodules}
\end{proposition}
%%%%%%%

On the other hand, Van den Bergh proved in \cite{MR2397630}, Prop. 3.3.4, that the $A$-bimodule $\diff A$ of noncommutative differential $1$-forms satisfies the Kontsevich--Rosenberg principle. 
Moreover, if $A$ is smooth, it also holds for the $A$-bimodule $\DDer A$ of double derivations, due to the isomorphisms of $A_{V}$-bimodules
\begin{equation}
\label{eq:KR-double-derivations-1-forms}
\big(\Omega^1_{\nc}A\big)_V\cong\Omega^1_{A_V} \text{ and }
\big(\DDer A\big)_V\cong\Der(A_V).
\end{equation}
He also stated in \cite{MR2397630}, Cor. 3.3.5, that there is an isomorphism $(\Omega^\bullet_{\nc}A)_V \cong \Lambda_{A_V}\Omega^\bullet(A_V)$ of $A_{V}$-algebras. 

Finally, W. Crawley-Boevey, P. Etingof and V. Ginzburg proved in \cite{MR2294224}, Thm. 6.4.3, (ii), that if a smooth algebra $A$ is endowed with a bisymplectic form $\omega$, then the representation scheme $\Rep(A,V)$ becomes a symplectic manifold. Using the Van den Bergh functor \eqref{eq:VdB-functor-def}, this result is a consequence of \eqref{eq:KR-double-derivations-1-forms} and the functoriality applied to the bijection given by the reduced contraction operator: $\iota(\omega)\colon\DDer A\to\Omega^1_{\nc}A$ (see \cite{arXiv:1708.02650} for the details). 

%%%%%%%%%%%%%
\subsection{ The Kontsevich--Rosenberg principle within the wheeled setting}
\label{sec:wheeled-KR}
%%%%%%%%%%%%%%

Given a finite dimensional vector space $V$, we define the functor 
\begin{equation}
\label{eq:functor-alg-whalg}
\Ec_{V} : \Alg \longrightarrow \Adm     
\end{equation}
sending an algebra $B$ in $\Vect$ to the wheelgebra $\EE^{\env}_{V} \otimes_{\gwh} \I_{\SG^{\env},0}(B)$ (see Example \ref{example:end-4}), and any morphism $f : B \rightarrow B'$ 
to the corresponding morphism of wheelgebras $\id_{\EE^{\env}_{V}} \otimes_{\gwh} \I_{\SG^{\env},0}(f)$. 

The following result originally appeared in \cite{MR2734329}, Thm. 6.1.7, but we give a more direct and simpler proof for completeness, following the more categorical approach explained in Section \ref{sec:Rep}.
%%%%%%%
\begin{proposition}
\label{proposition:functor-alg-whalg}
Given a finite dimensional vector space $V$, the functor 
\[     \Ec_{V} : \Alg \longrightarrow \Adm     \]
introduced in \eqref{eq:functor-alg-whalg} has a left adjoint functor 
\[     \whsqrt[V,\wh]{\place} : \Adm \longrightarrow \Alg.     \]
\end{proposition}
%%%%%%%
\begin{proof}
 Given a wheelgebra $\wA$, let 
 $\hat{\operatorname{B}}$ be the set of all algebra structures (up to isomorphism) on all the subsets of a fixed set of cardinality $\sum_{n \in \NN_{0}} \#\big(\wA(n)\big)$, and let $\mathscr{M}$ be the set formed by all 
 morphisms of wheelgebras $f : \wA \rightarrow  \otimes_{\gwh} \I_{\SG^{\env},0}(B_{f})$ for all $B_{f} \in \hat{\operatorname{B}}$. 
 Define 
 \[     F : \wA \longrightarrow \EE^{\env}_{V} \otimes_{\gwh} \I_{\SG^{\env},0}\bigg(\prod_{f \in \mathscr{M}} B_{f}\bigg)      \]
 as the morphism of wheelspaces such that $F(n)(a) = (f(n)(a))_{f \in \mathscr{M}}$, where we are using the canonical isomorphism of vector spaces $W \otimes (\prod_{i\in \mathtt{I}} B_{i}) \cong \prod_{i\in \mathtt{I}} W \otimes B_{i}$ for $W$ finite dimensional. 
 Since every $f \in \mathscr{M}$ is a morphism of wheelgebras, $F$ is so. 
 Let $B \subseteq \prod_{f \in \mathscr{M}} B_{f}$ be the smallest subalgebra satisfying that $\Img\big(F(n)\big) \subseteq \EE^{\env}_{V}(n) \otimes B$ for all $n \in \NN_{0}$. 
 Then, $F$ corestricts to a morphism of wheelgebras 
\[     F : \wA \longrightarrow \EE^{\env}_{V} \otimes_{\gwh} \I_{\SG^{\env},0}(B).      \]
We then set $\whsqrt[V,\wh]{\wA} = B$. 
The definition of the morphism $\whsqrt[V,\wh]{G} : \whsqrt[V,\wh]{\wA} \rightarrow \whsqrt[V,\wh]{\wA'}$ for any morphism of wheelgebras $G : \wA \rightarrow \wA'$ is similar. 
It is easy to verify that this construction is functorial and it satisfies the required left adjoint property. 
\end{proof}

It is possible to present an explicit construction of $\whsqrt[V,\wh]{\wA}$, which is vaguely hinted in the proof of Thm. 6.1.7 of \cite{MR2734329}. 

%%%%%%%
\begin{proposition}
\label{proposition:functor-alg-whalg-2}
Let $V$ be a finite dimensional vector space with distinguished basis $\{ e_{\alpha} : \alpha \in \mathtt{A} \}$, and $\wA$ be a wheelgebra. 
Then $\whsqrt[V,\wh]{\wA}$ is isomorphic to the quotient of the tensor algebra 
\[
T = T\Big(\bigoplus_{n \in \NN_{0}} \wA(n) \otimes \End(V^{\otimes n})\Big)
\]
by the ideal generated by 
\begin{enumerate}[label=(WA.\arabic*)]
    \item\label{item:wa1} $(a \otimes E_{\bar{\alpha},\bar{\beta}}) (a' \otimes E_{\bar{\alpha'},\bar{\beta'}}) - \mu_{m,n}(a,a') \otimes E_{\bar{\alpha} \sqcup \bar{\alpha}',\bar{\beta} \sqcup \bar{\beta}'}$, for all $m, n \in \NN_{0}$, $a \in \wA(m)$, $a' \in \wA(n)$, 
    $\bar{\alpha},\bar{\beta} \in \mathtt{A}^{m}$, $\bar{\alpha}',\bar{\beta}' \in \mathtt{A}^{n}$, 
    \item $1_{\wA} - 1_{T}$, 
    \item\label{item:wa3} $(\sigma' \cdot a \cdot \sigma) \otimes E_{\sigma' \cdot \bar{\alpha},\bar{\beta} \cdot \sigma} 
    - a \otimes E_{\bar{\alpha},\bar{\beta}}$, for all $n \in \NN_{0}$, $a \in \wA(n)$, 
    $\bar{\alpha},\bar{\beta} \in \mathtt{A}^{n}$, $\sigma, \sigma' \in \SG_{n}$,
    \item\label{item:wa4} ${}^{n}_{j}t_{i}(a) \otimes E_{\bar{\alpha},\bar{\beta}} - \sum_{\gamma \in \mathtt{A}} a \otimes E_{\incl_{j,\gamma}(\bar{\alpha}),\incl_{i,\gamma}(\bar{\beta})}$, for all $n \in \NN$, $a \in \wA(n)$, 
    $\bar{\alpha},\bar{\beta} \in \mathtt{A}^{n-1}$,
\end{enumerate}
where we denote in \ref{item:wa1} the product of $T$ simply by juxtaposition and we are using the notation of Examples \ref{example:end-2} and \ref{example:end-3}, 
and the canonical morphism 
\begin{equation}
\label{eq:morp-univ}
F : \wA \longrightarrow \EE^{\env}_{V} \otimes_{\gwh} \I_{\SG^{\env},0}\Big(\whsqrt[V,\wh]{\wA}\Big)      
\end{equation}
of whelgebras that satisfies the corresponding universal property and is given by 
\[     F(n)(a) = \sum_{\bar{\alpha}, \bar{\beta} \in \mathtt{A}^{n}} E_{\bar{\alpha},\bar{\beta}} \otimes \overline{\Big( a \otimes E_{\bar{\alpha},\bar{\beta}} \Big)}      \]
for all $n \in \NN_{0}$ and $a \in \wA(n)$, where the bar denotes the class of an element of $T$. 
\end{proposition}
%%%%%%%
\begin{proof}
Let us denote by $\bar{T}$ the quotient of the tensor algebra $T$ by the ideal generated by the elements \ref{item:wa1}--\ref{item:wa4}. 
We will prove that, together with $F$, it satisfies the universal property defining $\whsqrt[V,\wh]{\wA}$. 

Let $\phi : \wA \rightarrow \EE^{\env}_{V} \otimes_{\gwh} \I_{\SG^{\env},0}(B)$ be a morphism of wheelgebras, for some algebra $B$. 
Given $n \in \NN_{0}$ and $a \in \wA(n)$, we can write 
\[     \phi(n)(a) = \sum_{\bar{\alpha}, \bar{\beta} \in \mathtt{A}^{n}} E_{\bar{\alpha},\bar{\beta}} \otimes \phi(n)_{\bar{\alpha}, \bar{\beta}}(a),      \]
with $\phi(n)_{\bar{\alpha}, \bar{\beta}}(a) \in B$. 
Let $\psi_{n} : \wA(n) \otimes \End(V^{\otimes n}) \rightarrow B$ be the unique linear map given by 
\[     \psi_{n}\Big(a \otimes E_{\bar{\alpha},\bar{\beta}}\Big) = \phi(n)_{\bar{\alpha}, \bar{\beta}}(a),     \] 
for all $n \in \NN_{0}$, $a \in \wA(n)$ and $\bar{\alpha}, \bar{\beta} \in \mathtt{A}^{n}$, 
and let $\psi : T \rightarrow B$ be the unique morphism of algebras whose restriction to $\wA(n) \otimes \End(V^{\otimes n})$ is $\psi_{n}$ for all $n \in \NN_{0}$. 
The fact that $\phi$ is a morphism of wheelgebras tells us precisely that $\psi$ vanishes on the ideal generated by the elements \ref{item:wa1}--\ref{item:wa4}, so it induces a morphism of algebras $\bar{\psi} : \bar{T} \rightarrow B$. 
Moreover, a direct calculation shows that 
\[     
\phi = \Big(\id_{\EE^{\env}_{V}} \otimes_{\gwh} \I_{\SG^{\env},0}(\bar{\psi})\Big) \circ F. 
\]
Furthermore, the uniqueness of the morphism of algebras $\bar{\psi}$ satisfying the previous identity is also immediate, which proves the proposition. 
\end{proof}

Recall the abelianization functor introduced in \eqref{eq:abel}, which is left adjoint to the inclusion functor $\inc\colon \CAlg \hookrightarrow \Alg$. 

%%%%%%%
\begin{definition}
\label{definition:functor-alg-whalg-2}
Given a finite dimensional vector space $V$, define the functor 
\[     (\place)_{V,\wh} : \Adm \longrightarrow \CAlg.     \]
as the composition of the functor \hskip -1mm $\sqrt[V,\wh]{\place}$ given in Proposition 
\ref{proposition:functor-alg-whalg} and the functor\hskip 1mm $(\place)_{\ab}$.
\end{definition}
%%%%%%%

The following result is an immediate consequence of the definition. 
%%%%%%%
\begin{corollary}
\label{corollary:functor-alg-whalg}
Let $V$ be a finite dimensional vector space, and we define the functor 
\begin{equation}
\CAlg \longrightarrow \Adm     \label{eq: left-adjoint-V-wh}
\end{equation}
 as the restriction to the category $\CAlg$ of commutative algebras in $\Vect$ of the functor $\Ec_{V}$ introduced in \eqref{eq:functor-alg-whalg}. Then  the functor 
$     (\place)_{V,\wh} : \Adm \rightarrow \CAlg$ in Definition \ref{definition:functor-alg-whalg-2} is left adjoint to the functor \eqref{eq: left-adjoint-V-wh}.
\end{corollary}
%%%%%%%

Recall the functors  $\whsqrt[V]{\place}$ and $(\place)_{V}$ introduced in \eqref{eq:root-VdB-def-BKR}. 
As a direct consequence of Theorem \ref{theorem:fock-adjoint}, Proposition \ref{proposition:functor-alg-whalg}, and Corollary \ref{corollary:functor-alg-whalg}, we obtain the following result (\textit{cf.} \cite{MR2734329}, Thm. 6.1.8, (i)). 
%%%%%%%
\begin{corollary}
\label{corollary:functor-alg-whalg-2}
Given a finite dimensional vector space $V$, 
we have natural isomorphisms \[
\whsqrt[V,\wh]{\Fock(A)} \cong \whsqrt[V]{A},\quad 
\text{and} \quad 
\Fock(A)_{V,\wh} \cong A_{V}
\]
of algebras for all algebras $A$.
\end{corollary}
%%%%%%%
\begin{proof}
Theorem \ref{theorem:fock-adjoint} and Proposition \ref{proposition:functor-alg-whalg} tells us that
$\sqrt[V,\wh]{\IA \circ \Fock(\place)} : \Alg \rightarrow \Alg$ is left adjoint to the composition $\alg \circ \Ec_{V} : \Alg \rightarrow \Alg$ of the functor $\Ec_{V}$ given in \eqref{eq:functor-alg-whalg} and the functor $ \alg$ defined in \eqref{eq:wheel-alg}. 
Since the algebra $\alg\big(\EE^{\env}_{V} \otimes_{\gwh} \I_{\SG^{\env},0}(B)\big)$ is isomorphic to $\End(V) \otimes B$ (see Example \ref{example:end-4}), the functor $\sqrt[V,\wh]{\IA \circ \Fock(\place)} : \Alg \rightarrow \Alg$ is left adjoint to the functor $\End(V) \otimes (\place) : \Alg \rightarrow \Alg$. 
Since the latter functor is right adjoint to the functor $\whsqrt[V]{\place} : \Alg \rightarrow \Alg$, by \cite{Khachatryan}, Thm. 107, we get the natural isomorphism $\whsqrt[V,\wh]{\Fock(A)} \cong \whsqrt[V]{A}$ of algebras for all algebras $A$, and by applying the functor $(\place)_{\ab}$ we also get the natural isomorphism $\Fock(A)_{V,\wh} \cong A_{V}$ of commutative algebras for all algebras $A$. 
\end{proof}

Given a finite dimensional vector space $V$, an algebra $B$, and a $B$-bimodule $N$, note that the wheelspace (see Example \ref{example:end-4}):
\begin{equation}
\Ec_{V}^{\mod}(N) = \EE^{\env}_{V} \otimes_{\gwh} \I_{\SG^{\env},0}(N)
\label{eq:Ec-bimodules-wheel}
\end{equation}
has a natural structure of wheelbimodule over the wheelgebra $\Ec_{V}(B) = \EE^{\env}_{V} \otimes_{\gwh} \I_{\SG^{\env},0}(B)$ given by the action
\[     \rho_{m,n,p} : \Ec_{V}(B)(m) \otimes \Ec_{V}^{\mod}(N)(n) \otimes \Ec_{V}(B)(p) \longrightarrow \Ec_{V}^{\mod}(N)(m+n+p)     \]
defined by 
\begin{equation}
\label{eq:ac-ecb}
\begin{split}
 \rho_{m,n,p}\Big((f_{1} \otimes \dots \otimes f_{m} \otimes b), &(f_{m+1} \otimes \dots \otimes f_{m+n} \otimes x), (f_{m+n+1} \otimes \dots \otimes f_{m+n+p} \otimes b')\Big) 
 \\&= f_{1} \otimes \dots \otimes f_{m+n+p} \otimes b.x.b'    
 \end{split}
 \end{equation}
for all $f_{1}, \dots, f_{m+n+p} \in \End(V)$, $b, b' \in B$ and $x \in N$. 

Let $\wA$ be a wheelgebra and let $N$ be a bimodule over $\whsqrt[V,\wh]{\wA}$. 
By composing the universal morphism \eqref{eq:morp-univ}
with the expression \eqref{eq:ac-ecb} giving the action of $\Ec_{V}(\whsqrt[V,\wh]{\wA})$ on $\Ec_{V}^{\mod}(N)$, 
we see that $\Ec_{V}^{\mod}(N)$ has a natural structure of wheelbimodule over $\wA$, simply by replacing the first and third argument of \eqref{eq:ac-ecb} by elements of the form $F(m)(a)$ and $F(p)(a')$, respectively, with $a \in \wA(m)$ and $a \in \wA(p)$. 

Define the functor 
\begin{equation}
\label{eq:functor-alg-whmod}
\Ec_{V}^{\mod} : {}_{\whsqrt[V,\wh]{\wA}}\Mod_{\whsqrt[V,\wh]{\wA}} \longrightarrow {}_{\wA}\AMod_{\wA}   
\end{equation}
sending a $\whsqrt[V,\wh]{\wA}$-bimodule $N$ to the wheelbimodule $\Ec_{V}^{\mod}(N)$ defined in \eqref{eq:Ec-bimodules-wheel}, and any morphism $f : N \rightarrow N'$ of $\whsqrt[V,\wh]{\wA}$-bimodules  
to the corresponding morphism of wheelbimodules $\id_{\EE^{\env}_{V}} \otimes_{\gwh} \I_{\SG^{\env},0}(f)$ over $\wA$.

%%%%%%%
\begin{proposition}
\label{proposition:functor-alg-whmod-2}
Given a finite dimensional vector space $V$ with distinguished basis $\{ e_{\alpha} : \alpha \in \mathtt{A} \}$ and a wheelgebra $\wA$, the functor $\Ec_{V}^{\mod}$ introduced in \eqref{eq:functor-alg-whmod} has a left adjoint functor 
\begin{equation}     
\label{eq:root-wh-mod-adj}
\whsqrt[V,\wh]{\place} : {}_{\wA}\AMod_{\wA} \longrightarrow {}_{\whsqrt[V,\wh]{\wA}}\Mod_{\whsqrt[V,\wh]{\wA}}.     
\end{equation}
More precisely, if $\mathscr{M}$ is a wheelbimodule over $\wA$, $\whsqrt[V,\wh]{\mathscr{M}}$ is given as the quotient of the free bimodule 
\[
T_{\mathscr{M}} = T \otimes \Big(\bigoplus_{n \in \NN_{0}} \mathscr{M}(n) \otimes \End(V^{\otimes n})\Big) \otimes T
\]
over the tensor algebra $T = T(\oplus_{n \in \NN_{0}} \wA(n) \otimes \End(V^{\otimes n}))$ by the subbimodule generated by 
\begin{enumerate}[label=(WM.\arabic*)]
    \item\label{item:wm1} $(a \otimes E_{\bar{\alpha},\bar{\beta}}) (x \otimes E_{\bar{\alpha'},\bar{\beta'}}) (a' \otimes E_{\bar{\alpha},\bar{\beta}}) - \rho_{m,n,p}(a,x,a') \otimes E_{\bar{\alpha} \sqcup \bar{\alpha}' \sqcup \bar{\alpha}'',\bar{\beta} \sqcup \bar{\beta}'\sqcup \bar{\beta}''}$, for all $m, n, p \in \NN_{0}$, $a \in \wA(m)$, $a' \in \wA(p)$, $x \in \mathscr{M}(n)$, 
    $\bar{\alpha},\bar{\beta} \in \mathtt{A}^{m}$, $\bar{\alpha}',\bar{\beta}' \in \mathtt{A}^{n}$, 
    and $\bar{\alpha}'',\bar{\beta}'' \in \mathtt{A}^{p}$, 
    \item $(\sigma' \cdot x \cdot \sigma) \otimes E_{\sigma' \cdot \bar{\alpha},\bar{\beta} \cdot \sigma} 
    - x \otimes E_{\bar{\alpha},\bar{\beta}}$, for all $n \in \NN_{0}$, $x \in \mathscr{M}(n)$, 
    $\bar{\alpha},\bar{\beta} \in \mathtt{A}^{n}$, $\sigma, \sigma' \in \SG_{n}$,
    \item\label{item:wm3} ${}^{n}_{j}t_{i}(x) \otimes E_{\bar{\alpha},\bar{\beta}} - \sum_{\gamma \in \mathtt{A}} x \otimes E_{\incl_{j,\gamma}(\bar{\alpha}),\incl_{i,\gamma}(\bar{\beta})}$, for all $n \in \NN$, $x \in \mathscr{M}(n)$, 
    $\bar{\alpha},\bar{\beta} \in \mathtt{A}^{n-1}$,
\end{enumerate}
where we denote in \ref{item:wm1} the action of $T$ simply by juxtaposition and we are using the notation of Examples \ref{example:end-2} and \ref{example:end-3}, 
and the canonical morphism 
\[     F_{\mathcal{M}} : \mathscr{M} \longrightarrow \EE^{\env}_{V} \otimes_{\gwh} \I_{\SG^{\env},0}\Big(\whsqrt[V,\wh]{\mathscr{M}}\Big)      \]
of whelbimodules over $\wA$ satisfying the corresponding universal property is given by 
\[     F_{\mathcal{M}}(n)(x) = \sum_{\bar{\alpha}, \bar{\beta} \in \mathtt{A}^{n}} E_{\bar{\alpha},\bar{\beta}} \otimes \overline{\Big( x \otimes E_{\bar{\alpha},\bar{\beta}} \Big)}      \]
for all $n \in \NN_{0}$ and $x \in \mathscr{M}(n)$, where the bar denotes the class of an element of $T_{\mathscr{M}}$. 
\end{proposition}
%%%%%%%
\begin{proof}
The proof of this result is analogous to that of Proposition \ref{proposition:functor-alg-whalg-2}.     
\end{proof}

Define now the functor  
\begin{equation}
\label{eq:root-wh-mod-comm}
(\place)_{V,\wh} : {}_{\wA}\AMod_{\wA} \longrightarrow {}_{\wA_{V}}\Mod^{\s}_{\wA_{V}} = {}_{\wA_{V}}\Mod      
\end{equation}
given as the composition of \eqref{eq:root-wh-mod-adj} and \eqref{eq:abel-mod}. 

%%%%%%%
\begin{corollary}
\label{corollary:functor-alg-whalg-mod-2} 
Let $V$ be finite dimensional vector space, $A$ an algebra, and $N$ be an $A$-bimodule. Then
\begin{enumerate}
\item [(1)]
We have a natural isomorphism $\whsqrt[V,\wh]{\WW(N)} \cong \whsqrt[V]{N}$ of bimodules over $\whsqrt[V,\wh]{\Fock(A)} \cong \whsqrt[V]{A}$, where the functor $\whsqrt[V]{\place} : {}_{A}\Mod_{A} \rightarrow {}_{\whsqrt[V]{A}}\Mod_{\whsqrt[V]{A}}$ was recalled in 
 \eqref{eq:VdB-root-functor-def.a}.
\item [(2)]
We have a natural isomorphism 
 $\WW(N)_{V,\wh} \cong N_{V}$ of bimodules over $\Fock(A)_{V,\wh} \cong A_{V}$, where the functor $(\place)_{V} : {}_{A}\Mod_{A} \rightarrow {}_{A_{V}}\Mod^{\s}_{A_{V}}$ was recalled in \eqref{eq:VdB-functor-def}.
 \end{enumerate} 
\end{corollary}
%%%%%%%
\begin{proof}
The proof of this result is analogous to that of Corollary \ref{corollary:functor-alg-whalg-2}.     
\end{proof}

We conclude by presenting a result demonstrating that our definition of a wheelsymplectic wheelgebra satisfies the Kontsevich--Rosenberg principle for any wheelgebra $\Fock(B)$ associated with a smooth algebra $B$. 
%%%%%%%
\begin{theorem}
\label{theorem:w-symp-K-R} 
Let $B$ be a smooth algebra and let $V$ be finite dimensional vector space. 
If $\Fock(B)$ admits a wheelsymplectic form $\varpi \in \Omega^{2}_{\Fock(B)}(0)$, then $\Fock(B)_{V,\wh}$ is natually endowed with a symplectic structure
\end{theorem}
%%%%%%%
\begin{proof}
To simplify, let us denote $\Fock(B)$ simply by $\wC$. 
We first recall that a symplectic structure on a smooth commutative algebra $A$ is equivalently defined by an isomorphism 
\[
f : \Der(A) 
\longrightarrow 
\Omega^{1}_{A}
\]
of $A$-modules such that the induced pairing
\begin{equation}
\label{eq:symp-smooth}
\Omega^{1}_{A} \otimes 
\Omega^{1}_{A}
\longrightarrow 
A
\end{equation}
given by $X \otimes Y \mapsto f^{-1}(X)(Y)$
is antisymmetric, where we are using the canonical isomorphism $\Der(A) = \Hom_{A}(\Omega^{1}_{A},A)$ of $A$-modules for any commutative algebra $A$, and that the induced bracket 
\begin{equation}
\label{eq:symp-smooth-2}
A \otimes A
\longrightarrow 
A
\end{equation}
given as the composition of $\operatorname{d} \otimes \operatorname{d} : A \otimes A \rightarrow \Omega^{1}_{A} \otimes 
\Omega^{1}_{A}$ and \eqref{eq:symp-smooth} satisfies the Jacobi identity (see for instance \cite{MR1058984}, Section 3). 

Hence, it suffices to show that there is an isomorphism 
\[
f : \Der(\wC_{V,\wh}) 
\longrightarrow 
\Omega^{1}_{\wC_{V,\wh}}
\]
of $\wC_{V,\wh}$-modules 
satisfying the previous properties.  

Now, since $\wC$ is a symplectic wheelgebra, we have the isomorphism of $\wC$-modules 
\[
    \cano_{\varpi} : \IDer(\wC) \longrightarrow \Omega_{\wh}^{1}\wC
\]
given in \eqref{eq:wheelcont}, which, by Lemma \ref{lemma:wwderdoble-wder-iso}, is equivalent to the isomorphism 
\[
\WW\!\DDer B \longrightarrow \WW\Omega^1_{\nc}B
\]
given in \eqref{eq:wheel-symplectic-isom-def}. 
The latter is explicitly given by sending $X \in \WW\!\DDer B(n)$ to $\{\varpi,X\}^{\imath}_{0,n}$. 
We remark that $\wC_{V,\wh} = \Fock(B)_{V,\wh} \cong B_{V}$ is a smooth 
commutative algebra, since $B$ is smooth.  

By functoriality of $(-)_{V,\wh}$ we get an isomorphism 
\[
\Der(\wC_{V,\wh}) \cong (\WW\!\DDer B)_{V,\wh} 
\overset{\bar{f}}{\longrightarrow} 
(\WW\Omega^1_{\nc}B)_{V,\wh} \cong \Omega^{1}_{\wC_{V,\wh}},
\]
where the first and last isomorphisms follow from Corollary \ref{corollary:functor-alg-whalg-mod-2} and \eqref{eq:KR-double-derivations-1-forms}. 
It suffices to show that the pairing \eqref{eq:symp-smooth} induced by $\bar{f}$ is antisymmetric, which  follows from a long but straightforward computation using the fact that the wheeled big bracket is antisymmetric, and that the induced bracket on $\wC$ given by \eqref{eq:symp-smooth-2} satisfies the Jacobi identity, which follows from the Jacobi identity for the wheeled big bracket.
\end{proof}

%%%%%%%
\begin{remark}
\label{remark:KR-bisimplectic}   
Note that Theorem \ref{theorem:w-symp-K-R} is an improvement of the obvious consequence of combining Theorem \ref{theorem:bisymplectic-gives-wheel-symp} and Corollary \ref{corollary:functor-alg-whalg-mod-2} with the well-known Kontsevich--Rosenberg principle for smooth bisymplectic algebras (see \cite{MR2294224}, Thm. 6.4.3). 
In fact, the abovementioned consequence is a special case of our Theorem \ref{theorem:w-symp-K-R}, specifically when the wheelsymplectic structure on $\Fock(B)$ is induced by a bisymplectic structure  on $B$. 
\end{remark}
%%%%%%%

\newpage
%%%%%%%%%%%%%%%%%%%%%%%%%%%%%%%%%%%%%%%%%%%%%%%%%%%%%%%%%%%%%%%%%%%%%%%%%%%%%%%%%%%%%%%%%%%%%%%%%%%%%%%%%%%%%%%%%%%%%%%%%%%%%%%%%%%%%%%%%%%%%%%%%%%%
\appendix
\section{Glossary of notation}
\allowdisplaybreaks

For the convenience of the reader, this appendix lists the most commonly used symbols in the article, along with an oversimplified description and where to find them. We give the page number when no better indicator is available, such as the numbering of the environment (\textit{e.g.} proposition, lemma, theorem, equation) where it can be found.

%%%%%%%%%%%%%%%%%%%%%%%%%%%%%%%%%%%%%%%%%%%%%%%%%%%%%%%%%%%%%%%%%%%%%%%%%%%%%%%%%%%%%%%%%%%%%%%%%%%%%%%%%%%%%%%%%%%%%%%%%%%%%%%%%%%%%%%%%%%%%%%%
\begin{center}
\renewcommand{\arraystretch}{1.4}
\begin{longtable}{@{}>{\raggedright\arraybackslash}p{0.15\textwidth}%
>{\raggedright\arraybackslash}p{0.55\textwidth}%
>{\raggedright\arraybackslash}p{0.20\textwidth}@{}}
\toprule
\textbf{Symbol} & \textbf{Meaning} & \textbf{1$^{st}$ occurrence} \\
\midrule
\endfirsthead
\toprule
\textbf{Symbol} & \textbf{Meaning} & \textbf{1$^{st}$ occurrence} \\
\midrule
\endhead
\bottomrule
\endfoot

\multicolumn{3}{@{}c@{}}{\textbf{Section \ref{subsection:not-conv}}} \\  \hline
\addlinespace
$\mathbb{N}_0$ & Nonnegative integers $\{0,1,2,\dots\}$ & p. \pageref{eq:NN_0} \\
$\mathbb{N}$ & Positive integers $\{1,2,\dots\}$ & p. \pageref{eq:NN} \\
$\llbracket a, b \rrbracket$ & Set $\{ n \in \ZZ : a \leq n \leq b \}$ & p. \pageref{eq:interval-integers} \\
$X^{(\NN)}$ & Set of all finite tuples of elements of the set $X$ & p. \pageref{eq:set-tuples} \\
$\bar{n}$ & $\ell$-tuple $(n_{1},\dots,n_{\ell})$ in $\NN_{0}^{\ell}$ & p. \pageref{eq:tuple} \\
$|\bar{n}|$ & $\sum^\ell_{i=1}n_i$ & p. \pageref{eq:length-tuple} \\
$\bar{e}^{p,\ell}$ (or $\bar{e}^{p}$) & Vector: $|\bar{e}^{p,\ell}| = 1$ whose $p$-th component is the only nonzero & p. \pageref{eq:vector-Kronecker} \\
$\len(\bar{n})$ (or $\len$) & Length of a(n ordered) partition $\bar{n}$ & p. \pageref{eq:length-partition} \\
$\mathtt{P}(n)$ & Set of all partitions of $n \in \mathbb{N}$ & p. \pageref{eq:set-partitions} \\
$\mathtt{P}_\ell(n)$ & Set of all partitions of $n \in \mathbb{N}$ of length $\ell$ & p. \pageref{eq:set-partitions-given-length} \\
$\ord$ & Package map $\Part(n) \rightarrow \NN^{\llbracket 1 , n \rrbracket}$ & p. \pageref{eq:map-partition-O} \\
$\sqcup$ & Concatenation of partitions & p. \pageref{eq:cup-par} \\
$\mathtt{A}$ & An arbitrary set & p. \pageref{index:arbitrary-set} \\
$\SG_n$ & Symmetric group of $n$ elements & p. \pageref{symmetric group} \\
\addlinespace
\hline
\multicolumn{3}{@{}c@{}}{\textbf{Section \ref{subsection:categorical}}} \\
\hline
\addlinespace
$\kk$ & Field of characteristic zero & p. \pageref{eq:field-char-0} \\
$\mathcal{C}$ & Complete and cocomplete $\kk$-linear category & p. \pageref{eq:generic-category} \\
$(\C, \otimes_{\C}, \mathbf{I}_{\C})$ & Monoidal category structure on $\C$ & p. \pageref{eq:generic-monoidal-str} \\
$\tau^{\C}$ & Symmetric braiding on $\C$ & p. \pageref{eq:generic-braiding} \\
$\otimes_{\C}$ & Tensor product in $\C$ & Ex. \ref{example:tensor-alg} \\
${}_{\Bbbk}\Vect$ & Category of $\kk$-vector spaces & Ex. \ref{example:vect} \\
$(A,\mu)$ & Associative algebra in $\C$ & p. \pageref{eq:generic-algebra-def} \\
$(A,\tau^\C_{A,A})$ & Commutative (associative) algebra in $\C$ & p. \pageref{eq:generic-comm-algebra-category} \\
$(A,\eta)$ & Unitary (associative) algebra in $\C$ & p. \pageref{eq:generic-algebra-def} \\
$\Alg(\C)$ & Category of associative algebras in $\C$ & p. \pageref{eq:generic-category-of-algebras} \\
$\uAlg(\C)$ & Category of unitary algebras in $\C$ & p. \pageref{eq:generic-category-of-unitary-algebras} \\
$\CAlg(\C)$ & Category of commutative algebras in $\C$ & p. \pageref{eq:generic-category-of-commutative-algebras} \\
\addlinespace
\hline
\multicolumn{3}{@{}c@{}}{\textbf{Section \ref{subsection:res-ind}}} \\
\hline
\addlinespace
$\Res_{\varphi}$ 
&   Restriction functor  $\Mod_{A'}(\C) \rightarrow \Mod_{A}(\C)$
& \eqref{eq:res}
\\
$\Ind_{\varphi}$ 
& Induction functor $ \Mod_{A}(\C) \rightarrow \Mod_{A'}(\C)$
&  \eqref{eq:ind}
\\
$\Phi_{\varphi}(\hskip 0.6mm ,)$ 
& Adjunction $\Res_{\varphi}$ and $\Ind_{\varphi}$, $\varphi: A\to A'$
& \eqref{eq:ind-res} 
\\
$\phi^{\operatorname{I}}(\psi,\varphi)$
& Natural isomorphism $\Ind_{\psi} \circ \Ind_{\varphi} \rightarrow \Ind_{\psi \circ \varphi}$
& Fact \ref{fact:nat-res-ind}
\\
$\mathfrak{p}_{M} $ 
&   Morphism $ \Res_{\varphi}\big(\Ind_{\varphi}(M)\big) \rightarrow M $ in $\Mod_{\Bbbk H}$
& p. \pageref{eq:p-M} 
\\
$\Psi_{\varphi}(N,M)$
& Adjunction $\Res_{\varphi}$ and $\Ind_{\varphi}$, $\varphi: \Bbbk H\to \Bbbk G$
& 	\eqref{eq:res-ind} 
\\
\addlinespace
\hline
\multicolumn{3}{@{}c@{}}{\textbf{Section \ref{subsection:basic-gr}}} \\
\hline
\addlinespace
$\double_{n}$
& Special case of \eqref{eq:group-diag} for $G=\SG_{n} $ 
& p. \pageref{eq:group-diag-symm-gp}
\\
$\Delta_n$ 
& Diagonal morphism $\Bbbk \SG_{n} \to (\Bbbk \SG_{n})^{\env}$ 
&  \eqref{eq:diag-n} 
\\
$\Delta_{n,m}$ 
& Diagonal morphism $\Bbbk \SG_{n}\otimes  \kk\SG_{m} \to (\Bbbk \SG_{n}\otimes \Bbbk \SG_{m})^{\env}$ 
& p. \pageref{diagonal-morphism-n-m}
\\
$(j \ \dots \ i)$  & Cycle in $\SG_{n}$ associated to $\iota$
& p. \pageref{cycle}
\\
$\Bbbk \SG_{\bar{n}}$
&  Tensor product $\Bbbk \SG_{n_{1}} \otimes \dots \otimes \Bbbk \SG_{n_{\ell}}$ 
& p. \pageref{tesnor-product-group-algebras}
\\
$\sump_{m_{1}, \dots,m_{n}} $ 
& \begin{tabular}{@{}l@{}} Ordered sum of permutations map \\ $\SG_{m_{1}} \times \dots \times \SG_{m_{n}} \to \SG_{m}$ \end{tabular}
&   \eqref{eq:inc-1} 
\\
$\Bbbk \sump_{m_{1}, \dots,m_{n}}$
 & Morphism of $\Bbbk$-algebras $ \Bbbk\SG_{\bar{m}} \rightarrow \Bbbk\SG_{m}$
 & \eqref{eq:inc-1-alg} 
\\
$\block_{m_{1}, \dots,m_{n}}$ 
& Block permutation map $\SG_n\to\SG_{m_{1} + \dots + m_{n}}$ 
& \eqref{eq:inc-2} 
\\
$\texttt{S}$
& An arbitrary set
& p. \pageref{index:set-S}
\\
\addlinespace
\hline
\multicolumn{3}{@{}c@{}}{\textbf{Section \ref{subsec:generalities-s-mod}}} \\
\hline
\addlinespace
$\Hom_{\SG}(\hskip 0.6mm, )$ 
& Set of morphisms of $\SG$-modules 
& p. \pageref{eq: srt-morphisms-Mod-S}
\\
$\SMod$ 
& Category of $\SG$-modules 
& p. \pageref{def-categ-S-mod}
\\
$\mathbb{T}_{V}(n)$ 
& $\SG$-module $V^{\otimes n}$   
& Ex. \ref{example:ten-1} 
\\
$\mathbb{E}_{V}(n)$ 
& $\SG$-module $\Hom(V^{\otimes n}, V)$
& Ex. \ref{example:end-1} 
\\
$\Hom_{\SG^{\env}}(\hskip 0.6mm, )$
& Set of morphisms of diagonal $\SG$-modules  
& p. \pageref{eq:set-morphims-DMod-S}
\\
$\DMod$ 
& Category of diagonal $\SG$-bimodules 
& p. \pageref{def-categ-S-bimod} 
\\
$\SG_{V}^{\env}(n)$ 
& Diagonal $\SG$-bimodule $\Ind_{\Delta_{n}} (\mathbb{T}_{V}(n))$ 
& Ex. \ref{example:ten-2}
\\
$\mathbb{E}_{V}^{\env}(n)$ 
& Diagonal $\SG$-bimodule $\Hom(V^{\otimes n},V^{\otimes n})$ 
& Ex. \ref{example:end-2}
\\
$\mathbbl{i}_{A}'(n) $ 
& Morphism $\mathbb{T}_{A}(n) \to\Res_{\Delta_{n}}\big(\mathbb{E}_{A}^{\env}(n)\big)$ in $\SMod$
& Ex. \ref{example:ten-3}	
\\ 
$\mathbbl{i}_{A} (n)$ 
& Morphism $\mathbbl{i}_{A} : \SG_{A}^{\env}(n) \to \mathbb{E}_{A}^{\env}(n)$ in $\DMod$ 
& Ex. \ref{example:ten-3}	
\\
$\prod$
& Categorical product
& p. \pageref{eq:categorical-product}. 
\\
$\coprod$
& Categorical coproduct
& p. \pageref{de-coproduct}
\\
$\bigcap$
& Categorical intersection
& \eqref{eq:intersection-s-mod}
\\
$ \otimes_{\SG}$ 
& Cauchy tensor product in $\SMod$ 
&  \eqref{eq:cauchy} 
\\
$\otimes_{\SG^{\env}}$
& Cauchy tensor product in $\DMod$ 
& \eqref{eq:cauchydiagonal} 
\\
$\mathbf{1}_{\SG}$
& Unit in $\SMod$ (as a monoidal category)
& Lem. \ref{lemma:cauchy}
\\
$\mathbf{1}_{\SG^{\env}}$
& Unit in $\DMod$ (as a monoidal category)
& Lem. \ref{lemma:cauchy}
\\
$\Bi_{\SG}(\hskip 0.6mm ,\hskip 0.6mm;)$
& Set of sequences of maps $\{ f_{n,m}\}$ in $\SMod$ 
& Fact \ref{fact:morph-tensor-s-bimod}
\\
$\Bi_{\SG^{\env}}(\hskip 0.6mm ,\hskip 0.6mm;)$
& Set of sequences of maps $\{ f_{n,m}\}$ in $\DMod$ 
& Fact \ref{fact:morph-tensor-s-bimod}
\\
$\tau^{\SG}$ 
& Symmetric braiding in $\SMod$ 
&  \eqref{eq:tau-s} 
\\
$\tau^{\SG^{\env}}$ 
& Symmetric braiding in $\DMod$ 
&  \eqref{eq:tau-s-e} 
\\
$\Res$ 
& Restriction functor $ \DMod \to \SMod$ 
& \eqref{eq:diag}
\\
$\varphi_{0}$
& Structure morphism $ \mathbf{1}_{\SG} \rightarrow \Res(\mathbf{1}_{\SG^{\env}}) $
& Lem. \ref{lemma:res}
\\
$\varphi_{2}(S,T)$
& Morphism $\Res(S) \otimes_{\SG} \Res(T) \rightarrow \Res (S \otimes_{\SG^{\env}} T)  $
& Lem. \ref{lemma:res}
\\
$\I_{\SG,i}$ & Functor $ \Vect \to \SMod$ with $S(i) = V$
&  \eqref{eq:vect-s-se} 
\\
$\I_{\SG^{\env},i}$ 
& Functor $ \Vect \to \DMod$ with $S(i) = V$  &  \eqref{eq:vect-s-se} 
\\
   $ \linc_{N,N+k}$
   & Right inclusion group morphism $\SG_{N} \to \SG_{k+N}$
   & \eqref{eq:def-rinc}
   \\
   $ \rinc_{N,N+k}$
   & Left inclusion group morphism $\SG_{N} \to \SG_{k+N}$
   & \eqref{eq:def-rinc}
   \\
$\sh_{k}$
& Shift functor in $\SMod$
& p. \pageref{def-shifted-module-bimod}
\\
$\sh_{k}^{\env}$
& Shift functor in $\DMod$
& p. \pageref{def-shifted-module-bimod}
\\
$ \IHom_{\SG}(\hskip 0.6mm ,)$
& Internal $\Hom$ in $\Mod_\SG$
& \eqref{eq:int-hom-s-mod}
\\
$ \IHom_{\SG^\env}(\hskip 0.6mm ,)$
& Internal $\Hom$ in $\DMod$ 
& \eqref{eq:int-hom-s-mod}
\\
\addlinespace
\hline
\multicolumn{3}{@{}c@{}}{\textbf{Section \ref{sec:algebraic-structures-diagonal-bimodules}}} \\
\hline
\addlinespace
$\mu_{n,m}$
& Algebra multiplication in $\DMod$
& Fact \ref{fact:algebra-dmod}
\\
$\Sym_{\SG}(S)$
& Symmetric algebra of $S$ in $\SMod$
& p. \pageref{eq-symm-alg-MOD}
\\
$\Sym_{\SG^{\env}}(S)$
& Symmetric algebra of $S$ in $\DMod$
& p. \pageref{eq-symm-alg-DMOD}
\\
$\{ \hskip 0.6mm , \}_{n,m}  $ & (Poisson) bracket in $\DMod$
& Fact \ref{fact:s-bimod-poisson-algebra-ex}
\\
$\rho_{n,m,p}$
& Bimodule structure in $\DMod$
& Fact \ref{fact:bimod-dmod}
\\
\addlinespace
\hline
\multicolumn{3}{@{}c@{}}{\textbf{Section \ref{subsection:wh-basic}}} \\
\hline
\addlinespace
	$\linc_{n} $
 & Shorthand for $\linc_{n,n+1}\colon  \SG_{n} \to \SG_{n+1}$ 
 & p. \pageref{eq:e-n}
	\\
	$\mathbbl{l}_{i}^{n}$ & Map $\SG_{n} \to \SG_{n-1}$ &  \eqref{eq:r-n-def}
		\\
	$	\mathbbl{r}_{i}^{n}$ & Map $  \SG_{n} \to \SG_{n-1}$ & \eqref{eq:r-n-def}
	\\
	${}^{\bar{n}}_{j}t_{i}$ & Contractions on generalized wheelspaces 
    & Def. \ref{definition:gwh}
     \\
  $\GWMod$ 
  & Category of generalized wheelspaces 
  & p. \pageref{eq:def-gen-wheelsp-notation}
     \\
$\WMod$ 
& Category of wheelspaces 
& Def. \ref{definition:wh}
  \\
   ${}_{j}^{n}t_{i}$ 
   & Contractions on wheelspaces  
   & Rk. \ref{rem:def-wheelspaces-W1-W2}
\\
    $(S, {}_{\bullet}^{\bullet}t_{\bullet})$
    & An arbitrary wheelspace 
    &p. \pageref{index:wheelsp-notation-with-bulles}
\\
$\I_{\SG^{\env},0}(V)$
& Wheelspace $(\I_{\SG^{\env},0}(V), 0)$
& Ex. \ref{example:wh0}
\\
$\mathbb{E}_{V}^{\env}$ 
& Wheelspace $(\mathbb{E}_{V}^{\env}, {}_{j}^{n}t_{i})$
&  Ex. \ref{example:end-3}
\\
$\incl_{j,\beta}(\bar{\alpha})$ & $(n+1)$-tuple $(\alpha_{1},\dots,\alpha_{j-1},\beta,\alpha_{j},\dots,\alpha_{n})$
& Ex. \ref{example:end-3}
\\
$\gFo $
& Forgetful functor $\GWMod \rightarrow \DMod$
& \eqref{eq:forgetful-gwh-dmod}
\\
$\langle T \rangle$
&  Generalized subwheelspace generated by $T$
& p. \pageref{index:subwheelsp-generated-by-T}
\\
\addlinespace
\hline
\multicolumn{3}{@{}c@{}}{\textbf{Section \ref{subsection:wh-basic-2}}} \\
\hline
\addlinespace
 ${}^{n}\mathcalboondox{t}$ & Partial contractions $ \Res_{\Bbbk \linc_{n-1}^{\env}} \big(S(n)\big) \rightarrow S(n-1)$ 
 & Def. \ref{def:partial-wheelspace}
 \\
$\pWMod$ & Category of partial wheelspaces & Def. \ref{def:partial-wheelspace}
\\
\addlinespace
\hline
\multicolumn{3}{@{}c@{}}{\textbf{Section \ref{subsection:gwh-mono}}} \\
\hline
\addlinespace
$\otimes_{\gwh}$
& Tensor product on $\GWMod$ 
&  Lem/Def. \ref{lem-def:gwh-tensor-product} 
\\
$  {}^{\bar{N}}_{j}T_{i}$ & Contractions on $S\otimes_{\gwh}S'$ 
& \eqref{eq:T-gwh}
\\
$\Fo$ 
& Forgetful functor $\WMod \rightarrow \DMod$ 
& \eqref{eq:forgetful-wh-dmod}
\\
$\Bil_{\gwh}(\hskip 0.4mm, \hskip 0.4mm;)$ &  Set formed by the families of linear maps $ f_{\bar{n},\bar{m}} $ & Prop. \ref{proposition:morph-tensor-gwh}
\\
\addlinespace
\hline
\multicolumn{3}{@{}c@{}}{\textbf{Section \ref{subsection:contradiction}}} \\
\hline
\addlinespace
$\Set$
& Category of sets
& p. \pageref{index:categ-sets}
\\
 $\Bil_{\wh}( S, S';\hskip 0.4mm)$ 
 & Functor $\Bil_{\gwh}(S,S';\hskip 0.4mm)|_{\WMod} : \WMod \rightarrow \Set$
 & \eqref{index:functor-restriction-gen-wheelsp-wheelsp}
\\
$\tau_{\leq N}$
& Truncation functor $ \GWMod \rightarrow  \GWMod    $
& \eqref{eq:truncation-functor-gwh}
\\
$\mathcalboondox{i} $
&  Inclusion functor $\D \rightarrow \C$
& p. \pageref{inclusion-index}
\\
$\mathcalboondox{p}(X) $
&  Representing object of functor $\Hom_{\C}\big(X , \mathcalboondox{i}(-)\big)$ 
& p. \pageref{reflector-index}
\\
\addlinespace
\hline
\multicolumn{3}{@{}c@{}}{\textbf{Section \ref{subsection:internal-homos}}} \\
\hline
\addlinespace
 $\sh_{k}^{\wh}$
 & Wheeled $k$-shifted functor $\WMod\to \WMod$
 & p. \pageref{eq:shifted-w-mod}
 \\
$\IHom_{\WMod}(\hskip 0.6mm, )$
& Wheelspace of internal homomorphisms 
& p. \pageref{eq:int-hom-w-mod}
\\
${}^{n}_{j}T_{i}$
& Contractions on $\IHom_{\WMod}(\hskip 0.4mm, \hskip 0.4mm)$
& p. \pageref{eq: contractions-internal-wheelsp-hom}
\\
\addlinespace
\hline
\multicolumn{3}{@{}c@{}}{\textbf{Section \ref{sec:alg-str}}} \\
\hline
\addlinespace
 $\WAlg$ 
 & Category of wheelgebras 
 & p. \pageref{eq:categ-wheelalg-def}
 \\
 ${}_{\wA}\GWhMod_{\wA}$
 & Category of generalized wheelbimodules
 & p.\pageref{index:categ-generalize-wheelbimod-def}
\\
${}_{\wA}\GWhMod^{\s}_{\wA}$
& Category of symmetric gen. wheelbimodules
& p. \pageref{index:categ-sym-generalize-wheelbimod-def}
\\
${}_{\wA}\WhMod_{\wA} $
& Category of $\wA$-wheelbimodules
& p. \pageref{index:categ-wheelbimod-def}
\\
${}_{\wA}\WhMod_{\wA}^{\s}$
& Category of symmetric $\wA$-wheelbimodules
& p. \pageref{index:categ-wheelbimod-sym-def}
\\
${}_{\wA}\WhMod $
& Category of $\wA$-wheelmodules
& p. \pageref{index:categ-wheelmod-def}
\\
${}_{\wA}\GWhMod$
&  Category of generalized $\wA$-wheelmodules
& p. \pageref{index:categ--generalize-wheelmod-def} 
\\
$\mu_{n,m}$
& Product on $\IHom_{\WMod}(S,S)$
& Ex. \ref{example:end-alg-w-mopd}
\\
$\langle T \rangle$
& Gen. $\wA$-subwheelbimodule generated by $T$
& p. \pageref{index:gen-wheelbimodule-generated-by}
\\
$\otimes_\wC$
& Tensor product in ${}_{\wC}\GWhMod$
& Fact \ref{fact:symm-mon-wheelmod}
\\
$\wad_{c}$
& Adjoint action on $\IHom_{\WMod}(\wM,\wM')$
& \eqref{eq:wad}
\\
$\IHom_{\wC}(\wM,\wM')$
& $\Bbbk \SG_{n}^{\env}$-subspace of $\IHom_{\WMod}(\wM,\wM')$
& \eqref{eq:hom-w-mod-c}
\\
   $\alg$ & Functor $ \CAdm \to \Alg $ 
   & \eqref{eq:wheel-alg}
   \\
   ${}_{\wA}\WhMod$ 
   & Category of $\wA$-wheelmodules 
   & p. \pageref{index:categ-wheelbimod-def}
   \\
${}_{\wA}\WhMod_{\wA}^{\s} $
   & Category of symmetric $\wA$-wheel\-bi\-mod\-ules
   & p. \pageref{index:categ-wheelbimod-sym-def} 
\\
   $\bmod$ 
   & Functor ${}_{\wA}\AMod_{\wA} \to {}_{\alg(\wA)}\Mod_{\alg(\wA)} $ 
   & \eqref{eq:wheel-mod} 
\\
\addlinespace
\hline
\multicolumn{3}{@{}c@{}}{\textbf{Section \ref{subsec:wheeldif-forms}}} \\
\hline
\addlinespace
$\IDer(\wC, \wN)$
& Wheelspace of wheelderivations of $\wC$ in $\wN$
& Ex. \ref{example:wheelder}
\\
$\Der(\wC, \wN)$
& Vector space of derivations of $\wC$ in $\wN$
& Rk. \ref{remark:der-wheelder}
\\
$\WDiff(\hskip 0.4mm, \hskip 0.4mm)$
& Wheelspace of wheeldifferential operators
& Ex. \ref{example:wheeldiff}
\\
$  \IDer(\wC, -)$
& Functor ${}_{\wC}\AMod \rightarrow {}_{\wC}\AMod$
& p. \pageref{eq:wheel-der-func}
\\
$\Omega^{1}_{\wh}\wC$
& Representing $\wC$-wheelmodule of $ \IDer(\wC, -)$
& \eqref{eq:wheel-der-rep}
\\
$\derdifw_{\wC}$
& Universal derivation of $\Omega^{1}_{\wh}\wC$
& p. \pageref{index:univ-derivation-symp-wheel}
\\
$\wcatA$
& Category of $\NN_{0}$-graded $\wC$-wheelgebras, $\wcatA^0=\wC$
& p. \pageref{index:categ-N-0-grad}
\\
$\Omega_{\wC}^{\bullet}, \Omega_{\wC}$
 & Symmetric algebra $\Sym_{\wcatA}(\Omega^{1}_{\wh}\wC)$
 & p. \pageref{index:wheel-diff-forms-comm}
 \\
 $\iota_{\Omega^{1}_{\wh}\wC}$
 & Canonical morphism $\Omega^{1}_{\wh}\wC \rightarrow \Omega_{\wC}^{\bullet}$
 & p. \pageref{index:wheel-contraction-Ex2.3}
 \\
  $\derdifwbul_{\wC}$ 
& Differential of degree 1 on $\Omega_{\wC}^{\bullet}$
  & p. \pageref{index:diffferential-full-diff-forms}
  \\
$\IDer^{-1}(\Omega_{\wC})$
& Subwheelspace of derivations of degree $-1$
& p. \pageref{index:subwheelsp-degree-1}
\\
    $\ext$
    & Map $ \IDer(\wC) \rightarrow \IDer(\Omega_{\wC}^{\bullet})$
    & \eqref{eq:der-ext}
\\
$\iotaw_{X}$
& Contraction associated with $X \in \Der(\wC)$
& p. \pageref{index:contraction-assoc-wheelderivat}
  \\
  $\Liew_{X}$
  & Lie derivative over $X \in \Der(\wC)$
   & p. \pageref{index:Lie-derivative-wheel}
\\
 $\varpi$
 & Wheelsymplectic form
 & Def. \ref{def:symplectic-wheelgebra-proto} 
 \\
 $\cano_{\varpi}$ 
&
Associated morphism $\IDer(\wC) \longrightarrow \Omega_{\wh}^{1}\wC$
&
Def. \ref{def:symplectic-wheelgebra-proto}
 \\
 \addlinespace
\hline
\multicolumn{3}{@{}c@{}}{\textbf{Section \ref{sec:a-symmetric-algebra-diagonal-bimodules}}} \\
\hline
\addlinespace
 $\mathcal{F}(V,W)$ 
 & Symmetric algebra $\Sym_{\SG^\env}\big(\I_{\SG^{\env},1}(V) \oplus \I_{\SG^{\env},0}(W)\big)$ 
 & p. \pageref{eq:def-symm-F-caligrafica} 
 \\
$ \mu_{n,m} $ 
& Products on $ \mathcal{F}(V,W)$
& \eqref{eq:prod-fock-s-e}
\\
$i$
& Unit of $\mathcal{F}(V,W)$
& p. \pageref{eq:unit-F-calig}
\\
 \addlinespace
\hline
\multicolumn{3}{@{}c@{}}{\textbf{Section \ref{sec:Fock-algebra-associated-algebra}}} \\
\hline
\addlinespace
$\Fock(A)$ 
& Fock wheelgebra 
& p. \pageref{index:Fock-wheelgebra} 
\\
     $\Gen(A)$ 
     & Diagonal $\SG$-bimodule $ \I_{\SG^{\env},1}(A) \oplus \I_{\SG^{\env},0}(A_{\cyc})$ 
     &     \eqref{eq:gen-s}
     \\
     ${}^{n}\mathcalboondox{t}$
     & Partial contractions on $\Fock(A)$
     & \eqref{eq:tij-fock}
     \\
     $a\mid b$ 
     & Shorthand for $a\otimes b$ 
     & Not. \ref{notation:elements}  
\\
 $\overline{a_{1} | \dots |a_{n}}$ & $ (a_{1} | \dots |a_{n}) \otimes_{\Bbbk \SG_{n}} (\id_{\llbracket 1 , n \rrbracket}|\id_{\llbracket 1 , n \rrbracket}) \otimes 1_{\Sym(A_{\cyc})}$  
 & Lem. \ref{lemma:contractions-fock}
  \\
  $\Sym_{n}(A_{\cyc})$
  & Component of degree $n$ of $\Sym(A_{\cyc})$ 
  & p. \pageref{index:component-of-degree-n-of-Sym}
  \\
$\Fock(A)_n(m)$
  &  Component of degree $n$ and weight $m$ of $\Fock(A)$ 
  & p. \pageref{eq:f-A-n}
  \\
$\El_{n}$ 
& Vector subspace of $\Fock(A)$ spanned by $\overline{a_{1} | \dots |a_{n}}$ 
& Lem. \ref{lemma:contractions-fock}
     \\
      $\Fock(-)$ 
      & Fock functor $\Alg \to \CAdm  $ 
      & Lem. \ref{lem:functor-fock}
      \\  
$\IA$
& Inclusion functor $ \CAdm \rightarrow \Adm$
& Thm. \ref{theorem:fock-adjoint} 
      \\
      $\WW$
& Functor ${}_B\Mod_B\rightarrow {}_{\Fock(B)}\AMod$
& \eqref{index:functor-WW}
\\
    $\mathcalboondox{CA}$ 
    & \begin{tabular}{@{}l@{}} Category of commutative 
    $\Fock(B)$-algebras \end{tabular} 
    & Rk. \ref{rem:symmetric-Fock-alg}
    \\
     \addlinespace
\hline
\multicolumn{3}{@{}c@{}}{\textbf{Section \ref{subsec:bisymplectic-algebras}}} \\
\hline
\addlinespace
$\diff A$ 
&  Noncommutative K\"ahler differential 1-forms  
& \eqref{eq:1-forms-Kahler-nc}
\\
$\du{}$ & Canonical derivation $A\to   \diff A$ 
& p. \pageref{index:canonical-deriv} 
\\
$i_\theta$ & Contraction operator $\diff\to M$
& \eqref{eq:prop-universal-CQ}
\\
$\diffb A$ 
& Algebra of noncommutative differential forms 
& p. \pageref{index:algebra-nc-diff-forms} 
\\
$|\phantom{\varpi}|$
& Degree of a form in $\diffb A$ 
& p. \pageref{index:form-degree}
\\
$\DR^\bullet A$ 
& Karoubi--de Rham complex 
& \eqref{eq:Karoubi-de-Rham-def}
\\
$\du_{\DR}$ 
& Differential on $\DR^\bullet A$ 
& p. \pageref{index:differential-DR}
\\
$(A\otimes A)_{\out}$ 
& Outer $A$-bimodule structure on $A\otimes A$ 
& p. \pageref{index:outer-str}
\\
$(-)^\vee$ 
& Dual functor 
& p. \pageref{index:dual}
\\
$\DDer A$ 
& Bimodule of double derivations
& p. \pageref{index: double-der}
\\
$i_\Theta$
& Contraction with respect to $\Theta\in\DDer A$
& \eqref{eq:contraction-extended-double-deriv-def}
\\
$(-)^\circ$
& Permutation operator $(\diffb A)^{\otimes 2}\to \diffb A$
& p. \pageref{index:circulito}
\\
$\iota_\Theta$ 
& Reduced contraction operator ${}^\circ(i_\Theta)$
 & p. \pageref{eq:reduced-contraction-double-der-def}
 \\
$\omega$
& Bisymplectic form
& Def. \ref{def:bisymplectic-algebras}
 \\
 $\lambda$ 
 & (Noncommutative) Liouville $1$-form 
 & Ex. \ref{thm:Cbeg-Thm-5.1.1}
 \\
 $H_a$ 
 & Hamiltonian double derivation for $a\in A$
 & \eqref{eq:Hamiltonian-double-deriv-def}
\\
     \addlinespace
\hline
\multicolumn{3}{@{}c@{}}{\textbf{Section \ref{sec:double-Poisson-and-wheelgebras}}} \\
\hline
\addlinespace
$\lr{\hskip 0.4mm,}_{A} $ 
& Double (Poisson) bracket on $A$ 
& Def. \ref{definition:po}
\\
 $\AD(a)$ 
 & Map $A \to A \otimes A$, $b\mapsto \lr{a,b}_A$
 & Def. \ref{definition:po}
 \\
 $C_n $ 
 & Subgroup of $\SG_n$ of cyclic permutations &  Def. \ref{definition:po}
\\
$\lr{\hskip 0.4mm, \hskip 0.4mm,}_{A,L}$
& Left extension $A^{\otimes 3} \rightarrow A^{\otimes 3}$ of $\lr{\hskip 0.4mm,}_{A} $ 
& Def. \ref{definition:po}
  \\
  $\Jac_{\lr{\hskip 0.4mm,}_A}(\hskip 0.4mm, \hskip 0.4mm,)$ 
  & Double Jacobiator of $\lr{\hskip 0.4mm,}$ 
  & p. \pageref{index:double-Jacobiator}
  \\
  $\lr{\hskip 0.4mm,}_\omega$ 
  & Double bracket induced by a $2$-form $\omega$ 
  &  \eqref{eq:corchete-def-por-bisympl.1}
\\
 $\{\hskip 0.4mm,\}_{A} $ 
 & Associated bracket $A\otimes A\to A$
 & p. \pageref{associated-to-double-bracket}  
 \\
$\pi$
& Canonical projection $\colon A\to A_{\cyc}$
& p. \pageref{index:canonical-projection}
\\
     \addlinespace
\hline
\multicolumn{3}{@{}c@{}}{\textbf{Section \ref{sec: Cartan calculus for Fock algebras}}} \\
\hline
\addlinespace
$\WW\!\DDer B$
& $\Fock(B)$-wheelmodule $\Fock\big(T_B\DDer B\big)_1$ 
& p. \pageref{index:WW-DDer-Omega}
\\
$\WW \diff B$
& $\Fock(B)$-wheelmodule $\WW \diff B =\Fock\big(T_B\diff B\big)_1$ 
& p. \pageref{index:WW-DDer-Omega}
\\
     \addlinespace
\hline
\multicolumn{3}{@{}c@{}}{\textbf{Section \ref{sec:bisymplectic-and-wheelgebras}}} \\
\hline
\addlinespace
$\mathbb{E}$ 
&
$\DDer B\oplus\diff B$
&
p. \pageref{index:double-der-one-form}
\\
$\lr{\hskip 0.6mm, }^{\imath}$
& Double big bracket $\big(T_B\mathbb{E}\big)^{\otimes 2}\rightarrow \big(T_B\mathbb{E}\big)^{\otimes 2}$
& \eqref{eq:double-big-bracket}
\\
$\{\hskip 0.6mm, \}^{\imath}_{n,m}$
& Wheeled big bracket 
& \eqref{eq:wheeled-big-bracket-bis}
 \\
      \addlinespace
\hline
\multicolumn{3}{@{}c@{}}{\textbf{Section \ref{sec:KR-after-BKR}}} \\
\hline
\addlinespace
$\CAlg$
& Category of commutative algebras
& p. \pageref{index:categ-comm-alg}
\\
$\Rep_V A$
& Functor $\CAlg\rightarrow \Set$
& \eqref{eq:functor-representation}
\\
$\Alg$
& Category of associative algebras
& p. \pageref{index:categ-ass-alg}
\\
$\widetilde{\Rep}_V A$
& Functor $\Alg\rightarrow \Set$
& \eqref{eq:functor-representation-ext}
\\
$\inc$
& Inclusion functor $ \CAlg \hookrightarrow \Alg$
& p. \pageref{index:inclusion-functor}
\\
$\sqrt[V]{-}$
& Matrix reduction functor $ \Alg \rightarrow \Alg$
& \eqref{eq:root-VdB-def-BKR.a}
\\
$(-)_V$
& Representation functor $ \Alg \rightarrow \CAlg$
& \eqref{eq:root-VdB-def-BKR.b}
\\
$(\place)_{\ab}   $
& Abelianization functor $ \Alg \rightarrow \CAlg$
& \eqref{eq:abel}
\\
$\Rep(A,V)$
& Representation affine scheme $\Spec(A_V)$
& p. \pageref{index:representation-scheme}
\\
$(\place)_{\ab}^{\mod} $
& Abelianization functor ${}_{B}\Mod_{B} \rightarrow {}_{B_{\ab}}\Mod^{\s}_{B_{\ab}}$
& \eqref{eq:abel-mod}
\\
$\sqrt[V] {\place}$
& Functor  ${}_A\Mod_{A}\rightarrow {}_{\sqrt[V]{A}}\Mod_{\sqrt[V]{A}}$
& \eqref{eq:VdB-root-functor-def.a}
\\
$(-)_V$ 
&  Van den Bergh (VdB) functor ${}_A\Mod_{A} \rightarrow\Mod_{A_V}$
& \eqref{eq:VdB-functor-def}
\\
      \addlinespace
\hline
\multicolumn{3}{@{}c@{}}{\textbf{Section \ref{sec:wheeled-KR}}} \\
\hline
\addlinespace
$\Ec_{V}$
& Functor $ \Alg \rightarrow \Adm   $
& \eqref{eq:functor-alg-whalg}
\\
 $\whsqrt[V,\wh]{\place} $
 & Wheeled matrix reduction functor $ \Adm \rightarrow \Alg$  
 & Prop. \ref{proposition:functor-alg-whalg}
 \\
$T$ 
& Tensor algebra $T\Big(\bigoplus_{n \in \NN_{0}} \wA(n) \otimes \End(V^{\otimes n})\Big)$
& Prop. \ref{proposition:functor-alg-whalg-2}
\\
$F$
& Canonical map $\wA \rightarrow \EE^{\env}_{V} \otimes_{\gwh} \I_{\SG^{\env},0}\Big(\whsqrt[V,\wh]{\wA}\Big)$
& \eqref{eq:morp-univ}
\\
$ (\place)_{V,\wh} $
& Wheeled representation functor $ \Adm \rightarrow \CAlg$
& Def.  \ref{definition:functor-alg-whalg-2}
\\
$(\Ec_{V}^{\mod}(N),\rho)$
& Wheelbimodule $ \EE^{\env}_{V} \otimes_{\gwh} \I_{\SG^{\env},0}(N)$, $N\in {}_B\Mod_B$
& \eqref{eq:Ec-bimodules-wheel}
\\
$\Ec_{V}^{\mod}$
& Functor $  {}_{\whsqrt[V,\wh]{\wA}}\Mod_{\whsqrt[V,\wh]{\wA}} \rightarrow {}_{\wA}\AMod_{\wA} $
& \eqref{eq:functor-alg-whmod}
\\
$\whsqrt[V,\wh]{\place}$
& Functor $ {}_{\wA}\AMod_{\wA} \rightarrow {}_{\whsqrt[V,\wh]{\wA}}\Mod_{\whsqrt[V,\wh]{\wA}}$
& \eqref{eq:root-wh-mod-adj}
\\
$T_{\mathscr{M}}$
& Free bimod. $T \otimes \big(\oplus_{n \in \NN_{0}} \mathscr{M}(n) \otimes \End(V^{\otimes n})\big) \otimes T$
& Prop. \ref{proposition:functor-alg-whmod-2}
\\
$F_{\mathcal{M}} $
& Canonical map $ \mathscr{M} \rightarrow \EE^{\env}_{V} \otimes_{\gwh} \I_{\SG^{\env},0}\Big(\whsqrt[V,\wh]{\mathscr{M}}\Big)  $
& Prop. \ref{proposition:functor-alg-whmod-2}
\\
$(\place)_{V,\wh}$
&  Wheeled VdB functor ${}_{\wA}\AMod_{\wA} \rightarrow   {}_{\wA_{V}}\Mod $
& \eqref{eq:root-wh-mod-comm}
\medskip
\end{longtable}
\end{center}

%%%%%%%%%%%%%%%%%%%%%%%%%%%%%%%%%%%%%%%%%%%%%%%%%%%%%%%%%%%%%%%%%%%%%%%%%%%%%%%%%%%%%%%%%%%%%%%%%%%%%%%%%%%%%%%%%%%%%%%%%%%%%%%%%%%%%%%%%%%%%%%%%%%%
\bibliographystyle{model1-num-names}
\addcontentsline{toc}{section}{References}
%%%%%%%%%%%%%%%%%%%%%%%%%%%%%%%%%%%%%%%%%%%%%%%%%%%%%%%%%%%%%%%%%%%%%%%%%%%%%%%%%%%%%%%%%%%%%%%%%%%%%%%%%%%%%%%%%%%%%%%%%%%%%%%%%%%%%%%%%%%%%%%%%%%%

\begin{small}
\vskip 0.5cm
\noindent \scshape{David Fern\'andez: ETSI Caminos, Canales y Puertos. Universidad Politécnica de Madrid. 
Calle del Profesor Aranguren, 3,
28040, Madrid, Spain.}
\\
\normalfont\textit{E-mail address}: \href{mailto:david.fernandezalv@upm.es}{\texttt{david.fernandezalv@upm.es}}.

\medskip 

\noindent \scshape{Estanislao Herscovich: Institut Fourier, Universit\'e Grenoble Alpes, 38610 Gi\`eres, France.}
\\
\normalfont\textit{E-mail address}: \href{mailto:Estanislao.Herscovich@univ-grenoble-alpes.fr}{\texttt{Estanislao.Herscovich@univ-grenoble-alpes.fr}}. 
\end{small}
\end{document}